\documentclass[12pt]{amsart}

\usepackage{geometry}
\usepackage{amsmath,amssymb}
\usepackage[hidelinks]{hyperref}
\usepackage[capitalise,noabbrev,nameinlink]{cleveref}
\usepackage{bm,bbm}
\usepackage{xcolor}
\usepackage{mathtools}
\usepackage[noend]{algpseudocode}
\usepackage{algorithm}
\usepackage{subcaption}
\usepackage{graphicx}
\usepackage{enumitem}
\usepackage{tikz}
\usepackage{booktabs}

\newcommand{\OT}{\mathsf{OT}}

\newcommand{\RR}{\mathbb R}

\newcommand{\rA}{\mathrm{A}}
\newcommand{\rB}{\mathrm{B}}
\newcommand{\rC}{\mathrm{C}}

\newcommand{\rM}{\mathrm{M}}
\newcommand{\rN}{\mathrm{N}}
\newcommand{\rP}{\mathrm{P}}
\newcommand{\rQ}{\mathrm{Q}}
\newcommand{\rT}{\mathrm{T}}
\newcommand{\rU}{\mathrm{U}}

\newcommand{\cK}{\mathcal{K}}

\newcommand{\sH}{\mathsf{H}}
\renewcommand{\vec}{\mathrm{vec}}

\DeclareMathOperator{\supp}{spt}

\DeclareMathOperator{\Id}{Id}

\DeclareMathOperator{\argmin}{argmin}
\DeclareMathOperator{\argmax}{argmax}

\newtheorem{proposition}{Proposition}
\newtheorem{theorem}{Theorem}
\newtheorem{remark}{Remark}
\newtheorem{lemma}{Lemma}
\newtheorem{corollary}{Corollary}
\newtheorem{assumption}{Assumption}

\usepackage{bm}
\usepackage{enumitem}
\usepackage{tikz}
 \title{Discrete Gromov--Wasserstein Duality: Algorithms and Isomorphism Testing}

\author[G. Rioux]{Gabriel Rioux}
\address[G. Rioux]{
Department of Mathematics, Imperial College London}
\email{g.rioux@imperial.ac.uk}
\author[J. Marks]{Joanna Marks}
\address[J. Marks]{
Department of Mathematics, Imperial College London}
\email{joanna.marks23@imperial.ac.uk}
\author[R. Passeggeri]{Riccardo Passeggeri}
\address[R. Passeggeri]{
Department of Mathematics, Imperial College London}
\email{riccardo.passeggeri@imperial.ac.uk}
\author[Z. Goldfeld]{Ziv Goldfeld}
\address[Z. Goldfeld]{
School of Electrical and Computer Engineering, Cornell University
}
\email{goldfeld@cornell.edu}
\thanks{
Z. Goldfeld is partially supported by NSF grants CCF-1947801,  CCF-2046018, and DMS-2210368, and the 2020 IBM Academic Award. J. Marks is supported by EPSRC through the StatML CDT programme, grant no. EP/Y034813/1. R. Passeggeri is partially supported by EPSRC grant UKRI2396.
}

\begin{document}
\begin{abstract}
The Gromov--Wasserstein (GW) distance provides a principled framework for aligning metric measure (mm) spaces based solely on their intrinsic structure. Its ability to identify isomorphic representations of distributions across spaces renders it valuable for comparing data where equality up to isomorphism occurs naturally such as in graphs or, more generally, distributions on graphs. Recently, a type of dual form for the GW distance between Euclidean distributions with the squared Euclidean or inner product costs was derived, spurring the development of new statistical and algorithmic results for this setting. This work furnishes a novel duality result for GW distances with and without entropic regularization that is applicable to all finitely supported mm spaces. Leveraging this result, we derive the sample complexity of empirical GW distances between finite mm spaces, as well as limit distributions under proper centering and scaling.   Furthermore, we propose new algorithms for solving the regularized GW problem which are subject to formal convergence guarantees. These statistical and algorithmic advancements give rise to a principled and efficient framework for testing whether two distributions on the set of graphs with a fixed number of nodes are isomorphic based on samples.
\end{abstract}
\maketitle

\section{Introduction}

Gromov--Wasserstein (GW) distances serve as an optimal transport (OT) based framework for aligning distributions supported on distinct spaces.  The $(p,q)$-GW distance between two metric measure (mm) spaces $(\mathcal X_0,\mathsf d_0,\mu_0)$ and $(\mathcal X_1,\mathsf d_1,\mu_1)$ is given by \cite{memoli2011gromov} 
\begin{equation}
\label{eq:GWIntro}
\mathsf{GW}_{p,q}(\mu_0,\mu_1)\coloneqq \inf_{\pi\in\Pi(\mu_0,\mu_1)}\left(\int|\mathsf d_0^q(x,x')-\mathsf d_1^q(y,y')|^pd\pi\otimes \pi(x,y,x',y')\right)^{\frac 1 p},
\end{equation}
where $\Pi(\mu_0,\mu_1)$ is the set of all couplings of $\mu_0$ and $\mu_1$. Thus, $\mathsf{GW}_{p,q}$ can be interpreted as an $L^p$ relaxation of the Gromov--Hausdorff distance and, notably, serves as a metric on the space of all compact mm spaces modulo the equivalence class of isomorphic mm spaces.\footnote{\label{foot:isomorphic}$(\mathcal X_0,\mathsf d_0,\mu_0)$ and  $(\mathcal X_1,\mathsf d_1,\mu_1)$ are isomorphic if $\mu_0\circ T^{-1}=\mu_1$ for some isometry $T:\supp(\mu_0)\to \supp(\mu_1)$.}  %
These properties make GW a natural tool for comparing heterogeneous datasets and identifying isomorphic representations of objects, spurring its use in diverse applications including single-cell genomics \cite{cao2022manifold,demetci2022scot}, cross-domain learning \cite{alvarez-melis-jaakkola-2018-gromov,yuan2025optimal}, and graph comparison \cite{chen2020graph,xu2019gromov} among many others.

An important open problem in the study of GW distances concerns the existence of a dual form for \eqref{eq:GWIntro}. The primary obstacle is that \eqref{eq:GWIntro}  is generally a nonconvex quadratic program so that standard convex duality arguments do not apply. By contrast, the OT problem is a linear program and hence  benefits from a well-established duality theory which serves as a central tool for analysis. To date, a type of duality theory for GW problems has been provided in only a limited number of cases. The first such result was provided in \cite{zhang2024gromov} for $\mathsf{GW}_{2,2}$ between  distributions $\mu_0\in\mathcal P(\mathbb R^{d_0}),\mu_1\in\mathcal P(\mathbb R^{d_1})$ on Euclidean spaces where the metric is the Euclidean distance; a similar representation was provided in \cite{rioux2024entropic} for inner products in place of squared Euclidean distances. This result connects $\mathsf{GW}_{2,2}$ to a class of parameterized OT problems and has  served as the catalyst for recent developments in the computational \cite{rioux2024entropic}, statistical \cite{groppe2024lower,rioux2024limit,zhang2024gromov}, and geometric \cite{zhang2026gradient} study of Euclidean GW distances.
\cite{karumanchi2025approximation} derived a similar formulation for finite-dimensional concave quadratic programs, see the literature review for details.

This work develops a novel dual representation which holds for any finitely discrete mm spaces. This result enables us to propose new algorithms for solving the regularized GW problem which admit nonasymptotic convergence guarantees and to establish a statistical theory for empirical GW between discrete distributions encompassing parametric finite-sample rates, limit laws, and a framework  for testing for equality of discrete mm spaces up to isomorphism based on samples. This hypothesis testing framework is applied to the problem of discerning if two distributions on the set of all graphs with a fixed number of nodes are isomorphic.

\subsection{Literature review}
A dual formulation for $\mathsf{GW}_{2,2}$ and its entropic variant was first provided in \cite{zhang2024gromov} for distributions on Euclidean spaces with finite fourth moments when the metric is the Euclidean distance. If $\mu_0\in\mathcal P(\mathbb R^{d_0}),\mu_1\in\mathcal P(\mathbb R^{d_1})$ are, without loss of generality, centered   
$      \mathsf {GW}_{2,2}^2(\mu_0,\mu_1)= \mathsf S(\mu_0,\mu_1)+\inf_{\mathrm U\in\mathbb R^{d_0\times d_1}}\left\{32\|\mathrm U\|_{\mathrm F}^2+\mathsf{OT}_{\mathrm U}(\mu_0,\mu_1) \right\},$
where $\|\cdot\|_{\mathrm{F}}$ is the Frobenius norm, $\mathsf S(\mu_0,\mu_1)$ is a constant depending only on the moments of the marginals, and $\mathsf{OT}_{\mathrm U}(\mu_0,\mu_1)$ is an OT problem with cost function $c_{\mathrm U}:(x,y)\in\mathbb R^{d_0}\times \mathbb R^{d_1}\mapsto -4\|x\|^2\|y\|^2-32x^{\intercal}\mathrm U y$. Leveraging this representation, that work derived sharp sample complexity results for empirical GW between compactly supported distributions.  \cite{groppe2024lower} further demonstrated that the derived rates adapt to a notion of intrinsic dimensionality. Namely, these works illustrate that empirical GW suffers from the same curse of dimensionality inherent to empirical OT in that $\mathbb E[\mathsf{GW}_{2,2}(\hat{\mu}_{n},\mu)]\lesssim n^{-\frac{2}{\max\{4,d\}}}$ up to a log factor when $d=4$, where $\hat \mu_n$ is the empirical measure from $n$ independent and identically distributed (i.i.d.) samples from a compactly supported measure $\mu$, see \cite{kato2025convergence} for an extension to unbounded marginals. With entropic regularization, the rate is known to be parametric under certain tail conditions, see \cite{zhang2024gromov}.
Later,  \cite{rioux2024limit} built on these sample complexity results by deriving limit laws for the empirical GW distance between Euclidean distributions under different assumptions on the populations,  again using this variational form as a central tool for analysis. The same work proposed a framework for testing if two distributions on graphs are isomorphic under the assumption that the graphs have independently sampled edges.    

On the computational side, \cite{rioux2024entropic} proposed algorithms for computing the entropically regularized $\mathsf{GW}_{2,2}$ distance with the Euclidean metric based on optimizing a regularized variational form. The corresponding objective was shown to be smooth with $L_{\varepsilon}$-Lipschitz continuous gradient where $L_{\varepsilon}$ depends on the regularization strength $\varepsilon>0$. %
Using these properties, that work proposed accelerated gradient methods with approximate gradients which are subject to 
 nonasymptotic convergence rates.     

Finally, \cite{karumanchi2025approximation} showed that if $\rM =- \rB^{\intercal}\rB$ for a rank $r$ matrix, $\rB$,
\begin{equation}
\label{eq:varformQPintro}
    \min_{\substack{\mathrm Ax=b\\x\geq 0}}\left\{\frac 12 x^{\intercal} \rM x-c^{\intercal} x-\varepsilon\mathsf H(x)\right\}=\inf_{u\in\mathbb R^r}\left\{\frac 12\|u\|^2+\inf_{\substack{\mathrm{A}x=b\\x\geq 0}}\left\{-u^{\intercal}\rB x-c^{\intercal}x-\varepsilon\mathsf H(x)\right\}\right\}  
\end{equation}
where $\varepsilon\geq 0$  and $\mathsf H$ is Shannon's entropy. Though the main focus of that work was to show that solutions of this problem converge, as $\varepsilon\downarrow 0$, to those of the unregularized problem ($\varepsilon=0$) exponentially fast and only linearly if $\rM$ is not negative semidefinite, sufficient conditions for GW problems to be concave quadratic programs were also derived.

\section{Notation and Background}

For $N\in\mathbb N$, we let $[N]\coloneqq \{1,\dots, N\}$.
 For a function $f:S\to \mathbb R$, $\|f\|_{\infty, S} = \sup_{s\in S}|f(s)|$. The closure of a set $A\subset \mathbb R^d$ is denoted by $\bar A$.  For  matrices $\mathrm{A},\mathrm{B}\in\mathbb R^{d\times d'}$, $\langle\mathrm A,\mathrm B\rangle_{\mathrm{F}}=\sum_{(i,j)\in[d]\times[d']} \mathrm{A}_{ij}\mathrm{B}_{ij}$ and $\|\cdot\|_{\mathrm{F}}$ is the induced Frobenius norm. $\|\cdot\|_{\mathrm{op}}$ is the operator norm.

The collection of probability measures on a finite set $S$ is denoted $\mathcal P(S)$. The pushforward of a measure $\mu\in\mathcal P(S)$ through a map $T:S\to R$, where $R$ is another set of finite cardinality, is the probability measure $T_{\sharp}\mu\in\mathcal P(R)$ defined by the property that $T_{\sharp}\mu(A)=\mu\left(T^{-1}(A)\right)$ for every measurable set $A$. 
The multivariate normal distribution on $\mathbb R^d$ with mean $\mu\in\mathbb R^d$ and covariance $\Sigma\in\mathbb R^{d\times d}$ is denoted by $N(\mu,\Sigma)$. Equality in distribution is denoted by $\stackrel d = $ whereas convergence in distribution is denoted by $\stackrel d\to$.

\subsection{Optimal Transport}  For finite sets $\mathcal X_0=\left(x_0^{(i)}\right)_{i=1}^{N_0},\mathcal X_1=\left(x_1^{(j)}\right)_{j=1}^{N_1}$, a cost function $c:\mathcal X_0\times \mathcal X_1\to \mathbb R$, and probability measures $\mu_0\in\mathcal P(\mathcal X_0),\mu_1\in\mathcal P(\mathcal X_1)$, the optimal transport (OT) problem is 
$    \mathsf{OT}_c(\mu_0,\mu_1)=\inf_{\pi\in\Pi(\mu_0,\mu_1)}\int cd\pi,
$
where $\Pi(\mu_0,\mu_1)$ is the set of all couplings of $(\mu_0,\mu_1)$. By linear programming duality, we have the dual form  
\begin{equation}
    \label{eq:OTDuality}
\OT_c(\mu_0,\mu_1)=\sup_{\substack{(\varphi_0,\varphi_1)\in\mathcal F(\mathcal X_0)\times \mathcal F(\mathcal X_1)\\\varphi_0\oplus \varphi_1\leq c}} \left\{\int \varphi_0 d\mu_0+\int \varphi_1 d\mu_1\right\},
\end{equation}
where  $\varphi_0\oplus \varphi_1\leq c$ indicates that $\varphi_0\left( x_0^{(i)}\right)+\varphi_1\left( x_1^{(j)}\right)\leq c\left(x_0^{(i)}, x_1^{(j)}\right)$  for every $(i,j)\in[N_0]\times[N_1]$ and $\mathcal F(\mathcal X_0)$ is the set of all functions on $\mathcal X_0$ with $\mathcal F(\mathcal X_1)$ defined similarly. Both the primal (minimization) problem and the dual (maximization) problem admit solutions, which we refer to as  \textit{OT plans} and \textit{OT potentials} for $\OT_c(\mu_0,\mu_1)$. OT potentials are not unique: if $(\varphi_0^{\star},\varphi_1^{\star})$ is a pair of OT potentials for $\OT_c(\mu_0,\mu_1)$, then so is $(\varphi_0^{\star}+a,\varphi_1^{\star}-a)$ for any $a\in\mathbb R$. OT potentials related in this fashion are referred to as \textit{versions} of one another. 

As stated next, there always exists a version of the potentials which is bounded in terms of $\|c\|_{\infty,\mathcal X_0\times \mathcal X_1}$. The result is standard, but we include its proof in \cref{proof:lem:potentialBounds} for completeness. 

\begin{proposition}[Potential estimates]
\label{lem:potentialBounds}
            Fix $(\mu_0,\mu_1)\in\mathcal P(\mathcal X_0)\times\mathcal P(\mathcal X_1)$ with full support  and $c:\mathcal X_0\times \mathcal X_1\to\mathbb R$. Then, any pair of OT potentials $(\varphi_0,\varphi_1)$ for $\mathsf{OT}_c(\mu_0,\mu_1)$ admits a version $(\tilde \varphi_0,\tilde \varphi_1)$ satisfying $\max\left\{\|\tilde \varphi_0\|_{\infty,\mathcal X_0},\|\tilde \varphi_1\|_{\infty,\mathcal X_1}\right\}\leq \frac 32 \|c\|_{\infty,\mathcal X_0\times \mathcal X_1}$.  
\end{proposition}     

\subsection{Entropic Optimal Transport}
\label{sec:EOT}
The entropic optimal transport (EOT) problem, introduced in \cite{cuturi2013lightspeed}, is defined by regularizing the OT problem with the Kullback-Leibler (KL) divergence:  
\begin{equation}
   \begin{aligned} 
\label{eq:EOTPrimal_intro}
\mathsf{EOT}_c^{\varepsilon}(\mu_0,\mu_1)\coloneqq    \inf_{\pi\in\Pi(\mu_0,\mu_1)} \int cd\pi+\varepsilon{\mathsf{KL}}(\pi\|\mu_0\otimes \mu_1),
    \end{aligned}
\end{equation} 
where ${\mathsf{KL}}(\pi\|\mu_0\otimes \mu_1)\coloneqq\int\log(d\pi/d(\mu_0\otimes \mu_1))d\pi$ if $\pi\ll \mu_0\otimes \mu_1$ and takes the value $+\infty$ otherwise.
Akin to standard OT, \eqref{eq:EOTPrimal_intro} admits a dual problem, given as follows 
\begin{equation}
\label{eq:EOTDual}
    \sup_{(\varphi_0,\varphi_1)\in \mathcal F(\mathcal X_0)\times \mathcal F(\mathcal X_1)}\left\{ \int \varphi_0d\mu_0+\int\varphi_1d\mu_1-\varepsilon \int e^{\frac{\varphi_0\oplus \varphi_1-c}\varepsilon}d\mu_0\otimes \mu_1+\varepsilon\right\}.
\end{equation}
A solution, $(\varphi_0^{\star},\varphi_1^{\star})$, to \eqref{eq:EOTDual} is called a pair of \emph{EOT potentials} and is known to be a.s. unique up to additive constants (i.e. any other solution $(\psi_0^{\star},\psi_1^{\star})$ satisfies $(\psi_0^{\star},\psi_1^{\star})=(\varphi_0^{\star}+a,\varphi_1^{\star}-a)$ $\mu_0\otimes \mu_1$-a.s. for some $a\in\mathbb R$). Moreover, \eqref{eq:EOTPrimal_intro} has a unique solution $\pi^{\star}\in\Pi(\mu_0,\mu_1)$, called the \emph{EOT plan}, which can be expressed in terms of the marginal measures and the EOT potentials as
    $\frac{d\pi^{\star}}{d\mu_0\otimes \mu_1}=e^{\frac{\varphi_0^{\star}\oplus \varphi_1^{\star}-c}{\varepsilon}},\quad \mu_0\otimes \mu_1\text{-a.s.}$
Further,  $(\varphi_0,\varphi_1)\in L^1(\mu_0)\times L^1(\mu_1)$ solves \eqref{eq:EOTDual} if and only if $(\varphi_0,\varphi_1)$ solves  the  Schr\"odinger system. 
\begin{equation}
\label{eq:SchrodingerSystem}
    \int e^{\frac{\varphi_0(x)+\varphi_1-c(x,\cdot)}\varepsilon}d\mu_1=1,
    \int e^{\frac{\varphi_0+\varphi_1(y)-c(\cdot,y)}\varepsilon}d\mu_0=1, \text{for } \text{$\mu_0$-a.e. $x\in\mathcal X_0$, } \text{$\mu_1$-a.e. $y\in\mathcal X_1$}. 
\end{equation}  
See \cite{nutz2021introduction} for a comprehensive introduction to EOT. 

EOT potentials admitting  \emph{a priori} bounds depending only on $\|c\|_{\infty,\mathcal X_0\times \mathcal X_1}$ always exist. This result follows from Theorem 2.1 and Remark 2.2 in \cite{groppe2024lower} which is based on   
the work of \cite{marino2020optimal}. 
\begin{proposition}[EOT potential estimates]
\label{lem:EOTPotentialBounds}
           Fix $(\mu_0,\mu_1)\in\mathcal P(\mathcal X_{0})\times \mathcal P(\mathcal X_1)$ and $c:\mathcal X_0\times \mathcal X_1\to\mathbb R$. Then,  there exist EOT potentials $(\varphi_0,\varphi_1)$ for \eqref{eq:EOTDual} defined on $\mathcal X_0$ and $\mathcal X_1$  which satisfy  \eqref{eq:SchrodingerSystem} on $\mathcal X_0\times \mathcal X_1$ and for which $\max\left\{\| \varphi_0\|_{\infty,\mathcal X_0},\| \varphi_1\|_{\infty,\mathcal X_1}\right\}\leq \frac 32 \|c\|_{\infty,\mathcal X_0\times \mathcal X_1}$. As such, the EOT plan $\pi^{\star}$ satisfies
           $
                \frac{d\pi^{\star}}{d\mu_0\otimes \mu_1}(x,y)\geq e^{\frac{-4\|c\|_{\infty,\mathcal X_0\times \mathcal X_1}}{\varepsilon}}>0\text{  for each $(x,y)\in\supp(\mu_0)\times \supp(\mu_1)$.}           $
\end{proposition}

\subsection{Gromov--Wasserstein Problems} The Gromov--Wasserstein (GW) problem extends the OT setting to incompatible spaces by aligning inter-space relational structure. For $(\mu_0,\mu_1)\in\mathcal P(\mathcal X_0)\times \mathcal P(\mathcal X_1)$ and similarity measures $\kappa_0:\mathcal X_0\times \mathcal X_0\to\mathbb R$ and $\kappa_1:\mathcal X_1\times \mathcal X_1\to\mathbb R$, the $p$-GW problem with $p>0$ is given by \cite{chowdhury2023distances,memoli2011gromov} 
\begin{equation}
    \label{eq:GW}
    \mathsf{GW}_{p}(\mu_0,\mu_1)\coloneqq \inf_{\pi\in\Pi(\mu_0,\mu_1)}\left(\int \left|\kappa_0(x,x')-\kappa_1(y,y')\right|^pd\pi\otimes \pi(x,y,x',y')\right)^{1/p}.
\end{equation}
If we enforce that $\kappa_0=\mathsf d_0^q$ and $\kappa_1=\mathsf d_1^q$ for some $q>0$, the resulting functional is denoted by $\mathsf{GW}_{p,q}$ and, when $p\geq 1$, defines a metric on the quotient space of metric measure spaces obtained by identifying isomorphic metric measure spaces, see Footnote \ref{foot:isomorphic}.   
  
By complete analogy with EOT, \cite{peyre2016gromov,solomon2016entropic} proposed the following 
entropic regularization of $p$-GW,
\begin{equation}
    \label{eq:EGW}
    \mathsf{EGW}^{\varepsilon}_{p}(\mu_0,\mu_1)\coloneqq \inf_{\pi\in\Pi(\mu_0,\mu_1)}\int \left|\kappa_0(x,x')-\kappa_1(y,y')\right|^pd\pi\otimes \pi(x,y,x',y')+\varepsilon\mathsf{KL}(\pi\|\mu_0\otimes \mu_1).
\end{equation}
With our notation, $\mathsf{EGW}^0_{p}(\mu_0,\mu_1)=\mathsf{GW}_p(\mu_0,\mu_1)^p$. Consequently, we often state results for ``any $\varepsilon\geq 0$'' to indicate that it holds for both the regularized and unregularized problems. 

For each $\varepsilon\geq 0$, \eqref{eq:EGW} can be expressed as a regularized QP in $k=N_0N_1$ variables, 
\begin{equation}
\label{eq:regularizedQP}
    \mathsf{EGW}_{p}^{\varepsilon}(\mu_0,\mu_1) = \inf_{\substack{\rA x = b\\x\geq 0}} \left\{\frac{1}{2}x^{\intercal}\rC x-\varepsilon\big(\mathsf{H}(x)-\mathsf{H}_{\mu_0}-\mathsf{H}_{\mu_1}\big)\right\},  
\end{equation}
where 
 the constraint set $\Pi(\mu_0,\mu_1)$ is recast as $\{x\in\mathbb R^{k}:\mathrm {A}x=b,x\geq 0\}$ for suitable matrices $\mathrm A\in\mathbb R^{(N_0+N_1)\times k}$ and $b \in \mathbb R^{N_0+N_1}$, $\mathrm C\in\mathbb R^{k\times k}$ is the resulting cost matrix, and we have rewritten the KL divergence in terms of the Shannon entropies of $x,$ $\mu_0$, and $\mu_1$: 
$
\mathsf H(x)=-\sum_{i=1}^{k} x_i\log(x_i),\text{ and } \mathsf H_{\mu_i} = - \sum_{j=1}^{N_i}\mu_i(\{x_i^{(j)}\})\log\left(\mu_i(\{x_i^{(j)}\})\right)$  for $i\in\{0,1\},
$
  see \cref{sec:Ab} for exact expressions and details.  

\section{The Variational Form}

Our first main contribution is to establish, for any $\varepsilon\geq 0$, a variational formulation for the EGW problem \eqref{eq:regularizedQP} which trades the constrained quadratic minimization for an unconstrained $\min$-$\max$ problem whose objective involves the optimal value of a linear program. For notational convenience, we set $
    \cK\coloneqq \{x\in\RR^{k}:\rA x = b, x\geq 0\},$
where $\rA,b$ are as above and write \eqref{eq:regularizedQP} as  
$
\mathsf{EGW}_{p}^{\varepsilon}(\mu_0,\mu_1) =\inf_{x\in\mathcal K} \left\{\frac{1}{2}x^{\intercal}\rM x-\varepsilon\mathsf{H}(x)\right\} + \varepsilon\left(\mathsf{H}_{\mu_0}+\mathsf{H}_{\mu_1}\right),$
where $\rM=\frac{1}{2}\left(\rC+\rC^{\intercal}\right)$ which is symmetric. There is no loss in generality as $z^{\intercal}\rM z = z^{\intercal}\mathrm{C}z$ for each $z\in\mathbb R^k$. Dropping the constant entropy terms, the effective problem is thus 
\begin{equation}
    \label{eq:entropicQP}
    \inf_{x\in\mathcal K} \left\{\frac 12x^{\intercal} \rM x-\varepsilon\sH(x)\right\},
\end{equation} 
so that \eqref{eq:entropicQP} and $\mathsf{EGW}_{p}^{\varepsilon}(\mu_0,\mu_1)$ are equivalent problems. The main result is as follows.

\begin{theorem}[Variational form]\label{thm:variationalForm} 
    Any symmetric matrix $\rM\in\mathbb R^{k\times k}$ can be decomposed as $\rM=\rB_1^{\intercal}\rB_1-\rB_0^{\intercal}\rB_0$ for some $\rB_0\in\mathbb R^{r_0\times k},\rB_1\in\mathbb R^{r_1\times k}$ with $r_0,r_1\leq k$. With this, 
\begin{equation}
    \label{eq:variationalForm}
    \inf_{u\in\mathbb R^{r_0}}\sup_{v\in\mathbb R^{r_1}}\left\{ \frac12\|u\|^2-\frac 12 \|v\|^2+\inf_{x\in\mathcal K}\left\{\left( \rB_1^{\intercal}v-\rB_0^{\intercal} u \right)^{\intercal}x-\varepsilon\sH(x)  \right\} \right\}
\end{equation} is equivalent to \eqref{eq:entropicQP}  
in the sense that the optimal values of both problems coincide and  
\begin{enumerate}[leftmargin=*]
    \item If $x^{\star}$ solves \eqref{eq:entropicQP}, then $(u^{\star},v^{\star})=(\rB_0 x^{\star},\rB_1 x^{\star})$ solves the outer $\inf$-$\sup$ problem in \eqref{eq:variationalForm} and $x^{\star}\in\argmin_{x\in\mathcal K}\left\{\left( \rB_1^{\intercal}v^{\star}-\rB_0^{\intercal}u^{\star} \right)^{\intercal}x-\varepsilon\sH(x)  \right\}$. 
    \item If, conversely, $(u^{\star},v^{\star})$ solves the outer $\inf$-$\sup$ problem in \eqref{eq:variationalForm}, then there exists  
    $
    \bar x \in \argmin_{x\in\mathcal K}\left\{\left( \rB^{\intercal}_1v^{\star}-\rB_0^{\intercal} u^{\star} \right)^{\intercal}x-\varepsilon\sH(x)  \right\}$ which
    solves \eqref{eq:entropicQP} and $(u^{\star},v^{\star})=(\rB_0 \bar x, \rB_1\bar x)$.  
\end{enumerate} 
Moreover, $\argmax_{v\in\mathbb R^{r_1}}\left\{ \frac12\|u\|^2-\frac 12 \|v\|^2+\inf_{x\in\mathcal K}\left\{\left( \rB_1^{\intercal}v-\rB_0^{\intercal} u \right)^{\intercal}x-\varepsilon\sH(x)  \right\} \right\}\subset \rB_1\mathcal K$ for any $u\in \mathbb R^{r_0}$. Thus, the outer $\inf$ and $\sup$ in \eqref{eq:variationalForm} can be restricted to $\mathrm{B_0}\mathcal K$ and $\mathrm{B_1}\mathcal K$.
\end{theorem}

The proof of \cref{thm:variationalForm} leverages the fact that the positive semidefinite (psd) quadratic form $f_1(x)=\frac 12 x^{\intercal}\rB_1^{\intercal}\rB_1 x$ can be expressed as the convex conjugate of its convex conjugate via the Fenchel-Moreau theorem \cite[Theorem 11.1]{rockafellar1998variational}.  This enables expressing $f_1(x)$ as $\sup_{v\in\mathbb R^{r_1}}\left\{ x^{\intercal}\rB_1^{\intercal}v -\frac 12 \|v\|^2\right\}$, with a similar representation holding for $f_0(x)=\frac 12 x^{\intercal}\rB_0^{\intercal}\rB_0 x$. With this, it remains to interchange the supremum over the auxiliary variable $v$ with the infimum over $x\in\mathcal K$, leading to \eqref{eq:variationalForm}. See \cref{proof:thm:variationalForm} for the derivation.

\begin{remark}[Comparison with previous results]
\label{rmk:variationalComparison} 
As noted in the literature review, similar representations for GW distances have appeared previously. The first results in this direction, Theorem 3.4 and Corollary 4.1 in \cite{zhang2024gromov}, hold for  Euclidean distributions with finite fourth moments where $\kappa_0,\kappa_1$ are squared Euclidean distances and $p=2$ with and without entropy. \cite{karumanchi2025approximation} provided a similar formulation for \eqref{eq:entropicQP} in the case where $\rM$ is negative semidefinite (nsd) which coincides with \cref{thm:variationalForm} with $\rB_1=0$, recall \eqref{eq:varformQPintro}. Section 6 in that work shows that if $\kappa_0$ and $\kappa_1$ are both psd or both nsd kernels, then their result applies for the GW problem between finitely discrete measures with $p=2$. 

 By comparison,  \cref{thm:variationalForm} is applicable for any finitely discrete marginals without assumptions on $\kappa_0$, $\kappa_1$, or $p>0$. This generality comes at the cost of an additional maximization operation, but still yields a clear connection between GW and OT with a parameterized~cost.   
\end{remark}

Though \cref{thm:variationalForm} is written using Shannon's entropy, the result naturally yields a correspondence between the EGW problem and a parameterized family of EOT problems. %

\begin{corollary}[From EGW to EOT] 
\label{eq:EGWtoEOT}
For any $\varepsilon \geq 0$,   
$   \mathsf{EGW}_p^{\varepsilon}(\mu_0,\mu_1)$ can be expressed as $ \inf_{u\in\RR^{r_0}}\sup_{v\in\RR^{r_1}}\left\{\frac 12 \|u\|^2-\frac 12 \|v\|^2+\mathsf{EOT}^{\varepsilon}_{\rB_1^{\intercal}v-\rB_0^{\intercal}u}(\mu_0,\mu_1)\right\},
$
where $\mathsf{EOT}^{\varepsilon}_{\rB_1^{\intercal}v-\rB_0^{\intercal}u}(\mu_0,\mu_1)$ is the EOT problem $\min_{x\in\mathcal K}\left\{\left( \rB^{\intercal}_1v-\rB_0^{\intercal} u \right)^{\intercal}x-\varepsilon(\sH(x)-\sH_{\mu_0}-\sH_{\mu_1})  \right\}$.
\end{corollary}
This result is a direct consequence of \cref{thm:variationalForm}. %
           This type of connection has served as the catalyst for previous statistical and computational advancements for the quadratic GW distances between Euclidean distributions  \cite{rioux2024entropic,rioux2024limit,zhang2024gromov}. The remainder of this paper concerns the development of an analogous theory for distributions supported on \emph{arbitrary} finite metric spaces using \cref{thm:variationalForm}.    %

        \section{Statistical Analysis}
        \label{sec:statistics}
    
        In this section we consider the problem of statistical estimation of the GW distance with finitely discrete marginals. 
        Supposing that we have access to independent collections of i.i.d. samples $X_{0,1},\dots, X_{0,n}$ and $X_{1,1},\dots, X_{1,n}$ from $\mu_0$ and $\mu_1$ it is natural to estimate $\mathsf{EGW}^{\varepsilon}_p(\mu_{0}, \mu_{1})$ for a choice of kernels via the plug-in estimator $\mathsf{EGW}^{\varepsilon}_p(\hat \mu_{0,n},\hat \mu_{1,n})$, where $\hat \mu_{0,n}\coloneqq \frac 1n \sum_{i=1}^n\delta_{X_{0,i}}$ and $\hat \mu_{1,n}\coloneqq \frac 1n \sum_{i=1}^n\delta_{X_{1,i}}$ are the empirical measures from $n$ samples. In what follows, we establish that the plug-in estimator converges to $\mathsf{EGW}_p^{\varepsilon}(\mu_0,\mu_1)$ in expectation at the parametric rate $O(n^{-1/2})$ as $n\to \infty$ in \cref{sec:parametricRates}. This result hints that the asymptotic fluctuations of $\mathsf{EGW}_p^{\varepsilon}(\hat \mu_{0,n},\hat \mu_{1,n})$ about $\mathsf{EGW}_p^{\varepsilon}(\mu_{0},\mu_{1})$ with the $\sqrt n$ scaling should be stable. We formalize this heuristic in \cref{thm:limitTheorems} of \cref{sec:limitlaws} by deriving the limit distribution for empirical EGW. The derived limit laws enable us to formulate a framework for testing if two discrete mm spaces are isomorphic by using the quantiles of the limit law to set critical values for the test statistic, see \cref{sec:hypothesisTesting}.

        \subsection{Sample Complexity}
        \label{sec:parametricRates}
        We begin by characterizing the expected rate of convergence of the empirical cost. To this end, it will be convenient to identify a probability measure on a finite space with its vector representation. Namely, for a probability measure $\mu_i\in\mathcal P(\mathcal X_i)$, we let $w_{\mu_i}\coloneqq\left(\mu_i(\{x_i^{(1)}\}),\dots,\mu_i(\{x_i^{(N_i)}\})\right)$ for $i=0,1$ denote the vector of weights of $\mu_i$. %

       We first establish a uniform Lipschitz continuity property for the OT and EOT costs. This will enable us to reduce the problem of expected convergence analysis for empirical GW to the problem of controlling the expected difference between the empirical and population weights for multinomial distributions by leveraging the variational form.  This result will also prove useful when establishing limit laws for the empirical GW distance in \cref{sec:limitlaws}.   
        \begin{proposition}[Lipschitz stability]
        \label{prop:LipschitzStability}
 Fix $\varepsilon\geq 0$ and let  $\mathcal B$ be the unit ball in $\mathbb R^{k}$ with the $1$-norm. Then, for any choice of $(\nu_0,\nu_1),(\nu_0',\nu_1')\in\mathcal P(\mathcal X_0)\times \mathcal P(\mathcal X_1)$ and $(u,v)\in \rB_0\mathcal B\times \rB_1\mathcal B$, 
$
    \left|\mathsf{EOT}^{\varepsilon}_{\rB_1^{\intercal}v-\rB_0^{\intercal}u}(\nu_0,\nu_1)-\mathsf{EOT}^{\varepsilon}_{\rB_1^{\intercal}v-\rB_0^{\intercal}u}(\nu_0',\nu_1')\right|\leq C_{\rB_0,\rB_1}\left(\|w_{\nu_0}-w_{\nu_0'}\|_1+\|w_{\nu_1}-w_{\nu_1'}\|_1\right),
$
where $C_{\rB_0,\rB_1} =\frac3 2\sup_{\substack{\|x_0\|_1= 1\\\|x_1\|_1= 1}}\|\rB_1^{\intercal}\rB_1x_1-\rB_0^{\intercal}\rB_0x_0\|_{\infty}$. 
       \end{proposition}
We emphasize that $\mathsf{EOT}^0_c(\nu_0,\nu_1)=\mathsf{OT}_c(\nu_0,\nu_1)$ in the same way that $\mathsf{EGW}_p^{0}(\nu_0,\nu_1)=\mathsf{GW}_p(\nu_0,\nu_1)^p$. The result follows  from the dual forms of OT and EOT and by leveraging the estimates on the potentials from \cref{lem:potentialBounds,lem:EOTPotentialBounds}, see \cref{proof:prop:LipschitzStability} for the proof. 

With \cref{prop:LipschitzStability} in hand, we readily obtain parametric rates for the empirical EGW and GW distances by applying the strategy discussed above, see \cref{proof:thm:parametricRates} for details. 

        \begin{theorem}[Parametric convergence rates]
        \label{thm:parametricRates} In the setting of \cref{prop:LipschitzStability}, we have that
           \[
                \mathbb E\left[\left|\mathsf{EGW}^{\varepsilon}_p(\hat \mu_{0,n},\hat \mu_{1,n})-\mathsf{EGW}^{\varepsilon}_p( \mu_{0}, \mu_{1})\right|\right]\leq C_{\rB_0,\rB_1}
                \left(\sqrt{{N_0-1}{}}+\sqrt{N_1-1}\right)n^{-1/2}.
            \]
        \end{theorem}

\subsection{Directional Differentiability and Limit Distributions}
\label{sec:limitlaws}

We complement the derived sample complexity result by furnishing limit theorems for the empirical GW and EGW distance between finitely supported distributions with the same $\sqrt n$ scaling. In both cases, we leverage the extended delta method (cf. e.g., Theorem 1 in \cite{romisch2014delta}) which we describe in \cref{sec:deltaMethod} for clarity.
 
As the underlying spaces $\mathcal X_0,\mathcal X_1$ are finite, it  will be convenient to write the integral of $f:\mathcal X_0 \to \RR$ with respect to $\eta\in\mathcal P(\mathcal X_0)$ as
$
\int f d\eta = \sum_{i=1}^{N_0} f\left(x_0^{(i)}\right)\eta\left(\{x_0^{(i)}\}\right) = \langle v_f,w_{\eta}\rangle,  
$
where $v_f \coloneqq  \left(f\left(x^{(1)}_0\right),\dots,f\left(x^{(N_0)}_0\right) \right)\in\mathbb R^{N_0}$ and $w_{\eta}$ is the vector whose entries are the weights $\eta$ assigns to the points $\mathcal X_0= \left(x^{(i)}_0\right)_{i=1}^{N_0}$.  We will use the same notation for functions and measures on $\mathcal X_1$ or on the joint space $\mathcal X_0\times \mathcal X_1$.

We assume that 
$
\Delta(x,y,x',y')= |\kappa_0(x,x')-\kappa_1(y,y')|^p
$
satisfies $\Delta(x,y,x',y')=\Delta(x',y',x,y)$. There is no loss of generality, as $\tilde \Delta(x,y,x',y')=\frac 12 \left(\Delta(x,y,x',y')+\Delta(x',y',x,y)\right)$ satisfies this symmetry condition and
$
\int \tilde \Delta d\eta\otimes \eta = \int \Delta d\eta \otimes \eta,
$
for each $\eta \in\mathcal P(\mathcal X_0\times \mathcal X_1)$.

With these preliminaries in hand, we now state the main result.
\begin{theorem}[Limit theorems] 
    \label{thm:limitTheorems} Fix $(\mu_0,\mu_1)\in\mathcal P(\mathcal X_0)\times\mathcal P(\mathcal X_1)$ with full support. If the sequence $(\mu_{0,n},\mu_{1,n})_{n\in\mathbb N}\subset \mathcal P(\mathcal X_0)\times  \mathcal P(\mathcal X_1)$ is such that   
    $\sqrt n ( (w_{\mu_{0,n}},w_{\mu_{1,n}})-(w_{ \mu_{0}},w_{ \mu_{1}}))\stackrel{d}{\to}(G_{\mu_0},G_{\mu_1})$, 
    \begin{enumerate}[leftmargin=*]
        \item   
        $\sqrt n\left({\mathsf{GW}_p(\mu_{0,n},\mu_{1,n})^p- \mathsf{GW}_p(\mu_{0},\mu_{1})^p}\right)
     \stackrel{d}{\to}  \inf_{\pi^{\star}\in\Pi^{\star}(\mu_0,\mu_1)}\sup_{(\varphi_0,\varphi_1)\in\mathcal D_{\pi^{\star}}}\left\{v_{ \varphi_0}^{\intercal}G_{\mu_0}+v_{ \varphi_1}^{\intercal}G_{\mu_1}\right\}$, where $\Pi^{\star}(\mu_0,\mu_1)$ is the set of optimal plans for $\mathsf{GW}_p(\mu_{0},\mu_{1})$ and $\mathcal D_{\pi^{\star}}$ is the set of OT potentials for $\mathsf{OT}_{c_{\pi^{\star}}}(\mu_0,\mu_1)$ for $c_{\pi^{\star}}(x,y)=2\int \Delta(x,y,x',y')d\pi^{\star}(x',y')$.    
    \item If $\varepsilon >0 $, 
   $     \sqrt n\left({\mathsf{EGW}^{\varepsilon}_p(\mu_{0,n},\mu_{1,n})- \mathsf{EGW}_p^{\varepsilon}(\mu_{0},\mu_{1})}\right)
     \stackrel{d}{\to}  \inf_{\pi^{\star}\in\Pi^{\star}(\mu_0,\mu_1)}\left\{v_{ \varphi_0^{\pi^{\star}}}^{\intercal}G_{\mu_0}+v_{ \varphi_1^{\pi^{\star}}}^{\intercal}G_{\mu_1}\right\},
   $ 
    where $\Pi^{\star}(\mu_0,\mu_1)$ is the set of optimal plans for $\mathsf{EGW}_p^{\varepsilon}(\mu_{0},\mu_{1})$ and $(\varphi_0^{\pi^{\star}},\varphi_1^{\pi^{\star}})$ is the unique (up to constants) pair of EOT potentials for $\mathsf{EOT}^{\varepsilon}_{c_{\pi^{\star}}}(\mu_0,\mu_1)$
    \end{enumerate}
\end{theorem}

As discussed above, the proof of \cref{thm:limitTheorems}, included in \cref{proof:thm:limitTheorems}, relies on the delta method. In our application, it suffices to show that, for each $\varepsilon\geq 0$, the map $(\mu_0,\mu_1)\in\mathcal P(\mathcal X_0)\times \mathcal P(\mathcal X_1)\mapsto \mathsf{EGW}_p^{\varepsilon}(\mu_0,\mu_1)$ is Lipschitz continuous (which was shown in \cref{prop:LipschitzStability}) and directionally differentiable at $(\mu_0,\mu_1)$. In the case $\varepsilon=0$, the proof of differentiability follows similar lines to the proof of Theorem 2 in \cite{rioux2024limit}, which  established the directional derivative for the quadratic GW distance between finitely supported Euclidean distributions. The main insight is to linearize the GW distance along certain perturbations of the marginals and to tighten the resulting linearization. We also  leverage standard  techniques for proving directional stability of the optimal value function in nonlinear programming \cite{bonnans2013perturbation} as well as the stability of solutions to perturbed quadratic programs \cite{klatte1985Lipschitz}. The proof technique when $\varepsilon>0$ is similar to that of the unregularized case with some important distinctions stemming from the fact that the underlying problems are no longer standard quadratic programs. Similar results were obtained in \cite{rioux2024limit}, but only in the case of Euclidean measures for the quadratic GW distance.

\begin{remark}[Asymptotic normality] The limit laws derived in \cref{thm:limitTheorems} are not guaranteed to be normal due to the outer minimization and maximization. However, from the proof of \cref{prop:LipschitzContinuityGradient} ahead, if $\varepsilon>\lambda_{\max}(\rB_0^{\intercal}\rB_0)-\lambda_{\min}(\rB_1^{\intercal}\rB_1)$, the objective in \eqref{eq:variationalForm} is strictly convex as a function of $u$ and hence admits a unique minimizer. In this setting, Item (1) in \cref{thm:variationalForm} implies that $\Pi^{\star}(\mu_0,\mu_1)$ is a singleton by uniqueness of EOT plans so that the limit distribution for $\mathsf{EGW}_p^{\varepsilon}$ is normal if $G_{\mu_0}$ and $G_{\mu_1}$ are independent Gaussians.
\label{rmk:asymptotic}
\end{remark}

As an application of this result, the following section considers the problem of testing if two finite metric measure spaces are isometric based only on samples.

\subsection{Hypothesis Testing}
\label{sec:hypothesisTesting}
As noted previously, one of the most interesting properties of GW problems is that when $(\mathcal X_0,\mathsf d_0,\mu_0)$ and $(\mathcal X_1,\mathsf d_1,\mu_1)$ are metric measure spaces, their GW cost is $0$ if and only if $\mu_1=T_{\sharp}\mu_0$ for some isometry $T:\supp(\mu_0)\to \supp(\mu_1)$. In this setting the limit distributions derived in \cref{thm:limitTheorems} enable us to test if two population measures are related by some isometry based on independent collections of i.i.d. samples, provided that the quantiles of the limit distribution can be consistently estimated. 
 
As the limit distributions in \cref{thm:limitTheorems} depend on the population distributions in a nontrivial manner, it is unclear if the limit can be estimated in a direct way. Further,  
even if $\mu_{0,n}$ and $\mu_{1,n}$ are chosen to be the empirical measures, the naive bootstrap may fail to be consistent {in general due to the nonlinearity of the directional derivative; see \cite{dumbgen1993nondifferentiable}}. One possible approach is to apply the $m$-out-of-$n$ bootstrap \cite{bickel1996resampling}, but this approach is sensitive to subsampling size and requires solving multiple GW problems with bootstrapped measures which is computationally burdensome.     

This section develops an alternative approach based on the observation that,
under suitable assumptions, a simplified limit law holds under the null hypothesis, $\mathsf{GW}_p(\mu_0,\mu_1)=0$. %

\subsubsection{Limit Laws Under the Null}

Working with arbitrary kernels $\kappa_0,\kappa_1$ in place of metrics, the characterization of measures for which $\mathsf{GW}_p(\mu_0,\mu_1)=0$ is not as simple.  Nevertheless, the following connection holds.
\begin{proposition}[GW indiscernibles]
\label{lem:vanishingGW}
   Suppose that $\kappa_0$ and $\kappa_1$ are dissimilarity measures on $\mathcal X_0=\supp(\mu_0)$ and $\mathcal X_1=\supp(\mu_1)$ for which $\kappa_0(x,\cdot)\neq \kappa_0(x',\cdot)$ as functions on $\mathcal X_0$ for each $x\neq x'$ and similarly for $\kappa_1$. Then, $\mathsf{GW}_p(\mu_0,\mu_1)=0$ if and only if $\mu_1 =S_{\sharp}\mu_0$ for some map 
  $
       S\in\mathcal S\coloneqq\left\{S:\mathcal X_0\to \mathcal X_1: S\text{ is bijective and } \kappa_0(x,x')=\kappa_1(S(x),S(x'))\text{ for every  }x,x'\in\mathcal X_0\right\}.  
  $  
  In particular,  $\mathsf{GW}_p(\mu_0,\mu_1)=0$ if and only if every solution $\pi^{\star}$ of $\mathsf{GW}_p(\mu_0,\mu_1)$ is such that $\pi^{\star}=(\Id,S)_{\sharp}\mu_0$ for some $S\in\mathcal S$.
   \end{proposition}
The proof of this result is straightforward, but is included in \cref{proof:lem:vanishingGW} for completeness. Importantly,  \cref{lem:vanishingGW} can inform the choice of kernels if we wish to identify certain invariants. This perspective is used when applying these statistical results for testing for graph distribution isomorphisms in \cref{sec:graphIsomorphism} ahead. 
We underscore that a corresponding result does not hold for the regularized distance due to the added KL divergence term. %

In the sequel we assume that $\kappa_0$ and $\kappa_1$ satisfy the  conditions from \cref{lem:vanishingGW} and consider the problem of testing %
$\mathrm{H}_0 : \mu_1=S_{\sharp}\mu_0 \text{ for some $S\in\mathcal S$, corresponding to }  \mathsf{GW}_p(\mu_0,\mu_1)=0$, versus
$
\mathrm{H}_1 : \mu_1\neq S_{\sharp}\mu_0 \text{ for any $S\in\mathcal S$, corresponding to } \mathsf{GW}_p(\mu_0,\mu_1)>0$. To this end, we study the limit law under the null  $\mathsf{GW}_p(\mu_0,\mu_1)=0$.

\begin{corollary}[Limit law under the null]
\label{cor:simplifiedLimitLaw}
   Suppose that $\kappa_0,\kappa_1$ are dissimilarity measures satisfying the conditions of \cref{lem:vanishingGW}. %
 Then, in the setting of \cref{thm:limitTheorems}, if $\mathsf{GW}_p(\mu_0,\mu_1)=0$, $   \sqrt n \mathsf{GW}_p(\mu_{0,n},\mu_{1,n})^p\stackrel{d}{\to}\inf_{\substack{S\in\mathcal S\\\mu_1=S_{\sharp}\mu_0}}\sup_{h\in\mathcal H_{\mu_0}} \left\{v_{h}^{\intercal}G_{\mu_0}+v_{-h\circ S^{-1}}^{\intercal}G_{\mu_1}\right\}\eqqcolon L_*,\text{ for }$  
   \[
   \begin{gathered}
        \mathcal H_{\mu_0}\coloneqq \left\{h:\mathcal X_0\to \mathbb R\mspace{2mu}|\mspace{2mu}h(x)-h(\tilde x)\leq  2\int |\kappa_0(x,x')-\kappa_0(\tilde x,x')|^pd\mu_0(x')\text{ for each }x,\tilde x\in\mathcal X_0\right\}.
   \end{gathered} 
    \]
   \end{corollary}
 \cref{cor:simplifiedLimitLaw} follows by characterizing the dual potentials under the null   using \cref{lem:vanishingGW} and properties of solutions of linear programs, see 
   \cref{proof:cor:simplifiedLimitLaw} for complete details. 

\subsubsection{Conservative Quantile Estimates}

 While the limiting distribution simplifies under the null, the $\inf$-$\sup$ structure in \cref{cor:simplifiedLimitLaw} makes it difficult to directly estimate its quantiles and hence obtain the critical value for a test of exact asymptotic size $\alpha$. To account for this, we propose a conservative, {yet computationally efficient}, testing procedure based on overestimating the quantiles of $L_*$ in \cref{cor:simplifiedLimitLaw} and characterize when this overestimate is exact. The computational tractability comes from the fact that this test only requires solving random linear programs. We develop this estimator in a general setting which will be used in the test for isomorphism of graph distributions studied in \cref{sec:graphIsomorphism}.  The required assumptions are as follows. 
 
  \begin{assumption}
        \label{assn:generalSettingLimitLaw}
        Under the null $\mathsf{GW}_p(\mu_0,\mu_1)=0$ and in the  setting of \cref{cor:simplifiedLimitLaw},  
        \begin{enumerate}[leftmargin=*]
         \item $G_{\mu_0}\sim N(0,\Sigma_q)$ where $\Sigma_{q}=\mathrm{diag}(q)-qq^{\intercal}$ for a simplex vector $q\in\mathbb R^{N_0}$, 
            \item  For any $S\in\mathcal S$ satisfying $\mu_1=S_{\sharp}\mu_0$, $T_S(G_{\mu_1})$ is an independent copy of $G_{\mu_0}$ and $T_S:\mathbb R^{N_0}\to \mathbb R^{N_0}$ is such that $\left(T_S(x)\right)_i = x_{\sigma_{S}(i)}$ for each $i\in [N_0]$, where $\sigma_S:[N_0]\to [N_0]$ is the permutation induced by $S$ (i.e., $\sigma_S(i)=j$ if  $S(x^{(i)}_0)=x^{(j)}_1$),
            \end{enumerate}
            Furthermore, we assume that we can construct estimators $\hat \Sigma_n$ and $\hat u_n$   
            based on samples $Y_1,\dots, Y_n$ from an auxiliary distribution $\nu$ on some (perhaps distinct) space         satisfying  
            \begin{enumerate}[leftmargin=*]
            \item[3.] Given almost every realization of $Y_1,Y_2,\dots,$ $\hat \Sigma_n\to \Sigma_q$ and $\hat u_n\to u_{\mu_0}$ conditionally on the data, for
            $
            (u_{\mu_0})_l =  2\int |\kappa_0(x_0^{(\lfloor (l-1)/N_0\rfloor+1)},x')-\kappa_0(x_0^{((l-1)\bmod N_0)+1},x')|^pd\mu_0(x'), l\in [N_0^2],
            $%
        \item[4.] $\hat \Sigma_n$ is such that if $Z\sim N(0,\hat \Sigma_n)$, then $\mathbb P(Z^{\intercal}\mathbf 1=0)=1$ for all $n$ sufficiently large.
        \end{enumerate}
        \end{assumption}

 \cref{assn:generalSettingLimitLaw} is naturally met in the case of empirical measures, see  \cref{thm:convergenceDirectEstimator}.

    With these preliminaries in hand, we show that 
    $    L\coloneqq \sup_{h\in\mathcal H_{\mu_0}}\left\{\sqrt 2 w_h^{\intercal} G_{\mu_0}\right\}$
    satisfies the property that its $\beta$-quantile dominates that of $L_*$, the limit law from \cref{cor:simplifiedLimitLaw}, for each $\beta\in(0,1)$. With this, the test which rejects the null if the test statistic exceeds the $(1-\alpha)$-quantile of $L$ is of asymptotic size no more than $\alpha$ (i.e. it is of asymptotic level $\alpha$).
    \begin{proposition}[Conservative quantiles] 
    \label{prop:quantiles}
    Under \cref{assn:generalSettingLimitLaw} (1)-(2),  $\mathbb P(L\leq t)\leq \mathbb P(L_*\leq t)$ for each $t\in\mathbb R$ with equality if $\{S\in\mathcal S:\mu_1=S_{\sharp}\mu_0\}$ is a singleton. In particular, the $\beta$-quantile of $L$ is larger than or equal to that of $L_*$ for each $\beta\in(0,1)$. Finally,  $\{S\in\mathcal S:\mu_1=S_{\sharp}\mu_0\}$ is a singleton if and only if the identity is the unique map $T:\mathcal X_0\to\mathcal X_0$ satisfying $T_{\sharp}\mu_0=\mu_0$ and $\kappa_0(x,x')=\kappa_0(T(x),T(x'))$ for each $x,x'\in\mathcal X_0$.     
    \end{proposition}

    The proof of \cref{prop:quantiles} follows by showing that, for any $S\in\mathcal S$ satisfying $\mu_1=S_{\sharp}\mu_0$, $\sup_{h\in\mathcal H_{\mu_0}}\left\{w^{\intercal}_h G_{\mu_0} + w^{\intercal}_{-h\circ S^{-1}} G_{\mu_1}\right\}$ dominates $L_{\ast}$ in the stochastic order and by leveraging \cref{assn:generalSettingLimitLaw} and the definition of $\mathcal S$, see \cref{proof:prop:quantiles} for complete details.

    \subsubsection{Computation and Consistency of Conservative Quantiles}    
    \label{sec:consistencyConservativeQuantiles}
    Given samples, we may  compute the estimators $\hat \Sigma_n$ and $\hat u_n$ of $\Sigma_q$ and $u_{\mu_0}$ satisfying the conditions of \cref{assn:generalSettingLimitLaw} and approximate $\mathcal H_{\mu_0}$ and $L$ as  
    \begin{align} 
    \label{eq:constraintSetEstimator}
            \hat{\mathcal H}_{n}&\coloneqq \left\{w : \mathrm{R} w \leq \hat u_n\right\}, \text{ where } (\mathrm{R}w)_l= w_{\lfloor (l-1)/N_0\rfloor+1}-w_{ (l-1)\bmod N_0+1}, 
    \\
    \label{eq:LnDistribution}
   \hat L_n&\coloneqq \max_{Rw\leq \hat u_n} \sqrt 2 Z^{\intercal} w \text{ for } Z\sim N(0,\hat \Sigma_n).        
    \end{align}
    To sample from $\hat L_n$, it suffices to draw  a sample $Z\sim N(0,\hat \Sigma_n)$ and solve the corresponding linear program, see \cref{alg:directEstimator}. Practical implementation details are included in \cref{rmk:implementationEstimator}. 
    
    \begin{algorithm}
\caption{Draw $K$ samples $L_1,\dots, L_K$ from $\hat L_n$}
\label{alg:directEstimator}
\begin{algorithmic}[1] 
\Statex Given samples $Y_1,\dots,Y_n$ and desired number of samples from $\hat L_n$, $K$,
\State Compute $\hat \Sigma_n$ and $\hat u_n$ satisfying \cref{assn:generalSettingLimitLaw} and construct $\mathrm R$ as defined in \eqref{eq:constraintSetEstimator}
\For{$k=1,\dots, K$}
    \State $L_k\gets 
    \sqrt 2 \max_{\mathrm R w\leq \hat u_n} Z^{\intercal} w$ for a sample $Z\sim N(0,\hat \Sigma_n)$
\EndFor
\end{algorithmic}
\end{algorithm}

Importantly, the conditional quantiles of $\hat L_n$ consistently estimate  those of $L$. 
   
        \begin{proposition}[Consistency of $\hat L_n$]
        \label{thm:convergenceDirectEstimator} 
        In the setting of \cref{cor:simplifiedLimitLaw} and  \cref{assn:generalSettingLimitLaw}, for almost every realization of $Y_1,Y_2,\dots,$  $
        \lim_{n\to \infty}\sup_{t\geq 0}\left| \mathbb P(\hat L_n\leq t|Y_1,\dots,Y_n)-\mathbb P(L\leq t)\right|= 0,$
        provided that $\mu_0$ is not a point mass and that $\Sigma_q\neq 0$. In particular, if  $(\mu_{0,n},\mu_{1,n})=(\hat\mu_{0,n},\hat\mu_{1,n})$ and $\mu_0$ is not a point mass, then all relevant assumptions are satisfied by taking the empirical estimates of $\Sigma_{\mu_0}$ and $u_{\mu_0}$. 
        \end{proposition}
    \cref{thm:convergenceDirectEstimator} follows essentially by showing that the optimal value of the estimated random linear program, $\hat L_n$, converges conditionally in distribution to $L$ given almost every realization of the samples. Under the present conditions, this is effectively a consequence of known stability results for the optimal value of linear programs and the particular structure of the linear program in \eqref{eq:LnDistribution}.  The uniform convergence of the CDFs then follows by showing that the distribution function of $L$ is continuous. The argument can be found in   \cref{proof:thm:convergenceDirectEstimator}. The following result follows from \cref{thm:convergenceDirectEstimator,prop:quantiles}, see \cref{proof:thm:conservativeTesting} for details.
{ 

    \begin{theorem}[Conservative test]
    \label{thm:conservativeTesting} Assume the setting of \cref{thm:convergenceDirectEstimator}. 
    Given almost every realization of the samples,
        the test which rejects the null hypothesis if $\sqrt n\mathsf{GW}_p(\mu_{0,n},\mu_{1,n})^p$ exceeds the $(1-\alpha)$-quantile of $\hat L_n$ is of asymptotic level $\alpha$. If $\{S\in\mathcal S:\mu_1=S_{\sharp}\mu_0\}$ is a singleton, $\hat L_n$ consistently estimates the distribution of $L_*$. Otherwise, if $\mathsf{GW}_p(\mu_0,\mu_1)>0$, the asymptotic power of the test is $1$ given almost every realization of the samples.   
    \end{theorem}
   } 
     
    This result enables us to calibrate critical values for the test statistic $\sqrt n \mathsf{GW}_p(\mu_{0,n},\mu_{1,n})^p$ which achieve a desired asymptotic significance level by using the empirical quantiles from samples generated by  \cref{alg:directEstimator}. This procedure only requires solving linear programs which can be done in polynomial time \cite{karmarkar1984new} whereas computing the test statistic itself generally requires solving a nonconvex quadratic program. 
    
    The next section proposes new algorithms for solving the EGW problem based on \eqref{eq:variationalForm}. Similar algorithms for the unregularized distance are used in the application in \cref{sec:graphIsomorphism}, though the convergence guarantees derived in the sequel do not extend to that setting.

     \section{Stability and Algorithms}
\label{sec:computation}

Given finitely discrete probability measures $\mu_0\in\mathcal P(\mathcal X_0),\mu_1\in\mathcal P(\mathcal X_1)$,  with full support  and dissimilarity measures $\kappa_0:\mathcal X_0\times \mathcal X_0\to \mathbb R,\kappa_1:\mathcal X_1\times \mathcal X_1\to \mathbb R$, we aim to solve  
$
    \inf_{x\in\mathcal K}\left\{\frac 12 x^{\intercal}\rM x-\varepsilon\mathsf{H}(x)\right\}$
by leveraging its variational representation 
\begin{equation}
\label{eq:objectiveGW}
    \inf_{u\in\mathbb R^{r_0}} \ell_{\varepsilon}(u) \coloneqq \frac12\|u\|^2+ \sup_{v\in\mathbb R^{r_1}}\left\{  -\frac{1}{2}\|v\|^2 +\inf_{x\in\mathcal K}\left\{\left( \rB^{\intercal}_1 v-\rB_0^{\intercal} u\right)^{\intercal} x -\varepsilon\sH(x)\right\}\right\}, 
\end{equation}
recalling that $\mathsf H$ is used in place of the KL divergence only for notational convenience. 

We first address the properties of the objective $\ell_{\varepsilon}$ and later provide algorithms for minimizing it which are subject to nonasymptotic convergence rate results and account for imprecision in gradient computations. We focus on the case where $\varepsilon>0$, as entropic GW is most commonly used for large-scale problems arising in applications. In this setting, the objective $\ell_{\varepsilon}$ is differentiable in the classical sense with gradient 
\begin{equation}
\label{eq:gradientObjective}
    \nabla \ell_{\varepsilon}(u)=u-\rB_0 x^{\star}_{u}, \text{ where } \{x^{\star}_{u}\}=\argmin_{x\in\mathcal K}\left\{\frac 12 x^{\intercal}\rB_1^{\intercal}\rB_1x-u^{\intercal}\rB_0x -\varepsilon \mathsf{H}(x)\right\},
\end{equation}
as noted in \cref{lem:derivativeObjective}. The algorithms described in the sequel also naturally admit unregularized counterparts which we use in \cref{sec:graphIsomorphismExperiments}, though the derived convergence rate results do not readily transfer to the $\varepsilon=0$ case. 

We defer numerical validation of the subsequent results and a comparison between the proposed methods and the popular mirror descent algorithm from \cite{peyre2016gromov} to \cref{sec:validation}.

\subsection{Objective Stability} We first establish that $\ell_{\varepsilon}$ is $L$-smooth, that is, its gradient is $L$-Lipschitz continuous; this property is central to establishing the rate of convergence for the algorithms in \cref{subsec:algorithms}. We also elucidate the convexity properties of $\ell_{\varepsilon}$ which, naturally, depend on the choice of $\varepsilon$. In the sequel, we write $\ell$ in place of $\ell_{\varepsilon}$ to emphasize that $\varepsilon>0$ is fixed.

\begin{proposition}[$L$-smoothness of $\ell$]
\label{prop:LipschitzContinuityGradient}
    For any $\varepsilon>0$, the gradient of $\ell$ is Lipschitz continuous with constant 
    $
        L= \max\left\{1,\frac{\lambda_{\max}(\rB_0^{\intercal}\rB_0)}{\lambda_{\min}(\rB_1^{\intercal}\rB_1)+\varepsilon}-1 \right\}.
    $
    If $\varepsilon\geq \lambda_{\max}(\rB_0^{\intercal}\rB_0)-\lambda_{\min}(\rB_1^{\intercal}\rB_1)$, $\ell$ is  convex.
\end{proposition}

This result follows by characterizing the derivative of the map $u\in\mathbb R^{r_0}\mapsto x^{\star}_u$ via the implicit function theorem, enabling an analysis of the Hessian of $\ell$. This is done by leveraging the structural properties of solutions to EOT problems, see
\cref{proof:prop:LipschitzContinuityGradient} for details.

While \cref{prop:LipschitzContinuityGradient} establishes some useful properties of $\ell$, computing its gradient \eqref{eq:gradientObjective} still requires the resolution of a regularized convex quadratic program. To simplify computation, we establish that $x^{\star}_u$ can be obtained by solving an optimization problem whose objective depends on the constraint set $\mathcal K$ only through an entropic linear program which can thus be solved efficiently via Sinkhorn's algorithm \cite{cuturi2013lightspeed,sinkhorn1967diagonal}, see  \cref{sec:inexactOracle}. 
\begin{proposition}[Gradient via EOT]
    \label{lem:alternativeGradient}
    For any $u\in\mathbb R^{r_0}$ and $\varepsilon>0$, let $x^{\star}_u$ be the unique solution of the regularized quadratic program $\min_{x\in\mathcal K}\left\{\frac 12 x^{\intercal}\rB_1^{\intercal}\rB_1x-u^{\intercal}\rB_0x -\varepsilon \mathsf{H}(x)\right\}$. Then, $x^{\star}_u$ coincides with $x^{\star}_{(u,v^{\star}(u))}$, the unique solution of $\min_{x\in\mathcal K}\left\{\left(\rB_1^{\intercal} v^{\star}(u)-\rB_0^{\intercal} u\right)^{\intercal}x-\varepsilon \mathsf{H}(x)\right\}$, where $v^{\star}(u)$ is the unique solution of the strongly concave maximization problem 
    \begin{equation}
    \label{eq:innerMax}
        \sup_{v\in\mathbb R^{r_1}} f_{u}(v)\coloneqq-\frac 12 \|v\|^2+\min_{x\in\mathcal K}\left\{\left(\rB_1^{\intercal} v-\rB_0^{\intercal} u\right)^{\intercal}x-\varepsilon \mathsf{H}(x)\right\}.
    \end{equation}
\end{proposition}
\cref{lem:alternativeGradient} is an application of \cref{thm:variationalForm} in the convex setting, see  
\cref{proof:lem:alternativeGradient}.

With these results in hand, the algorithmic framework becomes clear. A first-order method can be used to locally minimize $\ell$, where the gradient in \eqref{eq:gradientObjective} is computed by solving the maximization problem \eqref{eq:innerMax} again via a first-order method, recalling that, by \cref{lem:derivativeObjective}, 
\begin{equation}
\label{eq:fuDerivative}
\nabla f_u(v) = -v+\rB_1 x^{\star}_{(u,v)}, \text{ where } \{x^{\star}_{(u,v)}\}=\argmin_{x\in\mathcal K}\left\{\left(\rB_1^{\intercal} v-\rB_0^{\intercal} u\right)^{\intercal}x-\varepsilon \mathsf{H}(x)\right\}.
\end{equation}
 Following the proof of  \cref{prop:LipschitzContinuityGradient}, we can also establish the smoothness and concavity properties of $f_u$, see \cref{proof:lemma:fu-proprties}.

\begin{proposition}[$L$-smoothness of $f_u$]\label{lemma:fu-proprties}
    For each $u\in\mathbb R^{r_0}$, $f_u$ is $1$-strongly concave and, if $\varepsilon>0$, its gradient is $1+\varepsilon^{-1}\|\rB_1\|_{\mathrm{op}}^2$-Lipschitz.
\end{proposition}

\subsection{Inexact Gradient Method} 
\label{subsec:algorithms}
Motivated by the previous discussion, we propose to minimize~$\ell$ via a projected accelerated gradient method, \cref{algo:outer-loop-accelerated}. While gradient methods without projection can be used, our analysis hinges on minimizing over a compact set. There is, however, no loss of generality as \cref{thm:variationalForm} asserts that every minimizer of $\ell$ is an element of $\rB_0\mathcal K\subset \mathbb B_{R_0}$, the closed ball in $\mathbb R^{r_0}$ of radius $R_0=\|\rB_0\|_{1,2}$,\footnote{For $p,q\in(0,\infty]$ and $\mathrm{A}\in\mathbb R^{n\times m}$, $\|\mathrm{A}\|_{p,q}=\sup_{\|x\|_p\leq 1}\|\mathrm{A}x\|_q$.} and that, for any fixed $u\in\mathbb R^{r_0}$, the maximizer of $f_u$ is an element of $\rB_1\mathcal K\subset \mathbb B_{R_1}$, the closed ball in $\mathbb R^{r_1}$ of radius $R_1=\|\rB_1\|_{1,2}$.

Our proposed accelerated gradient method, given below, enjoys rate adaptivity in the sense that its convergence rate becomes provably faster if $\ell$ is convex. In \cref{sec:standardGradientMethod}, we also analyze a gradient method without acceleration which does not benefit from this adaptivity; we defer the analysis of this simpler algorithm to the appendix, as the overall presentation and proof techniques are essentially the same.           

\begin{algorithm}
\caption{Inexact Accelerated Gradient Method for $\ell$}
\label{algo:outer-loop-accelerated}
\begin{algorithmic}[1]
\Statex Given initialization  $z_1=s_0\in\mathbb B_{R_0}$, the Lipschitz constant of $\nabla\ell$, $L$, and step sequences $\beta_j=\frac 1{2L}$, $\gamma_j=\frac{j}{4L}$, and $\tau_j=\frac{2}{j+2}$.
\For{$j = 1, \ldots, J$}
\State $g_j\gets \texttt{inexactGradientOracle}(z_j)$\Comment{ compute an approximation, $\tilde{\nabla}  \ell(z_j)$, of $\nabla  \ell(z_j)$ }
\State $u_{j}\gets \min\left(1,\frac{R_0}{\|z_j-\beta_jg_{j}\|}\right)(z_j-\beta_jg_{j})$
\State $s_j\gets \min\left(1,\frac{R_0}{\|s_{j-1}-\gamma_jg_{j}\|}\right)(s_{j-1}-\gamma_jg_{j})$
\State $z_{j+1}\gets \tau_js_j+(1-\tau_j)u_j$
\EndFor
\end{algorithmic}
\end{algorithm}

Note that the algorithm relies on the access to an inexact gradient oracle which, given a point $u\in \mathbb B_{R_0}$, returns an approximation, $\tilde{\nabla} \ell(u)$, of $\nabla \ell(u)$. We describe a method for accessing such a gradient approximation in \cref{sec:inexactOracle} and assume for the moment that a suitable oracle is available. A similar method was analyzed in \cite{rioux2024entropic} for solving the quadratic GW problem between Euclidean distributions. In fact, it can be shown that the algorithm analyzed in that work can be obtained from \cref{algo:outer-loop-accelerated} by leveraging Remark 4 in \cite{karumanchi2025approximation}. 

Assuming that the gradient oracle admits a fixed uniform precision (which will be established for the oracle proposed in \cref{sec:inexactOracle}), we obtain the following convergence rate result. 
\begin{theorem}[Convergence of \cref{algo:outer-loop-accelerated}]
\label{lem:convergenceRateAccMin}
    Fix $\eta'>0$ and assume that $\|\tilde \nabla \ell(u)-\nabla \ell(u)\|\leq \eta'$ for each $u\in\mathbb B_{R_0}$. Then, we may set $L=\max\left\{1,\frac{\lambda_{\max}(\rB_0^{\intercal}\rB_0)}{\lambda_{\min}(\rB_1^{\intercal}\rB_1)+\varepsilon}-1\right\}$ in \cref{algo:outer-loop-accelerated} and
    \begin{enumerate}[leftmargin=*]
        \item If $\ell$ is convex and $u^{\star}$ minimizes $\ell$,\footnote{This is the case if, for instance, $\varepsilon\geq \lambda_{\max}(\rB_0^{\intercal}\rB_0)-\lambda_{\min}(\rB_1^{\intercal}\rB_1)$, see \cref{prop:LipschitzContinuityGradient}.} the iterates $u_j$ of \cref{algo:outer-loop-accelerated} satisfy  
        \[
        \min_{j=1}^{J}\|{2L}\left(u_{j}-z_j\right)\|^2\le \frac{96L^2}{J(J+1)(J+2)}\|s_0-u^{\star}\|^2 + 16LR_0\eta'.
        \]
    \item If $\ell$ is nonconvex and $L'$ is a Lipschitz constant of $\nabla(\ell-\frac 12 \|\cdot \|^2)$, e.g., $L'=\frac{\lambda_{\max}(\rB_0^{\intercal}\rB_0)}{\lambda_{\min}(\rB_1^{\intercal}\rB_1)+\varepsilon}$, 
    \[
        \min_{j=1}^{J}\|{2L}\left(u_{j}-z_j\right)\|^2\le \frac{96L^2}{J(J+1)(J+2)}\|s_0-u^{\star}\|^2+\frac{24LL'}{J}\left(\|u^{\star}\|^2+\frac{5}{4}R_0^2\right) + 16LR_0\eta'.
    \]
    \end{enumerate} 
\end{theorem} 

\cref{lem:convergenceRateAccMin} follows from  Theorem 11 in \cite{rioux2024entropic} which, in turn, builds on the results of \cite{aspremont2008smooth}~and~\cite{ghadimi2016accelerated}. Indeed, the  gradient oracle used herein is shown in \cref{proof:lem:convergenceRateAccMin} to yield an approximate gradient as analyzed in \cite{aspremont2008smooth} so that the analysis from \cite{rioux2024entropic} holds verbatim given \cref{prop:LipschitzContinuityGradient}.

The derived convergence rates include a fixed error term that accounts for the inexact gradient evaluations and a term that tracks the algorithmic progress. As mentioned previously, \cref{algo:outer-loop-accelerated} is adaptive in the sense that it yields two distinct rates depending on if $\ell$ is convex or not, namely $O(J^{-3})$ or $O(J^{-1})$.  As discussed in Remark 13 in \cite{rioux2024entropic}, these rates coincide with the best known convergence rates for solving the relevant classes of problems using first-order methods. Furthermore, as shown in Corollary 12 in \cite{rioux2024entropic}, if $u_j$ is an interior point of $\mathbb B_{R_0}$, $\|\nabla \ell(u_j)\|$ can be bounded above in terms of $\|2L(u_j-z_j)\|$ and an additive term which accounts for the gradient oracle error. This shows that \cref{algo:outer-loop-accelerated} can yield approximate stationary points of $\ell$ if the iterates are interior points. As such, we may consider terminating \cref{algo:outer-loop-accelerated} if $\|g_j\|$ is small.

\subsection{An Inexact Gradient Oracle}\label{sec:inexactOracle} 
This section proposes an inexact gradient oracle that instantiates the function $\texttt{inexactGradientOracle}$ from \cref{algo:outer-loop-accelerated} and satisfies the uniform error condition from  \cref{lem:convergenceRateAccMin}. To this end, we recall from \cref{lem:alternativeGradient} that $\nabla \ell(u)=u-\rB_0 x^{\star}_{(u,v^{\star}(u))}$ where $v^{\star}(u)$ is the unique maximizer of the strongly concave function $f_u$ in \eqref{eq:innerMax} and $x^{\star}_{(u,v^{\star}(u))}$ is the unique solution of the EOT problem $\min_{x\in\mathcal K}\left\{\left(\rB_1^{\intercal} v^{\star}(u)-\rB_0^{\intercal} u\right)^{\intercal}x-\varepsilon \mathsf{H}(x)\right\}$. While there is no guarantee that Sinkhorn's method  returns a global solution of a given EOT problem after finitely many iterations, its accuracy can still be characterized as follows. 
\begin{lemma}[Proposition 8 in \cite{rioux2024entropic}]
\label{lem:SinkhornApproxError}
Let  $x^{\star}$ be the unique solution of   $\mathsf{EOT}_{c}(\mu_0,\mu_1)$. Then, for each $\delta>0$, there exists a number of steps $k$ depending on $\delta,\varepsilon, \|c\|_{\infty},\mu_0,$ and $\mu_1$ after which Sinkhorn's algorithm, as written in Algorithm 3 in \cite{rioux2024entropic}, outputs a vector,  $\tilde x$, satisfying $\|\tilde x-x^{\star}\|_{\infty}\leq e^{\delta}-1$.  
\end{lemma}
A precise bound on the number of steps required to achieve the claimed accuracy can be found in Equation (28) of \cite{rioux2024entropic}. By restricting the pair $(u,v)$ in the variational form to lie in $\mathbb B_{R_0}\times \mathbb B_{R_1}$, it follows from \cref{lem:SinkhornApproxError} that the EOT problems $\mathsf{EOT}_{\rB_1^{\intercal}v-\rB_0^{\intercal} u}(\mu_0,\mu_1)$ can be solved to within a desired precision $e^{\delta}-1$ with a number of steps that is independent of $(u,v)\in \mathbb B_{R_0}\times \mathbb B_{R_1}$.  

In a similar vein, optimization routines for maximizing $f_u$ generally return an approximate maximizer $\tilde v(u)$ in place of the true solution $v^{\star}(u)$. This issue is compounded by the fact that computing $\nabla f_u(v) = -v + \rB_1 x^{\star}_{(u,v)}$ as in \eqref{eq:fuDerivative} requires solving another EOT problem. Nevertheless, we show that the inexact gradient oracle given in \cref{algo:inexactGradient} below can satisfy the conditions of \cref{lem:convergenceRateAccMin} for any desired accuracy.

\begin{algorithm}
\caption{\texttt{inexactGradientOracle}$(u)$}
\label{algo:inexactGradient}
\begin{algorithmic}[1]
\Statex Given $u\in\mathbb B_{R_0}$,
\State $\tilde v(u)\gets $ approximate solution of $\sup_{\mathbb B_{R_1}}f_{u}$ 
\State $\tilde x_{(u,\tilde v(u))} \gets\text{approximate solution of } \inf_{x\in\mathcal K}\left\{\left( \rB^{\intercal}_1 \tilde v(u)-\rB_0^{\intercal} u\right)^{\intercal} x -\varepsilon\sH(x)\right\}$
\State 
\Return{$\tilde{\nabla}\ell(u)\coloneqq u-\rB_0\tilde x_{(u,\tilde v(u))}$}
\end{algorithmic}
\end{algorithm}

\begin{theorem}[Inexact oracle]
    \label{thm:inexactGradientOracle} 
    Fix $\delta,\tau >0$ and, in \cref{algo:inexactGradient}, suppose that $\|\tilde v(u)-v^{\star}(u)\|\leq \tau$ and $\|\tilde x_{(u,\tilde v(u))}-x^{\star}_{(u,\tilde v(u))}\|_{\infty}\leq e^{\delta}-1$ for each $u\in\mathbb B_{R_0}$. Then, \cref{algo:inexactGradient} is an inexact gradient oracle for $\ell$ satisfying 
    $
    \sup_{u\in\mathbb B_{R_0}}\|\tilde \nabla \ell(u)-\nabla \ell(u)\|\leq \|\rB_0\|_{\infty,2}\left(e^{\delta}-1 +2\varepsilon^{-1}\|\rB_1\|_{\mathrm{op}}\tau\right)$.
\end{theorem}
The proof of this result follows from known stability results on solutions to nonlinear programs with certain structural properties, see \cref{proof:thm:inexactGradientOracle} for the complete proof. Note that \cref{thm:inexactGradientOracle} is generic in the sense that any methods for approximately solving $\sup_{\mathbb B_{R_1}} f_u$ and the EOT problem can be applied. Importantly, \cref{thm:inexactGradientOracle} asserts that \cref{algo:inexactGradient} yields a gradient approximation that is compatible with the assumptions of \cref{lem:convergenceRateAccMin} for any desired accuracy. For completeness, we conclude this section with a brief comment on using inexact gradient methods to approximately solve $\sup_{\mathbb B_{R_1}} f_u$ in \cref{algo:inexactGradient}.

\begin{remark}[Maximizing $f_u$]
\label{rmk:oracleImplementation}
A natural approach to approximating $\{v^{\star}(u)\}=\argmax_{\mathbb B_{R_1}} f_u$ is via gradient methods. Given a gradient oracle $\tilde \nabla f_u$ with $\sup_{v\in\mathbb B_{R_1}}\|\tilde \nabla f_u(v)-\nabla f_{u}(v)\|\leq \eta$ for some $\eta>0$, we show in \cref{sec:oracleImplementation} that the output, $\tilde v(u)$, of inexact gradient methods with (\cref{algo:inner-loop-accelerated}) and without (\cref{algo:inner-loop-simple}) acceleration after $J$ iterations satisfy
\begin{equation}
\label{eq:approxSoln}
    \|\tilde v(u)-v^{\star}(u)\|^2\leq \begin{cases}
    \frac{4\left(\varepsilon^{-1}\|\rB_1\|^{2}_{\mathrm{op}}+1\right)\|v^{\star}(u)\|^2}{J(J+1)}+12R_1\eta, \;\text{using \cref{algo:inner-loop-accelerated}},
    \\
        \frac{\left(\varepsilon^{-1}\|\rB_1\|^{2}_{\mathrm{op}}+1\right)\|v_0-v^{\star}(u)\|^2}{J}+8R_1\eta, \;\text{using \cref{algo:inner-loop-simple} initialized at $v_0$},
    \end{cases}
\end{equation}
In particular, using $\tilde \nabla f_{u}(v)=-v+\rB_1 \tilde x_{(u,v)}$, where $\tilde x_{(u,v)}$ is the output of Sinkhorn's algorithm for solving $\inf_{x\in\mathcal K}\left\{(\rB_1^{\intercal}v-\rB_0^{\intercal} u)^{\intercal} x-\varepsilon \sH(x)\right\}$ for each $(u,v)\in \mathbb B_{R_0}\times \mathbb B_{R_1}$, one has $\sup_{v\in\mathbb B_{R_1}}\|\tilde \nabla f_u(v)-\nabla f_{u}(v)\|\leq \|\rB_1\|_{\infty,2}(e^{\delta}-1)$. In sum, by taking sufficiently many gradient descent steps and solving each EOT problem to a sufficient precision, the assumptions of \cref{thm:inexactGradientOracle}  and hence \cref{lem:convergenceRateAccMin} can be met.     
\end{remark}

\subsection{An Unaccelerated Gradient Method}
\label{sec:standardGradientMethod}

Following the presentation in \cref{subsec:algorithms}, we establish the convergence of a gradient method without acceleration for minimizing $\ell$. 

\begin{algorithm}
\caption{Inexact Gradient Method for $\ell$}
\label{algo:outer-loop-simple}
\begin{algorithmic}[1]
\Statex Given initialization  $u_0\in\mathbb B_{R_0}$, maximum number of iterations, $J$, and  the Lipschitz constant of $\nabla\left(\ell-\frac 12 \|\cdot\|^2\right)$, $L'$.
\For{$j = 0, \ldots, J-1$}
\State $g_j\gets \texttt{inexactGradientOracle}(u_j)$\Comment{ compute an approximation, $\tilde{\nabla}  \ell(u_j)$, of $\nabla  \ell(u_j)$ }
\State $u_{j+1}\gets \min\left(1,\frac{R_0}{\|u_j-(L'+1)^{-1}g_{j}\|}\right)(u_j-(L'+1)^{-1}g_{j})$ 
\EndFor
\end{algorithmic}
\end{algorithm}

Again, we rely on an oracle with a fixed uniform error to obtain approximate gradients for $\ell$ at each point $u\in\mathbb B_{R_0}$, see \cref{sec:inexactOracle} for additional discussion. 

\begin{proposition}[Convergence of \cref{algo:outer-loop-simple}]
\label{lem:convergenceRateMin}
    Fix $\eta'>0$ and set $h:u\in\mathbb R^{r_0}\mapsto \ell(u)-\frac 12 \|u\|^2.$ Then, $\nabla h$ is Lipschitz continuous with constant $L'=\frac{\lambda_{\max}(\rB_0^{\intercal}\rB_0)}{\lambda_{\min}(\rB_1^{\intercal}\rB_1)+\varepsilon}$ and, if $\|\tilde \nabla \ell(u)-\nabla \ell(u)\|\leq \eta'$ for each $u\in\mathbb B_{R_0}$, the iterates of \cref{algo:outer-loop-simple} satisfy
    \[
        \min_{j=0}^{J-1}\|{L'}\left(u_{j+1}-u_j\right)\|^2\le \frac{2L'(\ell(u_0)-\inf_{\mathbb R^{r_0}}\ell)}{J} + 8L'R_0\eta'.
    \]
\end{proposition}
The convergence properties of \cref{algo:outer-loop-simple} follow from the results of
\cite{nabou2025proximal} upon showing that our notion of approximate gradient corresponds to their gradient oracle model, see 
\cref{proof:lem:convergenceRateMin} for the proof.

As with \cref{algo:outer-loop-accelerated}, the quantity $\|{L'}\left(u_{j+1}-u_j\right)\|^2$ can be related to the gradient of $\ell$ so that it is reasonable to terminate \cref{algo:outer-loop-simple} if $\|g_j\|$ is small.

\begin{remark}[Comparison with \cref{algo:outer-loop-accelerated}]
   Compared with \cref{algo:outer-loop-accelerated}, \cref{algo:outer-loop-simple} does not benefit from an adaptive rate of convergence.  Nevertheless, this is the best known rate of convergence for solving such nonconvex problems using first-order methods, see Remark 13 in \cite{rioux2024entropic}.
   
    While Example 1 and Remark 4 in \cite{nabou2025proximal} imply that the bounds in \cref{lem:convergenceRateMin} can be improved so that the error term is $2(\eta')^2$ in place of $8L'R_0\eta'$, this is done at the cost of changing the step size of $\frac{1}{L'+1}$ in \cref{algo:outer-loop-simple} to $\frac{1}{2L'+1}$ which may yield slower convergence in practice. 

    Finally, the convergence rate in \cref{lem:convergenceRateMin} is measured in terms of $\min_{j=0}^{J-1}\|L'(u_{j+1}-u_j)\|^2$. To see why this is a reasonable figure of merit,  Theorem 3 in \cite{nabou2025proximal} asserts that, 
    \begin{equation}
    \label{eq:distToStationarity}
        \mathrm{dist}(0,\partial(\ell + \mathcal I_{\mathbb B_{R_0}})(u_{j+1}))\leq 2\|L'(u_{j+1}-u_j)\|+\eta',
    \end{equation}
    where $\mathcal I$ is the indicator function defined in \eqref{eq:indicator}.
    Thus, if $\eta'$ and  $\min_{j=0}^{J-1}\|L'(u_{j+1}-u_j)\|$ are small (which can be assured by taking sufficiently many steps in \cref{algo:outer-loop-simple}), one of the iterates will be near stationary in the sense that there exists $s\in\partial(\ell + \mathcal I_{\mathbb B_{R_0}})(u_{j+1})$ with small norm.  
    Moreover, if the iterates are interior points of $\mathbb B_{R_0}$, then $\partial(\ell + \mathcal I_{\mathbb B_{R_0}})(u_{j+1})=\nabla \ell(u_{j+1})$ so that \eqref{eq:distToStationarity} applies with $\|\nabla \ell(u_{j+1})\|$ in place of the distance to the subdifferential. 
\end{remark}

As noted in \cref{rmk:oracleImplementation}, gradient methods with and without acceleration can be leveraged to obtain suitable gradient oracles for \cref{algo:outer-loop-accelerated,algo:outer-loop-simple}.

\subsection{Gradient Methods as Inexact Oracles} 
\label{sec:oracleImplementation}
This section provides  the theoretical justification for the claims in \cref{rmk:oracleImplementation}. Namely, we establish the bounds on the distance to optimality presented in \eqref{eq:approxSoln} for \cref{algo:inner-loop-accelerated,algo:inner-loop-simple}. 

\begin{algorithm}
\caption{Accelerated Inexact Gradient Method for $f_{u}$}
\label{algo:inner-loop-accelerated}
\begin{algorithmic}[1] 
\Statex  Given $u\in\mathbb B_{R_0}$, initialization $y_0=0 \in \mathbb R^{r_1}$, maximum number of iterations, $J$, and  the Lipschitz constant of $\nabla f_u$, $L$, fix $\alpha_j=\frac{j+1}{2}$ and $\tau_j=\frac{2}{j+3}$.

\For{$j = 0, 1, \ldots, J-1$}
\State $\tilde x_j \gets\text{approximate solution of } \inf_{x\in\mathcal K}\left\{\left( \rB^{\intercal}_1 y_{j}-\rB_0^{\intercal} u\right)^{\intercal} x -\varepsilon\sH(x)\right\}$
\State $g_{j}\gets -\tilde{\nabla}f_u(y_j)=y_j-\rB_1\tilde x_j$ 
\State $w_j\gets \sum_{k=0}^j\alpha_kg_k$
\State $v_j\gets \min\left(1,\frac{R_1}{\|y_j-L^{-1}g_j\|}\right)(y_j-L^{-1}g_j)$
\State $z_{j}\gets -\min\left(1,\frac{R_1}{\|L^{-1}w_j\|}\right)L^{-1}w_j$
\State $y_{j+1}\gets \tau_jz_j+(1-\tau_j)v_j$
\EndFor
\end{algorithmic}
\end{algorithm}

\cref{algo:inner-loop-accelerated} corresponds to an accelerated variant of the gradient method with an inexact gradient oracle as studied in \cite{aspremont2008smooth}, where the inexactness stems from the fact that computing the gradient requires solving an EOT problem via, e.g., Sinkhorn's algorithm (recall that this can be done to within a uniform accuracy $e^{\delta}-1$ by \cref{lem:SinkhornApproxError}). The following convergence result follows essentially from the analysis in \cite{aspremont2008smooth} and from strong convexity of $-f_u$, see \cref{proof:lem:convergenceRateAccMax} for complete details.   

\begin{proposition}[Convergence of \cref{algo:inner-loop-accelerated}]
\label{lem:convergenceRateAccMax}
    Fix $\eta>0$ and assume that $\|\tilde \nabla f_u(v)-\nabla f_u(v)\|\leq \eta$ for each $v$ satisfying $\|v\|\leq R_1$. Then, for  $j=0,\dots, J-1$, the iterates, $v_j$, of \cref{algo:inner-loop-accelerated} satisfy
    \[
        \frac{1}{2}\|v_j - v^{\star} \|^2 \le   f_u(v^{\star}) -f_{u}(v_j) \le \frac{2(\varepsilon^{-1}\|\rB_1\|_{\mathrm{op}}^2+1) \|v^{\star} \|^2}{(j+1)(j+2)} +6{R_1}\eta,
    \]
    where $v^{\star}$ is the unique maximizer of $f_u$. Furthermore, we may set $L=1+\varepsilon^{-1}\|\rB_1\|^{2}_{\mathrm{op}}$ in \cref{algo:inner-loop-accelerated}. 
\end{proposition} 

As with the other inexact methods studied previously, the convergence rate result involves one term that accounts for the gradient error and another which tracks the optimization error. This latter term converges to $0$ at the rate of $O(1/j^2)$ which is known to be optimal among first-order methods for this choice of convergence metric, see \cite{nesterov2018lectures}.  

As noted previously, it is standard to terminate gradient methods early if $\|g_j\|<\epsilon$ at a given iteration for some desired tolerance $\epsilon>0$. Given that $f_u$ is $1$-strongly concave, we have from Exercise 12.59 in \cite{rockafellar1998variational} that 
\begin{equation}
\label{eq:boundtoOpt}
    \|y_j-v^{\star}\|^2\leq -\nabla f_u(y_j)^{\intercal}(y_j-v^{\star})\leq (\|g_j\|+\|\tilde \nabla f_u(y_j)-\nabla f_u(y_j)\|)\|y_j-v^{\star}\|,  
\end{equation}
noting that $\nabla f_u(v^{\star})=0$. Under the inexact gradient assumption in \cref{lem:convergenceRateAccMax}, it follows that \[
\|v_j-v^{\star}\|\leq \|y_j-v^{\star}\|+\|y_j-v_j\|\leq \|g_j\|+\eta + L^{-1}\|g_j\| 
\]
if $\|y_j-L^{-1}g_j\|\leq R_1$ so that terminating the algorithm once $\|g_j\|$ is small can still yield a sensible approximation of the true maximizer of $f_u$ which is compatible with the conditions of \cref{thm:inexactGradientOracle}.

We conclude this section by establishing a similar result for a gradient method without acceleration, \cref{algo:inner-loop-simple}.

\begin{algorithm}
\caption{Inexact Gradient Method for $f_{u}$}
\label{algo:inner-loop-simple}
\begin{algorithmic}[1] 
\Statex Given $u\in\mathbb R^{r_0}$, initialization $v_0 \in \mathbb B_{R_1}$, maximum number of iterations, $J$, and  the Lipschitz constant of $\nabla f_u$, $L$,
\For{$j = 0, 1, \ldots, J-1$}
\State $\tilde x_j \gets\text{approximate solution of } \inf_{x\in\mathcal K}\left\{\left( \rB^{\intercal}_1 v_{j}-\rB_0^{\intercal} u\right)^{\intercal} x -\varepsilon\sH(x)\right\}$
\State $g_{j}\gets -\tilde{\nabla}f_u(v_j)\coloneqq v_j-\rB_1\tilde x_j$
\State $v_{j+1}\gets \min\left(1,\frac{R_1}{\|v_j-L^{-1}g_{j}\|}\right)(v_j-L^{-1}g_{j})$
   \EndFor
\end{algorithmic}
\end{algorithm}

\begin{proposition}[Convergence of \cref{algo:inner-loop-simple}]
\label{lem:convergenceRateMax}
    Fix $\eta>0$ and assume that $\|\tilde \nabla f_u(v)-\nabla f_u(v)\|\leq \eta$ for each $v$ satisfying $\|v\|\leq R_1$. Then,  the iterates, $v_j$, of \cref{algo:inner-loop-simple} satisfy
    \[
        \frac{1}{2}\min_{j=0}^{J-1}\|v_j - v^{\star} \|^2 \le\min_{j=0}^{J-1}\left(   f_u(v^{\star}) -f_{u}(v_j)\right) \le \frac{(\varepsilon^{-1}\|\rB_1\|_{\mathrm{op}}^2+1) \|v_0 - v^{\star} \|^2}{2J} + 4{R_1}\eta,
    \]
    where $v^{\star}$ is the unique maximizer of $f_u$. Furthermore, we may set $L= 1+\varepsilon^{-1}\|\rB_1\|^{2}_{\mathrm{op}}$ in \cref{algo:inner-loop-simple}. 
\end{proposition}

This result follows by showing that our notion of approximate gradient is compatible with the results of \cite{devolder2014first}, see \cref{proof:lem:convergenceRateMax}. While the theoretical convergence rate derived in \cref{lem:convergenceRateMax} is slower than that in \cref{lem:convergenceRateAccMax}, this does not necessarily imply that \cref{algo:inner-loop-accelerated} outperforms \cref{algo:inner-loop-simple} on any given problem.  

While \cref{lem:convergenceRateMax} only applies to the minimum iterate, which is generally unknown, Equation (34) in \cite{devolder2014first} asserts that the same rate holds for $f_u(v^{\star})-f_u(\frac 1 J\sum_{j=0}^{J-1}v_j)$ and $\|\frac 1J\sum_{j=0}^{J-1}v_j-v^{\star}\|^2$ in place of $\min_{j=0}^{J-1}\left(   f_u(v^{\star}) -f_{u}(v_j)\right)$ and $\min_{j=0}^{J-1}\|v_j - v^{\star} \|^2$. Thus, that average point can be used to justify the discussion in \cref{rmk:oracleImplementation}. Alternatively, if $\|g_j\|$ is sufficiently small at a given iteration, \eqref{eq:boundtoOpt} yields a bound on $\|v_j-v^{\star}\|$.

\subsection{Computational Complexity} 
\label{subsec:complexity}
We now discuss the complexity of this algorithm,   
 separating the discussion of the complexity of implementing this approach into an analysis of the cost of decomposing $\rM$ as $\rB_1^{\intercal}\rB_1-\rB_0^{\intercal}\rB_0$ and that of running the algorithm for a prescribed number of steps given  $\rB_0$ and $\rB_1$. 

 {\textbf{Memory-efficient and computation-efficient cost decomposition.}} To decompose  $\rM$ as $\rB_1^{\intercal}\rB_1-\rB_0^{\intercal}\rB_0$, 
we first note that some problems admit canonical decompositions, e.g., when the cost is the squared Euclidean distance or inner product and $p=2$, see, e.g., \cite{scetbon2022linear}. A generic approach, described in the proof of \cref{thm:variationalForm}, is to compute the eigendecomposition of $\rM$ and let $\rB_1=\mathrm{diag}(\sqrt{\lambda_1},\dots,\sqrt{\lambda_{r_1}}) \mathbf V_1$, where $(\lambda_1,\dots,\lambda_{r_1})$ is the collection of all positive eigenvalues of $\rM$ and $\mathbf V_1\in\mathbb R^{r_1\times k}$ has $i$-th row given by the (transpose of the) $i$-th eigenvector $v_i$. $\rB_0$ is constructed similarly using the absolute value of the negative eigenvalues. Directly  diagonalizing $\rM$, an  $N_0N_1\times N_0N_1$ matrix, has a complexity of $O((N_0N_1)^3)$ and can thus be prohibitively expensive, see Chapter 42 in \cite{hogben2006handbook}. 
 
For the case $p=2$, this complexity can be reduced to $O(\max\{N_0^3,N_1^3\})$ by leveraging properties of Kronecker products as explained below. 
We note that when $p=2$, 
 \[
 \begin{aligned}
\mathsf{EGW}^{\varepsilon}_{2}(\mu_0,\mu_1)&=\int \kappa_0(x,x')^2 d\mu_0\otimes \mu_0(x,x') + \int \kappa_1(y,y')^2 d\mu_1\otimes \mu_1(y,y') 
\\
&+ \inf_{\pi\in\Pi(\mu_0,\mu_1)}\left\{-2\int \kappa_0(x,x')\kappa_1(y,y')d\pi\otimes \pi(x,y,x',y')+\varepsilon\mathsf{KL}(\pi\|\mu_0\otimes \mu_1)\right\},
\end{aligned} 
 \]
 so that 
 it suffices to solve the regularized quadratic program,  \eqref{eq:entropicQP}, with $\rM=-4\mathrm K_0\otimes \mathrm{K}_1$, where $\mathrm K_0\otimes \mathrm{K}_1$ is the Kronecker product of
 the pairwise cost matrices $\mathrm{K}_0\in\mathbb R^{N_0\times N_0},\mathrm{K}_1\in\mathbb R^{N_1\times N_1}$ whose $ij$-th entries are given by $\kappa_0\left(x_0^{(i)},x_0^{(j)}\right)$ and $\kappa_1\left(x_1^{(i)},x_1^{(j)}\right)$ respectively. Now, if $(\lambda_0,v_0)$ and $(\lambda_1,v_1)$ are eigenvalue-eigenvector pairs for $\mathrm{K}_0$ and $\mathrm{K}_1$ respectively, then all eigenvalue-eigenvector pairs of $\mathrm{K}_0\otimes \mathrm{K}_1$ are of the form $(\lambda_0\lambda_1,v_0\otimes v_1)$ by Theorem 4.2.12 in \cite{horn2012matrix}. It suffices, therefore, to separately diagonalize $\mathrm{K}_0$ and $\mathrm{K}_1$ and then to construct $\rB_0$ and $\rB_1$ as described above by separating the positive and negative eigenvalue pairs. The overall complexity of this routine is thus $O(\max\{N_0^3,N_1^3\})$, which is dictated by the cost of diagonalizing both matrices. %

 Although separately diagonalizing $\mathrm K_0$ and $\mathrm K_1$ reduces the cost of the decomposition, explicitly forming the matrices $\rB_0$ and $\rB_1$ remains memory-intensive, since they have sizes $r_0\times N_0N_1$,$r_1\times N_0N_1$  respectively. To overcome this, we propose a memory-efficient implementation based on factorizing these matrices first. Indeed, if $\mathrm K_0$ and $\mathrm K_1$ are symmetric, we may express them as
\[
\mathrm{K_0}= (\mathrm{K_0^+})^{\intercal}\mathrm{K_0^+}-(\mathrm{K_0^-})^{\intercal}\mathrm{K_0^-},
\qquad
\mathrm{K_1}= (\mathrm{K_1^+})^{\intercal}\mathrm{K_1^+}-(\mathrm{K_1^-})^{\intercal}\mathrm{K_1^-},
\]
where $\mathrm{K_0}^{+}\in \mathbb R^{r_{K_0^+}\times N_0}$, $\mathrm{K_0}^{-}\in \mathbb R^{r_{K_0^-}\times N_0}$, $\mathrm{K_1}^{+}\in \mathbb R^{r_{K_1^+}\times N_1}$, $\mathrm{K_1}^{-}\in \mathbb R^{r_{K_1^-}\times N_1}$. Thus,
applying the mixed-product property, $(\mathrm{A}\otimes \mathrm{B})(\mathrm{C}\otimes \mathrm{D})=(\mathrm{A}\mathrm{C})\otimes(\mathrm{B}\mathrm{D})$ for matrices $\mathrm{A},\mathrm{B},\mathrm{C},$ and $\mathrm{D}$ of a compatible size,
we see that
\[
\begin{aligned}
\mathrm K_0\otimes\mathrm K_1
&=
(\mathrm{K_0^+})^{\intercal}\mathrm{K_0^+}\otimes (\mathrm{K_1^+})^{\intercal}\mathrm{K_1^+}
+
(\mathrm{K_0^-})^{\intercal}\mathrm{K_0^-}\otimes (\mathrm{K_1^-})^{\intercal}\mathrm{K_1^-}
\\
&\quad
-
(\mathrm{K_0^+})^{\intercal}\mathrm{K_0^+}\otimes (\mathrm{K_1^-})^{\intercal}\mathrm{K_1^-}
-
(\mathrm{K_0^-})^{\intercal}\mathrm{K_0^-}\otimes (\mathrm{K_1^+})^{\intercal}\mathrm{K_1^+}
\\
&=
\left(\mathrm K_0^{+}\otimes \mathrm K_1^{+}\right)^{\intercal}\left(\mathrm K_0^{+}\otimes \mathrm K_1^{+}\right)
+
\left(\mathrm K_0^{-}\otimes \mathrm K_1^{-}\right)^{\intercal}\left(\mathrm K_0^{-}\otimes \mathrm K_1^{-}\right)
\\
&\quad
-
\left(\mathrm K_0^{+}\otimes \mathrm K_1^{-}\right)^{\intercal}\left(\mathrm K_0^{+}\otimes \mathrm K_1^{-}\right)
-
\left(\mathrm K_0^{-}\otimes \mathrm K_1^{+}\right)^{\intercal}\left(\mathrm K_0^{-}\otimes \mathrm K_1^{+}\right).
\end{aligned}
\]
Conclude that $\rM = \rB_1^{\intercal}\rB_1-\rB_0^{\intercal}\rB_0$ with
\[
\rB_0= 2
\begin{pmatrix}
\mathrm K_0^{+}\otimes \mathrm K_1^{+}\\
\mathrm K_0^{-}\otimes \mathrm K_1^{-}
\end{pmatrix}
\in \mathbb R^{r_1\times N_0N_1},
\qquad
\rB_1=2
\begin{pmatrix}
\mathrm K_0^{+}\otimes \mathrm K_1^{-}\\
\mathrm K_0^{-}\otimes \mathrm K_1^{+}
\end{pmatrix}
\in \mathbb R^{r_0\times N_0N_1},
\]
where we set 
$
r_0=r_{K_0^+}r_{K_1^+}+r_{K_0^-}r_{K_1^-},
r_1=r_{K_0^+}r_{K_1^-}+r_{K_0^-}r_{K_1^+}.
$

Crucially, this enables us to perform matrix-vector products with $\rB_0$ and $\rB_1$ without computing the underlying Kronecker products. Indeed, if $x=\mathrm{vec}(X)$ with $X\in\mathbb R^{N_1\times N_0}$, we have that 
\[
\rB_0x=2
\begin{pmatrix}
\mathrm{vec}\!\left(\mathrm K_1^{+}X(\mathrm K_0^{+})^{\intercal}\right)\\
\mathrm{vec}\!\left(\mathrm K_1^{-}X(\mathrm K_0^{-})^{\intercal}\right)
\end{pmatrix},
\qquad \rB_1x=2
\begin{pmatrix}
\mathrm{vec}\!\left(\mathrm K_1^{-}X(\mathrm K_0^{+})^{\intercal}\right)\\
\mathrm{vec}\!\left(\mathrm K_1^{+}X(\mathrm K_0^{-})^{\intercal}\right)
\end{pmatrix}
,
\]
and that $\rB_0^{\intercal}u=2
\mathrm{vec}\!\left(
(\mathrm K_1^{+})^{\intercal}U_1\mathrm K_0^{+}
+
(\mathrm K_1^{-})^{\intercal}U_2\mathrm K_0^{-}
\right),$
\text{ for } 
\[
v=
\begin{pmatrix}
\mathrm{vec}(U_1)\\
\mathrm{vec}(U_2)
\end{pmatrix}\text{ with }
U_1\in\mathbb R^{r_{K_1^+}\times r_{K_0^+}},
\;
U_2\in\mathbb R^{r_{K_1^-}\times r_{K_0^-}}, 
\]
and  
$\rB_0^{\intercal}v=2
\mathrm{vec}\!\left(
(\mathrm K_1^{-})^{\intercal}V_1\mathrm K_0^{+}
+
(\mathrm K_1^{+})^{\intercal}V_2\mathrm K_0^{-}
\right),$
\text{ for } 
\[
u=
\begin{pmatrix}
\mathrm{vec}(V_1)\\
\mathrm{vec}(V_2)
\end{pmatrix}\text{ with }
V_1\in\mathbb R^{r_{K_1^-}\times r_{K_0^+}},
\;
V_2\in\mathbb R^{r_{K_1^+}\times r_{K_0^-}}.
\]
This implementation therefore only stores  the factors $\mathrm K_0^\pm,\mathrm K_1^\pm$ which require total storage of size $(r_{\mathrm K_0^{+}}+r_{\mathrm K_0^{-}})N_0 + (r_{\mathrm K_1^{+}}+r_{\mathrm K_1^{-}})N_1$ compared to $(r_0 + r_1)N_0N_1$ for the naive implementation. This substantially reduces the memory cost for large $N_0$ and $N_1$.

{\textbf{Hyperparameters.}} When using \cref{algo:outer-loop-accelerated} with the theoretically convergent step sizes, it is required to compute the Lipschitz constant, $L$, of $\nabla \ell$ which is expressed in terms of $\lambda_{\max}(\rB_0^{\intercal}\rB_0)$ and $\lambda_{\min}(\rB_1^{\intercal}\rB_1)$. Given that eigenvalue decompositions are used to form $\rB_0$ and $\rB_1$, it is straightforward to obtain these values. To set the radii $R_0 = \|\rB_0\|_{1,2}$ and $R_1=\|\rB_1\|_{1,2}$,  note that $\|\cdot\|_{1,2}$ is the maximum $2$-norm of the columns of the input matrix.

{\textbf{Complexity of the algorithms.}} Once  the matrices $\rB_0,\rB_1$ for the decomposition of the cost matrix  $\rM=\rB_1^{\intercal}\rB_1-\rB_0^{\intercal}\rB_0$ have been obtained,  each iteration of Algorithm \ref{algo:outer-loop-accelerated} requires solving an EOT problem which can be done using $O(SN_0 N_1)$ operations up to $\log$ factors using $S$ iterations of Sinkhorn's algorithm, see Section 4.3 in \cite{peyre2019computational}. 

Beyond this, matrix vector products with $\rB_i,\rB_i^{\intercal}$ for $i\in\{0,1\}$ are required; the naive implementation of these products requires $O(r_iN_0N_1)$ operations.
As discussed above when $p=2$ we use an equivalent implementation which reduces the memory overhead of the method, but the complexity still  remains $O(r_iN_0N_1)$.

The cost of querying the gradient oracle depends on the specific implementation. In the case that $K$ iterations of a first-order method are used, as described in \cref{rmk:oracleImplementation}, the complexity of obtaining an approximate gradient is $O(K(S+\max\{r_0,r_1\})N_0N_1)$ up to $\log$ factors assuming that the EOT problems are solved using $S$ iterations of Sinkhorn's method.  

Altogether, $J$ iterations of \cref{algo:outer-loop-accelerated} requires a number of operations scaling as $O(JK(S+\max\{r_0,r_1\})N_0N_1)$ under the assumption that $\rB_0$ and $\rB_1$ are given and that all EOT problems are solved by $S$ iterations of Sinkhorn's algorithm.

\subsection{Numerical Validation} 
\label{sec:validation}
We now validate the convergence rates derived in \cref{lem:convergenceRateMin,lem:convergenceRateAccMin} and compare these methods to the  mirror descent algorithm \cite{peyre2016gromov,scetbon2022linear} which is most often used in practice. Throughout, we use the POT package's implementation of Sinkhorn's algorithm \cite{flamary2021pot} and, unless  otherwise stated, use the gradient method (\cref{algo:inner-loop-simple}) to solve the concave subproblems $\sup_{\mathbb B_{R_1}} f_u$ in \cref{algo:outer-loop-simple,algo:outer-loop-accelerated}. 

\begin{figure}
    \centering
    \begin{subfigure}[b]{0.48\textwidth}
         \centering
         \includegraphics[width=0.7\textwidth]{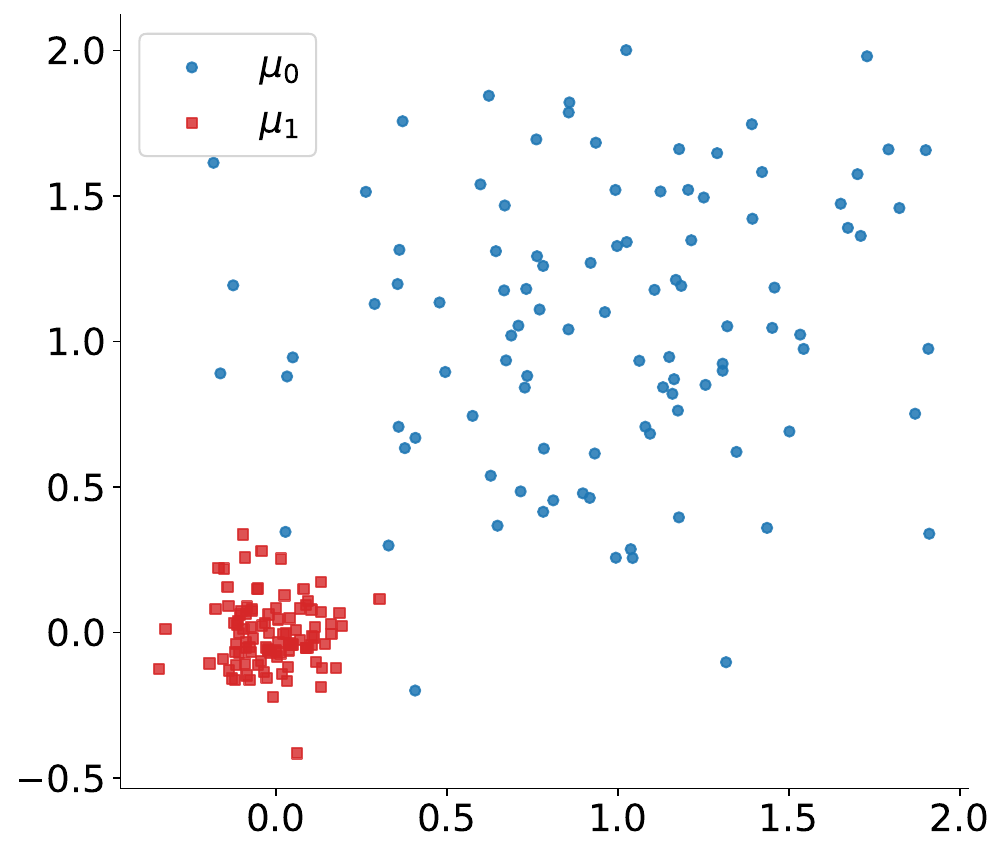}
         \caption{}
         \label{fig:euclidaen_visualisation}
     \end{subfigure}
     \begin{subfigure}[b]{0.48\textwidth}
         \centering
         \includegraphics[width=0.6\textwidth]{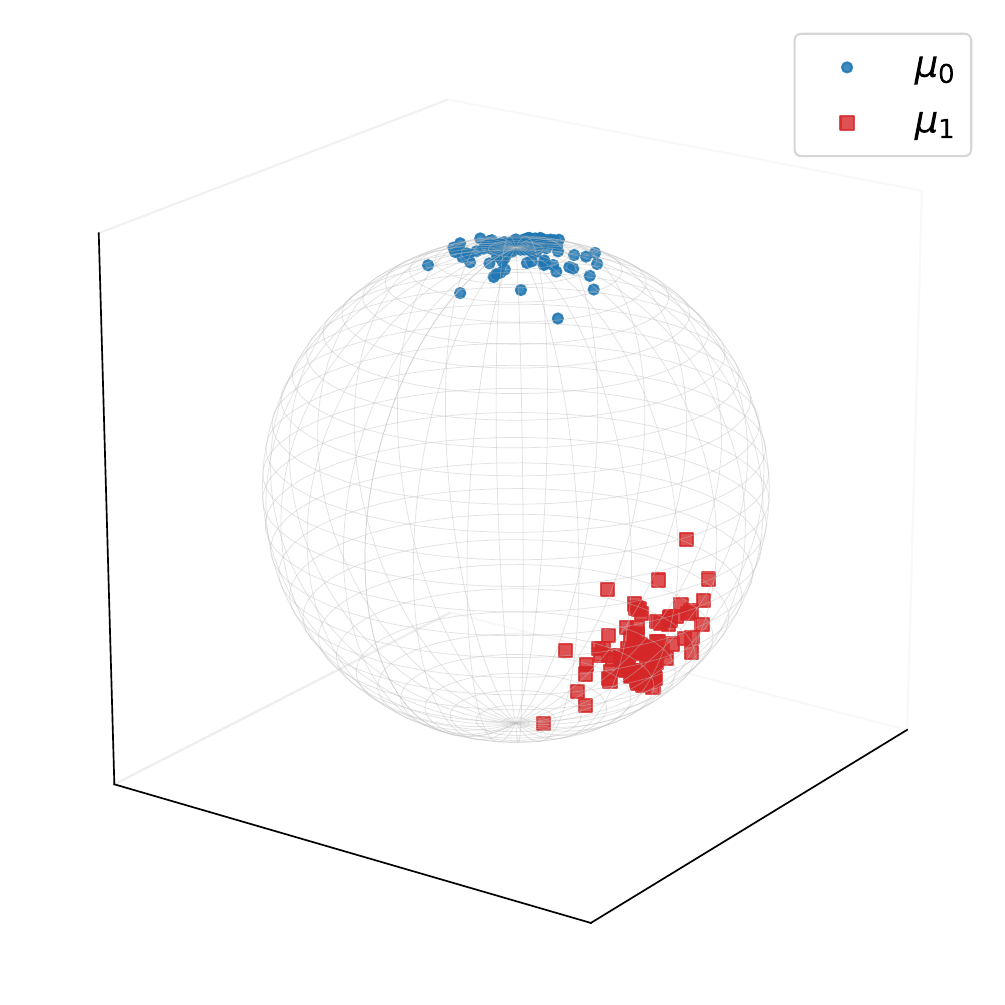}
         \caption{}
         \label{fig:sphere_visualisation}
     \end{subfigure}
     \caption{Visualizations of the datasets used in the convergence rate experiments. \textbf{(a):} samples from two-dimensional Gaussian distributions with $\mu_0$ and $\mu_1$ set to be uniform over the realizations. 
     \textbf{(b):} points on the unit sphere on which $\mu_0$ and $\mu_1$ are uniformly supported. In this case $\mu_1$ is obtained from $\mu_0$ by a fixed rotation applied to the support points.}
\end{figure}

{\textbf{Convergence rates.}} To empirically validate \cref{lem:convergenceRateAccMin,lem:convergenceRateMin},
we track the progress of  
 $\|u_{j+1} - u_j \|^2 $ for the gradient method (\cref{algo:outer-loop-simple}), and $\min_{j =1}^k \|u_{j} - z_j \|^2$ (at the $k$-th iteration of the loop) for the accelerated gradient method (\cref{algo:outer-loop-accelerated}) on some examples in the convex and nonconvex regimes. 
 
In the first example, $\mu_0$ is the uniform distribution on $N=100$ data points sampled i.i.d. from
$
{N}\left(
\left(
1\;
1
\right)^{\intercal},
0.5^2 \Id
\right)$ and
$\mu_1$ is defined analogously with $100$ i.i.d. samples from
$
N\left(
\left(
0\;
0
\right)^{\intercal},
0.11^2 \Id\right)$.
In this case, we choose $\kappa_0=\kappa_1$ to be the squared Euclidean cost; for $p=2$, the resulting  matrix $\rM$ used in the variational form \eqref{eq:variationalForm} can be decomposed as $\rB_1^{\intercal}\rB_1-\rB_0^{\intercal}\rB_0$ where, for this particular example, $r_0 = 6$ and $r_1 = 10$. We consider both the convex and nonconvex regimes, in the former case we take $\varepsilon =  \lambda_{\max}{(\rB_0^\intercal \rB_0)} - \lambda_{\min}{(\rB_1^\intercal \rB_1)} \approx 2367.34$ and in the latter,  $\varepsilon = 0.01$. The progress of the iterates when computing $\mathsf{EGW}^{\varepsilon}_2(\mu_0,\mu_1)$ using these choices of $\varepsilon$ are presented in \cref{fig:convergence_all}.   

  In the second example, $\mu_0$ is the uniform distribution on $N=100$ data points from the unit sphere in $\mathbb R^3$ and $\mu_1$ is obtained by rotating $\mu_0$ according to a fixed rotation matrix. In this case, we choose $\kappa_0=\kappa_1$ to be the geodesic distance on the  sphere and set $p=2$. For this particular example, $r_0 = 9802$ and $r_1 = 198$. Plots of the progress of the iterates are included in \cref{fig:convergence_all} both in the convex ($\varepsilon = \lambda_{\max}{(\rB_0^\intercal \rB_0)} - \lambda_{\min}{(\rB_1^\intercal \rB_1) \approx 2815.17}$) and nonconvex ($\varepsilon = 0.01$) regimes. The details of the hyperparameters and experimental setup are outlined in \cref{sec:experimental-details}. 

As illustrated in \cref{fig:convergence_all}, the iterates are consistent with the theoretical convergence rates. We observe that the iterates of the accelerated gradient method exhibit a staircase shape which is standard in momentum-based methods.

 \begin{figure}[htbp]
     \centering
     \begin{subfigure}[b]{0.48\textwidth}
         \centering
         \includegraphics[width=\textwidth]{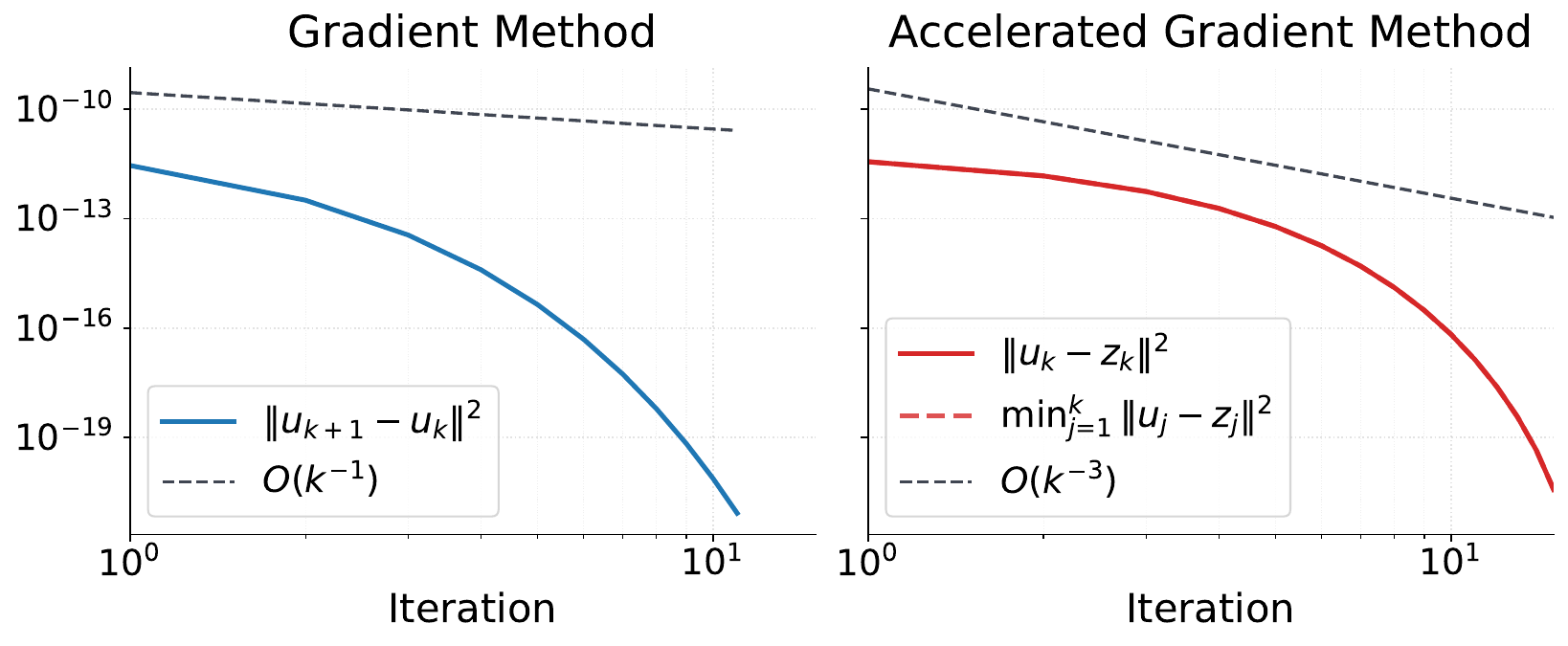}
         \caption{}
         \label{fig:simple_conv}
     \end{subfigure}
     \hfill
     \begin{subfigure}[b]{0.48\textwidth}
         \centering
         \includegraphics[width=\textwidth]{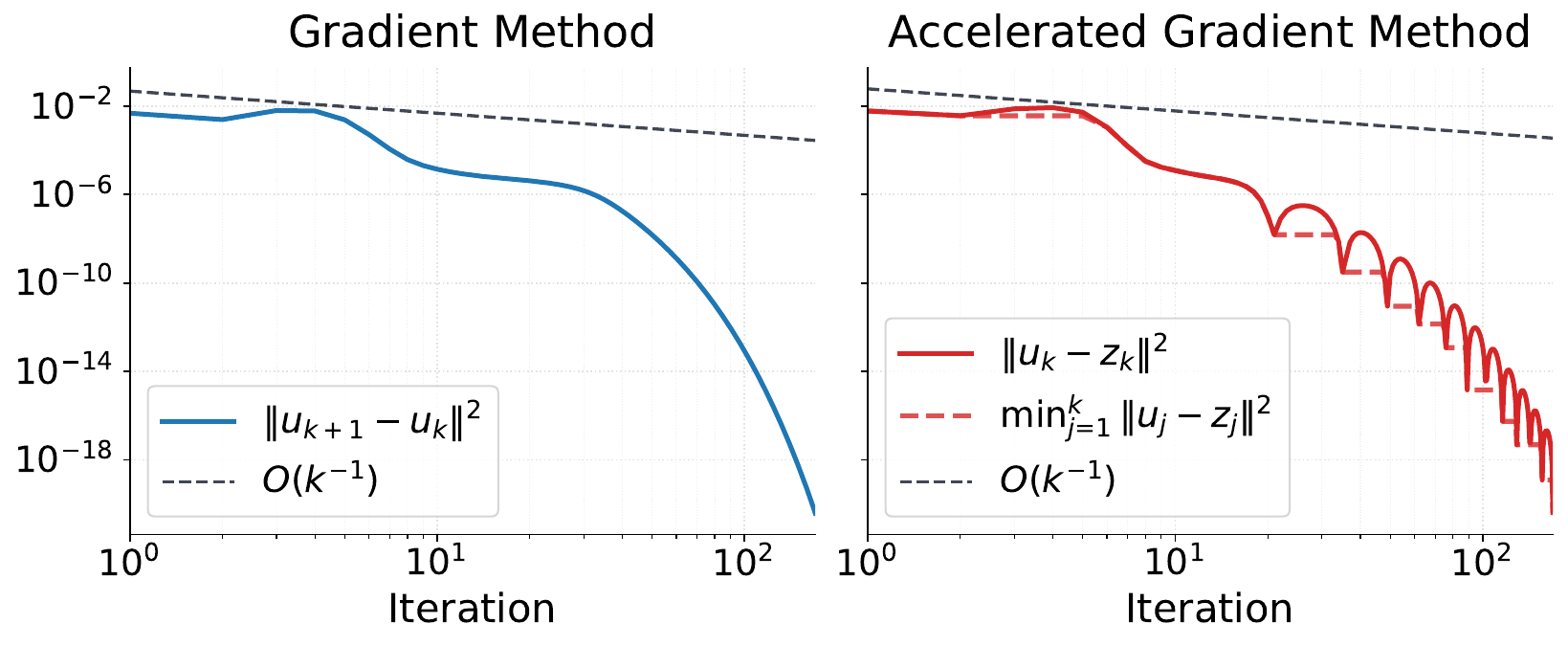}
         \caption{}
         \label{fig:simple_nonconv}
     \end{subfigure}

     \vspace{1em} %

     \begin{subfigure}[b]{0.48\textwidth}
         \centering
         \includegraphics[width=\textwidth]{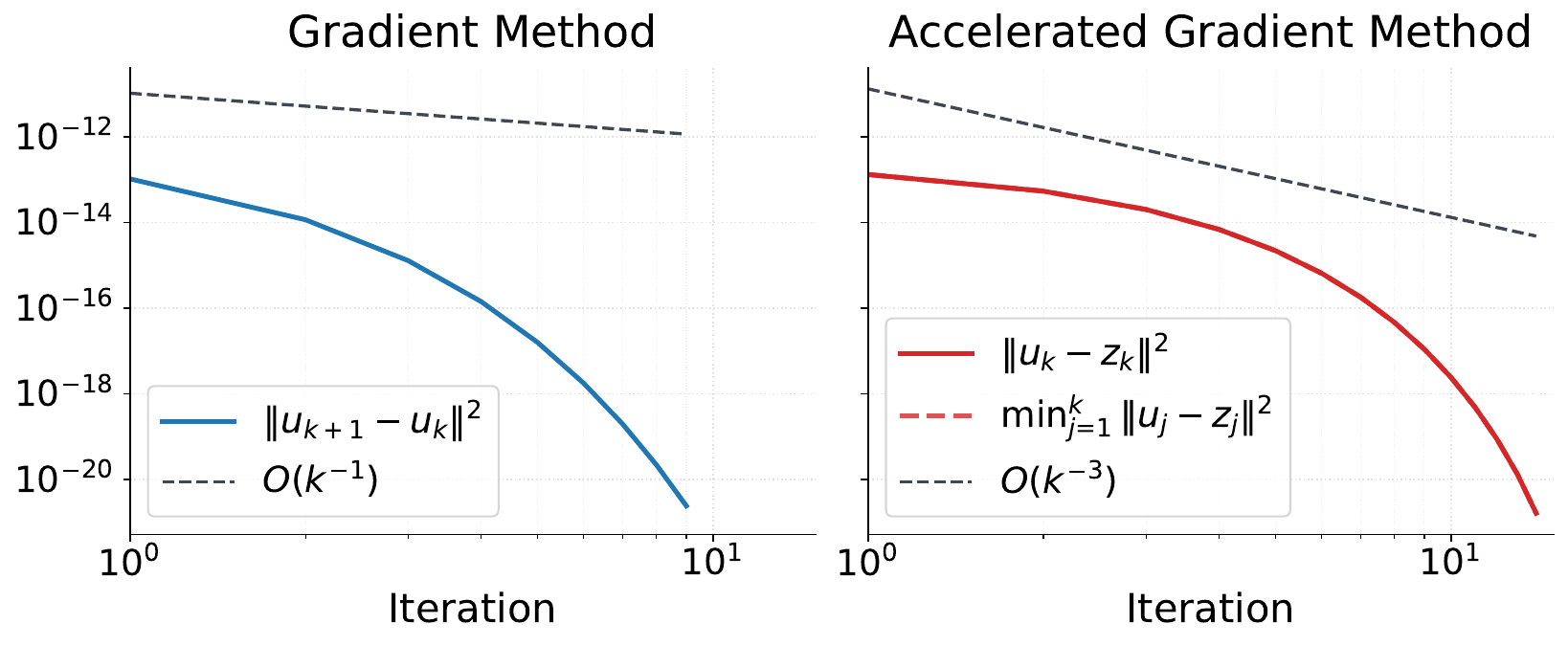}
         \caption{}
         \label{fig:sphere_conv}
     \end{subfigure}
     \hfill
     \begin{subfigure}[b]{0.48\textwidth}
         \centering
         \includegraphics[width=\textwidth]{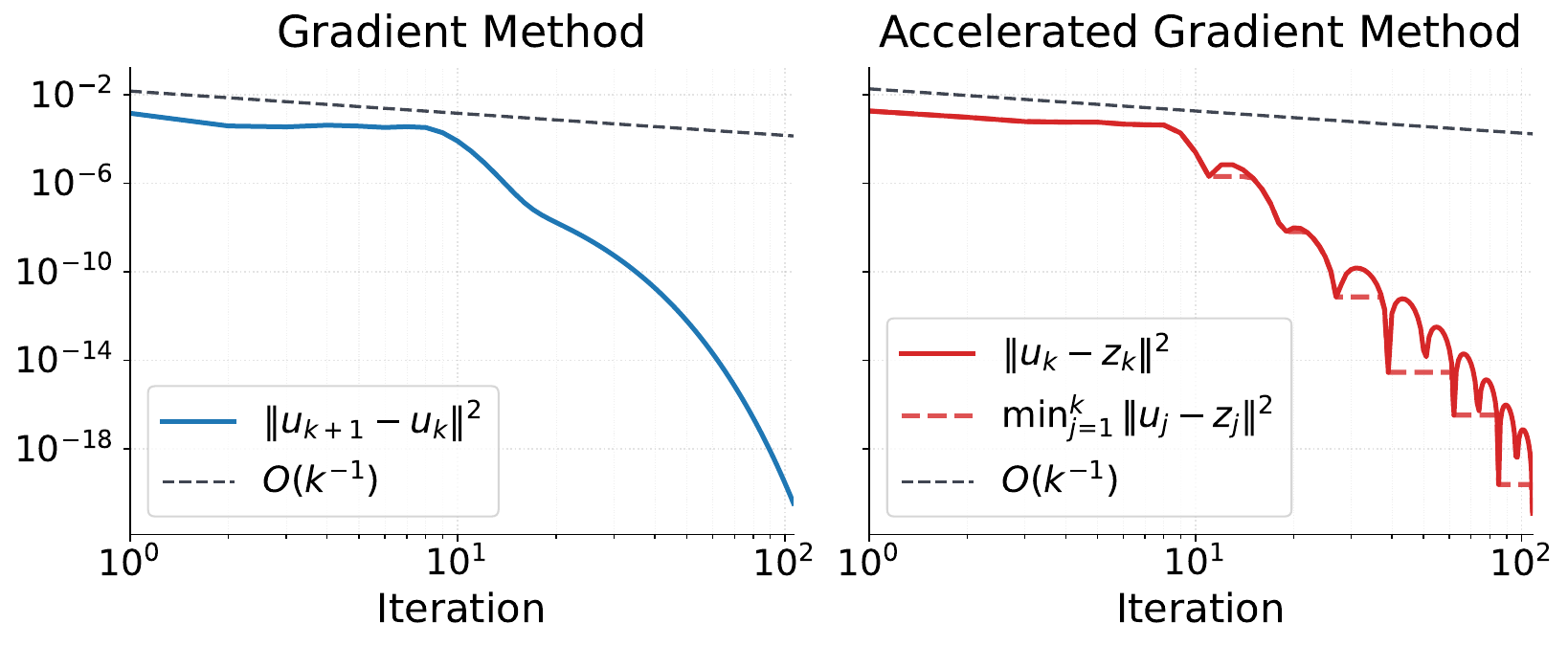}
         \caption{}
         \label{fig:sphere_nonconv}
     \end{subfigure}
     
     \caption{Progress of the iterates for the theoretical convergence rate validation. \textbf{(a)-(b):} progress for the first example, where $\mu_0$ and $\mu_1$ are constructed from samples of normal distributions and $\kappa_0,\kappa_1$ are taken to be squared Euclidean distances. \textbf{(a)} corresponds to the convex case where $\varepsilon\geq\lambda_{\max}{(\rB_0^\intercal \rB_0)}) - \lambda_{\min}{(\rB_1^\intercal \rB_1)}$ whereas \textbf{(b)} depicts the case where $\varepsilon<\lambda_{\max}{(\rB_0^\intercal \rB_0)}) - \lambda_{\min}{(\rB_1^\intercal \rB_1)}$ and convexity is not guaranteed. \textbf{(c)-(d)} are arranged similarly for the second example, where $\mu_0$ and $\mu_1$ are constructed from points on the unit sphere and $\kappa_0,\kappa_1$ are the geodesic distances.}
     \label{fig:convergence_all}
\end{figure}

{\textbf{Runtime comparison.}} As a supplement to the previous experiment, we compare the runtime of our methods to the mirror descent algorithm \cite{peyre2016gromov,scetbon2022linear} as implemented in  \cite[Algorithm 2]{scetbon2022linear}.

{To study the effect of the rank, we construct the cost matrix $\rM$ so that it exhibits a prescribed low-rank structure. Precisely, we generate matrices $\mathrm A_0,\mathrm A_1\in\mathbb R^{r\times N}$ with entries sampled uniformly at random from the interval $[0,1]$, and define the within-space cost matrices $\mathrm K_0=\mathrm A_0^{\intercal}\mathrm A_0, \mathrm K_1=\mathrm A_1^{\intercal}\mathrm A_1.$ We normalize each factor so that the maximum entry of its corresponding cost matrix is one, ensuring that both cost matrices are of a similar scale. Since $\mathrm K_0$ and $\mathrm K_1$ are psd, their Kronecker product is also psd, and hence $\rM=-4(\mathrm K_0\otimes\mathrm K_1)=-\rB_0^{\intercal}\rB_0$ is nsd with $\rB_0=2(\mathrm A_0\otimes\mathrm A_1)$ and the decomposition of $\rM$ is therefore known explicitly. As such, it is not required to solve a concave maximization problem to obtain a gradient oracle for minimizing $\ell$, recalling \cref{rmk:oracleImplementation}.  Moreover, $\rM\in\mathbb R^{N^2\times N^2}$ has rank at most $r^2$.}
We underscore that the mirror descent algorithm can also leverage such a decomposition to improve its per-iteration complexity \cite{scetbon2022linear}. 

With this setup in mind, we  generate $\rM$ according to the above procedure for each combination of $N\in\{2^k\}_{k=4}^{14}$ and $r\in\{4^k\}_{k=0}^3$, solve the resulting EGW problem, \eqref{eq:entropicQP}, where $\mu_0$ and $\mu_1$ are taken to be uniformly weighted on $N$ points using \cref{algo:outer-loop-simple,algo:outer-loop-accelerated}, and the mirror descent algorithm. For each such problem, we record the time required for each algorithm to converge neglecting the time required to construct $\rA_0,\rA_1$ and report the average runtime over $50$ randomly generated instances for each pair $(N,r)$ in \cref{fig:runtimeplot}. The shaded area between curves represents the range between the maximum and minimum runtime for the $50$ runs for a given pair $(N,r)$. Again we consider two regimes: convex, where we set $\varepsilon =  \lambda_{\max}{(\rB_0^\intercal \rB_0)}) - \lambda_{\min}{(\rB_1^\intercal \rB_1)}$ calculated separately for each problem instance (see \cref{sec:experimental-details} for exact values) and nonconvex where we set $\varepsilon=1$ for all problem instances.

\cref{fig:runtime-convex} illustrates the average runtimes for the convex setting. As shown, both the simple and accelerated methods adhere to the theoretical runtime complexity of $O(N^2)$. These methods achieve comparable runtimes across all rank values and for large values of $N$. Furthermore, both methods exhibit convergence times similar to mirror descent, with our algorithms converging slightly faster as $N$ increases. Similar trends are evident in the nonconvex setting (\cref{fig:runtime-nonconvex}), with the slight difference that the accelerated method requires more time to converge than both mirror descent and the gradient method. In this setting, the theoretical runtime complexity of $O(N^2)$ is again maintained by both methods. We also compare the final couplings from our algorithms and mirror descent. In the nonconvex setting, the maximum discrepancy in Frobenius norm over all tested configurations is $2.4022\times 10^{-8}$, while in the convex setting it is $6.0448\times 10^{-11}$, illustrating that both methods effectively converge to the same coupling for these examples.

\begin{figure}[!htb]
    \centering
    \begin{subfigure}[b]{0.48\textwidth}
         \centering
         \includegraphics[width=\textwidth]{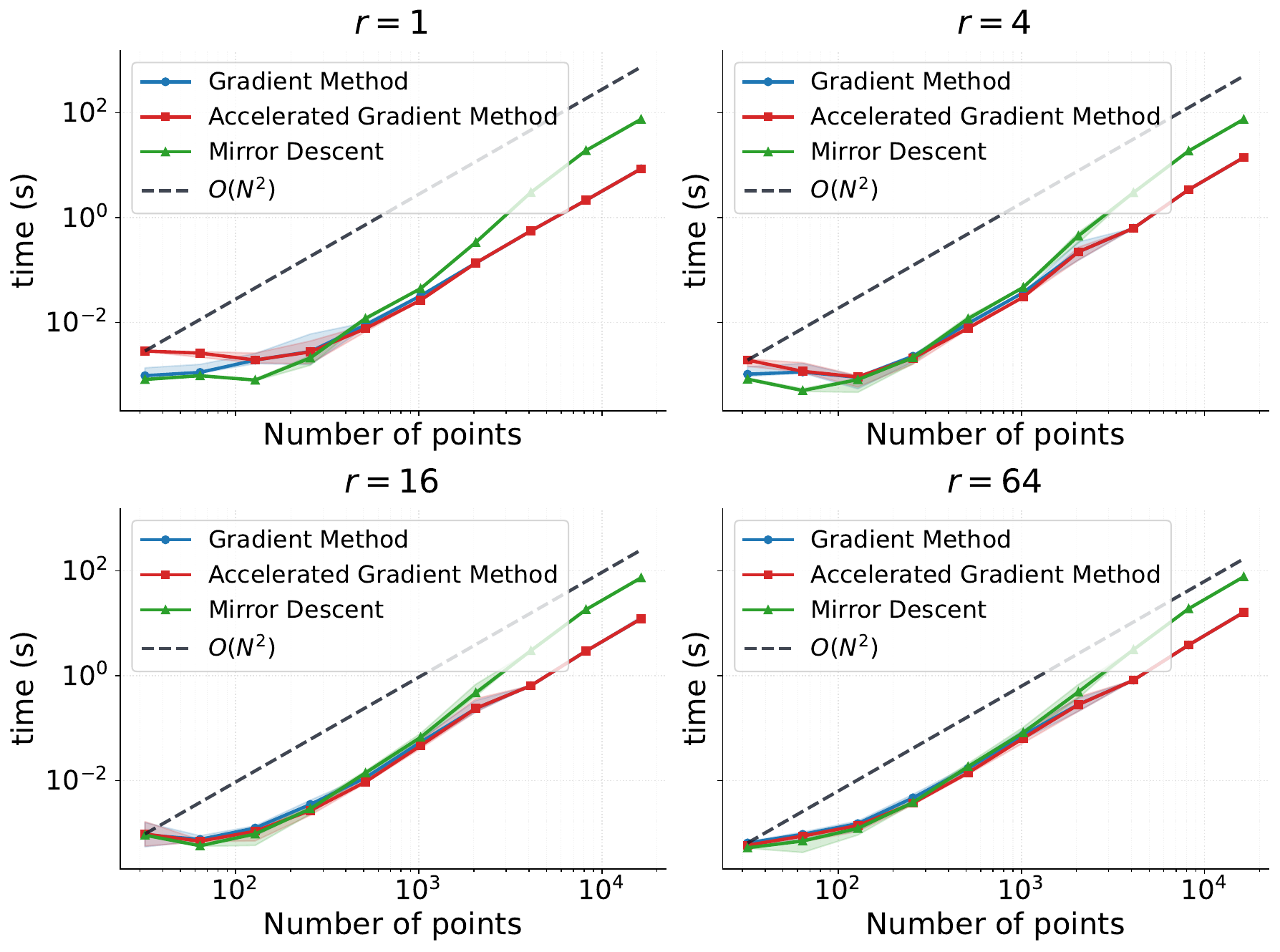}
         \caption{Convex setting}
         \label{fig:runtime-convex}
     \end{subfigure}
     \ \ \ \ \begin{subfigure}[b]{0.48\textwidth}
         \centering
         \includegraphics[width=\textwidth]{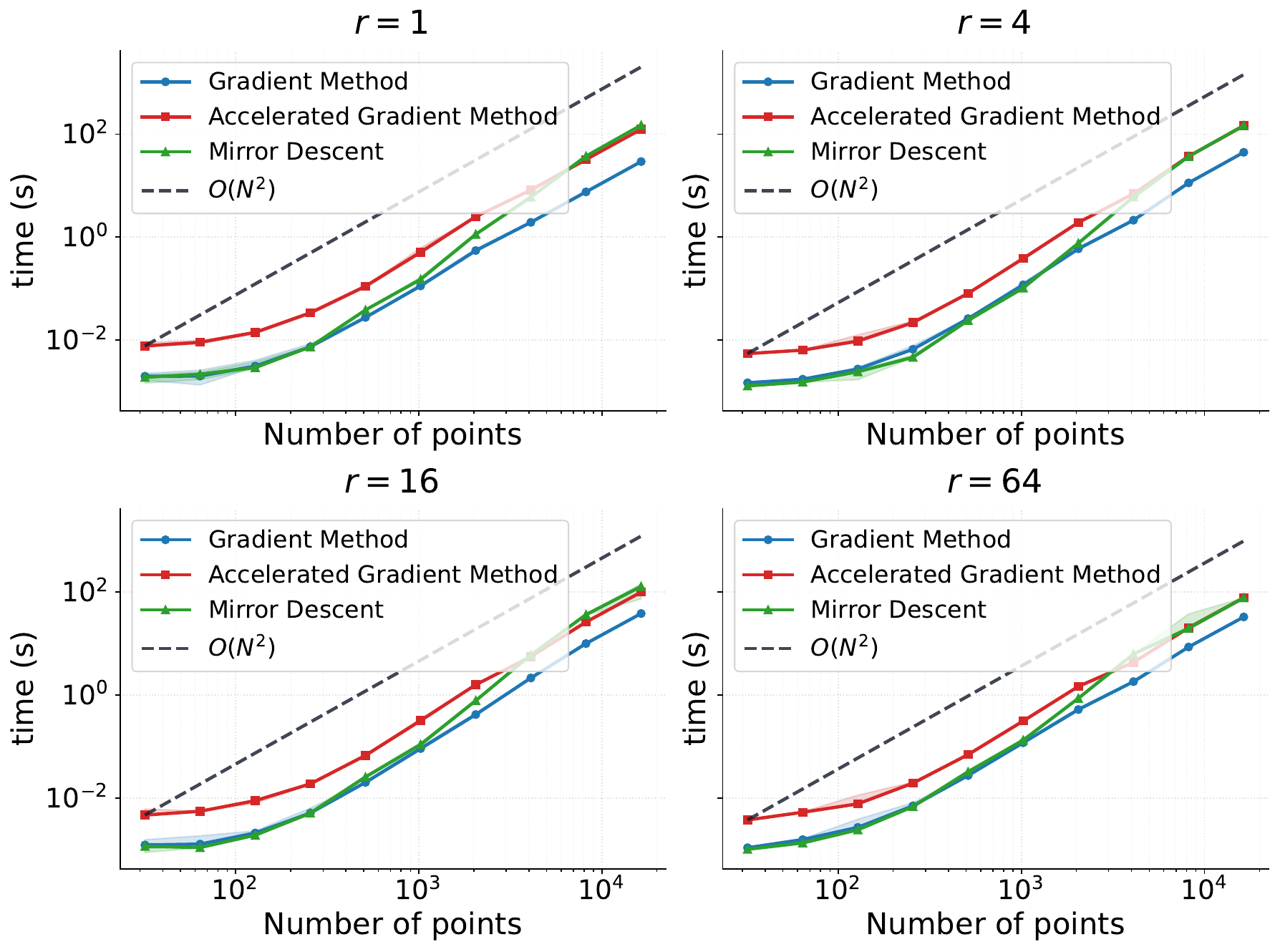}
         \caption{Nonconvex setting}
         \label{fig:runtime-nonconvex}
     \end{subfigure}
     \caption{Results of the runtime comparison experiment. \textbf{(a):} results for $\varepsilon = \lambda_{\max}{(\rB_0^\intercal \rB_0)} - \lambda_{\min}{(\rB_1^\intercal \rB_1)}$  so that each problem instance  is convex (see \cref{tab:L-details} for exact values). \textbf{(b):} results for $\varepsilon=1$ so that the problems are not guaranteed to be convex. All algorithms are seen to scale as $O(N^2)$ where $N$ is the number of support points.}
     \label{fig:runtimeplot}
\end{figure}

We make precise that the various algorithms use different termination conditions. For the gradient methods, we follow the discussion in \cref{subsec:algorithms} and terminate the algorithm if the norm of the approximate gradient at the given iterate is less than the tolerance $\delta = 10^{-6}$. By contrast, the standard implementation of the mirror descent algorithm \cite{scetbon2022linear} used for these experiments terminates once the distance between successive iterates is less than the tolerance $\delta=10^{-6}$. Furthermore, our method requires setting a step size. As remarked earlier, though \cref{lem:convergenceRateMin,lem:convergenceRateAccMin} establish principled choices of step size in terms of the constant $L$, this may be overly conservative in practice and lead to slow convergence of the algorithms. Similarly to the convergence experiments, we therefore tune $L$ by initializing at $L=L_0=2$ for each combination of $(N, r)$. If all runs converge for this choice of $L$, we multiply it by $0.99$ and repeat this procedure until at least one run fails to converge. The last choice of $L$ for which all runs were successful is then used to set the step size (see \cref{sec:experimental-details} for exact values per setting). By contrast, the mirror descent algorithm does not require setting a step size.  %

\section{Testing for Isomorphism of Graph Distributions} 
\label{sec:graphIsomorphism}

As an application of the theory developed in \cref{sec:statistics}, we consider the problem of testing if two  unknown distributions $\mu_0$ and $\mu_1$ on the set $\mathcal G_N$ of all labeled unweighted undirected graphs with $N$ vertices and no self-loops are such that $\mu_0$ and $\mu_1$ are equivalent up to a permutation of labels based only on independent collections of graphs $\{G_{0,i}\}_{i=1}^n$ and $\{G_{1,i}\}_{i=1}^n$ sampled i.i.d. from the populations. A similar application was explored in \cite{rioux2024limit} using limit laws for the quadratic Euclidean GW distance. Their method requires embedding the graph distributions into a Euclidean space and comparing the embedded distributions. An important limitation of their methodology, which arises as a byproduct of this embedding step, is that it only allows comparing distributions on graphs which have independent edges. By contrast, the method we propose in \cref{sec:hypothesisTesting} does not impose such a restrictive assumption and enables comparing arbitrary distributions on graphs. In effect, their method compares the marginal distributions of the edges up to isomorphism whereas the current approach compares the joint distribution of the edges, again up to isomorphism.

Concretely, for $N>1$, the populations $\mu_0$ and $\mu_1$ assign probabilities to each $G\in\mathcal G_N$ and we wish to test, based on the samples, if there exists a permutation of the labels $\sigma:[N]\to [N]$ such that $\mu_1=(P_{\sigma})_{\sharp}\mu_0$, where $P_{\sigma}(G)$ is the graph whose adjacency matrix,\footnote{The adjacency matrix, $\mathrm{A}_G$, of a graph $G\in\mathcal G_{N}$ has $ij$-th entry equal to $1$ if there is an edge connecting the $i$-th and $j$-th vertices and $0$ otherwise for $(i,j)\in[N]\times[N]$.} $\mathrm{A}_{P_{\sigma}(G)}$, satisfies $\left(\mathrm{A}_{P_{\sigma}(G)}\right)_{ij}= \left(\mathrm{A}_{G}\right)_{\sigma(i)\sigma(j)}$ for each $(i,j)\in[N]\times [N]$. To formulate this problem in terms of the GW framework, we consider the kernel
\[
    \kappa:(G,G')\in\mathcal G_N\times \mathcal G_N \mapsto 2 \langle \mathrm{A}_G\mathbf 1, \mathrm{A}_{G'}\mathbf 1\rangle+ \frac 12\langle \mathrm{A}_G, \mathrm{A}_{G'} \rangle_{\mathrm{F}},
\]
where $\mathrm{A}_G\mathbf 1$ is a vector that contains the degree of each vertex and $\langle \mathrm{A}_G, \mathrm{A}_{G'} \rangle_{\mathrm{F}}$ accounts for the number of edges that both graphs have in common. We first establish that this kernel is compatible with the desired hypothesis testing problem in the sense of \cref{lem:vanishingGW}.

\begin{proposition}[Kernel compatibility]
\label{prop:GWNullGraph}
Suppose that $\kappa_0=\kappa_1=\kappa$, $N\neq 4$, and that $\mu_0$ and $\mu_1$ assign positive mass to every graph with a single edge. Then, $\mathsf{GW}_p(\mu_0,\mu_1)=0$ if and only if $\mu_1=(P_{\sigma})_{\sharp} \mu_0$ for some permutation $\sigma:[N]\to [N]$.
\end{proposition}

It is easy to see that if $\mu_1=(P_{\sigma})_{\sharp} \mu_0$ for some permutation $\sigma:[N]\to [N]$, then $\mathsf{GW}_p(\mu_0,\mu_1)=0$ under this choice of kernel. The reverse implication requires slightly more care, whence the added assumptions that $N\neq 4$ and that each graph with a single edge occurs with nonzero probability. The result is proved by first showing that if $\mathsf{GW}_p(\mu_0,\mu_1)=0$ and $\mu_0,\mu_1$ assign some positive mass to each graph with a single edge, then $\mu_1=B_{\sharp}\mu_0$ for some bijection $B:\supp(\mu_0)\to\supp(\mu_1)$. Then it is argued that if $G$ is a graph with edges between the nodes $i,j$ and $i,k$, then $B(G)$ must have edges between the nodes $i',j'$ and $i',k'$ for some possibly distinct indices.   With this, it can be shown that if $N\neq 4$, $B$ is induced by a permutation, see \cref{proof:prop:GWNullGraph} for complete details. 

Given that \cref{prop:GWNullGraph} requires both $\mu_0$ and $\mu_1$ to assign some positive mass to each graph with a single edge, we may apply that result to a simple modification of $\mu_0$ and $\mu_1$ to discern whether they are related by an isomorphism. To state this result, let $\mathcal E$ be the collection of all graphs on $N$ nodes with one edge and $\mathcal U$ be the uniform measure on $\mathcal E$. 

\begin{corollary}[Isomorphic distributions]
\label{cor:tildeeta}
Fix $N\neq 4,\lambda\in(0,1)$ and, for each  $\eta \in\mathcal P(\mathcal G_N)$, let $\tilde \eta=(1-\lambda)\eta +\lambda \mathcal U$. %
    Then, for any $\mu_0,\mu_1\in\mathcal P(\mathcal G_N)$,  
    $\mathsf{GW}_p(\tilde \mu_0,\tilde \mu_1)=0$ if and only if $\mu_1=(P_{\sigma})_{\sharp}\mu_0$ for some permutation $\sigma:[N]\to [N]$.%
\end{corollary}

\cref{cor:tildeeta} is a direct consequence of \cref{prop:GWNullGraph}; its proof is thus omitted for brevity.
To apply this result for testing if population measures $\mu_0,\mu_1\in\mathcal P(\mathcal G_N)$ are isomorphic based on independent collections of  sampled graphs  $(G_{0,i})_{i\in\mathbb N}\stackrel{i.i.d.}{\sim}\mu_0$ and $(G_{1,i})_{i\in\mathbb N}\stackrel{i.i.d.}{\sim}\mu_1$, it suffices to define $\mu_{0,n} \coloneqq (1-\lambda)\hat \mu_{0,n}+\lambda \mathcal U$ and  $\mu_{1,n} \coloneqq (1-\lambda)\hat \mu_{1,n}+\lambda \mathcal U$ and use $\sqrt n \mathsf{GW}_p(\mu_{0,n},\mu_{1,n})^p$ as the test statistic. We may thus apply the results of \cref{sec:hypothesisTesting} with $\kappa_0=\kappa_1=\kappa$ to test $
\mathrm{H}_0 : \mu_1=(P_{\sigma})_{\sharp}\mu_0$ { for some permutation $\sigma:[N]\to [N]$, corresponding to }  $\mathsf{GW}_p(\tilde \mu_0,\tilde \mu_1)=0$ versus
$
\mathrm{H}_1 : \mu_1\neq (P_{\sigma})_{\sharp}\mu_0 \text{ for each permutation $\sigma:[N]\to [N]$, corresponding to }  \mathsf{GW}_p(\tilde \mu_0,\tilde \mu_1)>0.$
We set the critical values of this test for a given (asymptotic) level using a simple scaling of the conservative estimator from \cref{sec:consistencyConservativeQuantiles}  from which we  obtain the following result.
\begin{theorem}[Consistent Isomorphism Testing]
   \label{thm:consistencyLimitLawsGraphs}
    Fix $\lambda\in(0,1)$, $N\neq 4$,  $\mu_{i,n} = (1-\lambda)\hat \mu_{i,n}+\lambda \mathcal U$ for $i\in\{0,1\}$,  %
       $\hat \Sigma_n = \Sigma_{w_{\hat \mu_{0,n}}}$, and $\hat u_n=u_{\mu_{0,n}}$ (defined by analogy with $u_{\tilde \mu_0}$ in \cref{assn:generalSettingLimitLaw}). Then, 
         $\hat \Sigma_n\to \Sigma_{w_{\mu_0}}$ and $\hat u_n\to u_{\tilde \mu_0}$ conditionally on the data given almost every realization of $G_{0,1},G_{0,2},\dots,$ and, if $\mathsf{GW}_p(\tilde \mu_0,\tilde \mu_1)=0$,  the conservative testing procedure from \cref{sec:consistencyConservativeQuantiles} yields a test for graph distribution isomorphism of any given asymptotic level upon scaling $\hat L_n$ by $(1-\lambda)$. If $\mathsf{GW}_p(\tilde \mu_0,\tilde \mu_1)>0$, the probability that the test rejects the null hypothesis converges to $1$. 
\end{theorem}
 Since the measures $\mu_{0,n}$ and $\mu_{1,n}$ are simple modifications to the empirical measures, the result is a direct application of \cref{thm:conservativeTesting}. 

 We conclude this section by noting that graphs with $4$ nodes can be embedded into the space of graphs with $5$ nodes by adding an additional node to the sampled graphs. With this, we can leverage all of the results described above to treat the case $N=4$.

\subsection{Numerical Validation}
\label{sec:graphIsomorphismExperiments}
This section numerically validates the graph isomorphism testing framework from \cref{thm:consistencyLimitLawsGraphs}.

{\textbf{Population measures.}}
We construct $\mu_0\in\mathcal P(\mathcal G_5)$ by first assigning to the edge between nodes $i$ and $j$ a fixed probability $p_{ij}$, denoted $\mathrm{Ber}(p_{ij})$. We then sample independent Bernoulli random variables $X_{ij}$ with success probabilities $p_{ij}$ and draw the edge between the $i$-th and $j$-th nodes if $X_{ij}=1$ and do not draw it otherwise. As an exception to this rule, the edge between nodes $1$ and $3$ follows the distribution $(X_{12}+B_{1/4})\pmod{2}$, where $B_{1/4}\sim\mathrm{Ber}(1/4)$, see \cref{fig:graphModels} for the success probabilities.

\begin{figure}[!htb]
    \centering
    \includegraphics[width=0.6\linewidth]{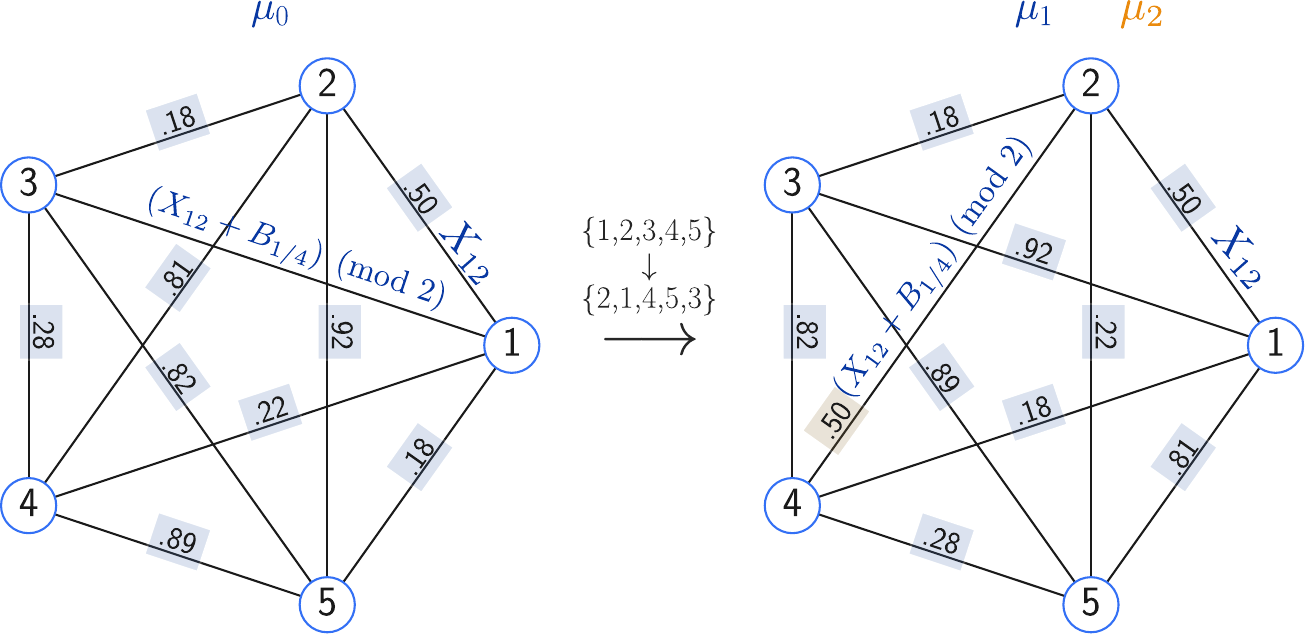}
    \caption{Population distributions used in \cref{sec:graphIsomorphismExperiments}. The boxed numbers on each edge  represent the probability that this edge is drawn for a given sample. The distribution of $\mu_1$ is obtained from $\mu_0$ by simply permuting the node labels so that $\mathsf{GW}_2(\mu_0,\mu_1)=0$. $\mu_2$ is identical to $\mu_1$ except that the edge between nodes $2$ and $4$ is sampled independently from the other edges with probability $1/2$ as illustrated by the boxed number, thus  $\mathsf{GW}_2(\mu_0,\mu_2)>0$.}
    \label{fig:graphModels}
\end{figure}

$\mu_1$ is obtained from $\mu_0$ by applying a fixed permutation $\sigma$ to the nodes as shown in \cref{fig:graphModels}, that is, $\mu_1=(P_{\sigma})_{\sharp}\mu_0$ whereby $\mathsf{GW}_2(\mu_0,\mu_1)=0$. In this experiment, $\sigma(1)=2,\sigma(2)=1,\sigma(3)=4,\sigma(4)=5$, and $\sigma(5)=3$. 
To assess the power of the proposed testing methodology, we also consider a distribution $\mu_2$ which is not isomorphic to $\mu_0$ and is also displayed in \cref{fig:graphModels}.  $\mu_2$ is constructed in the same way as $\mu_1$ except that the edge between the nodes $2$ and $4$ (corresponding to the edge between $1$ and $3$ after applying $\sigma$) has the  $\mathrm{Ber}(1/2)$ distribution so that all edges are drawn independently. The rationale for this choice is that the edge between nodes $2$ and $4$ has the same probability of being drawn under $\mu_1$ or $\mu_2$ for any given graph, but the edges are not drawn independently under $\mu_1$.  

This example highlights the difference between the current approach, which applies to arbitrary distributions, and that proposed in \cite{rioux2024limit}, which effectively compares the marginal distributions of the edges. As demonstrated in \cref{fig:type12error}, the type 2 error of our method decays to $0$ as the number of samples increases, while that of \cite{rioux2024limit} remains essentially~constant.

{\textbf{Implementing the test.}}
To implement the hypothesis test, it suffices to fix a significance level $\alpha\in (0,1)$, compute the test statistic $\sqrt n\mathsf{GW}_2^2(\mu_{0,n},\mu_{1,n})$, where $\mu_{i,n}=(1-\lambda)\hat \mu_{i,n}+\lambda\mathcal U$ for $i=0,1$, and use the conservative quantile estimate which entails solving random linear programs and scaling the optimal value by $1-\lambda$ (see \cref{alg:directEstimator} and \cref{thm:consistencyLimitLawsGraphs}). 

For a fixed $n\in\{10,50,100,500,1000,5000\}$ and $\lambda=0.01$, we generate $n$ i.i.d. samples from $\mu_0,\mu_1,$ and $\mu_2$, and compute $\mathsf{GW}_2(\mu_{0,n},\mu_{1,n})$ and $\mathsf{GW}_2(\mu_{0,n},\mu_{2,n})$ by applying \cref{algo:outer-loop-simple} with the subgradient derived in \cref{lem:derivativeObjective} for $\varepsilon = 0$ in place of the approximate gradient; this corresponds to a subgradient method for the unregularized variational form. We choose the $2$-GW as the test statistic, as this reduces the computational burden, see  \cref{sec:graphTestStatistic} for implementation details.

The conservative quantiles are approximated by generating $250$ samples from the empirical estimator via \cref{alg:directEstimator} using $\hat \Sigma_n$ and $\hat u_n$ from \cref{thm:consistencyLimitLawsGraphs}, scaling the values obtained by $1-\lambda$, and computing the resulting  sample quantiles, $q_{n,1-\alpha}$, for the significance levels $\alpha\in\{0.04\cdot k\}_{k=1}^{24}$. We then reject the null hypothesis if the test statistics $ \sqrt n\mathsf{GW}_2^2(\mu_{0,n},\mu_{1,n})$ or $ \sqrt n\mathsf{GW}_2^2(\mu_{0,n},\mu_{2,n})$ exceed the critical value of $q_{n,1-\alpha}$; as noted in \cref{thm:consistencyLimitLawsGraphs}, these tests are of asymptotic level $\alpha$. To estimate the probability of spuriously rejecting the null (when the samples are drawn from $\mu_0$ and $\mu_1$) or failing to reject the alternative (when the samples are drawn from $\mu_0$ and $\mu_2$), we repeat the above procedure $100$ times with fresh samples which we use to compute new values for the statistics and quantiles.

{\textbf{Results.}}
\cref{fig:type12error} compiles the results of applying the methodology described above for the population measures $\mu_0,\mu_1,$ and $\mu_2$ defined previously, see also \cref{fig:graphModels}. Although \cref{thm:consistencyLimitLawsGraphs} is asymptotic in nature, \cref{fig:type12error} illustrates that the proposed testing methodology yields good results even in the finite-sample regime. By contrast,  the approach from \cite{rioux2024limit} erroneously identifies $\mu_0$ and $\mu_2$ as being isomorphic graph distributions since (by construction) the marginal distributions of the edges are related by a permutation of the nodes and so the type $2$ error is effectively stable.

\begin{figure}[!htb]
    \centering
    \begin{subfigure}[b]{0.8\textwidth}
         \centering
         \includegraphics[width=\textwidth]{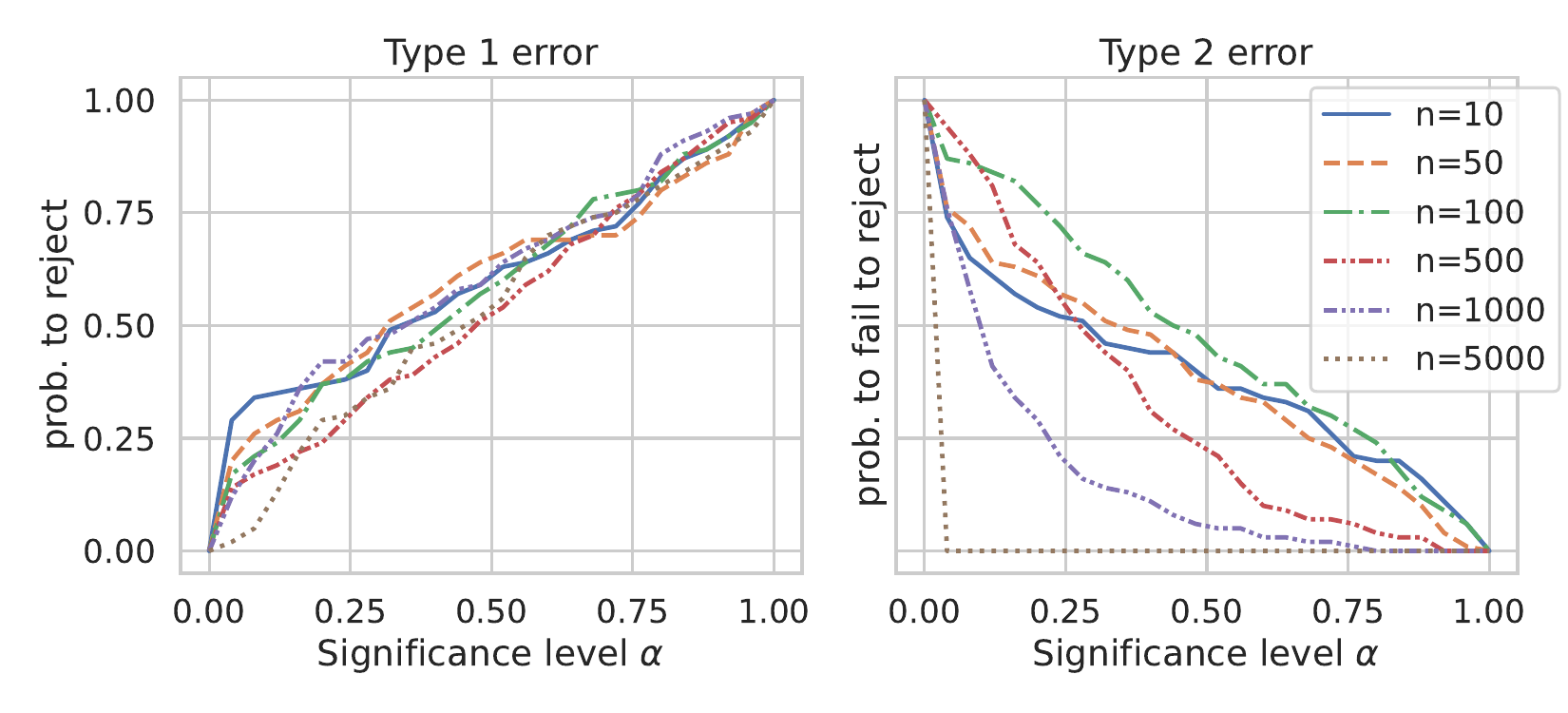}
         \caption{Our test}
         \label{fig:newtest}
     \end{subfigure}
     \ \ \ \ \begin{subfigure}[b]{0.8\textwidth}
         \centering
         \includegraphics[width=\textwidth]{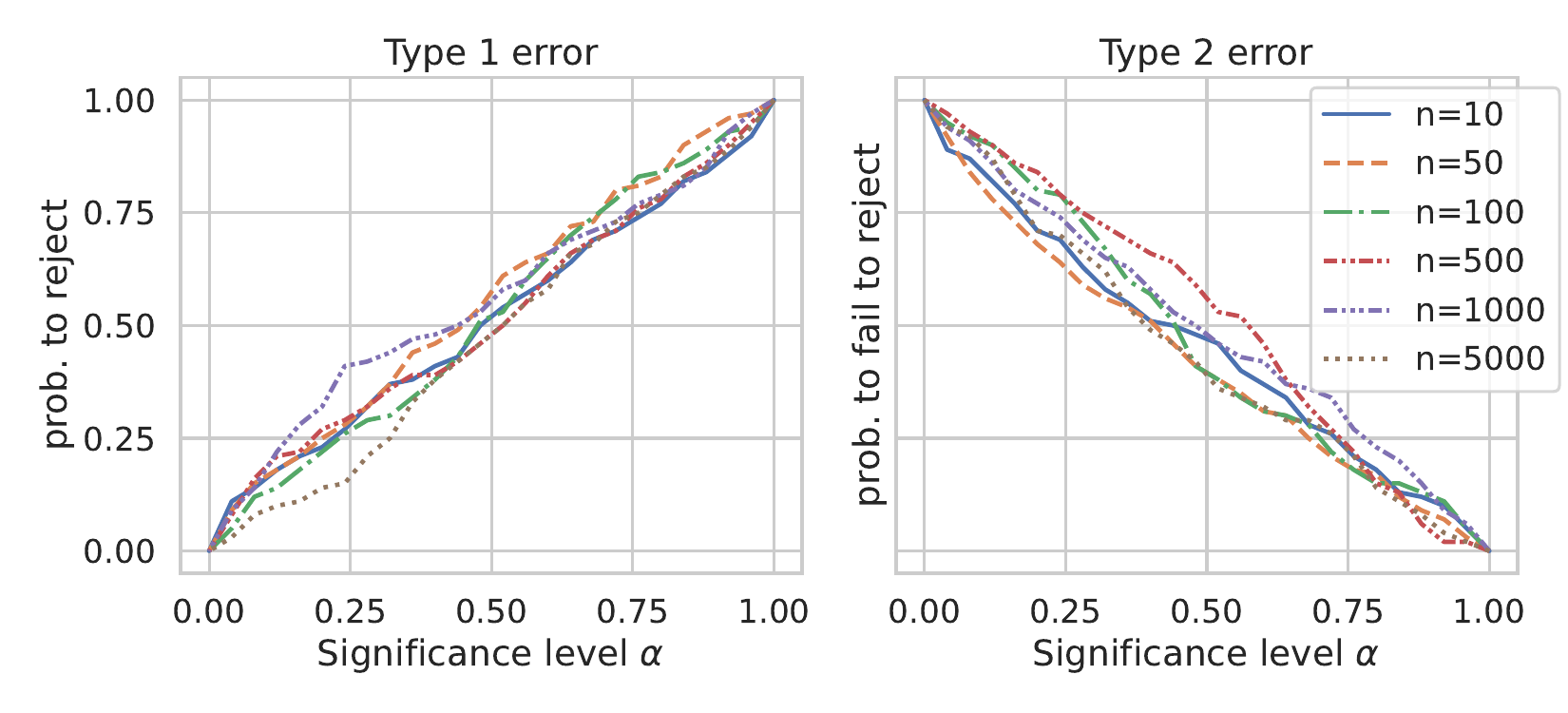}
         \caption{Test from \cite{rioux2024limit}}
         \label{fig:oldtest}
     \end{subfigure}
     \caption{Results of the experiment from \cref{sec:graphIsomorphismExperiments}.    \textbf{(a) left:} estimated probability that  $\sqrt n\mathsf{GW}_2(\mu_{0,n},\mu_{1,n})^2$ exceeds the estimated critical value as a function of the significance level, $\alpha$, based on $100$ repetitions with different sample sizes. \textbf{(a) right:} estimated probability that $\sqrt n\mathsf{GW}_2(\mu_{0,n},\mu_{2,n})^2$ is less than the estimated critical value. \textbf{(b)} is arranged similarly.     }
    \label{fig:type12error}
\end{figure}

\section{Proofs of Main Results}
\label{app:ProofsMain}
\subsection{Proof of \texorpdfstring{\cref{lem:potentialBounds}}{Proposition 1}}
\label{proof:lem:potentialBounds}
  Recall that $\OT_{c}(\mu_0,\mu_1)$ can be identified with a finite-dimensional linear program in the discrete setting.
By the complementary slackness conditions for linear programming, Theorem 4.5 in \cite{bertsimas1997introduction}, any pair of OT potentials $(\varphi_0,\varphi_1)$ for $\OT_{c}(\mu_0,\mu_1)$ satisfies $\varphi_0(x)+\varphi_1(y)=c(x,y)$ at each pair $(x,y)\in\mathcal X_0\times \mathcal X_1$ for which $\pi(\{(x,y)\})>0$ for some choice of OT plan $\pi$ for $\OT_{c}(\mu_0,\mu_1)$. As $\pi\in\Pi(\mu_0,\mu_1)$ and $\supp(\mu_i)=\mathcal X_i$ for $i=0,1$, for every $x\in\mathcal X_0$ there exists $y_x\in\mathcal X_1$ for which $\varphi_0(x)+\varphi_1(y_x)=c(x,y_x)$ and, similarly, for every $y\in\mathcal X_1$ there exists $x_y\in\mathcal X_0$ for which $\varphi_0(x_y)+\varphi_1(y)=c(x_y,y)$. Moreover, $y_x\in\argmin_{\mathcal X_1}\left\{c(x,\cdot)-\varphi_1\right\}$ and $x_y\in\argmin_{\mathcal X_0}\left\{c(\cdot,y)-\varphi_0\right\}$, as $\varphi_0\oplus\varphi_1\leq c$.  Whence,
\[
\begin{gathered}
-\|c\|_{\infty,\mathcal X_0\times \mathcal X_1}-\sup_{\mathcal X_1} \varphi_1
    \leq\varphi_0(x)=\inf_{\mathcal X_1}\left\{c(x,\cdot)-\varphi_1\right\}\leq \|c\|_{\infty,\mathcal X_0\times \mathcal X_1}-\sup_{\mathcal X_1} \varphi_1,
    \\
    -\|c\|_{\infty,\mathcal X_0\times \mathcal X_1}-\sup_{\mathcal X_0} \varphi_0
    \leq\varphi_1(y)=\inf_{\mathcal X_0}\left\{c(\cdot,y)-\varphi_0\right\}\leq \|c\|_{\infty,\mathcal X_0\times \mathcal X_1}-\sup_{\mathcal X_0} \varphi_0.
\end{gathered}
\]
Now, let $(\tilde \varphi_0,\tilde \varphi_1)= (\varphi_0+C, \varphi_1-C)$ for $C=-\sup_{\mathcal X_0}\varphi_0+\frac 12 \|c\|_{\infty,\mathcal X_0\times \mathcal X_1}$ so that $\sup_{\mathcal X_0}\tilde \varphi_0=\frac 12 \|c\|_{\infty,\mathcal X_0\times \mathcal X_1}$. From the display above, we have that $-\frac{3}{2}\|c\|_{\infty,\mathcal X_0\times\mathcal X_1}\leq \tilde \varphi_1(y) \leq\frac 12 \|c\|_{\infty,\mathcal X_0\times\mathcal X_1}$ and so $-\frac{3}{2}\|c\|_{\infty,\mathcal X_0\times\mathcal X_1}\leq \tilde \varphi_0(x) \leq\frac 12 \|c\|_{\infty,\mathcal X_0\times\mathcal X_1}$ for every $(x,y)\in\mathcal X_0\times \mathcal X_1$ as was desired. 
\qed

\subsection{Proof of \texorpdfstring{\cref{thm:variationalForm}}{Theorem 1}}
\label{proof:thm:variationalForm}
 First, by the spectral theorem, $\rM = \sum_{i=1}^r\lambda_i v_iv_i^{\intercal}$ where $r$ is the rank of $\rM$, $(\lambda_i)_{i=1}^r$ is the collection of nonzero eigenvalues of $\rM$ ordered from largest to smallest, and $(v_i)_{i=1}^r\subset \mathbb R^k$ is the collection of corresponding eigenvectors. $\rM$ can thus be expressed as $\rM=\rM_1-\rM_0$, where $\rM_0$ and $\rM_1$ are psd matrices of rank $r_0$ and $r_1$, respectively, where $r_0$ (resp. $r_1$) is  the number of negative (resp. positive) eigenvalues of $\rM$. Explicitly, $\rM_1=\rB_1^{\intercal}\rB_1$ where $\rB_1=\mathrm{diag}(\sqrt{\lambda_1},\dots, \sqrt{\lambda_{r_1}})\mathbf V_1$, where  $\mathbf V_1\in\mathbb R^{r_1\times k}$ has $i$-th row given by $v_i^{\intercal}$. Similarly, $\rM_0=\rB_0^{\intercal}\rB_0$ where $\rB_0\in\mathbb R^{r_0\times k}$ is defined analogously using the absolute value of the negative eigenvalues.

    We proceed by obtaining \eqref{eq:variationalForm} from \eqref{eq:entropicQP}. Setting $f_1(x)=\frac{1}{2}x^{\intercal}\rM_{1}x$ and $f_0(x)=\frac{1}{2}x^{\intercal}\rM_{0}x$, we write the objective in \eqref{eq:entropicQP} as $f_1-f_0-\varepsilon\mathsf H$ where $f_0$ and $f_1$ are, notably, convex functions.  Example 11.10 in \cite{rockafellar1998variational} asserts that the convex conjugate of $f_1$ is  
    \[
        f_1^{\star}:y\in\mathbb R^k \mapsto \sup_{z\in\mathbb R^k}\left\{  y^{\intercal} z-f_1(z)\right\}=\frac{1}{2}y^{\intercal}\rM^{\dagger}_{1}y+\mathcal I_{\mathrm{range}(\rM_{1})}(y),
    \]
    where, for a set $C$,
   \begin{equation}
   \label{eq:indicator}
        \mathcal I_{C}(y) = \begin{cases}
           0,&\text{if }y\in C,
           \\
           +\infty,&\text{otherwise},
        \end{cases}
   \end{equation}
   and
    $\rM_1^{\dagger}$ is the pseudoinverse of $\rM_1$  which satisfies the property that $\rM_1\rM_1^{\dagger}$ is the orthogonal projection onto $\mathrm{range}(\rM_1)\coloneqq \left\{\rM_1 x:x\in\mathbb R^{k}\right\}$ and, in particular,  $\rM_1\rM_1^{\dagger}\rM_1=\rM_1$.  
    Since $f_1$ is proper (i.e., $f_1$ is not identically $+\infty$ and does not take the value $-\infty$), lower semicontinuous, and convex, the Fenchel-Moreau theorem (cf. e.g. Theorem 11.1 in \cite{rockafellar1998variational}) yields that
    \[
        \begin{aligned}
        f_1(x)=f_1^{\star\star}(x)&=\sup_{z\in\mathbb R^k}\left\{  x^{\intercal} z-\frac{1}{2}z^{\intercal}\rM^{\dagger}_{1}z-\mathcal I_{\mathrm{range}(\rM_{1})}(z)\right\}
        \\
        &=\sup_{z'\in\mathbb R^k}\left\{  x^{\intercal} \rM_{1} z'-\frac{1}{2}(z')^{\intercal}\rM_{1}^{\intercal}\rM_{1}^{\dagger}\rM_{1} z'\right\}
        \\
        &=\sup_{z'\in\mathbb R^k}\left\{  x^{\intercal} \rM_{1} z'-\frac{1}{2}(z')^{\intercal}\rM_{1} z'\right\}
        \\
        &=\sup_{v\in\rB_1 \mathbb R^k}\left\{  x^{\intercal} \rB^{\intercal}_1 v-\frac{1}{2}\|v\|^2\right\},
        \end{aligned} 
    \] 
    where the second equality follows from the change of variables $z=\rM_{1} z'$, the third is due to the fact that $\rM_{1}\rM_{1}^{\dagger}\rM_{1}=\rM_{1}$, and the final equality is due to  the change of variables $v=\rB_1 z'$ using the fact that $\rM_{1}=\rB_1^{\intercal} \rB_1$. Note that, for any fixed $x\in\RR^k$, the function $v\in\mathbb R^{r_1}\mapsto x^{\intercal} \rB_1^{\intercal} v-\frac{1}{2}\|v\|^2$ achieves its global maximum at $v^{\star}=\rB_1 x$ so that  
    $
        f_1(x)=\sup_{v\in\mathbb R^{r_1}}\left\{  x^{\intercal} \rB_1^{\intercal} v-\frac{1}{2}\|v\|^2\right\}.
    $
    By the same logic, $f_0(x)=\sup_{u\in\mathbb R^{r_0}}\left\{  x^{\intercal} \rB_0^{\intercal} u-\frac{1}{2}\|u\|^2\right\}$. Altogether, we have that
    \[
    \begin{aligned} 
    \inf_{x\in\mathcal K} \left\{\frac 12x^{\intercal} \rM x-\varepsilon\sH(x)\right\}\mspace{-1mu}&=  \inf_{x\in\mathcal K} \left\{\sup_{v\in\mathbb R^{r_1}}\left\{  x^{\intercal} \rB^{\intercal}_1 v-\frac{1}{2}\|v\|^2\right\}-\sup_{u\in\mathbb R^{r_0}}\left\{ x^{\intercal}\rB_0^{\intercal} u-\frac12\|u\|^2 \right\}-\varepsilon\sH(x)\right\}
    \\
    &=  \inf_{x\in\mathcal K} \left\{\sup_{v\in\mathbb R^{r_1}}\left\{  x^{\intercal} \rB^{\intercal}_1 v-\frac{1}{2}\|v\|^2\right\}+\inf_{u\in\mathbb R^{r_0}}\left\{\frac12\|u\|^2 -x^{\intercal}\rB_0^{\intercal} u \right\}-\varepsilon\sH(x)\right\}
    \\
    &=\inf_{u\in\mathbb R^{r_0}}\inf_{x\in\mathcal K} \sup_{v\in\mathbb R^{r_1}}\left\{  x^{\intercal} \rB^{\intercal}_1 v-\frac{1}{2}\|v\|^2+\frac12\|u\|^2 -x^{\intercal}\rB_0^{\intercal} u -\varepsilon\sH(x)\right\}, 
    \end{aligned}
    \]
    where in the third equality we have interchanged the two infima. 
    As the objective in this final $\inf$-$\inf$-$\sup$ problem is (strongly) concave in $v$ for fixed $(x,u)\in\cK\times \RR^{r_0}$ and  convex in $x$ for fixed $(u,v)\in\RR^{r_0}\times \RR^{r_1}$ and $\mathcal K$ is compact, Sion's minimax theorem, Theorem 4.2' in \cite{sion1958minimax}, yields 
    \[
        \begin{aligned}
        &\inf_{u\in\mathbb R^{r_0}}\inf_{x\in\mathcal K} \sup_{v\in\mathbb R^{r_1}}\left\{  x^{\intercal} \rB^{\intercal}_1 v-\frac{1}{2}\|v\|^2+\frac12\|u\|^2 -x^{\intercal}\rB_0^{\intercal} u -\varepsilon\sH(x)\right\} 
        \\
        &\hspace{6em}=\inf_{u\in\mathbb R^{r_0}} \sup_{v\in\mathbb R^{r_1}}\left\{  \frac12\|u\|^2-\frac{1}{2}\|v\|^2 +\inf_{x\in\mathcal K}\left\{\left( \rB^{\intercal}_1 v-\rB_0^{\intercal} u\right)^{\intercal} x -\varepsilon\sH(x)\right\}\right\},
        \end{aligned} 
    \]
    Combining the two previous displayed equations, we obtain that 
    \begin{equation}
    \label{eq:equalityValues}
\inf_{x\in\mathcal K} \left\{\frac 12x^{\intercal} \rM x-\varepsilon\sH(x)\right\}= \inf_{u\in\mathbb R^{r_0}} \sup_{v\in\mathbb R^{r_1}}\left\{  \frac12\|u\|^2-\frac{1}{2}\|v\|^2 +\inf_{x\in\mathcal K}\left\{\left( \rB^{\intercal}_1 v-\rB_0^{\intercal} u\right)^{\intercal} x -\varepsilon\sH(x)\right\}\right\}.
    \end{equation}

    We now move to proving items (1) and (2) by leveraging the following technical lemmas whose proofs are deferred, respectively, to \cref{proof:lem:derivativeObjective,proof:lem:QPvsLPsoln}. 
    \begin{lemma}
        \label{lem:derivativeObjective} 
        For each $\varepsilon\geq0$, define $\ell_{\varepsilon}:u\in\RR^{r_0}\to\mathbb R$ via
        \[
            \ell_{\varepsilon}(u) = \sup_{v\in\mathbb R^{r_1}}\left\{  \frac12\|u\|^2-\frac{1}{2}\|v\|^2 +\inf_{x\in\mathcal K}\left\{\left( \rB^{\intercal}_1 v-\rB_0^{\intercal} u\right)^{\intercal} x -\varepsilon\sH(x)\right\}\right\}.
        \]
        Then, the Clarke subdifferential\footnote{The Clarke subdifferential \cite{clarke1990optimization} generalizes  the  gradient for functions which are locally Lipschitz continuous. Given a locally Lipschitz continuous function $f:X\subset \mathbb R^n\to \mathbb R$, its Clarke subdifferential at $x\in X$  is  
$
    \partial f(x)\coloneqq \left\{ \xi \in \mathbb R^n : \limsup_{\substack{y\to x,t\downarrow 0}}\frac{f(y+tv)-f(y)}{t}\geq \xi^{\intercal}v\text{ for all }v\in X\right\}.$ If $f$ is continuously differentiable at $x$, $\partial f(x)=\{\nabla f(x)\}$.} of $\ell_{\varepsilon}$ at $u\in\RR^{r_0}$ is given by 
        \begin{equation}
        \label{eq:lepsSubdiff}
            \partial \ell_{\varepsilon}(u) = u-\rB_0 \argmin_{x\in\mathcal K}\left\{\frac{1}{2}x^{\intercal}\rB^{\intercal}_1\rB_1 x- u^{\intercal}\rB_0 x-\varepsilon\sH(x)\right\}.
        \end{equation}
        Furthermore, for any fixed $u\in\mathbb R^{r_0}$, the Clarke subdifferential of the function 
        \[
            f_{\varepsilon,u}: v\in\RR^{r_1}\mapsto \frac12\|u\|^2-\frac{1}{2}\|v\|^2 +\inf_{x\in\mathcal K}\left\{\left( \rB^{\intercal}_1 v-\rB_0^{\intercal} u\right)^{\intercal} x -\varepsilon\sH(x)\right\}
        \]
        at any $v\in\RR^{r_1}$ is given by 
        \[
            \partial f_{\varepsilon,u}(v) = -v+\rB_1 \argmin_{x\in\mathcal K}\left\{\left( \rB^{\intercal}_1 v-\rB_0^{\intercal} u\right)^{\intercal} x -\varepsilon\sH(x)\right\}.
        \]
    \end{lemma}

        \begin{lemma}
        \label{lem:QPvsLPsoln}
           For any $\varepsilon\geq 0$, if $\bar x$ solves  $\min_{x\in\mathcal K}\left\{\frac{1}{2}x^{\intercal}\rB^{\intercal}_1\rB_1 x- u^{\intercal}\rB_0 x-\varepsilon\sH(x)\right\}$, then it also solves $\min_{x\in\mathcal K}\left\{\bar x^{\intercal}\rB^{\intercal}_1\rB_1 x- u^{\intercal}\rB_0 x-\varepsilon\sH(x)\right\}$. 
        \end{lemma}

        We proceed by proving item (1).
        Suppose that $x^{\star}$ solves \eqref{eq:entropicQP} and let $(\bar u,\bar v)=(\rB_0 x^{\star},\rB_1 x^{\star})$.  Then, we have that   
        \[
            \begin{aligned}
                 \ell_{\varepsilon}(\bar u)&=
\sup_{v\in\mathbb R^{r_1}}\left\{  \frac12(x^{\star})^{\intercal}\rB_0^{\intercal}\rB_0 x^{\star}-\frac{1}{2}\|v\|^2 +\inf_{x\in\mathcal K}\left\{\left( \rB^{\intercal}_1 v-\rB_0^{\intercal} \rB_0 x^{\star}\right)^{\intercal} x -\varepsilon\sH(x)\right\}\right\}
                \\
                &\leq \sup_{v\in\mathbb R^{r_1}}\left\{ -\frac12(x^{\star})^{\intercal}\rB_0^{\intercal}\rB_0 x^{\star}-\frac{1}{2}\|v\|^2 + v^{\intercal}\rB_1 x^{\star} -\varepsilon\sH(x^{\star})\right\}
                \\
                &= -\frac12(x^{\star})^{\intercal}\rB_0^{\intercal}\rB_0 x^{\star} +\frac 12  (x^{\star})^{\intercal}\rB_1^{\intercal}\rB_1 x^{\star} -\varepsilon\sH(x^{\star})
                \\
                &=\frac 12 \left( x^{\star}\right)^{\intercal}\rM x^{\star}-\varepsilon\sH(x^{\star})=\inf_{u\in\RR^{r_0}}\ell_{\varepsilon}(u),
            \end{aligned}
        \]
        where the inequality follows by taking $x=x^{\star}$ in the infimum over $\mathcal K$ and the subsequent equality follows by noting that the objective in the second line is strongly concave and is maximized at $\bar v = \rB_1 x^{\star}$. The final equality follows from  \eqref{eq:equalityValues}, which asserts that 
            $\inf_{x\in\mathcal K}\left\{\frac 12 x^{\intercal}\rM x -\varepsilon \mathsf H(x)\right\} = \inf_{u \in \mathbb R^{r_0}} \ell_{\varepsilon}(u).$ 
        Conclude that $\bar u=\rB_0 x^{\star}$ minimizes $\ell_{\varepsilon}$ and thus solves the outer infimum in \eqref{eq:variationalForm}.

        We now show that $\bar v=\rB_1 x^{\star}$ is optimal for the supremum in \eqref{eq:variationalForm}. Proposition 2.3. in \cite{clarke1990optimization} asserts that $\bar u$ satisfies the first-order optimality condition $0\in\partial \ell_{\varepsilon}(\bar u)$, which,   recalling \cref{lem:derivativeObjective}, means that $\bar u=\rB_0 \bar x$ for some $\bar x$ solving 
        $
        \min_{x\in\mathcal K}\left\{\frac 12 x^{\intercal} \rB_1^{\intercal}\rB_1 x - \bar u^{\intercal}\rB_0 x -\varepsilon\sH(x)\right\}.
        $
        Note that $x^{\star}$ also solves this problem. Indeed, if it does not,  
        \[
\frac 12 (\bar x)^{\intercal} \rB_1^{\intercal}\rB_1 \bar x - \frac 12 (\bar x)^{\intercal} \rB_0^{\intercal}\rB_0 \bar x -\varepsilon\sH(\bar x)<\frac 12 (x^{\star})^{\intercal} \rB_1^{\intercal}\rB_1 x^{\star}-\frac 12 (x^{\star})^{\intercal} \rB_0^{\intercal}\rB_0 x^{\star}  -\varepsilon\sH(x^{\star}),
        \]
        since $\bar u =\rB_0x^{\star}=\rB_0 \bar x$, but this contradicts   optimality of $x^{\star}$ for \eqref{eq:entropicQP}. Applying  
        \cref{lem:QPvsLPsoln}, it holds that   $x^{\star}\in \argmin_{x\in\mathcal K}\left\{ x^{\intercal} \rB^{\intercal}_1\rB_1 x^{\star}-\bar u^{\intercal}\rB_0x -\varepsilon \mathsf H(x)\right\}$ whereby
        \[
            \begin{aligned}
        \inf_{u\in \RR^{r_0}}\ell_{\varepsilon}(u)=\ell_{\varepsilon}(\bar u)&\geq \frac 12 \|\bar u\|^2-\frac 12 \|\bar v\|^2+\inf_{x\in\mathcal K}\left\{\left(\rB_1^{\intercal}\bar v-\rB_0^{\intercal} \bar u\right)^{\intercal} x-\varepsilon \sH(x)\right\}
                \\
               &= 
               \frac 12 \|\bar u\|^2-\frac 12 \|\bar v\|^2+\left(\rB_1^{\intercal}\bar v-\rB_0^{\intercal} \bar u\right)^{\intercal} x^{\star}-\varepsilon \sH(x^{\star})\\
                &=\frac 12 \left( x^{\star}\right)^{\intercal}\rM x^{\star}-\varepsilon\sH(x^{\star})=\inf_{u\in \RR^{r_0}}\ell_{\varepsilon}(u),
            \end{aligned} \]
            so that the lower bound, obtained by evaluating $f_{\varepsilon,\bar u}$ at $\bar v$, is in fact an equality, that is, $\bar v=\rB_1 x^{\star}\in\argmax_{\RR^{r_1}} f_{\varepsilon,\bar u}$. Altogether, this means that $(\bar u,\bar v)=(\rB_0 x^{\star},\rB_1 x^{\star})$ is optimal for the $\inf$-$\sup$ problem in \eqref{eq:variationalForm} and that $x^{\star}\in\argmin_{x\in\mathcal K}\left\{\left(\rB_1^{\intercal}\bar v-\rB_0^{\intercal}\bar u\right)^{\intercal} x -\varepsilon\sH(x)\right\}$ as  claimed.

        We now prove item (2). If $(u^{\star},v^{\star})$ solves the outer $\inf$-$\sup$ problem in \eqref{eq:variationalForm}, Proposition 2.3 in \cite{clarke1990optimization}  
        asserts that $u^{\star}$ satisfies the first-order optimality condition $0\in\partial \ell_{\varepsilon}(u^{\star})$. By  \cref{lem:derivativeObjective}, this means that there exists $\bar x\in \argmin_{x\in\mathcal K}\left\{\frac{1}{2}x^{\intercal}\rB^{\intercal}_1\rB_1 x - (u^{\star})^{\intercal}\rB_0 x-\varepsilon\sH(x)\right\}$ satisfying 
       $ u^{\star} =\rB_0 \bar x$. Thus, 
       \[
            \inf_{x\in\mathcal K}\left\{\frac 12 x^{\intercal}\rM x-\varepsilon\sH(x)\right\}=\inf_{u\in\RR^{r_0}}\ell_{\varepsilon}(u)=\frac 12 \bar x^{\intercal}\rB^{\intercal}_1\rB_1 \bar x-\frac 12 \bar x^{\intercal}\rB^{\intercal}_0\rB_0 \bar x-\varepsilon\sH(\bar x),
       \]
       where we have used the expression for $\ell_{\varepsilon}$ from \eqref{eq:QPDecomposition}. As such,  $\bar x$ solves \eqref{eq:entropicQP}. By \cref{lem:QPvsLPsoln}, $\bar x\in \argmin_{x\in\mathcal K}\left\{\bar x^{\intercal}\rB^{\intercal}_1\rB_1 x -(u^{\star})^{\intercal}\rB_0 x-\varepsilon\sH(x)\right\}$ so that  
       \[
            \inf_{x\in\mathcal K}\left\{\frac 12 x^{\intercal}\rM x-\varepsilon \sH(x)\right\}=\inf_{u\in\RR^{r_0}}\ell_{\varepsilon}(u)=\sup_{v\in\RR^{r_1}}f_{\varepsilon,u^{\star}}(v)\geq f_{\varepsilon,u^{\star}}(\rB_1 \bar x)=\frac 12 \bar x^{\intercal}\rM \bar x -\varepsilon\sH(\bar x).
       \]
       Conclude that the above inequality is tight so that $\rB_1 \bar x$ is the unique maximizer, $v^{\star}$, of $f_{\varepsilon,u^{\star}}$, $\bar x\in\argmin \left\{ \left(\rB^{\intercal}_1v^{\star}-\rB^{\intercal}_0 u^{\star}\right)^{\intercal} x -\varepsilon\sH(x)\right\}$, and $\bar x$ solves \eqref{eq:entropicQP}, proving  item (2). 

       The final assertion is a direct consequence of the first-order optimality condition for $f_{\varepsilon,u}(v)$. Namely, if $\bar v$ maximizes $f_{\varepsilon,u}$, then  
       \[
       0\in\partial f_{\varepsilon,u}(\bar v)=-\bar v+\rB_1 \argmin_{x\in\mathcal K}\left\{\left( \rB^{\intercal}_1 \bar v-\rB_0^{\intercal} u\right)^{\intercal} x -\varepsilon\sH(x)\right\},
       \]
       that is, $\bar v=\rB_1 x$ for some $x\in\mathcal K$. 
       \qed

\subsection{Proof of \texorpdfstring{\cref{prop:LipschitzStability}}{Proposition 3}}
\label{proof:prop:LipschitzStability}  

We separate the proof of \cref{prop:LipschitzStability} into two lemmas. The first result pertains to the entropically regularized case and the second treats the unregularized case. The proofs of these results are deferred to \cref{proof:lem:EOTLipschitz} and \cref{proof:lem:OTLipschitz}.
\begin{lemma}
        \label{lem:EOTLipschitz}
Suppose that $\varepsilon>0$. Then, for any choice of $(\nu_0,\nu_1),(\nu_0',\nu_1')\in\mathcal P(\mathcal X_0)\times \mathcal P(\mathcal X_1)$ and $(u,v)\in \rB_0\mathcal B\times \rB_1\mathcal B$, 
\[
    \left|\mathsf{EOT}^{\varepsilon}_{\rB_1^{\intercal}v-\rB_0^{\intercal}u}(\nu_0,\nu_1)-\mathsf{EOT}^{\varepsilon}_{\rB_1^{\intercal}v-\rB_0^{\intercal}u}(\nu_0',\nu_1')\right|\leq C_{\rB_0,\rB_1}\left(\|w_{\nu_0}-w_{\nu_0'}\|_1+\|w_{\nu_1}-w_{\nu_1'}\|_1\right),
\]
where $C_{\rB_0,\rB_1} =\frac3 2\sup_{\substack{\|x_0\|_1= 1\\\|x_1\|_1= 1}}\|\rB_1^{\intercal}\rB_1x_1-\rB_0^{\intercal}\rB_0x_0\|_{\infty}$. 
\end{lemma}

\begin{lemma}
\label{lem:OTLipschitz}
For any choice of $(\nu_0,\nu_1),(\nu_0',\nu_1')\in\mathcal P(\mathcal X_0)\times \mathcal P(\mathcal X_1)$ and $(u,v)\in \rB_0\mathcal B\times \rB_1\mathcal B$, 
\[
    \left|\mathsf{OT}_{\rB_1^{\intercal}v-\rB_0^{\intercal}u}(\nu_0,\nu_1)-\mathsf{OT}_{\rB_1^{\intercal}v-\rB_0^{\intercal}u}(\nu_0',\nu_1')\right|\leq C_{\rB_0,\rB_1}\left(\|w_{\nu_0}-w_{\nu_0'}\|_1+\|w_{\nu_1}-w_{\nu_1'}\|_1\right),
\]
where $C_{\rB_0,\rB_1} =\frac3 2\sup_{\substack{\|x_0\|_1= 1\\\|x_1\|_1= 1}}\|\rB_1^{\intercal}\rB_1x_1-\rB_0^{\intercal}\rB_0x_0\|_{\infty}$. 
\end{lemma}

\subsection{Proof of \texorpdfstring{\cref{thm:parametricRates}}{Theorem 2}}
\label{proof:thm:parametricRates}
Let $\hat{\mathcal K}_n$ denote the constraint set for the set of couplings with the empirical marginals, $\hat \mu_{0,n},\hat \mu_{1,n}$, and $\mathcal K$ be the analogous constraint set for $\mu,\nu$. Following \cref{thm:variationalForm} the solutions of the $\inf$-$\sup$ problems with the empirical or population marginals all lie in $\rB_0\mathcal B\times \rB_1\mathcal B$.  
Thus, from  \cref{eq:EGWtoEOT}, we have  
            \begin{equation}
            \label{eq:GWBoundEOT}
            \begin{aligned}
&\left|\mathsf{EGW}^{\varepsilon}_p(\hat \mu_{0,n},\hat \mu_{1,n})-\mathsf{EGW}^{\varepsilon}_p( \mu_{0}, \mu_{1})\right|
                \\
                & =\left|\inf_{u\in\rB_0\mathcal B}\sup_{v\in\rB_1\mathcal B}\left\{\frac 12 \|u\|^2-\frac 12 \|v\|^2+\mathsf{EOT}^{\varepsilon}_{\rB_1^{\intercal}v-\rB_0^{\intercal}u}(\hat \mu_{0,n},\hat \mu_{1,n})\right\}\right.
    \\
&\hspace{14em}\phantom{=}-
   \left. \inf_{u\in\rB_0\mathcal B}\sup_{v\in\rB_1\mathcal B}\left\{\frac 12 \|u\|^2-\frac 12 \|v\|^2+\mathsf{EOT}^{\varepsilon}_{\rB_1^{\intercal}v-\rB_0^{\intercal}u}(\mu_0,\mu_1)\right\}\right|
 \\
 &\leq \sup_{u\in\rB_0\mathcal B}\sup_{v\in\rB_1\mathcal B}
 \left|\mathsf{EOT}^{\varepsilon}_{\rB_1^{\intercal}v-\rB_0^{\intercal}u}(\hat \mu_{0,n},\hat \mu_{1,n})-   \mathsf{EOT}^{\varepsilon}_{\rB_1^{\intercal}v-\rB_0^{\intercal}u}(\mu_0,\mu_1)\right|,
                \end{aligned}
            \end{equation}
            where we have used the fact that if $f:A\to \mathbb R$ and $g:A\to\mathbb R$ are bounded functions on a set $A$, $|\inf_A f-\inf_A g|\leq \sup_A|f-g|$ and $|\sup_A f-\sup_A g|\leq \sup_A|f-g|$. Applying \cref{prop:LipschitzStability} we have, for any choice of $\varepsilon\geq 0$,  
            \[
            \begin{aligned}
            \sup_{u\in\rB_0\mathcal B}\sup_{v\in\rB_1\mathcal B}
 \left|\mathsf{EOT}^{\varepsilon}_{\rB_1^{\intercal}v-\rB_0^{\intercal}u}(\hat \mu_{0,n},\hat \mu_{1,n})-   \mathsf{EOT}^{\varepsilon}_{\rB_1^{\intercal}v-\rB_0^{\intercal}u}(\mu_0,\mu_1)\right|&
 \\
 &\hspace{-12em}\leq C_{\rB_0,\rB_1}\left(\|w_{\hat \mu_{0,n}}-w_{\mu_0}\|_1+\|w_{\hat \mu_{1,n}}-w_{\mu_1}\|_1\right). 
          \end{aligned} 
            \]
    We now bound the expectation of the final term so that the above display along with \eqref{eq:GWBoundEOT} yield the claimed result. Applying the Cauchy-Schwarz inequality and the fact that $\|\cdot\|_1\leq \sqrt{d}\|\cdot\|$ in $\mathbb R^d$, 
    \[
    \begin{aligned}
        \mathbb E\left[\|w_{\hat \mu_{0,n}}-w_{\mu_0}\|_1\right]\leq \sqrt{N_0}{\mathbb E[\|w_{\hat \mu_{0,n}}-w_{\mu_0}\|_2]}&\leq \sqrt{N_0}\sqrt{\mathbb E[\|w_{\hat \mu_{0,n}}-w_{\mu_0}\|_2^2]}.
        \end{aligned}
    \] 
    Note that $nw_{\hat \mu_{0,n}}$ can be identified with a sample from a multinomial distribution with $n$ trials and event probability vector $w_{\mu_0}$. As such, $\mathbb E\left[\left(n(w_{\hat \mu_{0,n}})_i-n(w_{\mu_0})_i\right)^2\right]= n(w_{\mu_0})_i(1-(w_{\mu_0})_i)$, yielding
    \[
    \begin{aligned}
        \mathbb E\left[\|w_{\hat \mu_{0,n}}-w_{\mu_0}\|_1\right]\leq \sqrt{\frac{N_0}{n}}\left(\sum_{i=1}^{N_0}(w_{\mu_0})_i(1-(w_{\mu_0})_i)\right)^{1/2}=\sqrt{\frac{N_0}{n}}\sqrt{\left(1-\|w_{\mu_0}\|^2\right)}.
        \end{aligned}
    \]
    Applying the same bound between the $1$- and $2$-norms, we get $\|w_{\mu_0}\|\geq N_0^{-1/2}\|w_{\mu_0}\|_1=N_0^{-1/2}$ so that $\mathbb E\left[\|w_{\hat\mu_{0,n}}-w_{\mu_0}\|_1\right]\leq \sqrt{\frac{N_0-1}{n}}$. A similar bound evidently holds for $\mathbb E\left[\|w_{\hat\mu_{1,n}}-w_{\mu_1}\|_1\right]$.  
\qed

\subsection{Proof of \texorpdfstring{\cref{thm:limitTheorems}}{Theorem 3}}
\label{proof:thm:limitTheorems}
As noted in the text, it suffices to verify the right
differentiability of \(\mathsf{EGW}_p^{\varepsilon}\) in the sense of condition (3) from \cref{prop:limitTheoremFramework}.

We start by computing the requisite directional derivative in the unregularized case,  $\varepsilon=0$. 

\begin{proposition}[GW directional derivative] \label{lem:GWDerivative}
      Let $(\nu_0,\nu_1)\in\mathcal P(\mathcal X_0)\times \mathcal P(\mathcal X_1)$ %
     be arbitrary and, for each $t\in(0,1)$, let $\mu_{i,t}\coloneqq \mu_i+t(\nu_i-\mu_i)$ with $i=0,1$. Then,  
   \[
    \begin{aligned}
        \lim_{t\downarrow 0}\frac{\mathsf{GW}_p(\mu_{0,t},\mu_{1,t})^p- \mathsf{GW}_p(\mu_{0},\mu_{1})^p}{t}\hspace{-3em}&\\    &=  \inf_{\pi^{\star}\in\Pi^{\star}(\mu_0,\mu_1)}\sup_{(\varphi_0,\varphi_1)\in\mathcal D_{\pi^{\star}}}\left\{\int \varphi_0d(\nu_0-\mu_0)+\int \varphi_1d(\nu_1-\mu_1)\right\}\\
     &=  \inf_{\pi^{\star}\in\Pi^{\star}(\mu_0,\mu_1)}\sup_{(\varphi_0,\varphi_1)\in\mathcal D_{\pi^{\star}}}\left\{v_{ \varphi_0}^{\intercal}(w_{\nu_0}-w_{\mu_0})+v_{ \varphi_1}^{\intercal}(w_{\nu_1}-w_{\mu_1})\right\}.
    \end{aligned} 
    \]
where $\Pi^{\star}(\mu_0,\mu_1)$ is the set of all optimal plans for $\mathsf{GW}_p(\mu_{0},\mu_{1})$ and $\mathcal D_{\pi^{\star}}$ is the set of  OT potentials for the OT problem $\mathsf{OT}_{c_{\pi^{\star}}}(\mu_0,\mu_1)$ with
$
c_{\pi^{\star}}(x,y)=2\int \left|\kappa_0(x,x')-\kappa_1(y,y')\right|^pd\pi^{\star}(x',y').
$
\end{proposition} 

The proof of \cref{lem:GWDerivative} follows by establishing matching bounds on the $\limsup$ and $\liminf$ of the difference quotient as stated in \cref{lem:differenceQuotientUpperBound,lem:differenceQuotientLowerBound} which are proved in \cref{proof:lem:differenceQuotientUpperBound,proof:lem:differenceQuotientLowerBound}, respectively.

\begin{lemma}
\label{lem:differenceQuotientUpperBound}
In the setting of \cref{lem:GWDerivative},  
   \[
 \begin{aligned}
    \limsup_{t\downarrow 0}t^{-1}\left(\mathsf{GW}_p(\mu_{0,t},\mu_{1,t})^p-\mathsf{GW}_p(\mu_{0},\mu_{1})^p\right)& 
    \\
    &\hspace{-6.1em}\leq \inf_{\pi^{\star}\in \Pi^{\star}(\mu_0,\mu_1)}\sup_{(\varphi_0,\varphi_1)\in\mathcal D_{\pi^{\star}}} \left\{ \int \varphi_0d(\nu_0-\mu_0)+\int \varphi_1d(\nu_1-\mu_1)\right\}.
\end{aligned}
\] 
\end{lemma}

\begin{lemma}
\label{lem:differenceQuotientLowerBound}
In the setting of \cref{lem:GWDerivative},  
   \[
 \begin{aligned}
    \liminf_{t\downarrow 0}t^{-1}\left(\mathsf{GW}_p(\mu_{0,t},\mu_{1,t})^p-\mathsf{GW}_p(\mu_{0},\mu_{1})^p\right)& 
    \\
    &\hspace{-6em}\geq \inf_{\pi^{\star}\in \Pi^{\star}(\mu_0,\mu_1)}\sup_{(\varphi_0,\varphi_1)\in\mathcal D_{\pi^{\star}}} \left\{ \int \varphi_0d(\nu_0-\mu_0)+\int \varphi_1d(\nu_1-\mu_1)\right\}.
\end{aligned}
\]   
\end{lemma}

Together, Lemmas \ref{lem:differenceQuotientUpperBound} and \ref{lem:differenceQuotientLowerBound} yield \cref{lem:GWDerivative}.

We proceed by treating the case where $\varepsilon>0$ is arbitrary. 

\begin{proposition}[EGW directional derivative] 
\label{lem:differenceQuotientEOT}
       Let $(\nu_0,\nu_1)\in\mathcal P(\mathcal X_0)\times \mathcal P(\mathcal X_1)$ be arbitrary and let $\mu_{i,t}\coloneqq \mu_i+t(\nu_i-\mu_i)$ for $t\in(0,1)$ with $i=0,1$. Then,  
  \[
 \begin{aligned}
    \lim_{t\downarrow 0}t^{-1}\left(\mathsf{EGW}_p^{\varepsilon}(\mu_{0,t},\mu_{1,t})-\mathsf{EGW}_p^{\varepsilon}(\mu_{0},\mu_{1})\right)& 
    \\
    &\hspace{-5em}= \inf_{\pi^{\star}\in \Pi^{\star}(\mu_0,\mu_1)}\left\{ \int \varphi_0^{\pi^{\star}}d(\nu_0-\mu_0)+\int \varphi_1^{\pi^{\star}}d(\nu_1-\mu_1)\right\}
    \\
    &\hspace{-5em}= \inf_{\pi^{\star}\in \Pi^{\star}(\mu_0,\mu_1)}\left\{  v_{\varphi_0^{\pi^{\star}}}^{\intercal}\left(w_{\nu_0}-w_{\mu_0}\right)+v_{\varphi_1^{\pi^{\star}}}^{\intercal}\left(w_{\nu_1}-w_{\mu_1}\right)\right\}
\end{aligned}
\]   
where $\Pi^{\star}(\mu_0,\mu_1)$ is the set of all optimal couplings for $\mathsf{EGW}_p^{\varepsilon}(\mu_{0},\mu_{1})$ and $(\varphi_0^{\pi^{\star}},\varphi_1^{\pi^{\star}})$ is the unique  (up to additive constants) pair of  EOT potentials for the EOT problem $\mathsf{EOT}^{\varepsilon}_{c_{\pi^{\star}}}(\mu_0,\mu_1)$ with cost function 
$
c_{\pi^{\star}}(x,y)=2\int \left|\kappa_0(x,x')-\kappa_1(y,y')\right|^pd\pi^{\star}(x',y').
$  
\end{proposition}

Again, the result follows by establishing matching bounds on the $\limsup$ and $\liminf$ of the difference quotient, see \cref{lem:differenceQuotientUpperBoundEOT,lem:differenceQuotientLowerBoundEOT} which are proved in \cref{proof:lem:differenceQuotientUpperBoundEOT,proof:lem:differenceQuotientLowerBoundEOT}.

\begin{lemma}
\label{lem:differenceQuotientUpperBoundEOT}
In the setting of \cref{lem:differenceQuotientEOT},  
   \[
 \begin{aligned}
    \limsup_{t\downarrow 0}t^{-1}\left(\mathsf{EGW}_p^{\varepsilon}(\mu_{0,t},\mu_{1,t})-\mathsf{EGW}_p^{\varepsilon}(\mu_{0},\mu_{1})\right)& 
    \\
    &\hspace{-5em}\leq \inf_{\pi^{\star}\in \Pi^{\star}(\mu_0,\mu_1)}\left\{ \int \varphi_0^{\pi^{\star}}d(\nu_0-\mu_0)+\int \varphi_1^{\pi^{\star}}d(\nu_1-\mu_1)\right\}.
\end{aligned}
\]   
\end{lemma}

\begin{lemma}
\label{lem:differenceQuotientLowerBoundEOT}
In the setting of \cref{lem:differenceQuotientEOT},    
   \[
 \begin{aligned}
    \liminf_{t\downarrow 0}t^{-1}\left(\mathsf{EGW}_p^{\varepsilon}(\mu_{0,t},\mu_{1,t})-\mathsf{EGW}_p^{\varepsilon}(\mu_{0},\mu_{1})\right)& 
    \\
    &\hspace{-5em}\geq \inf_{\pi^{\star}\in \Pi^{\star}(\mu_0,\mu_1)}\left\{ \int \varphi_0^{\pi^{\star}}d(\nu_0-\mu_0)+\int \varphi_1^{\pi^{\star}}d(\nu_1-\mu_1)\right\},
\end{aligned}
\]   
\end{lemma}

With these results in hand, the proof of \cref{thm:limitTheorems} follows directly from \cref{lem:GWDerivative} and \cref{lem:differenceQuotientEOT} by applying the delta method as stated in \cref{prop:limitTheoremFramework}.

\subsection{Proof of \texorpdfstring{\cref{lem:vanishingGW}}{Proposition 4}}
\label{proof:lem:vanishingGW}
   Let $\kappa_0$ and $\kappa_1$ be as in the statement and suppose  that $\mathsf{GW}_p(\mu_0,\mu_1)=0$ so that there exists a coupling, $\pi$, satisfying $\int |\kappa_0(x,x')-\kappa_1(y,y')|^pd\pi\otimes \pi(x,y,x',y')=0$, that is, $\kappa_0(x,x')=\kappa_1(y,y')$ for $\pi\otimes \pi$-a.e. $(x,y,x',y')$. Fix $i\in[N_0]$ and  suppose that $\pi(\{(x_0^{(i)},x_1^{(j)})\})>0,$ $\pi(\{(x_0^{(i)},x_1^{(j')})\})>0$ for some $j\neq j'$. Then,  for each $x\in\mathcal X_0$ and $y_x\in\mathcal X_1$ for which $\pi(\{(x,y_x)\})>0$,
    \[
    \kappa_0(x_0^{(i)},x)= \kappa_1(x_1^{(j)},y_x)= \kappa_1(x_1^{(j')},y_x).
    \]
    However, since $\mu_1$ has full support on $\mathcal X_1$, the above display implies that $\kappa_1(x_1^{(j)},\cdot)= \kappa_1(x_1^{(j')},\cdot)$ contradicting the condition on $\kappa_1$. Conclude that $\pi(\{(x_0^{(i)},x_1^{(j)})\})$ is nonzero for precisely one choice of $j\in[N_1]$ and is zero otherwise. By the same logic, for any fixed $j\in[N_1]$, $\pi(\{(x_0^{(i)},x_1^{(j)})\})$ is zero for all but one choice of $i\in[N_0]$. Conclude that $\pi$ is supported on the graph of a bijective map $S$ and so $\pi\otimes \pi$ is supported on the set $\{(x,S(x),x',S(x')):x,x'\in\mathcal X_0\}$. It follows that $\kappa_0(x,x')=\kappa_1(S(x),S(x'))$ for each $x,x'\in\mathcal X_0$, i.e., $S\in\mathcal S$ and $S_{\sharp}\mu_0=\mu_1$. 

    On the other hand, if $\mu_1=S_{\sharp}\mu_0$ for some $S\in\mathcal S$, the coupling $\pi=(\mathrm{Id},S)_{\sharp}\mu_0$ is valid for $\mathsf{GW}_p(\mu_0,\mu_1)$ and $\int \Delta d\pi\otimes \pi=0$ since $\kappa_0(x,x')=\kappa_1(S(x),S(x'))$. Conclude that $\pi$ is optimal since $\mathsf{GW}_p(\mu_0,\mu_1)\geq 0$ by definition.
\qed

\subsection{Proof of \texorpdfstring{\cref{cor:simplifiedLimitLaw}}{Corollary 2}}
\label{proof:cor:simplifiedLimitLaw}
       Under the stated conditions,  \cref{lem:vanishingGW} yields that  every optimal coupling for $\mathsf{GW}_p(\mu_0,\mu_1)$ is of the form $(\Id,S)_{\sharp} \mu_0$ for some $S\in \mathcal S$ and, by \cref{thm:limitTheorems}, 
        \[
\sqrt n\left({\mathsf{GW}_p(\mu_{0,n},\mu_{1,n})^p- \mathsf{GW}_p(\mu_{0},\mu_{1})^p}\right)
     \stackrel{d}{\to}  \inf_{\pi^{\star}\in\Pi^{\star}(\mu_0,\mu_1)}\sup_{(\varphi_0,\varphi_1)\in\mathcal D_{\pi^{\star}}}\left\{v_{ \varphi_0}^{\intercal}G_{\mu_0}+v_{ \varphi_1}^{\intercal}G_{\mu_1}\right\}. 
        \]

        We proceed by studying the structure of the optimal potentials. 
        Fix an optimal coupling $\pi^{\star}=(\mathrm{Id},S)_{\sharp}\mu_0$. By the complementary slackness conditions for linear programming \cite[Theorem 4.5]{bertsimas1997introduction},   $(\varphi_0,\varphi_1)\in\mathcal D_{\pi^{\star}}$ 
        if and only if 
        \[
            \varphi_0(x)+\varphi_1(y)\leq 2\int \Delta(x,y,x',y') d\pi^{\star}(x',y')= 2\int \Delta(x,y,x',S(x')) d\mu_0(x')
        \]
        \text{ with equality }$\pi^{\star}\text{-a.e.}$
        In particular, for every $x\in\mathcal X_0$, equality holds at all pairs $(x,y)=(x,S(x))$ whereby 
        \[
        \begin{aligned}
        \varphi_0(x)+\varphi_1(S(x))&= 2\int \Delta(x,S(x),x',S(x')) d\mu_0(x')\\
        &= 2\int |\kappa_0(x,x')-\kappa_1(S(x),S(x'))|^p d\mu_0(x') =0
        \end{aligned}
        \]
        since $\kappa_0(x,x')=\kappa_1(S(x),S(x'))$ for every $x,x'\in\mathcal X_0$ by definition of $\mathcal S$. Since $S$ is a bijection, conclude that 
        $\varphi_1=-\varphi_0\circ S^{-1}$ so that $(\varphi_0,\varphi_1)\in\mathcal D_{\pi^{\star}}$ if and only if $(\varphi_0,\varphi_1)=(h,-h\circ S^{-1})$ for a function $h:\mathcal X_0\to\mathbb R$ satisfying 
        \[
            h(x)-h(S^{-1}(y))\leq 2\int \Delta(x,y,x',S(x'))d\mu_0(x'), 
        \]
        with equality when $y=S(x)$ as noted above. Since $y=S(x_y)$ for some $x_y\in\mathcal X_0$,   
        \[
        \begin{aligned}
            h(x)\mspace{-1mu}-\mspace{-1mu}h(x_y)\leq 2\int \Delta(x,S( x_y),x',S(x'))d\mu_0(x')&= 2\int |\kappa_0(x,x')\mspace{-1mu}-\mspace{-1mu}\kappa_1(S( x_y),S(x'))|^pd\mu_0(x')
            \\
            &= 2\int |\kappa_0(x,x')\mspace{-1mu}-\mspace{-1mu}\kappa_0( x_y,x')|^pd\mu_0(x'), 
       \end{aligned} 
        \]
        where we have again used the fact $S\in\mathcal S$. The claimed result follows as the above characterization holds for each optimal coupling; i.e., for every coupling of the form $(\Id,S)_{\sharp}\mu_0$  with $S\in\mathcal S$.
    \qed

\subsection{Proof of \texorpdfstring{\cref{prop:quantiles}}{Proposition 5}}
\label{proof:prop:quantiles}
       Recall from \cref{cor:simplifiedLimitLaw} that  
        \[
            L_* = \inf_{\substack{S\in\mathcal S\\\mu_1=S_{\sharp}\mu_0}}\sup_{h\in\mathcal H_{\mu_0}}\left\{ v_h^{\intercal}G_{\mu_0}+  v_{-h\circ S^{-1}}^{\intercal}G_{\mu_1} \right\}.     
        \]
       For each $\omega$ in the sample space $\Omega$, $L_*(\omega)\leq\sup_{h\in\mathcal H_{\mu_0}}\left\{ v_h^{\intercal}G_{\mu_0}(\omega)+  v_{-h\circ S^{-1}}^{\intercal}G_{\mu_1}(\omega) \right\}$ for any fixed $S\in\mathcal S$ satisfying $\mu_1=S_{\sharp}\mu_0$. We now show that the distribution of 
       \[\sup_{h\in\mathcal H_{\mu_0}}\left\{ v_h^{\intercal}G_{\mu_0}+  v_{-h\circ S^{-1}}^{\intercal}G_{\mu_1}\right\}\] coincides with that of $L$ under the current assumptions so that Theorem 1.A.1 in \cite{shaked2007stochastic} yields that $L$ dominates ${L_*}$ in the stochastic order and so $\mathbb P(L\leq t)\leq \mathbb P({L_*}\leq t)$ for every $t\in\mathbb R$, see equation (1.A.2) in \cite{shaked2007stochastic}. 

        To this end, remark that, for any $t\in\mathbb R$, 
        \[
        \begin{aligned}
            \mathbb P\left(\sup_{h\in\mathcal H_{\mu_0}}\left\{ v_h^{\intercal}G_{\mu_0}+  v_{-h\circ S^{-1}}^{\intercal}G_{\mu_1} \right\}\leq t \right)&
            \\
            \\
             &\hspace{-3em}=\mathbb P\left(\sup_{h\in\mathcal H_{\mu_0}}\left\{\sum_{i=1}^{N_0}  h(x_0^{(i)})\left((G_{\mu_0})_i-(T_S(G_{\mu_1}))_{i}\right) \right\}\leq t \right). 
        \end{aligned}
        \]
        Letting $R_t\coloneqq \left\{r\in\mathbb R^{N_0}:\sup_{h\in\mathcal H_{\mu_0}}v_h^{\intercal}r\leq t \right\}$ it follows that   
        \[
        \begin{aligned}
        \mathbb P\left(\sup_{h\in\mathcal H_{\mu_0}}\left\{ v_h^{\intercal}G_{\mu_0}+  v_{-h\circ S^{-1}}^{\intercal}G_{\mu_1} \right\}\leq t \right)&=\mathbb P\left( G_{\mu_0}-T_S(G_{\mu_1})\in R_t \right)
       \\ 
        &=\mathbb P\left( \sqrt{2}G_{\mu_0}\in R_t \right)
        \\
        &= \mathbb P\left(\sup_{h\in\mathcal H_{\mu_0}} \sqrt 2 v_h^{\intercal}G_{\mu_0}\leq t \right)
        \end{aligned} 
        \]
        where the penultimate equality follows since $ G_{\mu_0}-T_S(G_{\mu_1})\stackrel{d}{=}\sqrt 2 G_{\mu_0}$ under the assumptions, concluding the proof of the first part of the statement. 

   We now establish that  $\mathcal Q\coloneqq \{S\in\mathcal S:\mu_1=S_{\sharp}\mu_0\}$ is a singleton if and only if the identity is the unique map $T:\mathcal X_0\to\mathcal X_0$ for which $T_{\sharp}\mu_0=\mu_0$ and $\kappa_0(x,x')=\kappa_0(T(x),T(x'))$ for each $x,x'\in\mathcal X_0$.
   
       Suppose first that there exists some $T:\mathcal X_0\to\mathcal X_0$ with $\kappa_0(T(x),T(x'))=\kappa_0(x,x')$ for each $x,x'\in\mathcal X_0$, $\mu_0=T_{\sharp}\mu_0$, and $T\neq \Id$. Then, for any $S\in\mathcal Q$, $\mu_1=S_{\sharp}T_{\sharp}\mu_0 = (S\circ T)_{\sharp}\mu_0$ and $\kappa_0(x,x')=\kappa_0(T(x),T(x'))=\kappa_1(S\circ T(x),S\circ T(x'))$ for each $x,x'\in\mathcal X_0$ so that $S\circ T\in \mathcal Q$ and, since $T\neq \mathrm{Id}$, $S\circ T\neq S$. This shows that $\mathcal Q$ is not a singleton if such a $T$ exists. 

       Conversely, if there exists $S,S'\in\mathcal Q$ with $S\neq S'$, then $\mu_0 =S^{-1}_{\sharp} \mu_1= S^{-1}_{\sharp} S'_{\sharp}\mu_0 = (S^{-1}\circ S')_{\sharp}\mu_0$ and for each $x,x'\in\mathcal X_0$, $\kappa_0(x,x') =\kappa_1(S'(x),S'(x')) =\kappa_0(S^{-1}\circ S'(x),S^{-1}\circ S'(x'))$. Since $S\neq S'$, $(S^{-1}\circ S')\neq \mathrm{Id}$, proving the claim.
    \qed

\subsection{Proof of\texorpdfstring{\cref{thm:convergenceDirectEstimator}}{Proposition 6}}
\label{proof:thm:convergenceDirectEstimator} 

The proof of \cref{thm:convergenceDirectEstimator} will follow by establishing the following technical lemmas. 

 \begin{lemma}
    \label{lem:nonemptyInterior}
        Fix a measure $\eta\in\mathcal P(\mathcal X_0)$ and let $\supp(\eta)=(x^{(i)})_{i=1}^{N_{\eta}}$. Then, the set 
        \[
 \Gamma_{\eta} \coloneqq \left\{w \in \mathbb R^{N_{\eta}}:w_i-w_j\leq  2\int |\kappa_0(x^{(i)},x')-\kappa_0(x^{(j)},x')|^pd\eta(x'),\forall (i,j)\in [N_{\eta}]\times [N_{\eta}]\right\},
        \] 
         has nonempty interior if and only if, for each pair $(i,j)\in [N_{\eta}]\times [N_{\eta}]$ with $i\neq j$, there exists $x'\in\supp(\eta)$ satisfying     $\kappa_0(x^{(i)},x')\neq \kappa_0(x^{(j)},x')$.
    \end{lemma} 

     Note that the condition in \cref{lem:nonemptyInterior} is the same as that in \cref{lem:vanishingGW} and is hence automatically satisfied for $\mu_0$ by assumption. The proof of \cref{lem:nonemptyInterior} can be found in \cref{proof:lem:nonemptyInterior}.
     
     While $\Gamma_{\eta}$ is an unbounded set (since if $w\in\Gamma_{\eta}$ then so too is $w+a$ for any $a\in\mathbb R$), for any cost $c\in\mathbb R^{N_{\eta}}$ satisfying $c^{\intercal}\mathbf 1=0$, the linear program $\inf_{w\in\Gamma_{\eta}}c^{\intercal} w$ has the same optimal value as $\inf_{\substack{w\in\Gamma_{\eta}\\ w_1\in[-\delta,\delta]}}c^{\intercal} w$ for any choice of $\delta>0$ as explained in  \cref{rmk:implementationEstimator}. As the constraint set in the  latter problem is bounded, both linear programs have finite optimal value. This observation readily implies the first part of the following result. The second part is due to Lemma 11 in \cite{rioux2024limit}.  
    \begin{lemma}
    \label{lem:orthogonality}
        Under the conditions of \cref{assn:generalSettingLimitLaw}, $\hat L_n$ is almost surely finite for sufficiently large $n$. Moreover, if $\hat \Sigma_n = \Sigma_{\hat p_n}$ for some choice of simplex vector $\hat p_n$, and $Z\sim N(0,\hat \Sigma_n)$, then $\mathbb P(Z^{\intercal}\mathbf 1=0)=1$ so that item (4) of \cref{assn:generalSettingLimitLaw} is satisfied.  
    \end{lemma}

     We now consider the following, generic, linear program  
     \begin{equation}
     \label{eq:LPproofConsistencyDelta}
    \begin{aligned}
        \sup_{w \in \mathbb R^{N_{0}}}& c^{\intercal}w
        \\ 
        \text{subject to }&\tilde{\mathrm{R}}w\leq \tilde u,\; w_{1}\leq \delta,\;
        -w_{1}\leq \delta, 
    \end{aligned} 
     \end{equation}
     where $\delta>0$ is any fixed constant, $\tilde {\mathrm{R}}\in\mathbb R^{N_0(N_0-1)\times N_0}$ enforces the constraints on $w_i-w_j$ for each $(i,j)\in[N_0]\times[N_0]$ with $i\neq j$ and $c\in\mathbb R^{N_0}$, $\tilde u\in\mathbb R^{N_0(N_0-1)}$ are arbitrary. This is the same as the constraint set generated by $\mathrm{R}$ and $u$ as considered previously, but the redundant constraints on $w_i-w_i$ for $i\in[N_0]$ are all removed. 
      
     We denote the  optimal value of this linear program by $v_{\delta}(c,\tilde u)$ noting that, by the above deliberations, the feasible set is compact so that $v_{\delta}(c,\tilde u)<\infty$. Furthermore, if $c^{\intercal}\mathbf 1=0$, $v_{\delta}(c,\tilde u)$ coincides with $v(c,\tilde u)$, the optimal value of
     \[
    \begin{aligned}
        \sup_{w \in \mathbb R^{N_{0}}}& c^{\intercal}w
        \\ 
        \text{subject to }&\mathrm{\tilde R}w\leq \tilde u 
    \end{aligned} 
     \]
     provided that the resulting constraint set is nonempty.
     The following result shows that $v_{\delta}(c,u)$ is stable under perturbations of $c$ and $u$.  
           \begin{lemma}[Lemma S.1.15. in \cite{rioux2024limit}] \label{lem:convergenceOptimalValue}
                Fix $\delta>0$ and suppose that $\tilde u_n\to \tilde u$ for some $\tilde u$ satisfying $\min_{i=1}^{N_0(N_0-1)} \tilde u_i>0$. Then, $v_{\delta}(\cdot, \tilde u_n)\to v_{\delta}(\cdot, \tilde u) $ uniformly on each compact subset of $\mathbb R^{N_0}$.  
           \end{lemma}

        Given the above results, we are now in the position to prove the desired result.
    
\begin{proof}[Proof of \cref{thm:convergenceDirectEstimator}] 
        By \cref{assn:generalSettingLimitLaw}, we have that, conditionally on the data, the estimators $\hat \Sigma_n$ and $\hat u_n$ converge deterministically to $\Sigma_q$ and $u_{\mu_0}$ given almost every realization of $Y_1,Y_2,\dots$. As noted in the proof of \cref{lem:nonemptyInterior}, the vector $\tilde u_{\mu_0}$ with $l=N_0(i-1) +j$-th entry for $(i,j)\in[N_0]\times [N_0]$ with $i\neq j$   given by 
        \[
            (\tilde u_{\mu_0})_l = 2\int \left|\kappa_0(x^{(i)},x')-\kappa_0(x^{(j)},x')\right|^pd\mu_0(x')
        \]
        admits a uniform lower bound
         $\underline\beta>0$. It follows that $\tilde {\hat u}_n$, defined by removing the same indices from $\hat u_n$ as were removed from $u_{\mu_0}$ to form $\tilde u_{\mu_0}$, converges to $\tilde u_{\mu_0}$ given a.e.~realization of the samples.\footnote{Note that even if $\hat u_n$ may not be defined on  $\mathbb R^{N_0^2}$ for small $n$, it will be for all $n$ sufficiently large as discussed in \cref{rmk:implementationEstimator}.}   
        By \cref{lem:convergenceOptimalValue}, for any $\delta>0$, $v_{\delta}(\cdot, \tilde{\hat u}_n)\to v_{\delta}(\cdot, \tilde u_{\mu_0})$ uniformly on each compact set of $\mathbb R^{N_0}$.

        Now, let $Z_n\sim N(0,\hat \Sigma_n)$ for each $n\in\mathbb N$. Conditionally on the data, $\hat \Sigma_n\to \Sigma_q$ deterministically given almost every realization of $Y_1,Y_2,\dots$, $Z_n\stackrel{d}{\to} Z\sim N(0,\Sigma_q)$. By applying the continuous mapping theorem (cf. e.g. Theorem 1.11.1 in \cite{van1996weak}) it follows that, conditionally on the data, $v_{\delta}(Z_n,\tilde{\hat u}_n)\stackrel{d}{\to}v_{\delta}(Z,\tilde{u}_{\mu_0})$ given almost every realization of $Y_1,Y_2,\dots$, where we note that $\sqrt 2 v_{\delta}(Z_n,\tilde{\hat u}_n)\stackrel{d}= \sqrt 2 v(Z_n,\tilde{\hat u}_n)\stackrel d =\hat L_n $ and $\sqrt 2 v_{\delta}(Z,\tilde{u}_{\mu_0})\stackrel d =\sqrt 2  v(Z,\tilde{u}_{\mu_0}) \stackrel d = L$ as follows from \cref{lem:orthogonality} and the preceding discussion. The first claim will then follow from Lemma 2.11 in \cite{van1998asymptotic} upon showing that the distribution function of $L$ is continuous.

        Recall from the proof of \cref{lem:nonemptyInterior} that
        $\Gamma_{\mu_0}$ contains an open ball centered at $0$ and is symmetric in the sense that if $w\in\Gamma_{\mu_0}$, then $-w\in\Gamma_{\mu_0}$. It follows that $z\in\bm 1^{\perp}\mapsto \sup_{w\in\Gamma_{\mu_0}} w^{\intercal}z$ defines a norm on $\bm 1^{\perp}$, the orthogonal subspace to the vector $\bm 1$, which we denote by $\|\cdot\|_{\Gamma_{\mu_0}}$. Thus, $L\stackrel{d}{=} \sqrt 2 \|Z\|_{\Gamma_{\mu_0}}$ and, since $Z$ is not a point mass at $0$, but contains $0$ in its support, Proposition 12.1 in \cite{davydov1998local}, asserts that the  distribution function of $L$ is absolutely continuous, proving the desired result.

        To see that the empirical measures satisfy all relevant conditions when taking the empirical versions of $\Sigma_{\mu_0}$ and $u_{\mu_0}$, note that items (1) and (2) of \cref{assn:generalSettingLimitLaw} have already been shown in the proof of \cref{cor:simplifiedLimitLaw}, item (3)  follows from the strong law of large numbers, and item (4) is a consequence of \cref{lem:orthogonality}. Furthermore, since $\mu_0$ is not a point mass, $\Sigma_{\mu_0}\neq 0$ so that all conditions are verified.
        \end{proof}
    
\subsection{Proof of  \texorpdfstring{\cref{thm:conservativeTesting}}{Theorem 4}}
\label{proof:thm:conservativeTesting}
The claims under the null hypothesis follow directly 
from \cref{prop:quantiles,thm:convergenceDirectEstimator}. Under the alternative, $\mathsf{GW}_p(\mu_0,\mu_1)>0$, \cref{thm:limitTheorems} yields that 
\[
    \sqrt n\left(\mathsf{GW}_p(\mu_{0,n},\mu_{1,n})^p- \mathsf{GW}_p(\mu_0,\mu_1)^p\right) \stackrel{d}{\to} \inf_{\pi^{\star}\in\Pi^{\star}(\mu_0,\mu_1)}\sup_{(\varphi_0,\varphi_1)\in\mathcal D_{\pi^{\star}}}\left\{v_{ \varphi_0}^{\intercal}G_{\mu_0}+v_{ \varphi_1}^{\intercal}G_{\mu_1}\right\}\coloneqq \Upsilon.
\]
By the continuous mapping theorem, $0=\bm 1^{\intercal}(\sqrt n (w_{\mu_{0,n}}-w_{\mu_0}))\stackrel{d}{\to} \bm 1^{\intercal} G_{\mu_0}$ so that $G_{\mu_0}$ is almost surely orthogonal to $\bm 1$; the same argument applied to $G_{\mu_1}$. For any $\omega\in\Omega$ where this orthogonality condition is met,  \[
\begin{aligned}
\inf_{\pi^{\star}\in\Pi^{\star}(\mu_0,\mu_1)} \sup_{(\varphi_0,\varphi_1)\in\mathcal D_{\pi^{\star}}}\left\{v_{ \varphi_0}^{\intercal}G_{\mu_0}(\omega)+v_{ \varphi_1}^{\intercal}G_{\mu_1}(\omega)\right\}
\\
&\hspace{-4em}\leq {3}\|\Delta\|_{\infty,\mathcal X_0\times \mathcal X_1\times \mathcal X_0\times \mathcal X_1}\|(\|G_{\mu_0}(\omega)\|_1+\|G_{\mu_0}(\omega)\|_1),
\end{aligned}
\]
as follows from applying the bounds from \cref{lem:potentialBounds}. Consequently, $\Upsilon$ is almost surely finite and so, for any $0<\epsilon<1$, there exists $s_{\epsilon}\in\mathbb R$ for which $\mathbb P(\Upsilon \leq s_{\epsilon})<\epsilon$. Now,     
\[
\begin{aligned}
    &\mathbb P\left( \sqrt n\mathsf{GW}_p(\mu_{0,n},\mu_{1,n})^p\leq t\right) 
    \\ 
    &= \mathbb P\left( \sqrt n\left(\mathsf{GW}_p(\mu_{0,n},\mu_{1,n})^p-\mathsf{GW}_p(\mu_{0},\mu_{1})^p\right)\leq t-\sqrt n\mathsf{GW}_p(\mu_{0},\mu_{1})^p\right) 
    \end{aligned}
\]
and, for $n$ sufficiently large, $t-\sqrt n\mathsf{GW}_p(\mu_{0},\mu_{1})^p\leq s_{\epsilon}$ so that 
\[
\begin{aligned}
    \limsup_{n\to \infty } \mathbb P\left( \sqrt n\mathsf{GW}_p(\mu_{0,n},\mu_{1,n})^p\leq t\right)&\\&\hspace{-12em}\leq \limsup_{n\to\infty }\mathbb P\left( \sqrt n\left(\mathsf{GW}_p(\mu_{0,n},\mu_{1,n})^p-\mathsf{GW}_p(\mu_{0},\mu_{1})^p\right)\leq s_{\epsilon}\right) 
    \leq \mathbb P(\Upsilon\leq s_{\epsilon})<\epsilon 
\end{aligned}
\]
by the Portmanteau theorem. As $\epsilon>0$ is arbitrary, $\lim_{n\to \infty } \mathbb P\left( \sqrt n\mathsf{GW}_p(\mu_{0,n},\mu_{1,n})^p\leq t\right) = 0 $ under the alternative for any $t\in\mathbb R$. 

Now, let $t_{\alpha}$ be the $\alpha$-quantile of $L$ for $\alpha\in(0,1)$. Since the distribution function of $L$ is continuous (recall the proof of \cref{thm:convergenceDirectEstimator}), for any $r>0$, $\mathbb P(L\leq t_{\alpha}+r)>\alpha$ so that \cref{prop:quantiles} implies that any conditional $\alpha$-quantile, $t_{n,\alpha}$,  of $\hat L_n$ satisfies $t_{n,\alpha}<t_{\alpha}+r$ with probability approaching $1$. As such, 
\[
\begin{aligned}
    \mathbb P(\sqrt n \mathsf{GW}_p(\mu_{0,n},\mu_{1,n}))^p \leq t_{n,\alpha})&\leq \mathbb P(\sqrt n \mathsf{GW}_p(\mu_{0,n},\mu_{1,n}))^p \leq t_{n,\alpha}\cap t_{n,\alpha}\geq M)
    \\
    &+\mathbb P(\sqrt n \mathsf{GW}_p(\mu_{0,n},\mu_{1,n}))^p \leq t_{n,\alpha}\cap t_{n,\alpha}< M)
\end{aligned}
\]
Evidently, $\mathbb P(\sqrt n \mathsf{GW}_p(\mu_{0,n},\mu_{1,n}))^p \leq t_{n,\alpha}\cap t_{n,\alpha}< M)\leq \mathbb P(\sqrt n \mathsf{GW}_p(\mu_{0,n},\mu_{1,n}))^p \leq M)$ whereas $\mathbb P(\sqrt n \mathsf{GW}_p(\mu_{0,n},\mu_{1,n}))^p \leq t_{n,\alpha}\cap t_{n,\alpha}\geq M)\leq\mathbb P(t_{n,\alpha}\geq M)$. Setting $M=t_{n,\alpha}+r$ as above for some $r>0$, we have from the previous deliberations that  the right hand side of the above display converges to $0$, i.e., $\lim_{n\to\infty} \mathbb P(\sqrt n \mathsf{GW}_p(\mu_{0,n},\mu_{1,n}))^p \leq t_{n,\alpha}) = 0$.
\qed

\subsection{Proof of \texorpdfstring{\cref{prop:LipschitzContinuityGradient}}{Proposition 7}}
\label{proof:prop:LipschitzContinuityGradient}

   For each $(u,v)\in \RR^{r_0}\times\RR^{r_1}$, let $c_{(u,v)}:\mathcal X_0\times \mathcal X_1\to \mathbb R$ be the cost function associated with the cost vector $\rB_1^{\intercal}v-\rB_0^{\intercal}u$. It follows from \cref{lem:EOTPotentialBounds} that the unique EOT coupling for $\mathsf{EOT}_{\rB_1^{\intercal}v-\rB_0^{\intercal}u}^{\varepsilon}(\mu_0,\mu_1)$, $\pi_{(u,v)}^{\star}$, satisfies
    \[
    \label{eq:couplingLowerBound}
    \frac{d\pi_{(u,v)}^{\star}}{d\mu_0\otimes \mu_1}\left(x,y\right)\geq e^{\frac{-4\|c_{(u,v)}\|_{\infty,\mathcal X_0\times \mathcal X_1}}\varepsilon}, \text{ for each }(x,y)\in\mathcal X_0\times \mathcal X_1.
    \]
    Now, let $K\subset \mathbb R^{r_0}$ be an open bounded set so that, for each $(u,v)\in K\times \rB_1\mathcal K$, it holds that 
   \[
\pi_{(u,v)}^{\star}\left(\left\{\left(x,y\right)\right\}\right)\geq e^{\frac{-4\sup_{(u,v)\in K\times \rB_1\mathcal K}\|c_{(u,v)}\|_{\infty,\mathcal X_0\times \mathcal X_1}}\varepsilon}l_{\mu_0} l_{\mu_1}\eqqcolon l_{\min} >0, \text{ for each }(x,y)\in\mathcal X_0\times \mathcal X_1,
   \]
   where $l_{\mu_0}=\min_{x\in\mathcal X_0}\mu_0(\{x\})>0$ and similarly for $ l_{\mu_1}$ noting that $\mu_0,\mu_1$ have full support. By \eqref{eq:gradientObjective}, $\nabla \ell_{\varepsilon}(u)=u-\rB_0 x^{\star}_{u}, \text{ where } \{x^{\star}_{u}\}=\argmin_{x\in\mathcal K}\left\{\frac 12 x^{\intercal}\rB_1^{\intercal}\rB_1x-u^{\intercal}\rB_0x -\varepsilon \mathsf{H}(x)\right\}$. From \cref{lem:QPvsLPsoln}, $x^{\star}_u$ is also the unique solution of $\inf_{x\in\mathcal K}\left\{ (x^{\star}_u)^{\intercal}\rB_1^{\intercal}\rB_1 x-u^{\intercal}\rB_0x -\varepsilon\mathsf H(x)\right\}$ and hence can be identified with a solution of the EOT problem $\mathsf{EOT}_{\rB_1^{\intercal}\rB_1x^{\star}_u-\rB_0^{\intercal}u}^{\varepsilon}(\mu_0,\mu_1)$ so that all entries of $x^{\star}_u$ are bounded below by $l_{\min}$ uniformly in the choice of $u\in K$ by the above deliberations. 
   
    As noted in the proof of \cref{lem:derivativeObjective}, the following expression for $\ell$ holds, see \eqref{eq:QPDecomposition},
   \[
        \begin{aligned}
        \ell(u)
                &= \frac 12 \|u\|^2+\inf_{x\in\mathcal K}\left\{\frac 12  x^{\intercal}\rB^{\intercal}_1\rB_1 x - u^{\intercal}\rB_0x-\varepsilon \mathsf H(x) \right\}.
        \end{aligned}
    \]
    The Lagrangian of the inner minimization problem is given by  
    \[
        \mathcal L: (x,\lambda)\in\mathbb R^{N_0N_1}\times \RR^{N_0+N_1}\mapsto \frac 12 x^{\intercal}\rB^{\intercal}_1\rB_1 x - u^{\intercal}\rB_0x-\varepsilon \mathsf H(x) + \lambda^{\intercal}\left( \rA x-b \right), 
    \]
    noting that the entropy term enforces nonnegativity of $x$. By Proposition 4.4.2 in \cite{bertsekas2016nonlinear}, $x^{\star}_u\in \argmin_{x\in\mathcal K}\left\{\frac 12 x^{\intercal}\rB^{\intercal}_1\rB_1 x - u^{\intercal}\rB_0x-\varepsilon \mathsf H(x) \right\}$ if and only if there exists $\lambda^{\star}$ for which 
    \begin{equation}
        \label{eq:lagrangePoint}
        x^{\star}_u\in \argmin_{x\in\RR^{N_0N_1}}\left\{\frac 12
x^{\intercal}\rB^{\intercal}_1\rB_1 x - u^{\intercal}\rB_0x-\varepsilon \mathsf H(x) + (\lambda^{\star})^{\intercal}\left( \rA x-b 
    \right)\right\}\text{ and }\rA x^{\star}_u=b.
\end{equation}
As noted above, for each $u\in K$, $x^{\star}_u\in\mathcal K_{l}\coloneqq \left\{ x\in \RR^{N_0N_1}:l<x_i<2, i \in [N_0N_1] \right\}$ for any choice of $0<l<l_{\min}$. Furthermore, since $b=\rA x'$ for any $x'\in \cK$, 
    \[
        \lambda^{\intercal}\left( \rA x-b \right)=\lambda^{\intercal}\left( \rA x-\rA x' \right)= (x-x')^{\intercal}\rA^{\intercal} \lambda = (x-x')^{\intercal}\rN^{\intercal} \zeta_{\lambda},
    \]
    where, letting $r=\mathrm{rank}(\rA)$ and $\rA = \rP\Sigma\rQ^{\intercal}$ be the compact SVD of $\rA$ (i.e. $\rP \in \mathbb R^{(N_0+N_1)\times r}$, $\Sigma\in\mathbb R^{r\times r}$, and $\rQ^{\intercal}\in\mathbb R^{r\times N_0N_1}$), $\rN=\Sigma \rQ^{\intercal}\in \mathbb R^{r\times N_0N_1}$ and $\zeta_{\lambda}=\rP^{\intercal}\lambda\in \RR^r$. By construction,  $\rP^{\intercal}$ has full row rank and $\rN^{\intercal}$ has full column rank hence $\rP^{\intercal}\mathbb R^{N_0+N_1}=\mathbb
    R^r$. With this, the condition \eqref{eq:lagrangePoint} can be recast as the existence of some $\zeta^{\star}$ for which  
 \begin{equation}
 \label{eq:LagrangianEquationFinal}
     x^{\star}_u\in \argmin_{x\in\RR^{N_0N_1}}\left\{\frac12 x^{\intercal}\rB^{\intercal}_1\rB_1 x - u^{\intercal}\rB_0x-\varepsilon \mathsf H(x) + (\zeta^{\star})^{\intercal}\rN (x-x')
        \right\}\text{ and }\rN(x^{\star}_u-x')=0,
\end{equation}
where we note that the condition $\rA x^{\star}_u=b$ is equivalent to $\rP\rN (x^{\star}_u-x')=0$ which, in turn, is equivalent to $\rN (x^{\star}_u-x')=0$ since $\rP$ has full column rank (i.e., it is injective).
We now characterize the solution of \eqref{eq:LagrangianEquationFinal} (which is known to lie in $\mathcal K_l$) by analyzing the critical points of the function 
\[
    r_{u,\zeta}:x\in \mathcal K_{l} \mapsto \frac 12 x^{\intercal}\rB^{\intercal}_1\rB_1 x - u^{\intercal}\rB_0x-\varepsilon \mathsf H(x) + \zeta^{\intercal}\rN (x-x'), 
\]
for fixed $u\in K$ and $\zeta \in \mathbb R^r$. Here, $r_{u,\zeta}$ 
is smooth and strongly convex on $\mathcal K_l$ with gradient 
\[
\begin{aligned}
   \nabla r_{u,\zeta}(x) &=   \rB^{\intercal}_1\rB_1 x - \rB_0^{\intercal}u+\varepsilon(\mathbf 1+\log(x)) +\rN^{\intercal} \zeta,
\end{aligned}
\]
where $\mathbf 1$ is the vector of length $N_0N_1$ consisting of all 1's and $\log(x)$ is the componentwise application of $\log$ to the vector $x$. With this, we define   $
g:(u,x,\zeta)\in K \times\cK_{l}\times \mathbb R^{r}\mapsto
   \left(
       \nabla r_{u,\zeta}(x),\rN(x-x')
   \right) \in \mathbb R^{N_0N_1}\times \mathbb R^{r}$
and observe that \eqref{eq:LagrangianEquationFinal} can be recast as solving the system $g(u,\cdot)= (0,0)$.

We now prove the claimed result by applying the implicit function theorem. It is easy to see that $g$ is smooth and its Jacobians with respect to $u$ and $(x,\zeta)$ are  
\[
    D_u g(u,x,\zeta)= \begin{pmatrix}-\rB_0^{\intercal}\\ 0\end{pmatrix},\quad  D_{x,\zeta}g(u,x,\zeta)=\begin{pmatrix} 
            \rB^{\intercal}_1\rB_1 +\varepsilon\mathrm{diag}(1/x)&\rN^{\intercal}\\ \rN&0
    \end{pmatrix}.
\]
We now show that $D_{x,\zeta} g(u,x,\zeta)$ is invertible. Suppose that $(y,\sigma)$ are such that 
\[           \rB^{\intercal}_1\rB_1 y+\varepsilon\mathrm{diag}(1/x) y+\rN^{\intercal}\sigma = 0 \text{ and } \rN y = 0. 
\]
Multiplying the first term above by $y$ and using the fact that $y^{\intercal}\rN^{\intercal}\sigma=0$, we must have that $y^{\intercal}\rB^{\intercal}_1\rB_1 y+\varepsilon y^{\intercal}\mathrm{diag}(1/x) y=0$ whereby $y=0$ since  the matrix $\rB^{\intercal}_1\rB_1 +\varepsilon \mathrm{diag}(1/x)$ is positive definite. It follows that $\rN^{\intercal} \sigma = 0$ so that $\sigma=0$, recalling that $\rN^{\intercal}$ has full column rank and is hence injective. Conclude that $D_{x,\zeta}g(u,x,\zeta)$ is an invertible matrix and, by the implicit function theorem, for each $u\in K$, there exists a neighborhood $N_u$ of $u$ and a unique function $Z_u:N_u\to \mathbb R^{N_0N_1}\times \mathbb R^r$ satisfying $g(s,Z_u(s))=(0,0)$ for every $s\in N_u$, i.e., $Z_u(s) = (x^{\star}_s,\zeta^{\star}_s)$, the unique pair solving \eqref{eq:LagrangianEquationFinal} for $u=s$. Furthermore, $Z_u$ is continuously differentiable with Jacobian
\[
    DZ_u(s) = - D_{x,\zeta}g(s,Z_u(s))^{-1}D_u g(s, Z_u(s)). 
\]
Letting $\rT_s=\rB^{\intercal}_1\rB_1 +\varepsilon\mathrm{diag}(1/x^{\star}(s))$, $\{x^{\star}(s)\}\coloneqq \argmin_{x\in\mathcal K}\left\{\frac 12 x^{\intercal}\rB^{\intercal}_1\rB_1 x-s^{\intercal}\rB_0x-\varepsilon\sH(x) \right\}$ for each  $s\in N_u $, and noting that $\rT_s$ is positive definite and $\rN\rT_s^{-1}\rN^{\intercal}$ is invertible, since $\rN^{\intercal}$ is injective, the inversion formula for block matrices, Exercise 5.16 in \cite{abadir2005matrix}, yields
\[
    D_{x,\zeta}g(s,Z_u(s))^{-1}= \begin{pmatrix}
        \rT_s^{-1}+\rT_s^{-1}\rN^{\intercal}\left(-\rN\rT_s^{-1}\rN^{\intercal}  \right)^{-1}\rN\rT_s^{-1} &-\rT_s^{-1}\rN^{\intercal}\left(-\rN\rT_s^{-1}\rN^{\intercal}  \right)^{-1}
        \\
        -\left(-\rN\rT_s^{-1}\rN^{\intercal}  \right)^{-1}\rN\rT_s^{-1}&\left(-\rN\rT_s^{-1}\rN^{\intercal}  \right)^{-1}
    \end{pmatrix}.
\]
Consequently, 
\[
   DZ_u(s) = \begin{pmatrix}
        \rT_s^{-1}\rB_0^{\intercal}+\rT_s^{-1}\rN^{\intercal}\left(-\rN\rT_s^{-1}\rN^{\intercal}  \right)^{-1}\rN\rT_s^{-1}\rB_0^{\intercal} 
        \\
        -\left(-\rN\rT_s^{-1}\rN^{\intercal}  \right)^{-1}\rN\rT_s^{-1}\rB_0^{\intercal}
    \end{pmatrix},
\]
where the top submatrix corresponds to the Jacobian of the map $s\in N_u\mapsto x^{\star}_s$.  As $\nabla \ell(u) = u - \rB_0x^{\star}_u$, conclude that 
\[
   D( \nabla \ell)(u)= D^2\ell(u) = \Id-\rB_0 \rT_u^{-1}\rB_0^{\intercal}+\rB_0\rT_u^{-1}\rN^{\intercal}\left(\rN\rT_u^{-1}\rN^{\intercal}  \right)^{-1}\rN\rT_u^{-1}\rB_0^{\intercal},
\]
is the Hessian of $\ell$ at $u$.
As $\ell(u)=\frac 12 \|u\|^2+\inf_{x\in\mathcal K}\left\{ \frac 12 x^{\intercal}\rB^{\intercal}_1\rB_1 x - u^{\intercal}\rB_0x-\varepsilon \mathsf H(x) \right\}$, we can identify the Hessian of the second term in the sum with 
$
-\rB_0 \rT_u^{-1}\rB_0^{\intercal}+\rB_0\rT_u^{-1}\rN^{\intercal}\left(\rN\rT_u^{-1}\rN^{\intercal}  \right)^{-1}\rN\rT_u^{-1}\rB_0^{\intercal}
$
and hence this matrix must be negative semidefinite, as 
\[
u\in\RR^{r_0}\mapsto \inf_{x\in\mathcal K}\left\{ \frac 12 x^{\intercal}\rB^{\intercal}_1\rB_1 x - u^{\intercal}\rB_0x-\varepsilon \mathsf H(x) \right\}
\]
is a concave function (cf. e.g., Proposition 2.9 in \cite{rockafellar1998variational}).
Thus, for any $w\in\mathbb R^{r_0}$ with $\|w\|=1$, 
\[
w^{\intercal}D^2 \ell(u)w=\|w\|^2- \left\| \rT_u^{-1/2}\rB_0^{\intercal}w\right\|^2+\left\|\left(\rN\rT_u^{-1}\rN^{\intercal}  \right)^{-1/2}\rN\rT_u^{-1}\rB_0^{\intercal}w\right\|^2\leq \|w\|^2,
\] 
so that  $\lambda_{\max}\left(D^2 \ell(u)\right)\leq 1$. We now bound the minimal eigenvalue of the Hessian. For any $q\in\mathbb R^{N_0N_1}$ with $\|q\|=1$,  
$
q^{\intercal} \rT_u q =  q^{\intercal}\rB_1^{\intercal}\rB_1 q+\varepsilon \sum_{i=1}^{N_0N_1} \frac{q_i^2}{(x^{\star}_u)_i}$, so that 
\[
\lambda_{\min}(\rT_u) 
\geq \lambda_{\min}(\rB_1^{\intercal}\rB_1)+ \frac{\varepsilon}{(x^{\star}_u)_{\max}}\geq \lambda_{\min}(\rB_1^{\intercal}\rB_1)+ \varepsilon, %
\]
noting that $\lambda_{\min}(\rB_1^{\intercal}\rB_1)\geq 0$ since 
 $\rB_1^{\intercal}\rB_1$ is positive semidefinite and that $(x^{\star}_u)_{\max}$, the maximum value of $x^{\star}_u$, is at most $1$. 
Conclude that 
\[
\left\| \rT_u^{-1/2}\rB_0^{\intercal}w\right\|^2\leq \lambda_{\max}\left(\rT_u^{-1/2}\right)^2 \lambda_{\max}(\rB_0^{\intercal}\rB_0)\leq \frac{\lambda_{\max}(\rB_0^{\intercal}\rB_0)}{\lambda_{\min}(\rB_1^{\intercal}\rB_1)+ \varepsilon}, 
\]
where we have 
used the fact that  $\lambda_{\max}\left( \rT_{u}^{-1/2}\right) =1/\sqrt{\lambda_{\min}\left( \rT_{u}\right)}$ and $\lambda_{\max}(\rB_0^{\intercal}\rB_0)=\lambda_{\max}(\rB_0\rB_0^{\intercal})$.  In sum, for any $w\in\mathbb R^{r_0}$ with $\|w\|=1$, 
\[
w^{\intercal}D^2 \ell(u)w\geq 1-\frac{\lambda_{\max}(\rB_0^{\intercal}\rB_0)}{\lambda_{\min}(\rB_1^{\intercal}\rB_1)+ \varepsilon}, \text{ so that } \lambda_{\min}\left(D^2\ell(u)\right)\geq 1-\frac{\lambda_{\max}(\rB_0^{\intercal}\rB_0)}{\lambda_{\min}(\rB_1^{\intercal}\rB_1)+ \varepsilon}.  
\]  
Since $\ell$ is twice continuously differentiable on $\mathbb R^{r_0}$, Theorem 4.5 in \cite{rockafellar1997convex} yields that $\ell$ is convex if and only if its Hessian is positive semidefinite at each point. This proviso is met if, for instance, $\varepsilon\geq \lambda_{\max}(\rB_0^{\intercal}\rB_0)-\lambda_{\min}(\rB_1^{\intercal}\rB_1)$.

We now show that $\nabla \ell$ is Lipschitz continuous with the claimed constant. By the mean value inequality (cf. e.g., Example 2 on p.356 of \cite{apostol1958mathematical}), for each $u,u'\in \mathbb R^{r_0}$, there exists $\bar u\in[u,u']$, the connecting line between $u$ and $u'$, satisfying 
\[
    \|\nabla \ell(u)-\nabla \ell(u')\|\leq \|D^2\ell(\bar u)(u-u')\|\leq \|D^2\ell(\bar u)\|_{\mathrm{op}}\|u-u'\|,
\]
and, by the previous deliberations, $\|D^2\ell(\bar u)\|_{\mathrm{op}}\leq \max\left\{\left|\lambda_{\min}\left(D^2\ell(\bar u)\right)\right|,\left|\lambda_{\max}\left(D^2\ell(\bar u)\right)\right|\right\}$ which can be bounded as $\max\left\{1, \frac{\lambda_{\max}(\rB_0^{\intercal}\rB_0)}{\lambda_{\min}(\rB_1^{\intercal}\rB_1)+ \varepsilon}-1\right\}$ which is independent of  $\bar u$ so that this Lipschitz constant is global.
\qed

\subsection{Proof of \texorpdfstring{\cref{lem:alternativeGradient}}{Propositon 8}}
\label{proof:lem:alternativeGradient}
   Let $v^{\star}(u)$ and $x^{\star}_{(u,v^{\star}(u))}$ be as defined in the statement for some fixed $u\in\mathbb R^{r_0}$. Then, by the first-order optimality conditions for  \eqref{eq:innerMax}, $v^{\star}(u)=\rB_1 x^{\star}_{(u,v^{\star}(u))}$ (recall \cref{lem:derivativeObjective}) so that 
   \[
        \sup_{v\in\mathbb R^{r_1}} f_u(v) =\frac 12   \left(x^{\star}_{(u,v^{\star}(u))}\right)^{\intercal}\rB_1^{\intercal}\rB_1 x^{\star}_{(u,v^{\star}(u))}-u^{\intercal}\rB_0x^{\star}_{(u,v^{\star}(u))} -\varepsilon \mathsf H(x^{\star}_{(u,v^{\star}(u))}).
   \]
   As noted in the proof of \cref{lem:derivativeObjective}, for each fixed $u\in\mathbb R^{r_0}$, 
   \[
\frac 12 \|u\|^2+\sup_{v\in\mathbb R^{r_1}} f_{u}(v)= \frac 12 \|u\|^2+\inf_{x\in\mathcal K}\left\{\frac 12 x^{\intercal}\rB_1^{\intercal}\rB_1x-u^{\intercal}\rB_0x -\varepsilon \mathsf{H}(x)\right\}.
   \]
   Comparing the previous two displayed equations, we discern that $x^{\star}_{(u,v^{\star}(u))}$ is the unique solution of the strictly convex problem  $\inf_{x\in\mathcal K}\left\{\frac 12 x^{\intercal}\rB_1^{\intercal}\rB_1x-u^{\intercal}\rB_0x -\varepsilon \mathsf{H}(x)\right\}$.
\qed

\subsection{Proof of \texorpdfstring{\cref{lemma:fu-proprties}}{Proposition 9}}
\label{proof:lemma:fu-proprties}
The $1$-strong concavity of the objective follows by noting that $f_u$ is the sum of the $1$-strongly concave function $-\frac 12 \|\cdot\|^2$ and the concave function $v\in\mathbb R^{r_1}\mapsto \min_{x\in\mathcal K}\left\{(\rB_1^{\intercal} v-\rB_0^{\intercal}u)^{\intercal}x -\varepsilon \mathsf H(x)\right\}$, see Proposition 2.9 in \cite{rockafellar1998variational}.

Fix a bounded open set $K'\subset \mathbb R^{r_1}$.
   Following the proof of \cref{prop:LipschitzContinuityGradient},  for each  $v\in K'$, the unique solution,  $x^{\star}_v$, of $\min_{x\in\mathcal K'} \left\{(\rB_1^{\intercal}v-\rB_0^{\intercal}u)^{\intercal}x-\varepsilon\sH(x)\right\}$ has entries bounded below by a constant $l_{\min}>0$ which is uniform in the choice of $v\in K'$. 
 The  Lagrangian for this problem is given by  
    \[
        \mathcal L: (x,\lambda)\in\mathbb R^{N_0N_1}\times \RR^{N_0+N_1}\mapsto v^{\intercal}\rB_1 x - u^{\intercal}\rB_0x-\varepsilon \mathsf H(x) + \lambda^{\intercal}\left( \rA x-b \right), 
    \]
    noting that the entropy term enforces nonnegativity of $x$. By Proposition 4.4.2 in \cite{bertsekas2016nonlinear}, $x^{\star}_v\in \argmin_{x\in\mathcal K'}\left\{v^{\intercal}\rB_1 x - u^{\intercal}\rB_0x-\varepsilon \mathsf H(x) \right\}$ if and only if there exists $\lambda^{\star}$ for which 
    \begin{equation}
        \label{eq:lagrangePointv}
        x^{\star}_v\in \argmin_{x\in\RR^{N_0N_1}}\left\{ v^{\intercal}\rB_1 x - u^{\intercal}\rB_0x-\varepsilon \mathsf H(x) + (\lambda^{\star})^{\intercal}\left( \rA x-b 
    \right)\right\}\text{ and }\rA x^{\star}_v=b.
\end{equation}
Letting $\mathcal K_{l}\coloneqq \left\{ x\in \RR^{N_0N_1}:l<x_i<2,i\in[N_0N_1] \right\}$ for some $0<l<l_{\min}$, we have 
   from the prior discussion that the unique solution of \eqref{eq:lagrangePointv} is contained in $\mathcal K_{l}$ for each $v\in K'$. Following the proof of \cref{prop:LipschitzContinuityGradient}, we write $\mathrm P\Sigma \mathrm Q^{\intercal}$ for the compact SVD of $\mathrm{A}$ and let $r=\mathrm{rank}(\mathrm A)$. With this, we recast \eqref{eq:lagrangePointv} as  
  
  the existence of some $\zeta^{\star}\in\mathbb R^r$ for which  
 \begin{equation}
 \label{eq:LagrangianEquationFinalv}
     x^{\star}_v\in \argmin_{x\in\RR^{N_0N_1}}\left\{v^{\intercal}\rB_1 x - u^{\intercal}\rB_0x-\varepsilon \mathsf H(x) + (\zeta^{\star})^{\intercal}\rN (x-x')
        \right\}\text{ and }\rN (x^{\star}_v-x')=0.
\end{equation}
For any fixed $v\in K'$ and $\zeta \in \mathbb R^r$, the function 
\[
    r_{v,\zeta}:x\in \mathcal K_{l} \mapsto v^{\intercal}\rB_1 x - u^{\intercal}\rB_0x-\varepsilon \mathsf H(x) + \zeta^{\intercal}\rN (x-x') 
\]
is smooth and strongly convex with gradient 
\[
\begin{aligned}
   \nabla r_{v,\zeta}(x) &=   \rB^{\intercal}_1v - \rB_0^{\intercal}u+\varepsilon(\mathbf 1+\log(x)) +\rN^{\intercal} \zeta,
\end{aligned}
\]
where $\mathbf 1$ is the vector of length $N_0N_1$ consisting of all 1's and $\log(x)$ is the componentwise application of $\log$ to the vector $x$. With this, we define   $
g:(v,x,\zeta)\in K' \times \cK_{l}\times \mathbb R^{r}\mapsto
   \left(
       \nabla r_{v,\zeta}(x),\rN( x-x')
   \right) \in \mathbb R^{N_0N_1+r}$
and observe that \eqref{eq:LagrangianEquationFinalv} can be recast as solving the system $g(v,\cdot)= (0,0)$.

With this representation in hand, we will apply the implicit function theorem. The Jacobians of $g$ with respect to $v$ and $(x,\zeta)$ are  
\[
    D_v g(v,x,\zeta)= \begin{pmatrix}\rB_1^{\intercal}\\ 0\end{pmatrix},\quad  D_{x,\zeta}g(v,x,\zeta)=\begin{pmatrix} 
            \varepsilon\mathrm{diag}(1/x)&\rN^{\intercal}\\ \rN&0
    \end{pmatrix},
\]
and invertibility of  $D_{x,\zeta} g(v,x,\zeta)$ follows by complete analogy with the proof of \cref{prop:LipschitzContinuityGradient}. By the implicit function theorem, for each $v\in K'$, there exists a neighborhood $N_v$ of $v$ and a unique function $Z:N_v\to \mathbb R^{N_0N_1}\times \mathbb R^r$ satisfying $g(s,Z(s))={(0,0)}$ for every $s\in N_v$, i.e. $Z(s) = (x^{\star}_s,\zeta^{\star}_s)$, the unique pair solving \eqref{eq:LagrangianEquationFinalv} for $v=s$. Furthermore, $Z$ is continuously differentiable with Jacobian
\[
    DZ(s) = - D_{x,\zeta}g(s,Z(s))^{-1}D_v g(s, Z(s)). 
\]
By the properties of block matrices used in \cref{prop:LipschitzContinuityGradient}, 
\[
    D_{x,\zeta}g(s,Z(s))^{-1}= \begin{pmatrix}
        \rU_s^{-1}+\rU_s^{-1}\rN^{\intercal}\left(-\rN\rU_s^{-1}\rN^{\intercal}  \right)^{-1}\rN\rU_s^{-1} &-\rU_s^{-1}\rN^{\intercal}\left(-\rN\rU_s^{-1}\rN^{\intercal}  \right)^{-1}
        \\
        -\left(-\rN\rU_s^{-1}\rN^{\intercal}  \right)^{-1}\rN\rU_s^{-1}&\left(-\rN\rU_s^{-1}\rN^{\intercal}  \right)^{-1}
    \end{pmatrix},
\]
where $\rU_s=\varepsilon\mathrm{diag}(1/x^{\star}_s)$ and $\{x^{\star}_s\}\coloneqq \argmin_{x\in\mathcal K}\left\{ s^{\intercal}\rB_1 x-u^{\intercal}\rB_0x-\varepsilon\sH(x) \right\}$ for each  $s\in N_v$. Consequently, 
\[
   DZ(s) = \begin{pmatrix}
        -\rU_s^{-1}\rB_1^{\intercal}-\rU_s^{-1}\rN^{\intercal}\left(-\rN\rU_s^{-1}\rN^{\intercal}  \right)^{-1}\rN\rU_s^{-1}\rB_1^{\intercal} 
        \\
        \left(-\rN\rU_s^{-1}\rN^{\intercal}  \right)^{-1}\rN\rU_s^{-1}\rB_1^{\intercal}
    \end{pmatrix},
\]
where the top submatrix corresponds to the Jacobian of the map $s\in N_v\mapsto x^{\star}_s$. 

Since $f_u$ is concave and twice continuously differentiable, its Hessian is negative semidefinite so that the Lipschitz constant for the gradient will follow by bounding the minimum eigenvalue of the Hessian from below at each point $v\in\mathbb R^{r_1}$, the claim will then follow from the mean value theorem as described in the proof of \cref{prop:LipschitzContinuityGradient}.  For any $v\in \mathbb R^{r_1}$ and 
 $w\in\mathbb R^{r_1}$ with $\|w\|=1$,
\[
w^{\intercal}\left(D^2 f_u(v)\right)w = -\|w\|^2+\|\left(\rN\rU_v^{-1}\rN^{\intercal}\right)^{-1/2}\rN\rU_v^{-1}\rB_1^{\intercal} w\|^2 -\|\rU_v^{-1/2}\rB_1^{\intercal} w\|^2,
\]
which satisfies the lower bound $w^{\intercal}\left(D^2 f_u(v)\right)w\geq -1-\|\rU_v^{-1/2}\rB_1^{\intercal} w\|^2$. Now, note that 
\[
    \|\rU_v^{-1/2}\rB_1^{\intercal} w\|\leq \|\rU_v^{-1/2}\|_{\mathrm{op}}\|\rB_1^{\intercal}\|_{\mathrm{op}}\leq  \sqrt{\frac{\max_{i=1}^{N_0N_1}\left(x^{\star}_v\right)_i}{\varepsilon}} \|\rB_1^{\intercal}\|_{\mathrm{op}}\leq \varepsilon^{-1/2}\|\rB_1^{\intercal}\|_{\mathrm{op}},
\]
so that $\lambda_{\min}\left( D^2f_u(v)\right)\geq -1-\varepsilon^{-1}\|\rB_1\|_{\mathrm{op}}^2$, proving the claim.
\qed

\subsection{Proof of \texorpdfstring{\cref{lem:convergenceRateAccMin}}{Theorem 5}}
\label{proof:lem:convergenceRateAccMin}
As noted previously, this result is essentially due to Theorem 11 in \cite{rioux2024entropic}. In effect, the analysis of their algorithm applies to minimizing an $L$-smooth  function $F=R+H$ over a closed ball of radius $M/2$  where $R$ is $L''$-smooth and convex, and $H$ is $L'$-smooth and nonconvex assuming that only approximate gradients for $F$, denoted $\tilde \nabla F$, satisfying 
\[
    \sup_{u,u',u''\in \mathbb B_{R/2}}\|(\tilde \nabla F(u)-\nabla F(u))^{\intercal}(u'-u'')\|\leq \delta' 
\]
are available. In the current setting, the gradient oracle satisfies
\[
    \sup_{u,u',u''\in \mathbb B_{R_0}}\|(\tilde \nabla F(u)-\nabla F(u))^{\intercal}(u'-u'')\|\leq \sup_{u,u',u''\in \mathbb B_{R_0}}\|\tilde \nabla F(u)-\nabla F(u)\|\|u'-u''\|\leq 2R_0\eta'  
\]
so that Theorem 11 in \cite{rioux2024entropic} directly yields the claimed result up to setting $M = 2R_0$ and $\delta' = 2R_0\eta'$ in that result. 
\qed

\subsection{Proof of \texorpdfstring{\cref{thm:inexactGradientOracle}}{Theorem 6}}
\label{proof:thm:inexactGradientOracle} The proof follows by noting that 
\begin{equation}
\label{eq:triangleIneq}
\|\tilde \nabla \ell(u)-\nabla \ell(u)\| = \|\rB_0(\tilde x_{(u,\tilde v(u))})-\rB_0( x^{\star}_{(u,v^{\star}(u))})\| \leq \|\rB_0\|_{\infty,2}\|\tilde x_{(u,\tilde v(u))}- x^{\star}_{(u,v^{\star}(u))}\|_{\infty},  
\end{equation}
so that it suffices to  control $\|\tilde x_{(u,\tilde v(u))}-x^{\star}_{(u,\tilde v(u))}\|_{\infty}+ \|x^{\star}_{(u,v^{\star}(u))}-x^{\star}_{(u,\tilde v(u))}\|_{\infty}$. By assumption, the first term is bounded as $e^{\delta}-1$ whereas the second term requires a more careful analysis.

   Following the proof of \cref{prop:LipschitzContinuityGradient}, for fixed $u\in\mathbb B_{R_0}$, the unique minimizer,  $x^{\star}_v$, of 
   \[
   g_{v}: x\in\mathbb R^{N_0N_1} \mapsto (\rB_1^{\intercal} v-\rB_0^{\intercal} u)^{\intercal} x - \varepsilon \sH(x) 
   \] on $\mathcal K$ has entries bounded below by $e^{\frac{-4\max_{i=1}^{N_0N_1}\left|\left(\rB_1^{\intercal}v-\rB_0^{\intercal}u\right)_i\right|}{\varepsilon}}l_{\mu_0}l_{\mu_1}\geq l_{\min}>0$ where  $l_{\min}$ is uniform over $v\in \mathbb B_{R_1}$. On the set $U_{m,M}\coloneqq \{x\in\mathbb R^{N_0N_1}:m<\min_{i=1}^{N_0N_1} x_i,\max_{i=1}^{N_0N_1} x_i<M\}$, $g_v$  is smooth with 
   \[
        \nabla g_v(x)= \rB_1^{\intercal} v-\rB_0^{\intercal} u +\varepsilon(\mathbf 1+\log(x)),\;D^2 g_v(x)= \varepsilon\mathrm{diag}(1/x).
   \] It follows that the eigenvalues of the Hessian are uniformly bounded below by $\frac{\varepsilon}{M}$ and above by $\frac{\varepsilon}{m}$ so that $g_v$ is strongly convex on $U_{m,M}$ with modulus $\frac{\varepsilon}{M}$. If $m<l_{\min}$, the unique minimizer, $x^{\star}(v)$, of $g_v$ over $\mathcal K$ is an element of $U_{m,1+s}$ for any $s>0$. Hence, for each $x\in\mathcal K\cap U_{m ,1+s}$ and $t\in[0,1]$, 
   \[
       g_v(x^{\star}(v))\leq g_v((1-t)x+tx^{\star}(v))\leq (1-t) g_v(x)+tg_v(x^{\star}(v))-\frac 12t(1-t)\frac{\varepsilon}{1+s}\|x-x^{\star}(v)\|^2,
   \]
   whereby $g_v(x^{\star}(v))+\frac{1}{2}\frac{t\varepsilon}{1+s}\|x-x^{\star}(v)\|^2\leq g_v(x)$ for each $t\in[0,1)$. As $s>0$ is arbitrary,
   \begin{equation}
   \label{eq:secondOrderGrowth}
        g_v(x^{\star}(v))+\frac{\varepsilon}{2}\|x-x^{\star}(v)\|^2\leq g_v(x).
   \end{equation}
   We also have that, for any $v,v'\in K'$ and $x,x'\in\mathcal K$, 
   \begin{equation}
   \label{eq:LipschitzDiff}
   \begin{aligned}
        \left|(g_v(x)-g_{v'}(x))-(g_v(x')-g_{v'}(x'))\right|&= \left|v^{\intercal}\rB_1x-(v')^{\intercal}\rB_1x-v^{\intercal}\rB_1x'+(v')^{\intercal}\rB_1x'\right|
        \\
        &=\left|(v-v')^{\intercal}\rB_1(x-x')\right|
        \\
        &\leq \|\rB_1\|_{\mathrm{op}}\|x-x'\|\|v-v'\|.
    \end{aligned} 
    \end{equation}
Together, \eqref{eq:secondOrderGrowth} and \eqref{eq:LipschitzDiff} imply that if $x^{\star}(v')\in\mathcal K$ minimizes $g_{v'}$ over $\mathcal K$ and similarly for $x^{\star}(v)$,
\[
    \|x^{\star}(v')-x^{\star}(v)\|\leq  2\varepsilon^{-1}\|\rB_1\|_{\mathrm{op}}\|v-v'\|, 
\]
as follows from Proposition 4.32 in \cite{bonnans2013perturbation}.

Combining this bound with \eqref{eq:triangleIneq} and using that $\|\tilde v(u)-v^{\star}(u)\|\leq \tau$, we have that 
\[
\|\tilde \nabla \ell(u)-\nabla \ell(u)\|  \leq \|\rB_0\|_{\infty,2}\left(e^{\delta}-1 +2\varepsilon^{-1}\|\rB_1\|_{\mathrm{op}} \tau \right),
\]
proving the claim.  
\qed
\subsection{Proof of \texorpdfstring{\cref{lem:convergenceRateMin}}{Proposition 10}} 
\label{proof:lem:convergenceRateMin} 
    We will apply the results of \cite{nabou2025proximal} which extend \cite{devolder2014first} in the case that the objective is of the form $h+r$ where $r$ is a closed convex function taking values in $\mathbb R\cup\{+\infty\}$ and $h$ is closed, but possibly nonconvex, bounded away from $-\infty$, and  an oracle for the gradient of $h$ at each $u\in\mathbb B_{R_0}$, $g_{\delta,L,q}(u)$ is available which satisfies 
    \[
        h(u)-\left(h(u')+ g_{\delta,L,q}(u)^{\intercal}(u-u')\right)\leq \frac{L}{2}\|u-u'\|^2+\delta\|u-u'\|^q \text{ for every }u'\in\mathbb B_{R_0},
    \]
    for some $q\in[0,2)$, $\delta\geq 0$, and $L\geq 0$. 
    In our case, $r=\frac 12 \|\cdot\|^2 +\mathcal I_{\mathbb B_{R_0}}$ and $h=\ell(u)-\frac 12\|u\|^2$. Remark that if we set  $\tilde \nabla h = \tilde \nabla \ell-u$, then  $\sup_{u\in \mathbb B_{R_0}}\|\tilde\nabla h(u)-\nabla h(u)\|=\sup_{u\in \mathbb B_{R_0}}\|\tilde\nabla \ell(u)-\nabla \ell(u)\|\leq \eta'$ by assumption. It follows  that the above inequality holds for $g_{\delta,L,q}(u)=\tilde \nabla h(u)$ with $q=0$, $\delta = 4R_0\eta'$,  and $L=L'$ by Remark 1 in \cite{nabou2025proximal}.  

    The claimed convergence rate then follows from item 1 of Corollary 1 in \cite{nabou2025proximal}, upon showing that  their Algorithm 1 corresponds to our \cref{algo:outer-loop-simple}.
    Their algorithm works by obtaining $g_{\delta,L,q}(u_k)$ and updating $u_{k+1}$ as
    \[ 
    \begin{aligned}
        u_{k+1}&=\mathrm{prox}_{\frac 1{L'} r}(u_k-(L')^{-1}g_{\delta,L,q}(u_k))
        \\
        &= \argmin_{u\in\mathbb B_{R_0}}\left\{\frac 12 \|u\|^2+\frac{L'}{2}\|(u_k-(L')^{-1}g_{\delta,L,q}(u_k))-u\|^2\right\}
        \\
        &= \argmin_{u\in\mathbb B_{R_0}}\left\{\frac 12\left(1+L'\right) \left\|u- L'\left(1+L'\right)^{-1}\left(u_k{-}(L')^{-1}g_{\delta,L,q}(u_k)\right)\right\|^2\right\}, 
    \end{aligned}
    \]
    where in the final line we have neglected terms which do not depend on $u$. From the above representation, we see that $u_{k+1}$ is obtained by projecting $L'\left(1+L'\right)^{-1}(u_k-(L')^{-1}g_{\delta,L,q}(u_k))$ onto $\mathbb B_{R_0}$. Note, moreover, that 
    \[
        \frac{L'}{1+L'}(u_k-(L')^{-1}g_{\delta,L,q}(u_k))= u_k-\frac{1}{1+L'}(u_k+ g_{\delta,L,q}(u_k))=u_k-\frac{1}{1+L'}g_k 
    \] 
    where $g_k$ is computed on Line 2 of \cref{algo:outer-loop-simple}. Conclude that the iterations in \cref{algo:outer-loop-simple} do indeed coincide with those of Algorithm 1 in \cite{nabou2025proximal} so that the claimed convergence rate result holds.
    
    The bound on $L'$ follows directly from the end of the proof of \cref{prop:LipschitzContinuityGradient}, noting that $h$ is concave and smooth so that the largest eigenvalue of its Hessian is nonpositive and its smallest eigenvalue is bounded below by  $-\frac{\lambda_{\max}(\rB_0^{\intercal}\rB_0)}{\lambda_{\min}(\rB_1^{\intercal}\rB_1)+\varepsilon}$. 
\qed

\subsection{Proof of \texorpdfstring{
\cref{lem:convergenceRateAccMax}}{Proposition 11}}
\label{proof:lem:convergenceRateAccMax}
  The proof is a direct consequence of Theorem 2.2 in \cite{aspremont2008smooth}, where we choose the prox-function in their setting as $\frac 12\|\cdot\|^2$ which is $1$-strongly convex. Of note is that since $\|\tilde \nabla f_u(v)-\nabla f_u(v)\|\leq \eta$, it holds that 
   \[
        \sup_{v,y,z\in \mathbb B_{R_1}}\left|\left\langle\tilde \nabla f_u(v)-\nabla f_u(v),y-z\right\rangle\right|\leq \sup_{v,y,z\in \mathbb B_{R_1}}\left\|\tilde \nabla f_u(v)-\nabla f_u(v)\right\|\left\|y-z\right\|\leq 2R_1\eta,
   \]
   so that Equation (2.3) in \cite{aspremont2008smooth} is satisfied by our notion of approximate gradient. As the gradient of $f_u$ is $(1+\varepsilon^{-1}\|\rB_1\|_{\mathrm{op}}^2)$-Lipschitz, Theorem 2.2 in \cite{aspremont2008smooth} yields that \[
        f_u(v^{\star})-f_u(v_j) \leq  \frac{(1+\varepsilon^{-1}\|\rB_1\|_{\mathrm{op}}^2)\frac 12 \|v^{\star}\|^2}{\frac 14(j+1)(j+2)}
        +6R_1\eta
   \]
   for this sequence of iterates and step size. Furthermore, we may lower bound $f_u(v^{\star})-f_u(v_j)$ by $\frac 12\|v_j-v^{\star}\|^2$ by using the $1$-strong convexity of $-f_u$ exactly as in the end of the proof of \cref{lem:convergenceRateMax}. 
   
   Following the same logic as in the proof of Theorem 9 in \cite{rioux2024entropic}, it is easily seen that the iterates of \cref{algo:inner-loop-accelerated} coincide with those studied in \cite{aspremont2008smooth}.

       It remains to show that $\frac 12\|v_j-v^{\star}\|^2\leq f_u(v^{\star})-f_u(v_j)$. As 
    $-f_u$ is $1$-strongly convex, 
    \[-f_u(v^{\star})\leq
        -f_u(tv_j+(1-t)v^{\star})\leq -tf_u(v_j)-(1-t)f_u(v^{\star})-\frac 12 t(1-t)\|v_j-v^{\star}\|^2 \text{ for all }t\in[0,1].
    \]
    Rearranging this equation yields $\frac 12 t(1-t)\|v_j-v^{\star}\|^2\leq t\left(f_u(v^{\star})-f_u(v_j)\right)$ so that $\frac 12 (1-t)\|v_j-v^{\star}\|^2\leq f_u(v^{\star})-f_u(v_j)$ for every $t\in(0,1]$. Taking the limit $t\downarrow 0$ yields the claimed lower bound.
\qed

\subsection{Proof of \texorpdfstring{\cref{lem:convergenceRateMax}}{Proposition 12}}
\label{proof:lem:convergenceRateMax}
   This result is a direct consequence of Theorem 2 and the discussion preceding equation (37) in \cite{devolder2014first} once it is shown that $-\tilde \nabla f_{u}(v_j)$ enables constructing a $(\delta,L)$-oracle for $-f_u$ at $v_j$ in the sense of Definition 1 in \cite{devolder2014first}. Precisely, a convex function $f$ on a convex set $Q\subset \mathbb R^d$ is endowed with a $(\delta,L)$-oracle if, for each $y\in Q$, a pair $(f_{\delta,L}(y),g_{\delta,L}(y))\in\mathbb R\times \mathbb R^d$ can be computed which satisfy
    \[
        0\leq f(x)-\left(f_{\delta,L}(y)+\left(g_{\delta,L}(y)\right)^{\intercal}(x-y)\right)\leq \frac L2\|x-y\|^2+\delta\text{ for each }x\in Q.
    \] 
    $f_{\delta,L}(y)$ and $g_{\delta,L}(y)$ can thus be thought of as approximations of $f$  and its gradient at the point $y$.
    We note that in \cref{algo:inner-loop-simple}, no function calls are made and, as such, we may assume that we have access to exact function evaluations  at each $y\in \mathbb B_{R_1}$, the closed ball of radius $R_1$ centered at $0$. 
   
    Following Section 2.3 b in \cite{devolder2014first}, a convex function $f$ with $M$-Lipschitz continuous gradient satisfies the property that if $\tilde \nabla f$ is an approximate gradient for $f$ satisfying $\sup_{y\in \mathbb B_{R_1}}\|\nabla f(y)-\tilde \nabla f(y)\|\leq \Delta$, then $\left(f(y)-2R_1\Delta ,\tilde\nabla f(y)\right)$ serves as a $(4R_1\Delta,M)$-oracle for $f$ on $\mathbb B_{R_1}$, noting that no function calls are used in the algorithm and hence they may be assumed to be exact. As $f_u$ has $M=1+\varepsilon^{-1}\|\rB_1\|_{\mathrm{op}}^2$-Lipschitz gradient by \cref{lemma:fu-proprties} and it is assumed that the accuracy of Sinkhorn's algorithm is such that $\sup_{v\in \mathbb B_{R_1}}\|\nabla f_u(v)-\tilde \nabla f_u(v)\|\leq \eta$, we see that $-\tilde \nabla f_u$ can be used to construct a $(4R_1\eta,(1+\varepsilon^{-1}\|\rB_1\|_{\mathrm{op}}^2))$-oracle for $-f_u$. It follows from Theorem 2 in \cite{devolder2014first} that 
    \[\min_{j=0}^{J-1}\left(f_u(v^{\star})-f_u(v_j)\right)\leq
     \frac{(\varepsilon^{-1}\|\rB_1\|_{\mathrm{op}}^2+1) \|v_0 - v^{\star} \|^2}{2J} + 4{R_1}\eta.
    \]
    Following the end of the  proof of \cref{lem:convergenceRateAccMax}, $\min_{j=0}^{J-1}\left(f_u(v^{\star})-f_u(v_j)\right)$ can be bounded below by $\min_{j=0}^{J-1}\frac 12\|v_j-v^{\star}\|^2$, proving the desired assertion. 
   \qed

\subsection{Proof of \texorpdfstring{\cref{prop:GWNullGraph}}{Proposition 13}}
\label{proof:prop:GWNullGraph}
We first show that if $\mu_1=(P_{\sigma})_{\sharp}\mu_0$, then $\mathsf{GW}_p(\mu_0,\mu_1)=0$ under this choice of kernel, see \cref{proof:lem:kappaVanishing} for the proof.

\begin{lemma}
\label{lem:kappaVanishing}
Suppose that $\kappa_0=\kappa_1=\kappa$ and that
    $\mu_1=(P_{\sigma})_{\sharp} \mu_0$ for some permutation $\sigma:[N]\to [N]$. Then, it holds that $\mathsf{GW}_p(\mu_0,\mu_1)=0$. 
\end{lemma}

For the opposite direction, recall from \cref{lem:vanishingGW} that $\mathsf{GW}_p(\mu_0,\mu_1)=0$ under the kernel $\kappa$ if and only if $\mu_1=B_{\sharp}\mu_0$ for a map satisfying $\kappa(G,G')=\kappa(B(G),B(G'))$ for every $G,G'\in\supp(\mu_0)$ provided that, for each $G,G'\in\supp(\mu_0)$ with $G\neq G'$, $\kappa(G,\cdot)\neq \kappa(G',\cdot)$ as functions on $\supp(\mu_0)$ and similarly for $\supp(\mu_1)$. To leverage this result, it will be convenient to first evaluate the kernel values for some simple graphs. 

Here and in the sequel, $G_{ij}$ denotes the graph with  a single edge connecting the $i$-th and $j$-th vertex for some $i,j\in[N]$.

\begin{lemma}
\label{lem:kappaFormulas}
Fix a graph $G\in\mathcal G_N$ with $m>0$ edges. Then, $\kappa(G,G)\geq m+4$ with equality if and only if $G=G_{ij}$ for some $i,j\in[N]$ with $i\neq j$. Furthermore, for each $i,j\in[N]$ with $i\neq j$,
 \begin{equation}
      \label{eq:kappaComp}
            \kappa(G,G_{ij}) = \begin{cases} 2(\mathrm{deg}(v_i(G))+\mathrm{deg}(v_j(G)))+1, & \text{if } (\mathrm{A}_G)_{ij}=1 ,
            \\
            2(\mathrm{deg}(v_i(G))+\mathrm{deg}(v_j(G))),&\text{otherwise},
            \end{cases}
      \end{equation}
        where $v_i(G)$ denotes the $i$-th vertex of $G$.
\end{lemma}

The proof of this result follows essentially by direct computation, see \cref{proof:lem:kappaFormulas}. 

With \cref{lem:kappaFormulas} in hand, the following result is a direct consequence of \cref{lem:vanishingGW}, see \cref{lem:pushforwardB} for complete details. 
\begin{lemma}
\label{lem:pushforwardB}
   In the setting of \cref{prop:GWNullGraph}, if  $\mathsf{GW}_p(\mu_0,\mu_1)=0$, then $\mu_1=B_{\sharp}\mu_0$ for a bijection $B:\supp(\mu_0)\to \supp(\mu_1)$ satisfying $\kappa(G,G')=\kappa(B(G),B(G'))$ for every $G,G'\in\supp(\mu_0)$. 
\end{lemma}

It remains to show that the map $B$ from \cref{lem:pushforwardB} is necessarily induced by a permutation of vertices. To this end, we first show that it corresponds to a permutation of nodes on the set of single edge graphs.

\begin{lemma}
\label{lem:permutationEdges}
    Let $B$ be as in \cref{lem:pushforwardB}. Then, for each $i,j\in[N]\times [N]$ with $i< j$, $B(G_{ij})=G_{\sigma(i)\sigma(j)}$ for some fixed permutation $\sigma:[N]\to [N]$. 
\end{lemma}

The proof of the previous result, see \cref{proof:lem:permutationEdges}, illustrates why the assumption $N\neq 4$ is made since that is the only case where there exist two distinct graphs with $N-1$ edges which satisfy the property that each pair of edges shares one vertex, but are not isomorphic, see \cref{fig:triangleClaw}.
\begin{figure}[!htb]
    \centering
\begin{tikzpicture}[every node/.style={circle, fill=black, inner sep=1.5pt}]

\begin{scope}[xshift=0cm]
  \node[label=above:$1$] (a) at (0,1) {};
  \node[label=below left:$2$] (b) at (-1,0) {};
  \node[label=below right:$3$] (c) at (1,0) {};
  \node[label=below:$4$] (x) at (0,-1) {};

  \draw (a) -- (b) -- (c) -- (a);
\end{scope}

\begin{scope}[xshift=6cm]
  \node[label=above:$1$] (u) at (0,1) {};
  \node[label=left:$2$] (v) at (-1,0) {};
  \node[label=right:$3$] (w) at (1,0) {};
  \node[label=below:$4$] (x) at (0,-1) {};

  \draw (x) -- (u);
  \draw (v) -- (u);
  \draw (w) -- (u);
\end{scope}
\end{tikzpicture}
\caption{Examples of graphs on $N=4$ nodes with $N-1$ edges for which each pair of edges shares a common vertex. The graph on the left is called a triangle graph and that on the right is a claw graph. Note that these graphs are not isomorphic; this can only occur in the case $N=4$.}
    \label{fig:triangleClaw}
\end{figure}
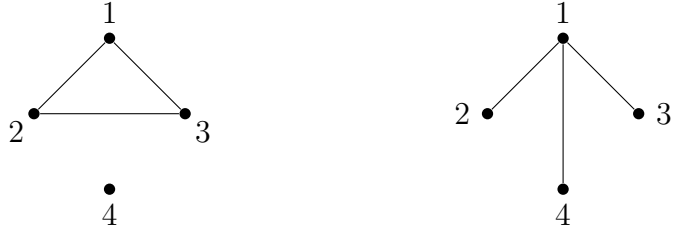

Given the above result, the remainder of the proof of  \cref{prop:GWNullGraph} is straightforward.

\begin{proof}[Proof of \cref{prop:GWNullGraph}]
    It remains to show that if $G\in\supp(\mu_0)$ is not a graph with a single edge, then $(\mathrm{A}_{G})_{\sigma(i)\sigma(j)} =(\mathrm{A}_{B(G)})_{ij}$. To this end, recall from \eqref{eq:kappaComp} that $\kappa(G,G_{ij})$ is odd if $(\mathrm{A}_G)_{ij}=1$ and is even otherwise. From \cref{lem:pushforwardB}, $\kappa(G,G_{ij})=\kappa(B(G),B(G_{ij}))=\kappa(B(G),G_{\varsigma(i)\varsigma(j)})$ where $\varsigma$ is the permutation from \cref{lem:permutationEdges}. It follows that $G$ contains an edge between the $i$-th and $j$-th vertex if and only if $B(G)$  contains an edge between the $\varsigma(i)$-th and $\varsigma(j)$-th vertex. Conclude that the desired equality holds with $\sigma=\varsigma^{-1}$.   
\end{proof}

\section{Concluding Remarks}
This work has furnished new variational forms for GW problems between arbitrary distributions supported on finite metric spaces. This result enables tying the GW problem with or without entropic regularization to a class of parameterized OT problems, allowing us to utilize powerful tools from  OT theory. %
This new representation was leveraged to obtain a novel and comprehensive statistical theory for finite GW, encompassing finite-sample convergence rates, limit laws, and a framework for testing if two distributions are isomorphic based on samples. On the computational side, new first-order algorithms were proposed for solving the regularized variational form and their nonasymptotic convergence properties were characterized even in the presence of inexactness in the gradient evaluations. This inexactness stems from the fact that computing gradients requires solving auxiliary optimization problems which do not generally admit closed-form solutions.

These developments were applied in the context of testing for isomorphisms of graph distributions. Namely, the aforementioned hypothesis testing framework was leveraged with particular choices of kernels to obtain consistent tests for the problem of verifying if two sets of sampled graphs come from the same distribution modulo a permutation of the labels. To compute the test statistic,  the algorithms developed for regularized problems were adapted and were seen to yield reasonable results even though the derived convergence guarantees do not apply without regularization.

Many natural research questions stem from this new representation of Gromov-Wasserstein distances. Chief among them is the problem of extending this result beyond the case of finite metric measure spaces. The main technical hurdle in doing so is that our proof uses the fact that a symmetric matrix $\rM$ can always be decomposed as the difference of two psd matrices. A natural analog in the continuous case would be to express $\Delta(x,y,x',y')=|\kappa_0(x,x')-\kappa_1(y,y')|^p$ as the difference of two psd kernels on $\mathcal X\times \mathcal Y$, but such a characterization does not appear to be straightforward to obtain in general. Related directions involve developing tools for  inference on optimal GW plans and designing provably convergent algorithms for solving finite GW problems without regularization.

\textbf{AI Statement.} ChatGPT 5.5 and 5.6 were used to aid with copyediting and assistance with coding for the simulations.

\bibliographystyle{amsplain}
\bibliography{ref}
\appendix

\pagebreak

\section{From Discrete GW to Quadratic Programming}
\label{sec:Ab} 

The constraint $\Pi(\mu_0,\mu_1)$ can be represented as $\{x\in\mathbb R^k:\rA x = b,x\geq 0\}$ where 

\begin{equation}
\label{eq:Ab}
\mathrm{A} = \begin{pmatrix}
   \mathbf{1}_{N_1}^{\intercal}&\mathbf 0_{N_1}^{\intercal}&\cdots&\mathbf 0_{N_1}^{\intercal}\\ 
   \mathbf 0_{N_1}^{\intercal}&\mathbf{1}_{N_1}^{\intercal}&\cdots&\mathbf 0_{N_1}^{\intercal}\\
   \vdots&\vdots &\ddots&\vdots\\
   \mathbf 0_{N_1}^{\intercal}&\mathbf 0_{N_1}^{\intercal}&\cdots&\mathbf{1}_{N_1}^{\intercal}\\
   \mathrm{Id}_{N_1}&\mathrm{Id}_{N_1}&\cdots&\mathrm{Id}_{N_1}
\end{pmatrix}\in \mathbb R^{(N_0+N_1)\times N_0N_1},\quad b = \begin{pmatrix}
   \mu_0\left(\left\{x^{(1)}_0\right\}\right)\\\vdots\\ \mu_0\left(\left\{x^{(N_0)}_0\right\}\right)\\ \mu_1\left(\left\{x^{(1)}_1\right\}\right)\\\vdots\\ \mu_1\left(\left\{x^{(N_1)}_1\right\}\right)
\end{pmatrix}\in\mathbb R^{N_0+N_1},
\end{equation}
with $\mathbf{1}_{N_1}$ and $\mathbf{0}_{N_1}$ as the vectors of all $1$'s and all $0$'s of length $N_1$, $\mathrm{Id}_{N_1}\in\mathbb R^{N_1\times N_1}$ the identity matrix, and $\{x_0^{(i)}\}_{i=1}^{N_0}$ is some ordering of the elements of $\mathcal X_0$ and similarly for $\{x_1^{(j)}\}_{j=1}^{N_1}$. Indeed,  if $x= \vec(\rP)$  where $\rP\in\mathbb R^{N_1\times N_0}$ has entries $\rP_{lm}=\pi\left(\left\{\left(x_0^{(m)},x_1^{(l)}\right)\right\}\right)$ and $\vec(\rP)\in\mathbb R^k$ is the vector obtained by stacking the columns of $\rP$, 
\[
(\mathrm{A} x)_l = \begin{cases}
\sum_{j=1}^{N_1} \pi\left(\left\{\left(x_0^{(l)},x_1^{(j)}\right)\right\}\right), &\text{if }l\in[N_0],
\\
\sum_{i=1}^{N_0} \pi\left(\left\{\left(x_0^{(i)},x_1^{(l-N_0)}\right)\right\}\right), &\text{if }l\in N_0+[N_1],
\end{cases}
\]
so that the constraints $\mathrm A x = b,x\geq 0$ imply that $x$ is the vectorization of a coupling of $(\mu_0,\mu_1)$.  

Under this convention, we have also used the fact that 
\[
\mathsf{KL}(\pi\|\mu_0\otimes \mu_1)=\sum_{\substack{i\in[N_0]\\j\in[N_1]}}\log\left(\frac{\pi{(\{(x_0^{(i)},x_1^{(j)})\})}}{\mu_0(\{x_0^{(i)}\})\mu_1(\{x_1^{(j)}\})}\right)\pi{(\{(x_0^{(i)},x_1^{(j)})\})}
\]
can be expressed in terms of the Shannon entropies of $\pi$, $\mu_0$, and $\mu_1$: 
\[
\begin{aligned}
\mathsf H_{\pi} = - \sum_{\substack{i\in[N_0]\\j\in[N_1]}}\pi(\{(x_0^{(i)},x_1^{(j)})\})\log\left(\pi(\{(x_0^{(i)},x_1^{(j)})\})\right),
\\
 \mathsf H_{\mu_i} = - \sum_{j=1}^{N_i}\mu_i(\{x_i^{(j)}\})\log\left(\mu_i(\{x_i^{(j)}\})\right)\text{ for }i\in\{0,1\},
 \end{aligned}
\] 
as $\mathsf{KL}(\pi\|\mu_0\otimes \mu_1)=-\mathsf H_{\pi}+\mathsf H_{\mu_0}+\mathsf H_{\mu_1}$.

\section{The Delta Method}
\label{sec:deltaMethod} 
Throughout, we
treat $\mathsf{EGW}_p^{\varepsilon}$ as a functional on $\mathcal P(\mathcal X_0)\times \mathcal P(\mathcal X_1)$ which can be identified with pairs of probability vectors. With these definitions in hand, we will apply the delta method-based approach provided in Proposition 1 and Remark 4 of \cite{goldfeld2024statistical} to derive limit laws for the empirical GW cost. A specific form of that framework amenable to the current setting is as follows.   

\begin{proposition}[Limit theorems, Proposition 1 in \cite{goldfeld2024statistical}] 
\label{prop:limitTheoremFramework}
Fix some $\varepsilon\geq 0$ and suppose that
    
\begin{enumerate}[leftmargin=*]
    \item the sequence $\sqrt n \big( (w_{\mu_{0,n}},w_{\mu_{1,n}})-(w_{ \mu_{0}},w_{ \mu_{1}})\big)$ converges in distribution to a pair of random vectors $(G_{\mu_0},G_{\mu_1})\in \mathbb R^{N_0}\times \mathbb R^{N_1}$,

    \item  the functional $(\nu_0,\nu_1)\in\mathcal P(\mathcal X_0)\times \mathcal P(\mathcal X_1)\mapsto\mathsf{EGW}_p^{\varepsilon}(\nu_0,\nu_1)$ is  Lipschitz continuous in the sense that $|\mathsf{EGW}_p^{\varepsilon}(\nu_0,\nu_1)-\mathsf{EGW}_p^{\varepsilon}(\nu_0'
    ,\nu_1')|\leq C\left(\|w_{\nu_0}-w_{\nu'_0}\|_1+\|w_{\nu_1}-w_{\nu_1'}\|_1\right)$ for  some constant $C<\infty$ which is independent of the choice of $(\nu_0,\nu_1),(\nu_0'
     ,\nu_1')\in\mathcal P(\mathcal X_0)\times \mathcal P(\mathcal X_1)$,

\item for every $(\nu_0,\nu_1)\in\mathcal P(\mathcal X_0)\times \mathcal P(\mathcal X_1)$, 
\[
   \lim_{t\downarrow 0} \frac{\mathsf{EGW}_p^{\varepsilon}(\mu_0+t(\nu_0-\mu_0),\mu_1+t(\nu_1-\mu_1))-\mathsf{EGW}_p^{\varepsilon}(\mu_0,\mu_1)}t = F_{(\mu_0,\mu_1)}(\nu_0-\mu_0,\nu_1-\mu_1),
\]
for some functional $F_{(\mu_0,\mu_1)}$ which may depend on $(\mu_0,\mu_1)$.
\end{enumerate}

\smallskip
 Under these assumptions, letting $H_{(\mu_0,\mu_1)}$ be such that $H_{(\mu_0,\mu_1)}(w_{\nu_0}-w_{\mu_0},w_{\nu_1}-w_{\mu_1})=F_{(\mu_0,\mu_1)}(\nu_0-\mu_0,\nu_1-\mu_1)$ for every $(\nu_0,\nu_1)\in\mathcal P(\mathcal X_0)\times \mathcal P(\mathcal X_1)$, it holds that:
 \begin{enumerate}[leftmargin=*]
     \item[(i)] $H_{(\mu_0,\mu_1)}$ extends to a continuous, positively homogeneous map on $$\mathcal T\coloneqq \overline{\left\{t\left( (w_{\eta_0},w_{\eta_1})-(w_{\mu_0},w_{\mu_1})\right):(\eta_0,\eta_1)\in\mathcal P(\mathcal X_0)\times \mathcal P(\mathcal X_1),t>0 \right\}},$$
     
     \item[(ii)] $(G_{\mu_0},G_{\mu_1})\in \mathcal T$ almost surely,

     \item[(iii)]  $\sqrt n \left(\mathsf{EGW}_p^{\varepsilon}(\mu_{0,n},\mu_{1,n})-\mathsf{EGW}_p^{\varepsilon}( \mu_{0},\mu_{1})\right)\stackrel{d}{\to} H_{(\mu_0,\mu_1)}(G_{\mu_0},G_{\mu_1})$.
 \end{enumerate} 
\end{proposition}

\section{Implementation details for \cref{alg:directEstimator}}
\label{rmk:implementationEstimator}

This section compiles some practical details for implementing \cref{alg:directEstimator}. 
\medskip

{\textbf{Storing $\mathrm{R}$.}} Given that the matrix $\mathrm{R}$ is of size $N_0^2\times N_0$, simply storing the full matrix can require a large memory overhead if $N_0$ is large. However, since the matrix $\mathrm R$ is sparse it can be represented in a memory-efficient manner by using sparse matrix libraries. Furthermore, when $\lfloor (l-1)/N_0\rfloor +1=(l-1)\bmod N_0+1$ in  \eqref{eq:constraintSetEstimator}, $(\mathrm{R}w)_l=0$ so that no constraint is enforced; such vacuous constraints can simply be removed.   
\medskip 

{\textbf{Size of $\hat \Sigma_n,\hat u_n$.}} Since $\hat \Sigma_n$ and $\hat u_n$ are estimated from data, they may not be of the same size as the population quantities $\Sigma_q$ and $u_{\mu_0}$. For instance, if the support of $\mu_0$ is unknown, then it is natural to estimate it as $\supp(\mu_{0,n})$, setting $\hat N_n$ to be the cardinality of this set, and taking $\hat \Sigma_n$ and $\hat u_n$ as elements of $\mathbb R^{\hat N_n\times \hat N_{n}}$ and $\mathbb R^{\hat N_n^2}$ respectively. For instance, if $\mu_{0,n}=\hat \mu_{0,n}$,  it is natural to set $\hat \Sigma_n=\Sigma_{w_{\hat \mu_{0,n}}}$ and $\hat u_n=u_{\hat \mu_{0,n}}$   (defined by analogy with the vector $u_{\mu_0}$ in item (3) of \cref{assn:generalSettingLimitLaw}). These matrices and vectors can also be thought of as elements of $\mathbb R^{N_0\times N_0}$ and $\mathbb R^{N_0^2}$ respectively by padding them in a suitable way. For the matrix, we may, without loss of generality, assume that $\supp(\mu_0)=(x^{(i)})_{i=1}^{N_0}$ is ordered so that the first $\hat N_n$ points coincide with the support of $\mu_{0,n}$ and extending $\hat \Sigma_n$ to $\tilde \Sigma_n\in \mathbb R^{N_0\times N_0}$ by adding blocks of zero matrices. With this, a sample $Y\sim N(0,\tilde \Sigma_n)$ is equal in distribution to $(Z,0)^{\intercal}$ where $Z\sim N(0,\hat \Sigma_n)$, that is, the last $N_0-\hat N_n$ entries of $Y$ are $0$. As such, treating the linear program in \eqref{eq:LnDistribution} as an optimization problem over $\tilde w\in\mathbb R^{N_0}$, the value of the last $N_0-\hat N_n$ entries of $\tilde w$ do not contribute to the objective value. Consequently, 
   \[
        \sup_{\substack{\tilde w\in\mathbb R^{N_0}\\{\mathrm R}\tilde w\leq \tilde u_n}} Y^{\intercal}\tilde w \stackrel{d}=  \sup_{\substack{\tilde w\in\mathbb R^{N_0}\\{\mathrm R}\tilde w\leq \tilde u_n}} (Z,0)^{\intercal}\tilde w\stackrel{d}= \sup_{\substack{ w\in\mathbb R^{\hat N_n}\\\hat{\mathrm R}_nw\leq \hat u_n}} Z^{\intercal}w, 
   \]
   where $\hat{\mathrm R}_n$ is defined by analogy with $\mathrm{R}$, but on $\hat N_n$ and $\tilde u_n$ enforces the constraints from $\hat u_n$ on the relevant variables, i.e., 
   \[
        \tilde w_{\lfloor (l-1)/\hat N_n\rfloor+1}-\tilde w_{(l-1)\bmod \hat N_n+1}\leq (\hat u_n)_l \text{ for $l\in[\hat N_n^2]$} 
   \]
   and enforces no constraints on $\tilde w_{i}$ for $i\geq \hat N_n$ (simply set the relevant entries of $\tilde u_n$ to be $+\infty$).   Evidently, $\hat N_n\to N_0$ with probability $1$ by definition of the support so that this digression is of a mostly practical interest.
\medskip

{\textbf{Bounding the Constraint Set.}} As the constraints in the linear program are enforced on differences of entries of $w$, the constraint set in \eqref{eq:constraintSetEstimator} is unbounded. In effect, if $w\in\hat{ \mathcal H}_n$, then $w+a\mathbf 1\in\hat{\mathcal H}_n$. Importantly, by item (4) of \cref{assn:generalSettingLimitLaw}, $\mathbb P(Z^{\intercal}\mathbf 1=0)=1$ when $Z\sim N(0,\hat \Sigma_n)$ (at least for sufficiently large $n$) so that $Z^{\intercal}(w+a\mathbf 1)=Z^{\intercal} w$ with probability one. As such, we may introduce an additional constraint $w_1\in[-\delta,\delta]$ for any $\delta>0$ in $\hat{\mathcal H}_n$ without changing the optimal value. This modification further implies that the remaining entries are bounded due to the box constraints on $w_1-w_i$ for each $i>1$.
   The resulting constraint set is bounded and the corresponding linear program admits a solution. It follows that the linear programs in \eqref{eq:LnDistribution} have finite optimal value with probability $1$. Moreover, when implementing \cref{alg:directEstimator}, we may add the constraint $w_1\in[-\delta,\delta]$ without loss of generality. This enables us to account for the fact that the sample $Z$ generated by computational methods may not be exactly orthogonal to $\mathbf 1$ due to numerical imprecision.

\section{Additional Experimental Details for \cref{sec:validation}}\label{sec:experimental-details}
~\\

\textbf{Convergence experiment.} We always initialize the gradient methods without acceleration using $u_0 = \rB_0 x_0, v_0 = \rB_1 x_0$ where  $x_0 = \vec(w_{\mu_0} w_{\mu_1}^{\intercal})$  from the independent coupling. For the accelerated gradient method, we initialize $s_0 = \rB_0 x_0, v_0 = \rB_1 x_0.$  In all cases, we  take a  maximum of $K=1000$ iterations when minimizing $\ell$. For all settings,  we maximize $f_u$ using the gradient method (\cref{algo:inner-loop-simple}) for the gradient oracle with  maximum number of iterations set to $J=100$. We terminate the algorithms early if the norm of the approximate gradient at the given iterate is less than the tolerance $\delta = 10^{-10}$ for both methods to minimize $\ell$ (\Cref{algo:outer-loop-accelerated} and \Cref{algo:outer-loop-simple}) and the gradient method used for the inexact gradient oracle (\Cref{algo:inner-loop-simple}); this condition was always met before the maximum number of steps is reached. Finally, the step size for \cref{algo:outer-loop-accelerated,algo:outer-loop-simple} is set using $L=1.5$ and for \cref{algo:inner-loop-simple} is set using $L=2$; even though these step sizes do not necessarily coincide with the theoretical step sizes, the results are seen to abide by the derived convergence rates. \\

\textbf{Runtime experiment.} For the runtime experiment, we use the same initialization strategy described above for our gradient method (\Cref{algo:outer-loop-simple}) and accelerated gradient method (\cref{algo:outer-loop-accelerated}). For the mirror descent algorithm, the initialization is also chosen to be the independent coupling $\pi_0=w_{\mu_0}w_{\mu_1}^{\intercal}$. For all settings we take a maximum of $K=5000$ iterations when minimizing $\ell$. Note that since, $\rB_1 = \textbf{0}$ we do not need to maximize $f_u$. We terminate the algorithms early if the norm of the approximate gradient at the given iterate is less than the tolerance $\delta = 10^{-6}$ for both methods to minimize $\ell$ (\Cref{algo:outer-loop-accelerated} and \Cref{algo:outer-loop-simple}); this condition is always met first before the maximum number of steps is met. As noted in the text, the value of $\varepsilon$ is set on a problem dependent basis and the values of $L$ used for the gradient and accelerated gradient methods were obtained via a line search. We compile the values of $L$  together with the corresponding  values of $\varepsilon$ taken for the objective to be convex in Table~\ref{tab:L-details}.

\begin{table}[!htb]
\centering
\tiny
\begin{tabular}{rrrrrrr}

\toprule
Rank & $N$
& \multicolumn{2}{c}{Nonconvex ($\varepsilon = 1$)}
& \multicolumn{3}{c}{Convex} \\
\cmidrule(lr){3-4}\cmidrule(lr){5-7}
& & $L_{\rm grad}$ & $L_{\rm acc}$
& $L_{\rm grad}$ & $L_{\rm acc}$ & $\varepsilon$ \\
\midrule
1 & 16 & 0.979781 & 1.46461 & 0.999674 & 1.24705 & 30.8068--135.152 \\
1 & 32 & 0.979781 & 1.47940 & 0.989677 & 1.06181 & 169.486--405.247 \\
1 & 64 & 0.969983 & 1.47940 & 0.950680 & 0.724744 & 667.353--1319.69 \\
1 & 128 & 0.969983 & 1.47940 & 0.801369 & 0.136655 & 3003.20--4743.51 \\
1 & 256 & 0.969983 & 1.47940 & 0.530732 & 0.094217 & $1.211{\times}10^4$--$1.729{\times}10^4$ \\
1 & 512 & 0.969983 & 1.47940 & 0.210530 & 0.0351867 & $5.159{\times}10^4$--$6.586{\times}10^4$ \\
1 & 1024 & 0.969983 & 1.47940 & 0.0605449 & 0.0132737 & $2.106{\times}10^5$--$2.639{\times}10^5$ \\
1 & 2048 & 0.969983 & 1.47940 & 0.0151264 & 0.0132737 & $8.813{\times}10^5$--$1.039{\times}10^6$ \\
1 & 4096 & 0.975350 & 1.47018 & 0.0132737 & 0.0132737 & $3.559{\times}10^6$--$3.912{\times}10^6$ \\
1 & 8192 & 0.975350 & 1.47018 & 0.0132737 & 0.0132737 & $1.458{\times}10^7$--$1.547{\times}10^7$ \\
1 & 16384 & 0.975350 & 1.47018 & 0.0132737 & 0.0132737 & $5.845{\times}10^7$--$6.074{\times}10^7$ \\
\midrule
4 & 16 & 0.979781 & 1.43546 & 0.989677 & 1.05119 & 73.9835--158.528 \\
4 & 32 & 0.979781 & 1.42111 & 0.922444 & 0.575167 & 281.318--477.826 \\
4 & 64 & 0.979781 & 1.42111 & 0.754473 & 0.127372 & 826.680--1499.73 \\
4 & 128 & 0.979781 & 1.39283 & 0.377114 & 0.0630285 & 3085.78--5848.21 \\
4 & 256 & 0.979781 & 1.39283 & 0.133936 & 0.0221614 & $1.286{\times}10^4$--$2.049{\times}10^4$ \\
4 & 512 & 0.979781 & 1.37890 & 0.030877 & 0.0132737 & $4.395{\times}10^4$--$7.204{\times}10^4$ \\
4 & 1024 & 0.979781 & 1.37890 & 0.0132737 & 0.0132737 & $1.750{\times}10^5$--$2.450{\times}10^5$ \\
4 & 2048 & 0.979781 & 1.37890 & 0.0132737 & 0.0132737 & $7.495{\times}10^5$--$9.192{\times}10^5$ \\
4 & 4096 & 0.975350 & 1.39667 & 0.0132737 & 0.0132737 & $2.742{\times}10^6$--$3.640{\times}10^6$ \\
4 & 8192 & 0.975350 & 1.39667 & 0.0132737 & 0.0132737 & $1.092{\times}10^7$--$1.326{\times}10^7$ \\
4 & 16384 &0.975350 & 1.39667 & 0.0132737 & 0.0132737 & $4.353{\times}10^7$--$5.238{\times}10^7$ \\
\midrule
16 & 16 & 0.999674 & 1.33794 & 0.877235 & 0.373343 & 103.366--193.908 \\
16 & 32 & 0.999674 & 1.35146 & 0.662066 & 0.111771 & 450.642--654.993 \\
16 & 64 & 0.999674 & 1.33794 & 0.299283 & 0.0500203 & 1628.78--2673.97 \\
16 & 128 & 0.999674 & 1.31132 & 0.0611565 & 0.0132737 & 6321.86--8588.77 \\
16 & 256 & 0.999674 & 1.28522 & 0.0139578 & 0.0132737 & $2.275{\times}10^4$--$3.423{\times}10^4$ \\
16 & 512 & 0.999674 & 1.27237 & 0.0132737 & 0.0132737 & $8.231{\times}10^4$--$1.233{\times}10^5$ \\
16 & 1024 & 0.999674 & 1.27237 & 0.0132737 & 0.0132737 & $3.425{\times}10^5$--$4.379{\times}10^5$ \\
16 & 2048 & 0.999674 & 1.25965 & 0.0132737 & 0.0132737 & $1.318{\times}10^6$--$1.649{\times}10^6$ \\
16 & 4096 & 0.975350 & 1.26050 & 0.0132737 & 0.0132737 & $4.966{\times}10^6$--$6.235{\times}10^6$ \\
16 & 8192 & 0.975350 & 1.26050 & 0.0132737 & 0.0132737 & $1.906{\times}10^7$--$2.454{\times}10^7$ \\
16 & 16384 & 0.975350 & 1.26050 & 0.0132737 & 0.0132737 & $7.500{\times}10^7$--$9.415{\times}10^7$ \\
\midrule
64 & 16 & 0.999674 & 1.24705 & 0.536093 & 0.0951687 & 179.246--234.532 \\
64 & 32 & 0.999674 & 1.21001 & 0.194265 & 0.0324683 & 669.445--872.077 \\
64 & 64 & 0.999674 & 1.17407 & 0.0456944 & 0.0132737 & 2492.33--3278.00 \\
64 & 128 & 0.989677 & 1.09431 & 0.0132737 & 0.0132737 & 9297.03--$1.253{\times}10^4$ \\
64 & 256 & 0.989677 & 1.05119 & 0.0132737 & 0.0132737 & $3.721{\times}10^4$--$4.885{\times}10^4$ \\
64 & 512 & 0.989677 & 1.03027 & 0.0132737 & 0.0132737 & $1.515{\times}10^5$--$1.804{\times}10^5$ \\
64 & 1024 & 0.979781 & 0.999674 & 0.0132737 & 0.0132737 & $5.645{\times}10^5$--$6.837{\times}10^5$ \\
64 & 2048 & 0.979781 & 0.960283 & 0.0132737 & 0.0132737 & $2.250{\times}10^6$--$2.648{\times}10^6$ \\
64 & 4096 & 0.975350 & 0.926582 & 0.0132737 & 0.0132737 & $8.721{\times}10^6$--$1.010{\times}10^7$ \\
64 & 8192 & 0.975350 & 0.926582 & 0.0132737 & 0.0132737 & $3.300{\times}10^7$--$4.071{\times}10^7$ \\
64 & 16384 & 0.975350 & 0.926582 & 0.0132737 & 0.0132737 & $1.324{\times}10^8$--$1.514{\times}10^8$ \\
\bottomrule
\end{tabular}
\caption{Values of $L$ and range of $\varepsilon$ used for the  runtime experiments.}
\label{tab:L-details}
\end{table}

\section{Implementing the Graph Isomorphism Test}
\label{sec:graphTestStatistic}

Given a collection of graphs $\{G_i\}_{i=1}^N\subset \mathcal G_N$, the matrix of pairwise kernel evaluations $K\in\mathbb R^{N\times N}$ with entries $K_{ij}= \kappa(G_i,G_j)$ for $i,j\in[N]$ can be computed as 
\[
    K =  \begin{pmatrix}
       \sqrt{2} (\mathrm{A}_{G_1} \bm 1)^{\intercal}
       \\
    \vdots
    \\
      \sqrt{2} (\mathrm{A}_{G_N} \bm 1)^{\intercal}
    \end{pmatrix} \begin{pmatrix}
       {\sqrt{2}}(\mathrm{A}_{G_1}\bm 1)^{\intercal}
       \\
    \vdots
    \\
      {\sqrt{2}} (\mathrm{A}_{G_N} \bm 1)^{\intercal}
    \end{pmatrix}^{\intercal}
   + 
    \begin{pmatrix}
       \frac{1}{\sqrt{2}}\vec(\mathrm{A}_{G_1})^{\intercal}
       \\
    \vdots
    \\
      \frac{1}{\sqrt{2}} \vec(\mathrm{A}_{G_N})^{\intercal}
    \end{pmatrix} \begin{pmatrix}
       \frac{1}{\sqrt{2}} \vec(\mathrm{A}_{G_1})^{\intercal}
       \\
    \vdots
    \\
      \frac{1}{\sqrt{2}} \vec(\mathrm{A}_{G_N})^{\intercal}.
    \end{pmatrix}^{\intercal}
\]
This representation demonstrates that $K$ is a psd matrix. Thus, if we choose $p=2$, we may take the GW cost matrix to be $\rM = -4\mathrm{K}_0\otimes \mathrm{K}_1$, where $\mathrm{K}_0$ and $\mathrm{K}_1$ are the pairwise cost matrices for the two collections of sampled graphs (recall \cref{subsec:complexity}). As the Kronecker product of psd matrices is psd, it follows that $\rM$ is nsd so that we may choose $\rB_1=0$ in our decomposition of the cost. 

With this, we solve the resulting $2$-GW problems using \cref{algo:outer-loop-simple}  with $\varepsilon=0$ and $\rB_1=0$. As such, the \texttt{inexactGradientOracle} method only requires an approximate solution of $\inf_{x\in\mathcal K}\left\{-u^{\intercal}\rB_0 x\right\}$ which is just an OT problem and can be solved using the OT solver from the POT package \cite{flamary2021pot}.

While the convergence result for \cref{algo:outer-loop-simple}, \cref{lem:convergenceRateMin}, does not apply to the case $\varepsilon=0$, we set $L'=19$, corresponding to a step size of $0.05$, and terminate the loop once a point with subgradient norm below $1\times 10^{-8}$ is met or $10000$ iterations have been performed; the former termination condition was always met first in this experiment.

For these tests to be statistically valid, the GW distances must be computed to a high accuracy. As the objective in the variational form is nonconvex and nonsmooth, we therefore implement a warm start procedure for the subgradient method. For the purposes of this experiment and to reduce the computational burden, this approach consists of initializing \cref{algo:outer-loop-simple} at $u_0=\rB_0 x^{\star}_{\mathrm{pop}}$, where $x^{\star}_{\mathrm{pop}}$ is the vectorization of the optimal coupling for the population level problem $\mathsf{GW}_2(\mu_0,\mu_1)$ which is induced by the true permutation $\sigma$. This initialization is justified by noting that if $(\mu_{0,n},\mu_{1,n})$ converges weakly to $(\mu_0,\mu_1)$, then each cluster point of solutions of the variational form for $\mathsf{GW}_2(\mu_{0,n},\mu_{1,n})$ is a solution for $\mathsf{GW}_2(\mu_{0},\mu_{1})$, see \cref{sec:convergenceMinimizers}. In general the null hypothesis posits that, at the population level, the optimal couplings are induced by permutations. Hence, even if the ground truth permutation $\sigma$ was not known, the same warm start procedure can be applied for each possible permutation in which case the reported test statistic is the minimum value attained over all initializations, though this approach scales poorly as $N$ increases.

\section{Convergence of Solutions for the Variational Form}
\label{sec:convergenceMinimizers}
In this section, we demonstrate the following stability property for solutions to the variational form under weak convergence of the marginals.

\begin{proposition}\label{prop:continuityVarSolns} Let $(\mu_{0,n})_{n\in\mathbb N}\subset \mathcal P(\mathcal X_0)$ and $(\mu_{1,n})_{n\in\mathbb N}\subset \mathcal P(\mathcal X_1)$ be such that  $\mu_{0,n}\stackrel{w}{\to} \mu_0\in\mathcal P(\mathcal X_0),$ $\mu_{1,n}\stackrel{w}{\to} \mu_1\in \mathcal P(\mathcal X_1)$.   
    For each $n\in\mathbb N$, let $(u_n,v_n)$ solve
    \[ 
    \inf_{u\in\mathbb R^{r_0}}\sup_{v\in\mathbb R^{r_1}}\left\{ \frac12\|u\|^2-\frac 12 \|v\|^2+\inf_{x\in\mathcal K_n}\left\{\left( \rB_1^{\intercal}v-\rB_0^{\intercal} u \right)^{\intercal}x  \right\} \right\}=\mathsf{GW}_p(\mu_{0,n},\mu_{1,n})^p,
    \]
    where $x\in\mathcal K_n$ is understood as the vector representation for $\pi\in\Pi(\mu_{0,n},\mu_{1,n})$. 
    Then, every cluster point of the sequence $(u_n,v_n)_{n\in\mathbb N}$ solves  
    \begin{equation}
    \label{eq:limitingProblem}
    \inf_{u\in\mathbb R^{r_0}}\sup_{v\in\mathbb R^{r_1}}\left\{ \frac12\|u\|^2-\frac 12 \|v\|^2+\inf_{x\in\mathcal K}\left\{\left( \rB_1^{\intercal}v-\rB_0^{\intercal} u \right)^{\intercal}x  \right\} \right\}.
    \end{equation}
\end{proposition}

We first establish that the minimizers of 
\[
\begin{aligned}
\ell_n(u) &\coloneqq \frac 12 \|u\|^2+\sup_{v\in\mathbb R^{r_1}}\left\{-\frac 12\|v\|^2 + \inf_{\mathcal K_n}\left\{(\rB_1^{\intercal}v-\rB_0^{\intercal}u)^{\intercal}x\right\}\right\}\\& =  \frac 12 \|u\|^2+\inf_{x\in\mathcal K_n}\left\{\frac 12 x^{\intercal}\rB_1^{\intercal}\rB_1x-u^{\intercal }\rB_0 x\right\},
\end{aligned}
\] 
converge to those of 
\[
\begin{aligned}\ell(u)&\coloneqq\frac 12 \|u\|^2+\sup_{v\in\mathbb R^{r_1}}\left\{-\frac 12\|v\|^2 + \inf_{\mathcal K}\left\{(\rB_1^{\intercal}v-\rB_0^{\intercal}u)^{\intercal}x\right\}\right\}\\
&=  \frac 12 \|u\|^2+\inf_{x\in\mathcal K}\left\{\frac 12 x^{\intercal}\rB_1^{\intercal}\rB_1x-u^{\intercal }\rB_0 x\right\},
\end{aligned}
\]
in a suitable sense
as $n\to \infty$, recalling the alternative representation of $\ell$ from the proof of \cref{lem:derivativeObjective}. Complete details can be found in \cref{proof:lem:convergenceMinimizers}.
\begin{lemma}
\label{lem:convergenceMinimizers}
   In the setting of \cref{prop:continuityVarSolns}, let $u_n$ minimize $\ell_n$ for each $n\in\mathbb N$. Then, every cluster point of the sequence $(u_n)_{n\in\mathbb N}$ minimizes $\ell$ and $\ell_n(u_n)\to \inf_{u\in\mathbb R^{r_0}}\ell(u)$.
\end{lemma}

With \cref{lem:convergenceMinimizers} in hand, we prove the main result. 

\begin{proof}[Proof of \cref{prop:continuityVarSolns}]
Let $(u,v)$ be a cluster point of $(u_n,v_n)$ and let $n'$ be a subsequence along which $(u_{n'},v_{n'})$ converges to $(u,v)$. By \cref{lem:convergenceMinimizers}, we have that 
     \[
            \frac 12 \|u_{n'}\|^2-\frac 12\|v_{n'}\|^2 + \inf_{x\in\mathcal K_{n'}}\left\{\left( \rB_1^{\intercal}v_{n'}-\rB_0^{\intercal} u_{n'} \right)^{\intercal}x  \right\}=\ell_{n'}(u_{n'})\to \inf_{u\in\mathbb R^{r_0}}\ell(u).
     \]
    Note that $\inf_{x\in\mathcal K_{n'}}\left\{\left( \rB_1^{\intercal}v_{n'}-\rB_0^{\intercal} u_{n'} \right)^{\intercal}x \right\}=\mathsf{OT}_{\rB_1^{\intercal}v_{n'}-\rB_0^{\intercal} u_{n'}}(\mu_{0,{n'}},\mu_{1,{n'}})\to \mathsf{OT}_{\rB_1^{\intercal}v-\rB_0^{\intercal} u}(\mu_{0},\mu_{1})$  by Theorem 5.20 in \cite{villani2008optimal}. Conclude that 
    \[
    \begin{aligned}
        \frac 12 \|u_{n'}\|^2-\frac 12\|v_{n'}\|^2 + \inf_{x\in\mathcal K_{n'}}\left\{\left( \rB_1^{\intercal}v_{n'}-\rB_0^{\intercal} u_{n'} \right)^{\intercal}x  \right\}&\\&\hspace{-5.5em}\to \frac 12 \|u\|^2-\frac 12\|v\|^2 + \inf_{x\in\mathcal K}\left\{\left( \rB_1^{\intercal}v-\rB_0^{\intercal} u \right)^{\intercal}x  \right\}=\inf_{u\in\mathbb R^{r_0}}\ell(u), 
    \end{aligned} 
    \] 
    where the final equality is due to the previous display. Conclude that  
    \[
        \begin{aligned}
        \inf_{u\in\mathbb R^{r_0}}\ell(u)&=\frac 12 \|u\|^2-\frac 12\|v\|^2 + \inf_{x\in\mathcal K}\left\{\left( \rB_1^{\intercal}v-\rB_0^{\intercal} u \right)^{\intercal}x  \right\}
        \\
        &\leq \sup_{v\in\mathbb R^{r_1}}\left\{\frac 12 \|u\|^2-\frac 12\|v\|^2 + \inf_{x\in\mathcal K}\left\{\left( \rB_1^{\intercal}v-\rB_0^{\intercal} u \right)^{\intercal}x  \right\}\right\} = \ell(u)=\inf_{u\in\mathbb R^{r_0}}\ell(u) 
        \end{aligned}
    \]
    noting that $u$ minimizes $\ell$ by \cref{lem:convergenceMinimizers}. Conclude that the pair $(u,v)$ solves \eqref{eq:limitingProblem}.
\end{proof}

\section{Proofs of Auxiliary Results}
\label{app:ProofsAux}
\subsection{Proof of \texorpdfstring{\cref{lem:derivativeObjective}}{Lemma 2}}
\label{proof:lem:derivativeObjective}
       By analogy with the first part of the proof of Theorem \ref{thm:variationalForm}, it follows from Sion's minimax principle that
        \[
               \ell_{\varepsilon}(u) =\inf_{x\in\mathcal K} 
               \sup_{v\in\mathbb R^{r_1}}\left\{  \frac12\|u\|^2-\frac{1}{2}\|v\|^2 +\left( \rB^{\intercal}_1 v-\rB_0^{\intercal} u\right)^{\intercal} x -\varepsilon\sH(x)\right\}.
        \]
        For any fixed $u\in\mathbb R^{r_0}$ and $x\in\mathcal K$, the objective in the above display is strongly concave in $v$ and its unique maximizer is given by $v^{\star}=\rB_1 x$. Conclude that  
        \begin{equation}
        \label{eq:QPDecomposition}
               \ell_{\varepsilon}(u) =\frac12\|u\|^2+ 
               \inf_{x\in\mathcal K} \left\{\frac{1}{2}x^{\intercal}\rB^{\intercal}_1\rB_1 x- u^{\intercal}\rB_0 x-\varepsilon\sH(x)\right\}\eqqcolon \frac 12 \|u\|^2-V_{\varepsilon}(u).
        \end{equation}
        As $V_{\varepsilon}(u)=\sup_{x\in\mathcal K}\left\{ -\frac{1}{2}x^{\intercal}\rB^{\intercal}_1\rB_1 x + u^{\intercal}\rB_0 x+\varepsilon\sH(x)\right\}$, Danskin's theorem, \cite[Proposition B.22]{bertsekas2016nonlinear}, asserts that $V_{\varepsilon}$ is convex and has  subdifferential (in the convex analytic sense) given by 
        \[
        \begin{aligned}
\partial V_{\varepsilon}(u)&=\rB_0\argmax_{x\in\mathcal K}\left\{  -\frac{1}{2}x^{\intercal}\rB^{\intercal}_1\rB_1 x + u^{\intercal}\rB_0 x+\varepsilon\sH(x)\right\}
\\
&=\rB_0\argmin_{x\in\mathcal K}\left\{  \frac{1}{2}x^{\intercal}\rB^{\intercal}_1\rB_1 x - u^{\intercal}\rB_0 x-\varepsilon\sH(x)\right\}.
        \end{aligned} 
        \]
       By  Propositions 2.2.7 and 2.3.1 in \cite{clarke1990optimization} the Clarke subdifferential of $-V_{\varepsilon}$ at $u\in\RR^{r_0}$ is $-\partial V_{\varepsilon}(u)$ so that the claimed formula for $\partial \ell_{\varepsilon}(u)$, \eqref{eq:lepsSubdiff}, follows from the subdifferential sum rule (see the Corollary on p.33 and Corollary 1 on p.39 in \cite{clarke1990optimization}).  

        The Clarke subdifferential of $f_{\varepsilon,u}$ for any $u\in\mathbb R^{r_0}$ can also be computed using Danskin's theorem and following the same steps as above. 
\qed

\subsection{Proof of \texorpdfstring{\cref{lem:QPvsLPsoln}}{Lemma 3}}
\label{proof:lem:QPvsLPsoln}
        Following the notation of \eqref{eq:indicator}, we express 
        \[\min_{x\in\mathcal K}\left\{\frac{1}{2}x^{\intercal}\rB^{\intercal}_1\rB_1 x- u^{\intercal}\rB_0 x-\varepsilon\sH(x)\right\}
        \]as 
        \[
            \min_{x\in\mathbb R^k}\left\{ f(x)+g(x)\right\},\text{ where }f(x)=\frac 12  x^{\intercal}\rB^{\intercal}_1\rB_1 x\text{ and }g(x)=- u^{\intercal}\rB_0 x-\varepsilon\sH(x)+\mathcal I_{\mathcal K}(x)
        \]
        As both $f$ and $g$ are proper, lower semicontinuous, and convex, Corollary 16.48 in \cite{bauschke2017convex} yields that $\partial(f+g)=\partial f +\partial g$. By Theorem 16.3 in \cite{bauschke2017convex}, a solution $\bar x$ of the problem $\min_{x\in\mathcal K}\left\{\frac{1}{2}x^{\intercal}\rB^{\intercal}_1\rB_1 x- u^{\intercal}\rB_0 x-\varepsilon\sH(x)\right\}$ satisfies 
        \[
            0\in \partial f(\bar x) +\partial g(\bar x), \text{ that is }, -\rB_1^{\intercal}\rB_1 \bar x \in \partial g(\bar x).
        \]
       By definition of the subdifferential, we have further that  
        \[
            g(x)-g(\bar x) \geq (-\rB_1^{\intercal}\rB_1\bar x)^{\intercal}(x-\bar x)\text{ for each }x\in\mathcal K. 
        \] Rearranging the above display, we see that 
        \[
            (\bar x)^{\intercal}\rB_1^{\intercal}\rB_1 x -u^{\intercal} \rB_0 x-\varepsilon \mathsf H(x)\geq 
            (\bar x)^{\intercal}\rB_1^{\intercal}\rB_1 \bar x -u^{\intercal} \rB_0 \bar x-\varepsilon \mathsf H(\bar x) 
        \]
        for each $x\in\mathcal K$, proving the claim.
\qed

\subsection{Proof of \texorpdfstring{\cref{lem:EOTLipschitz}}{Lemma 4}}
    \label{proof:lem:EOTLipschitz}
    Let $(\varphi_0^{(\nu_0,\nu_1)}, \varphi_1^{(\nu_0,\nu_1)})$ be a pair of EOT potentials for the problem $\mathsf{EOT}^{\varepsilon}_{\rB_1^{\intercal}v-\rB_0^{\intercal}u}(\nu_0,\nu_1)$ satisfying the properties of \cref{lem:EOTPotentialBounds} and define $(\varphi_0^{(\nu_0',\nu_1)}, \varphi_1^{(\nu_0',\nu_1)}),$ and $(\varphi_0^{(\nu_0',\nu_1')}, \varphi_1^{(\nu_0',\nu_1')})$ analogously for the different choices of marginals. Then, 
\[
\begin{aligned}
 \left|\mathsf{EOT}^{\varepsilon}_{\rB_1^{\intercal}v-\rB_0^{\intercal}u}(\nu_{0},\nu_{1})-   \mathsf{EOT}^{\varepsilon}_{\rB_1^{\intercal}v-\rB_0^{\intercal}u}(\nu_0',\nu_1')\right| &\leq \left|\mathsf{EOT}^{\varepsilon}_{\rB_1^{\intercal}v-\rB_0^{\intercal}u}(\nu_{0},\nu_{1})-   \mathsf{EOT}^{\varepsilon}_{\rB_1^{\intercal}v-\rB_0^{\intercal}u}(\nu_0',\nu_1)\right|
 \\&+
 \left|\mathsf{EOT}^{\varepsilon}_{\rB_1^{\intercal}v-\rB_0^{\intercal}u}(\nu_{0}',\nu_{1})-   \mathsf{EOT}^{\varepsilon}_{\rB_1^{\intercal}v-\rB_0^{\intercal}u}(\nu_0',\nu_1')\right|
\end{aligned}
\]
We bound the first term on the right hand side; an analogous bound will hold for the other term by symmetry. Letting $c_{u,v}$ be the cost function for the cost vector $\rB_1^{\intercal}v-\rB_0^{\intercal}u$,   
\[
\begin{aligned}
\mathsf{EOT}^{\varepsilon}_{\rB_1^{\intercal}v-\rB_0^{\intercal}u}(\nu_{0},\nu_{1})-   \mathsf{EOT}^{\varepsilon}_{\rB_1^{\intercal}v-\rB_0^{\intercal}u}(\nu_0',\nu_1)&
\\&\hspace{-14em}\leq \int \varphi_0^{(\nu_0,\nu_1)} d \nu_0+\int \varphi_1^{(\nu_0,\nu_1)} d \nu_1 - \varepsilon \int e^{\frac{\varphi_0^{(\nu_0,\nu_1)}\oplus\varphi_1^{(\nu_0,\nu_1)}-c_{u,v}}\varepsilon}d\nu_0\otimes \nu_1 + \varepsilon 
\\&\hspace{-14em}-\int \varphi_0^{(\nu_0,\nu_1)} d \nu_0'-\int \varphi_1^{(\nu_0,\nu_1)} d \nu_1 + \varepsilon \int e^{\frac{\varphi_0^{(\nu_0,\nu_1)}\oplus\varphi_1^{(\nu_0,\nu_1)}-c_{u,v}}\varepsilon}d\nu_0'\otimes \nu_1 - \varepsilon
\\
&\hspace{-14em}=\int \varphi_0^{(\nu_0,\nu_1)} d(\nu_0-\nu_0')- \varepsilon \int e^{\frac{\varphi_0^{(\nu_0,\nu_1)}\oplus\varphi_1^{(\nu_0,\nu_1)}-c_{u,v}}\varepsilon}d(\nu_0-\nu_0')\otimes \nu_1
\\&\hspace{-14em}
=\int \varphi_0^{(\nu_0,\nu_1)} d(\nu_0-\nu_0'),
\end{aligned}
\]
where the final equality follows by noting that $(\varphi_0^{(\nu_0,\nu_1)},\varphi_1^{(\nu_0,\nu_1)})$  satisfy the relevant Schr{\"o}dinger system \eqref{eq:SchrodingerSystem} on $\mathcal X_0\times \mathcal X_1$ so that 
\[
\int e^{\frac{\varphi_0^{(\nu_0,\nu_1)}\oplus\varphi_1^{(\nu_0,\nu_1)}-c_{u,v}}\varepsilon}d(\nu_0-\nu_0')\otimes \nu_1 = \int 1d(\nu_0-\nu_0')=0.
\]
Applying the same logic as above, we obtain the lower bound 
\[
\mathsf{EOT}^{\varepsilon}_{\rB_1^{\intercal}v-\rB_0^{\intercal}u}(\nu_{0},\nu_{1})-   \mathsf{EOT}^{\varepsilon}_{\rB_1^{\intercal}v-\rB_0^{\intercal}u}(\nu_0',\nu_1)\geq \int \varphi_0^{(\nu_0',\nu_1)} d(\nu_0-\nu_0'), 
\]
so that 
\[
\begin{aligned}
\left|\mathsf{EOT}^{\varepsilon}_{\rB_1^{\intercal}v-\rB_0^{\intercal}u}(\nu_{0},\nu_{1})-   \mathsf{EOT}^{\varepsilon}_{\rB_1^{\intercal}v-\rB_0^{\intercal}u}(\nu_0',\nu_1)\right|&\\&\hspace{-3em}\leq \max\left\{\left|\int \varphi_0^{(\nu_0,\nu_1)} d(\nu_0-\nu_0')\right|,\left|\int \varphi_0^{(\nu_0',\nu_1)} d(\nu_0-\nu_0')\right|\right\}. 
\end{aligned}
\]
It follows similarly that 
\[
\begin{aligned}
\left|\mathsf{EOT}^{\varepsilon}_{\rB_1^{\intercal}v-\rB_0^{\intercal}u}(\nu_{0}',\nu_{1}')-   \mathsf{EOT}^{\varepsilon}_{\rB_1^{\intercal}v-\rB_0^{\intercal}u}(\nu_0',\nu_1)\right|&
\\&\hspace{-3em}\leq \max\left\{\left|\int \varphi_1^{(\nu_0',\nu_1')} d(\nu_1-\nu_1')\right|,\left|\int \varphi_1^{(\nu_0',\nu_1)} d(\nu_1-\nu_1')\right|\right\}. 
\end{aligned}
\]
Applying the estimates on the potentials from \cref{lem:EOTPotentialBounds}, we obtain that 
\[
    \max\left\{\|\varphi_0^{(\nu_0,\nu_1)}\|_{\infty,\mathcal X_0},\|\varphi_0^{(\nu_0',\nu_1)}\|_{\infty,\mathcal X_0},\|\varphi_1^{(\nu_0',\nu_1')}\|_{\infty,\mathcal X_1},\|\varphi_1^{(\nu_0',\nu_1)}\|_{\infty,\mathcal X_1}\right\}\leq \frac{3}{2}\sup_{\substack{u\in\mathrm{B_0}\mathcal B\\v\in\mathrm{B_1}\mathcal B}}\|\rB_1^{\intercal}v-\rB_0^{\intercal}u\|_{\infty}. 
\]
Since $\sup_{\substack{u\in\mathrm{B_0}\mathcal B\\v\in\mathrm{B_1}\mathcal B}}\|\rB_1^{\intercal}v-\rB_0^{\intercal}u\|_{\infty}\leq\sup_{\substack{\|x_0\|_1=1\\\|x_1\|_1=1}}\|\rB_1^{\intercal}
\rB_1x_1-\rB_0^{\intercal}\rB_0x_0\|_{\infty}$, we have that
\[
\begin{aligned}
 \left|\mathsf{EOT}^{\varepsilon}_{\rB_1^{\intercal}v-\rB_0^{\intercal}u}(\nu_{0},\nu_{1})-   \mathsf{EOT}^{\varepsilon}_{\rB_1^{\intercal}v-\rB_0^{\intercal}u}(\nu_0',\nu_1')\right|&\\ &\hspace{-3em}\leq C_{\rB_0,\rB_1}\left(\|w_{\nu_0}-w_{\nu_0'}\|_1+\|w_{\nu_1}-w_{\nu_1'}\|_1\right) ,
\end{aligned}
\]
proving the claimed result.
\qed

\subsection{Proof of \texorpdfstring{\cref{lem:OTLipschitz}}{Lemma 5}}
\label{proof:lem:OTLipschitz}
    Remark that the marginals are not assumed to have full support so that   \cref{lem:potentialBounds} does not directly apply. Nevertheless, Propositions 1.11 and 1.34 in  \cite{santambrogio2015optimal} assert that there exists a pair of OT potentials, $(\varphi_0, \varphi_1)$, for $\mathsf{OT}_{\rB_1^{\intercal}v-\rB_0^{\intercal}u}(\nu_0,\nu_1)$ satisfying $\varphi_1(y) = \inf_{x\in\mathcal X_0}\left\{ c_{u,v}(x,y)-\varphi_0(x)\right\}$ and  
$\varphi_0(x) = \inf_{y\in\mathcal X_1}\left\{c_{u,v}(x,y)-\varphi_1(y)\right\}$ for every $(x,y)\in\mathcal X_0\times \mathcal X_1$ where $c_{u,v}$ is the cost function corresponding to $\rB_1^{\intercal}v-\rB_0^{\intercal} u$.

Letting $(\tilde \varphi_0,\tilde \varphi_1)= ( \varphi_0-\sup_{\mathcal X_0} \varphi_0 +\frac 12 \|c_{u,v}\|_{\infty,\mathcal X_0\times\mathcal X_1},\varphi_1+\sup_{\mathcal X_0} \varphi_0 -\frac 12 \|c_{u,v}\|_{\infty,\mathcal X_0\times\mathcal X_1})$,  
\[
    \tilde \varphi_1(y) = \inf_{x\in\mathcal X_0}\{ c_{u,v}(x,y)-\varphi_0(x)\} +\sup_{\mathcal X_0} \varphi_0 -\frac 12\|c_{u,v}\|_{\infty,\mathcal X_0\times\mathcal X_1} = \inf_{x\in\mathcal X_0}\{ c_{u,v}(x,y)-\tilde \varphi_0(x)\} 
\]
and, similarly, $\tilde \varphi_0(x)=\inf_{y\in\mathcal X_1}\{c_{u,v}(x,y)-\tilde \varphi_1(y)\}$ for each $(x,y)\in\mathcal X_0\times \mathcal X_1$. It follows that $\sup_{\mathcal X_0}\tilde \varphi_0=\frac 12 \|c_{u,v}\|_{\infty,\mathcal X_0\times \mathcal X_1}$ and, following the proof of \cref{lem:potentialBounds}, 
we obtain that
    $-\frac 32\|c_{u,v}\|_{\infty,\mathcal X_0\times \mathcal X_1} \leq  \tilde \varphi_1(y) \leq\frac 12 \|c_{u,v}\|_{\infty,\mathcal X_0\times \mathcal X_1}$ and, as such, $-\frac 32\|c_{u,v}\|_{\infty,\mathcal X_0\times \mathcal X_1} \leq  \tilde \varphi_0(x) \leq\frac 12 \|c_{u,v}\|_{\infty,\mathcal X_0\times \mathcal X_1} $ for each $(x,y)\in\mathcal X_0\times \mathcal X_1$.

It follows that there exists a choice of OT potentials, $(\varphi_0^{(\nu_0,\nu_1)}, \varphi_1^{(\nu_0,\nu_1)})$, for $\mathsf{OT}_{\rB_1^{\intercal}v-\rB_0^{\intercal}u}(\nu_0,\nu_1)$ admitting the above estimates;
let  $(\varphi_0^{(\nu_0',\nu_1')}, \varphi_1^{(\nu_0',\nu_1')})$ be OT potentials for $\mathsf{OT}_{\rB_1^{\intercal}v-\rB_0^{\intercal}u}(\nu_0',\nu_1')$ satisfying these same bounds. Then, 
\[
\begin{aligned}
\mathsf{OT}_{\rB_1^{\intercal}v-\rB_0^{\intercal}u}(\nu_{0},\nu_{1})-   \mathsf{OT}^{}_{\rB_1^{\intercal}v-\rB_0^{\intercal}u}(\nu_0',\nu_1')&\\&\hspace{-4em}\leq \int \varphi_0^{(\nu_0,\nu_1)} d\nu_0+ \int \varphi_1^{(\nu_0,\nu_1)} d\nu_1 - \int \varphi_0^{(\nu_0,\nu_1)} d\nu_0'-\int \varphi_1^{(\nu_0,\nu_1)} d\nu_1',
\\
\mathsf{OT}_{\rB_1^{\intercal}v-\rB_0^{\intercal}u}(\nu_{0},\nu_{1})-   \mathsf{OT}^{}_{\rB_1^{\intercal}v-\rB_0^{\intercal}u}(\nu_0',\nu_1')&\\&\hspace{-4em}\geq \int \varphi_0^{(\nu_0',\nu_1')} d\nu_0+ \int \varphi_1^{(\nu_0',\nu_1')} d\nu_1 - \int \varphi_0^{(\nu_0',\nu_1')} d\nu_0'-\int \varphi_1^{(\nu_0',\nu_1')} d\nu_1',
\end{aligned}
\]
so that 
\[
\begin{aligned}
\left|\mathsf{OT}_{\rB_1^{\intercal}v-\rB_0^{\intercal}u}(\nu_{0},\nu_{1})-   \mathsf{OT}^{}_{\rB_1^{\intercal}v-\rB_0^{\intercal}u}(\nu_0',\nu_1')\right| &\\&\hspace{-17em}\leq \max\left\{ \left|\int \varphi_0^{(\nu_0,\nu_1)} d(\nu_0-\nu_0')+ \int \varphi_1^{(\nu_0,\nu_1)} d(\nu_1-\nu_1')\right|,\right.&\\&\hspace{-4em}\left.\left|\int \varphi_0^{(\nu_0',\nu_1')} d(\nu_0-\nu_0')+ \int \varphi_1^{(\nu_0',\nu_1')} d(\nu_1-\nu_1')\right|\right\}.
\end{aligned}
\]
Applying the bounds on the potentials and the fact that $(u,v)\in\rB_0\mathcal B\times \rB_1\mathcal B$, we obtain
\[
\left|\mathsf{OT}_{\rB_1^{\intercal}v-\rB_0^{\intercal}u}(\nu_{0},\nu_{1})-   \mathsf{OT}^{}_{\rB_1^{\intercal}v-\rB_0^{\intercal}u}(\nu_0',\nu_1')\right| \leq C_{\rB_0,\rB_1}\left(\|w_{\nu_0}-w_{\nu_0'}\|_1+\|w_{\nu_1}-w_{\nu_1'}\|_1\right),
\]
by complete analogy with the end of the proof of \cref{lem:EOTLipschitz}.
\qed

\subsection{Proof of \texorpdfstring{\cref{lem:differenceQuotientUpperBound}}{Lemma 6}}
\label{proof:lem:differenceQuotientUpperBound}
Fix any $\pi^{\star}\in \Pi(\mu_0,\mu_1)$ solving $\mathsf{GW}_p(\mu_{0},\mu_{1})$ and let 
\[
\mathcal R_{\pi^{\star}}\coloneqq \left\{ (i,j)\in[N_0]\times [N_1]: \pi^{\star}\left(\left\{\left(x^{(i)}_0,x_1^{(j)}\right)\right\}\right)=0\right\}.
\]
Fix an arbitrary finite signed measure, $\gamma$, on $\supp(\mu_0)\times \supp(\mu_1)$ with marginals $\nu_0-\mu_0$ and $\nu_1-\mu_1$ which satisfies the property that $\gamma\left(\left\{\left(x^{(i)}_0,x_1^{(j)}\right)\right\}\right)\geq 0$  for every $(i,j)\in\mathcal R_{\pi^{\star}}$ (such a measure clearly exists, e.g., take any coupling of $(\nu_0,\nu_1)$ and take its difference with $\pi^{\star}$), whereby $\pi^{\star}+t\gamma \in \Pi(\mu_{0,t},\mu_{1,t})$ for every $t\in(0,1)$ sufficiently small.
Indeed, $\pi^{\star}+t\gamma$ always has the proper marginals, and proper normalization, but may fail to be a nonnegative measure unless $t$ is sufficiently small.   
It follows that 
\[
\begin{aligned}
       \mathsf{GW}_p(\mu_{0,t},\mu_{1,t})^p-\mathsf{GW}_p(\mu_{0},\mu_{1})^p &\leq \int \Delta d(\pi^{\star}+t\gamma)\otimes  (\pi^{\star}+t\gamma)-\int \Delta d\pi^{\star}\otimes \pi^{\star}
        \\
        &=2t\int \Delta d\pi^{\star}\otimes \gamma+t^2\int \Delta d\gamma \otimes \gamma
\end{aligned}
\]
for all sufficiently small $t$ and so 
\begin{equation}
\label{eq:expansionGWDifference}
    \limsup_{t\downarrow 0}t^{-1}\left(\mathsf{GW}_p(\mu_{0,t},\mu_{1,t})^p-\mathsf{GW}_p(\mu_{0},\mu_{1})^p\right) \leq 2\int \Delta d\pi^{\star}\otimes \gamma
\end{equation}
for any $\gamma$ as above. 
We now minimize the term on the right hand side over valid choices of $\gamma$. To this end, let $\mathrm{C}\in\mathbb R^{N_0\times N_1}$ satisfy  $\mathrm{C}_{ij}=2\int \Delta\left( x_0^{(i)},x_1^{(j)},x,y \right)d\pi^{\star}(x,y)$ for $(i,j)\in [N_0]\times [N_1]$ and consider the linear program  
\begin{equation}
    \label{eq:discreteLP}
    \begin{aligned}
    \inf_{\mathrm{G}\in \mathbb R^{N_0\times N_1}} \quad \langle \mathrm{C},\mathrm{G}\rangle_{\mathrm F}& \\
    \text{subject to} \hspace{0.95em} \quad
    \mathrm{G}\mathbf 1_{N_1} &= w_{\nu_0}-w_{\mu_0}, \\
    \mathrm{G}^{\intercal}\mathbf 1_{N_0} &= w_{\nu_1}-w_{\mu_1}, \\
    \mathrm{G}_{ij}&\geq 0, \;\forall\;(i,j)\in\mathcal R_{\pi^{\star}},
\end{aligned}
\end{equation}
where $\mathbf 1_{N_i}$ is the vector of all $1$'s of length $N_i$ for $i=0,1$. Note that the matrices $\mathrm{G}$ satisfying the constraint set in \eqref{eq:discreteLP} correspond to weight vectors of the finite signed measures $\gamma$ discussed above up to arranging them in matrix form. We now analyze the  dual problem to \eqref{eq:discreteLP}, which will enable us to derive the desired upper bound on the $\limsup$. The dual problem is given by
\begin{equation}
    \label{eq:discreteLPDual}
\begin{aligned}
    \sup_{\substack{r\in\mathbb R^{N_0},s\in\mathbb R^{N_1}\\ (S_{ij})_{ (i,j)\in \mathcal R_{\pi^{\star}}}}} \quad r^{\intercal}(w_{\nu_0}-w_{\mu_0})&+ s^{\intercal}(w_{\nu_1}-w_{\mu_1}) \\
    \text{subject to} \hspace{2.65em} \quad
    r_i+s_j&=\mathrm{C}_{ij},\;\forall\;(i,j)\not\in\mathcal R_{\pi^{\star}}, \\
    r_i+s_j+S_{ij}&=\mathrm{C}_{ij},\;\forall\;(i,j)\in\mathcal R_{\pi^{\star}}, \\
    S_{ij}&\geq 0.
\end{aligned}
\end{equation}
Remark that any feasible pair $(r,s)$ in \eqref{eq:discreteLPDual} satisfies $r_i+s_j\leq \mathrm{C}_{ij}$ with equality for all indices $(i,j)$ in the complement of $\mathcal R_{\pi^{\star}}$, i.e., for which $\pi^{\star}\left(\left\{\left(x^{(i)}_0,x_1^{(j)}\right)\right\}\right)>0$. We now show that any feasible $(r,s)$ can be identified with a pair of OT potentials for a particular OT problem. 

Since $\pi^{\star}$ solves $\mathsf{GW}_p(\mu_0,\mu_1)$, $x^{\star}=w_{\pi^{\star}}$ solves $\inf_{x\in\mathcal K}\frac 12 x^{\intercal}\rM x$. \cref{thm:variationalForm} further yields that $x^{\star}\in\argmin_{x\in\mathcal K}\left\{ (\rB_1^{\intercal}\rB_1x^{\star}-\rB_0^{\intercal}\rB_0x^{\star})^{\intercal}x \right\}$ where, by definition, for any $\pi\in\Pi(\mu_0,\mu_1)$,  
\[
(\rB_1^{\intercal}\rB_1w_{\pi^{\star}}-\rB_0^{\intercal}\rB_0w_{\pi^{\star}})^{\intercal}w_{\pi} = (w_{\pi^{\star}})^{\intercal}\rM w_{\pi} = 2\int \Delta d\pi^{\star}\otimes \pi.
\]
That is, $\rB_1^{\intercal}\rB_1w_{\pi^{\star}}-\rB_0^{\intercal}\rB_0w_{\pi^{\star}}$ corresponds to a vectorized version of the matrix $\mathrm C$ defined previously. It follows from the complementary slackness conditions for linear programming (cf. e.g., Theorem 4.5 in \cite{bertsimas1997introduction}) that $(\varphi_0,\varphi_1)$ is a pair of OT potentials for $\mathsf{OT}_{2\int \Delta d\pi^{\star}}(\mu_0,\mu_1)$ if and only if 
\[
    \varphi_0(x_0^{(i)})+\varphi_1(x_1^{(j)})\leq 2\int\Delta(x_0^{(i)},x_1^{(j)},x,y)d\pi^{\star}(x,y)=\mathrm{C}_{ij}\text{ for all }(i,j)\in[N_0]\times[N_1], 
\]
with equality $\pi^{\star}$-a.e. (since $\pi^{\star}$ is an OT plan for $\mathsf{OT}_{2\int \Delta d\pi^{\star}}(\mu_0,\mu_1)$), i.e., for every $(i,j)\not\in \mathcal R_{\pi^{\star}}$. It follows that every feasible pair $(r,s)$ in \eqref{eq:discreteLPDual} corresponds to the vector of values of some pair of OT potentials for $\mathsf{OT}_{2\int \Delta d\pi^{\star}}(\mu_0,\mu_1)$. 

Applying \cref{lem:potentialBounds}, for each feasible pair $(r,s)$ for \eqref{eq:discreteLPDual}, there is a constant $a\in\mathbb R$ depending on $(r,s)$ for which $(\tilde r,\tilde s)=(r+a,s-a)$ satisfies 
\[
    \tilde r^{\intercal} (w_{\nu_0}-w_{\mu_0}) + \tilde s^{\intercal} (w_{\nu_1}-w_{\mu_1}) =  r^{\intercal} (w_{\nu_0}-w_{\mu_0}) + s^{\intercal} (w_{\nu_1}-w_{\mu_1}),   
\]
as $\bm 1_{N_0}^{\intercal}(w_{\nu_0}-w_{\mu_0})=0$, $\bm 1_{N_1}^{\intercal}(w_{\nu_1}-w_{\mu_1})=0$, and $\max\{\|\tilde r\|_{\infty},\|\tilde s\|_{\infty}\}\leq \frac{3}{2} \max_{(i,j)\in[N_0]\times [N_1]}|\mathrm{C}_{ij}|$.  It follows that \eqref{eq:discreteLPDual} admits a solution, since it can be solved in the auxiliary variables $(\tilde r,\tilde s)$ defined above which are bounded. Conclude that \eqref{eq:discreteLP} also admits a solution, $\mathrm G^{\star}$, and that the optimal value of that problem coincides with that of \eqref{eq:discreteLPDual} by strong duality, see Theorem 4.4 in \cite{bertsimas1997introduction}. As noted above, $\mathrm{G}^{\star}$ corresponds to the matrix version of the weight vector of some finite signed measure $\gamma^{\star}$ with marginals $(\nu_0-\mu_0,\nu_1-\mu_1)$ and satisfying $\gamma^{\star}(\{(x_0^{(i)},x_1^{(j)})\})\geq 0$ if $(i,j)\in\mathcal R_{\pi^{\star}}$. 

Returning to \eqref{eq:expansionGWDifference}, we   insert this choice of $\gamma^{\star}$ to obtain that 
\[
    \limsup_{t\downarrow 0}t^{-1}\left(
       \mathsf{GW}_p(\mu_{0,t},\mu_{1,t})^p-\mathsf{GW}_p(\mu_{0},\mu_{1})^p\right) \leq 2\int \Delta d\pi^{\star}\otimes \gamma^{\star},
\]
where the right hand side is the optimal value in \eqref{eq:discreteLP} which we have shown coincides with the optimal value in \eqref{eq:discreteLPDual} and whose feasible set corresponds to the set of all vectors $(v_{\varphi_0},v_{\varphi_1})$ where $(\varphi_0,\varphi_1)$ is any choice of OT potentials for $\mathsf{OT}_{c_{\pi^{\star}}}$ where $c_{\pi^{\star}}(x,y) = 2\int \Delta(x,y,x',y')d\pi^{\star}(x',y')$. In the notation of  \cref{lem:differenceQuotientUpperBound} therefore, 
\[
\begin{aligned}
    \limsup_{t\downarrow 0}t^{-1}\left(
       \mathsf{GW}_p(\mu_{0,t},\mu_{1,t})^p-\mathsf{GW}_p(\mu_{0},\mu_{1})^p\right)&\\&\hspace{-2em} \leq \sup_{(\varphi_0,\varphi_1)\in\mathcal D_{\pi^{\star}}}\left\{\int \varphi_0d(\nu_0-\mu_0)+\int \varphi_1d(\nu_1-\mu_1)\right\}.
\end{aligned}
\]
As this argument does not depend on the choice of $\pi^{\star}$, the claimed bound follows by tightening the upper bound by taking the infimum over all optimal plans for $\mathsf{GW}_p(\mu_0,\mu_1)$. 
\qed

\subsection{Proof of \texorpdfstring{\cref{lem:differenceQuotientLowerBound}}{Lemma 7}}
\label{proof:lem:differenceQuotientLowerBound}
    Fix a sequence $t_n\downarrow 0$ and let $\pi^{\star}_{t_n}$ be any optimal coupling for $\mathsf{GW}_p(\mu_{0,t_n},\mu_{1,t_n})$. %
    As 
    \[
        \mathsf{GW}_p(\mu_{0,t_n},\mu_{1,t_n})^p = \inf_{\substack{x\in\mathcal K(\mu_{0,t_n},\mu_{1,t_n})}} \frac 12 x^{\intercal} \rM x,  
    \]
    where $\mathcal K(\mu_{0,t_n},\mu_{1,t_n})\coloneqq \left\{ x\in \mathbb R^{N_0N_1}:\mathrm{A}x=\begin{pmatrix}
       w_{\mu_{0,t_n}}\\w_{\mu_{1,t_n}} 
    \end{pmatrix},x\geq 0\right\}$ and $\rM$ and $\rA$ are fixed, Theorem 3 in \cite{klatte1985Lipschitz} implies that, for all $n$ sufficiently large,  
    \begin{equation}
    \label{eq:LipschitzPropertyGap}
        \argmin_{\substack{x\in\mathcal K(\mu_{0,t_n},\mu_{1,t_n})}} \frac 12 x^{\intercal} \rM x\subset \argmin_{\substack{x\in\mathcal K(\mu_{0},\mu_{1})}} \frac 12 x^{\intercal} \rM x+ Ct_n\left(\|w_{\nu_0}\mspace{-1mu}-\mspace{-1mu}w_{\mu_0}\|\mspace{-1mu}+\mspace{-1mu}\|w_{\nu_1}\mspace{-1mu}-\mspace{-1mu}w_{\mu_1}\|\right)\mathbb B
    \end{equation}
    where $C>0$ is a fixed constant, $\mathbb B$ is the closed ball of radius $1$ centered at $0$, and the sum is understood in the sense of Minkowski. %
    It follows that, for each $x^{\star}_{t_n} = w_{\pi^{\star}_{t_n}}$, there exists some choice of $\bar x_{t_n} \in \argmin_{x\in\mathcal K(\mu_0,\mu_1)} \frac 12 x^{\intercal}\rM x$ for which $\|x^{\star}_{t_n}-\bar x_{t_n}\|=O(t_n)$ as $t_n\downarrow 0$. 
    
    Applying \cref{thm:variationalForm}, we see that $\bar x_{t_n}$ also solves the linear program $\inf_{x\in\mathcal K}\bar x_{t_n}^{\intercal}\rM x$, i.e., it is a solution of $\mathsf{OT}_{c_n}(\mu_0,\mu_1)$ where $c_n(x,y)=2\int \Delta(x,y,x',y')d\bar \pi_{t_n}(x',y')$ is the cost function induced by $\rM\bar x_{t_n}$ and $\bar \pi_{t_n}\in\Pi(\mu_0,\mu_1)$ is the coupling associated with $\bar x_{t_n}$. By the complementary slackness conditions for linear programming, Theorem 4.5 in \cite{bertsimas1997introduction}, any pair of OT potentials $(\varphi_{0,n},\varphi_{1,n})$ for $\mathsf{OT}_{c_n}(\mu_0,\mu_1)$ must satisfy 
    $
       \varphi_{0,n}(x_0^{(i)})+ \varphi_{1,n}(x_1^{(j)})\leq c_n(x_0^{(i)},x_1^{(j)}) \text{ for all $(i,j)\in[N_0]\times [N_1]$ with equality if } \bar \pi_{t_n}(\{(x_0^{(i)},x_1^{(j)})\})>0.   
    $
    We apply this bound to lower bound the following expression,
    \begin{equation}
    \label{eq:GWLowerBoundIntermediate}
    \begin{aligned}
        \mathsf{GW}_p(\mu_{0,t_n},\mu_{1,t_n})^p-\mathsf{GW}_p(\mu_0,\mu_{1})^p&\\&\hspace{-2em}= \int \Delta d\pi^{\star}_{t_n}\otimes \pi^{\star}_{t_n} - \int \Delta d\bar \pi_{t_n}\otimes \bar \pi_{t_n} 
        \\
        &\hspace{-2em}=\int \Delta d(\pi^{\star}_{t_n}-\bar\pi_{t_n})\otimes (\pi^{\star}_{t_n}-\bar\pi_{t_n}) + 2\int \Delta d (\pi^{\star}_{t_n}-\bar \pi_{t_n})\otimes \bar \pi_{t_n}.
   \end{aligned} 
    \end{equation}
    Notably, the previous discussion implies that 
    \[
    \begin{aligned}
2\int \Delta d (\pi^{\star}_{t_n}-\bar \pi_{t_n})\otimes \bar \pi_{t_n}&=-\int \varphi_{0,n}(x)+\varphi_{1,n}(y)d\bar \pi_{t_n}(x,y) + \int c_n(x,y) d \pi^{\star}_{t_n}(x,y)
\\
&\geq \int \varphi_{0,n}(x)+\varphi_{1,n}(y)d(\pi^{\star}_{t_n}-\bar \pi_{t_n})(x,y)
\\
&=t_n \int \varphi_{0,n}d(\nu_0-\mu_0)+t_n\int \varphi_{1,n}d(\nu_1-\mu_1).
    \end{aligned} 
    \]
Furthermore, we have the bound 
\[
    \left|\int \Delta d(\pi^{\star}_{t_n}-\bar\pi_{t_n})\otimes (\pi^{\star}_{t_n}-\bar\pi_{t_n}) \right|\leq \|\rM\|_{\mathrm{op}}\|x^{\star}_{t_n}-\bar x_{t_n}\|^2=o(t_n)\text{ as }t_n\downarrow 0,  
    \]
    so that 
    \[  t_n^{-1}\left(\mathsf{GW}_p(\mu_{0,t_n},\mu_{1,t_n})^p-\mathsf{GW}_p(\mu_{0},\mu_{1})^p\right)
       \geq \int \varphi_{0,n}d(\nu_0-\mu_0)+\int \varphi_{1,n}d(\nu_1-\mu_1) +o(1),
    \] 
    where we recall that  $(\varphi_{0,n},\varphi_{1,n})$ is any pair of potentials for $\mathsf{OT}_{c_n}(\mu_0,\mu_1)$. We may thus tighten the lower bound by taking the supremum over all such potentials,
      \[
    \begin{aligned}t_n^{-1}\left(\mathsf{GW}_p(\mu_{0,t_n},\mu_{1,t_n})^p-\mathsf{GW}_p(\mu_{0},\mu_{1})^p\right)\hspace{-4em}& \\ 
       &\geq  \sup_{(\varphi_0,\varphi_1)\in\mathcal D^{\star}_{\bar \pi_{t_n}}}\left\{\int \varphi_{0}d(\nu_0-\mu_0)+\int \varphi_{1}d(\nu_1-\mu_1)\right\} +o(1).
      \end{aligned} 
       \]
       We conclude by observing that $\bar \pi_{t_n}$ was  a solution of $\mathsf{GW}_p(\mu_0,\mu_1)$ so that, taking the infimum over all such solutions, 
       \[
       \begin{aligned}
       t_n^{-1}\left(\mathsf{GW}_p(\mu_{0,t_n},\mu_{1,t_n})^p-\mathsf{GW}_p(\mu_{0},\mu_{1})^p\right)\hspace{-4em}& \\ 
       &\geq\inf_{\pi\in\Pi^{\star}}  \sup_{(\varphi_0,\varphi_1)\in\mathcal D^{\star}_{\pi}}\left\{\int \varphi_{0}d(\nu_0-\mu_0)+\int \varphi_{1}d(\nu_1-\mu_1)\right\} +o(1).
    \end{aligned} 
    \]
    Conclude that 
 \[
      \begin{aligned}\liminf_{t\downarrow 0}t^{-1}\left(\mathsf{GW}_p(\mu_{0,t},\mu_{1,t})^p-\mathsf{GW}_p(\mu_{0},\mu_{1})^p\right)&\\&\hspace{-3em}\geq\inf_{\pi\in\Pi^{\star}}  \sup_{(\varphi_0,\varphi_1)\in\mathcal D^{\star}_{\pi}}\left\{\int \varphi_{0}d(\nu_0-\mu_0)+\int \varphi_{1}d(\nu_1-\mu_1)\right\},
    \end{aligned} 
    \]
    proving the claimed result.

\subsection{Proof of \texorpdfstring{\cref{lem:differenceQuotientUpperBoundEOT}}{Lemma 8}}
\label{proof:lem:differenceQuotientUpperBoundEOT}

Fix any $\pi^{\star}\in \Pi(\mu_0,\mu_1)$ solving $\mathsf{EGW}_p^{\varepsilon}(\mu_{0},\mu_{1})$. By \cref{thm:variationalForm}, we have that $w_{\pi^{\star}}$ solves the discrete EOT problem with cost vector $c = \rB_1^{\intercal}\rB_1 w_{\pi^{\star}}-\rB_0^{\intercal}\rB_0 w_{\pi^{\star}}=\rM w_{\pi^{\star}}$ and marginals $(\mu_0,\mu_1)$. Now, $\|c\|_{\infty}\leq \|\rM\|_{\mathrm{op}}\|w_{\pi^{\star}}\|_2\leq \|\rM\|_{\mathrm{op}}$, using the fact that $\|\cdot\|_{\infty}\leq \|\cdot\|_2\leq \|\cdot\|_1$. It follows  from \cref{lem:EOTPotentialBounds} that 
\begin{equation}
\label{eq:lowerBoundPlans}
                {\pi^{\star}}\left(\{(x,y)\}\right)\geq  e^{\frac{-4\|\rM\|_{\mathrm{op}}}{\varepsilon}}\mu_0(\{x\}) \mu_1(\{y\})\geq l_{\min}>0\text{  for each $(x,y)\in\mathcal X_0\times \mathcal X_1$,}           
\end{equation}
where we set $l_{\min}\coloneqq e^{\frac{-4\|\rM\|_{\mathrm{op}}}{\varepsilon}}\inf_{(x,y)\in\mathcal X_0\times \mathcal X_1}\mu_0(\{x\}) \mu_1(\{y\})$ which is nonzero since $\mu_0,\mu_1$ assign positive mass to each point in $\mathcal X_0,\mathcal X_1$ which have finite cardinality.

Now, let $\gamma$ be any finite signed measure, on $\supp(\mu_0)\times \supp(\mu_1)$ with marginals $\nu_0-\mu_0$ and $\nu_1-\mu_1$. It follows from the above deliberations that  $\pi^{\star}+t\gamma \in \Pi(\mu_{0,t},\mu_{1,t})$ and assigns positive measure to every pair $(x,y)\in\mathcal X_0\times \mathcal X_1$ for every $t\in[0,1)$ sufficiently small. Letting $\Delta(x,y,x',y')\coloneqq \left|\kappa_0(x,x')-\kappa_1(y,y')\right|^p$, we have that
\begin{equation}
\label{eq:differenceExpansion}
\begin{aligned}
       &\mathsf{EGW}_p^{\varepsilon}(\mu_{0,t},\mu_{1,t})-\mathsf{EGW}_p^{\varepsilon}(\mu_{0},\mu_{1}) \\
       &\leq \int \Delta d(\pi^{\star}+t\gamma)\otimes  (\pi^{\star}+t\gamma)-\int \Delta d\pi^{\star}\otimes \pi^{\star}
       \\
       &-\varepsilon\left(\mathsf H_{\pi^{\star}+t\gamma} -\mathsf H_{\mu_{0,t}}-\mathsf H_{\mu_{1,t}}\right) + \varepsilon\left(\mathsf H_{\pi^{\star}} -\mathsf H_{\mu_{0}}-\mathsf H_{\mu_{1}}\right)
\end{aligned}
\end{equation}
provided that $t$ is sufficiently small  that all terms are well-defined.

Note that, for any probability measure $\eta$ with full support on $\mathcal X_0$ and any finite signed measure $\chi$ with total mass zero  on $\mathcal X_0$, $\mathsf H_{\eta+t\chi}$ is well-defined for all $t\in[0,1)$ sufficiently small and, by the mean value theorem, for all such $t$,  
\[
\begin{aligned}
-\mathsf H_{\eta+t\chi}+\mathsf H_{\eta} &= t\sum_{x\in\mathcal X_0}\left( \chi(\{x\})\log\left((\eta+s_t\chi)(\{x\})\right) + \chi(\{x\})\right)
\\
&= t\sum_{x\in\mathcal X_0}\chi(\{x\})\log\left((\eta+s_t\chi)(\{x\})\right)
\end{aligned}
\]
for some $s_t\in [0,t]$ since $\chi$ sums to $0$. An analogous expansion evidently holds if the base measure is defined on $\mathcal X_1$ or $\mathcal X_0\times \mathcal X_1$. It follows from \eqref{eq:differenceExpansion} and the continuity of the $\log$ away from $0$ that  
\begin{equation}
\label{eq:limsupEGW}
\begin{aligned}
&\limsup_{t\downarrow 0}t^{-1}\left(\mathsf{EGW}_p^{\varepsilon}(\mu_{0,t},\mu_{1,t})-\mathsf{EGW}_p^{\varepsilon}(\mu_{0},\mu_{1})\right) 
       \\&\leq 2\int \Delta d\gamma\otimes  \pi^{\star}
       +\varepsilon \sum_{(x,y)\in\mathcal X_0\times \mathcal X_1} \gamma(\{(x,y)\})\log\left(\pi^{\star}(\{(x,y)\})\right)
       \\
       &-\varepsilon \sum_{x\in\mathcal X_0}(\nu_0-\mu_0)(\{x\})\log\left(\mu_0(\{x\})\right)-\varepsilon \sum_{y\in\mathcal X_1}(\nu_1-\mu_1)(\{y\})\log\left(\mu_1(\{y\})\right).
\end{aligned}
\end{equation}
We now derive the choice of $\gamma$ which minimizes the right hand side subject to the marginal constraints. To this end, let $\mathrm{C}\in\mathbb R^{N_0\times N_1}$ be such that  $\mathrm{C}_{ij}=2\int \Delta\left( x_0^{(i)},x_1^{(j)},x,y \right)d\pi^{\star}(x,y)+\varepsilon \log\left(\pi^{\star}(\{(x_0^{(i)},x_1^{(j)})\})\right)$ for $(i,j)\in [N_0]\times [N_1]$ and consider the linear program  
\begin{equation}
    \label{eq:discreteLPEOT}
    \begin{aligned}
    \inf_{\mathrm{G}\in \mathbb R^{N_0\times N_1}} \quad \langle \mathrm{C},\mathrm{G}\rangle_{\mathrm F}& \\
    \text{subject to} \hspace{0.95em} \quad
    \mathrm{G}\bm 1_{N_1} &= w_{\nu_0}-w_{\mu_0}, \\
    \mathrm{G}^{\intercal}\bm 1_{N_0} &= w_{\nu_1}-w_{\mu_1}.
\end{aligned}
\end{equation}
The dual problem to \eqref{eq:discreteLPEOT} is given by
\begin{equation}
    \label{eq:discreteLPDualEOT}
\begin{aligned}
    \sup_{\substack{u\in\mathbb R^{N_0},v\in\mathbb R^{N_1}}} \quad u^{\intercal}(w_{\nu_0}-w_{\mu_0})&+ v^{\intercal}(w_{\nu_1}-w_{\mu_1}) \\
    \text{subject to} \hspace{2.65em} \quad
    u_i+v_j&=\mathrm{C}_{ij},\;\forall\;(i,j)\in[N_0]\times[N_1].
\end{aligned}
\end{equation}
The constraint in \eqref{eq:discreteLPDualEOT} can be written explicitly as 
\[ \pi^{\star}(\{(x_0^{(i)},x_1^{(j)})\})=
    e^{\frac{u_i+v_j- 2\int \Delta\left( x_0^{(i)},x_1^{(j)},x,y \right)d\pi^{\star}(x,y)}{\varepsilon}}, 
\]
so that 
\begin{equation}
\label{eq:optEGWDensity}
\frac{d\pi^{\star}}{d\mu_0\otimes \mu_1}(x_0^{(i)},x_1^{(j)})=
    e^{\frac{u_i+v_j- 2\int \Delta\left( x_0^{(i)},x_1^{(j)},x,y \right)d\pi^{\star}(x,y)-\varepsilon\log\left((w_{\mu_0})_i\right)-\varepsilon\log\left((w_{\mu_1})_j\right)}{\varepsilon}}. 
\end{equation}
Since $\pi^{\star}$ solves $\mathsf{EGW}_p^{\varepsilon}(\mu_0,\mu_1)$ by assumption, it follows from  \cref{thm:variationalForm} that 
\begin{equation}
\label{eq:solnVarFormProofEGW}
   {w_{\pi^{\star}}}\in \argmin_{x\in\mathcal K}\left\{w_{\pi^{\star}}^{\intercal}\rM x-\varepsilon \mathsf H(x)\right\}.  
\end{equation}
For any coupling $\pi\in \Pi(\mu_0,\mu_1)$, let $z_{\mu_0}\in\mathbb R^{N_0N_1},z_{\mu_1}\in\mathbb R^{N_0N_1}$ be such that 
\[
\begin{aligned}
z_{\mu_0}^{\intercal} w_{\pi}\coloneqq
    \sum_{(x_0,x_1)\in\mathcal X_0\times \mathcal X_1} \pi(\{(x_0,x_1)\})\log(\mu_0(\{x_0\})) &= \sum_{x_0\in\mathcal X_0} \mu_0(\{x_0\})\log(\mu_0(\{x_0\})) = -\mathsf H_{\mu_0},
    \\
 z_{\mu_1}^{\intercal} w_{\pi}\coloneqq  \sum_{(x_0,x_1)\in\mathcal X_0\times \mathcal X_1} \pi(\{(x_0,x_1)\})\log(\mu_1(\{x_1\})) &= \sum_{x_1\in\mathcal X_1} \mu_1(\{x_1\})\log(\mu_1(\{x_1\}))= -\mathsf H_{\mu_1}.
\end{aligned}
\]
Inserting these terms into \eqref{eq:solnVarFormProofEGW}, we obtain that 
\[ {w_{\pi^{\star}}}\in \argmin_{\pi\in\Pi(\mu_0,\mu_1)}\left\{w_{\pi^{\star}}^{\intercal}\rM w_{\pi}+\varepsilon z_{\mu_0}^{\intercal}w_{\pi}+\varepsilon z_{\mu_1}^{\intercal}w_{\pi}+\varepsilon \mathsf{KL}(\pi\|\mu_0\otimes \mu_1)\right\},  
\]
which is a discrete EOT problem with cost function 
\[
c^{\star}(x_0^{(i)},x_1^{(j)})= 2\int \Delta\left( x_0^{(i)},x_1^{(j)},x,y \right)d\pi^{\star}(x,y)+\varepsilon\log(\mu_0(\{x_0^{(i)}\}))+\varepsilon\log(\mu_1(\{x_1^{(j)}\}))
\]
so that \eqref{eq:optEGWDensity} and the characterization of EOT couplings in terms of EOT potentials, see \cref{sec:EOT}, imply that if $\varphi_0(x_0^{(i)})=u_i$ and $\varphi_1(x_1^{(j)})=v_j$, $(\varphi_0,\varphi_1)$ is a pair of EOT potentials for $\mathsf{EOT}_{c^{\star}}^{\varepsilon}(\mu_0,\mu_1)$. 
 
It follows that every feasible pair $(u,v)$ in \eqref{eq:discreteLPDualEOT} corresponds to a pair $(\varphi_0^{\pi^{\star}},\varphi_1^{\pi^{\star}})$ of EOT potentials for $\mathsf{EOT}_{c^{\star}}^{\varepsilon}(\mu_0,\mu_1)$ (which is unique up to additive constants) and hence all achieve the same value in the objective of \eqref{eq:discreteLPDualEOT} since $w_{\nu_0}-w_{\mu_0}$ and $w_{\nu_1}-w_{\mu_1}$ are vectors whose entries sum to $0$. Recalling the bounds on the potentials in  \cref{lem:EOTPotentialBounds}, we see that the dual problem admits a solution and hence strong duality, Theorem 4.4 in \cite{bertsimas1997introduction}, applies. That is, the optimal value in \eqref{eq:discreteLPEOT} coincides with that in \eqref{eq:discreteLPDualEOT}. Note, also, that if $(\varphi_0^{\pi^{\star}},\varphi_1^{\pi^{\star}})$ is a pair of EOT potentials for $\mathsf{EOT}^{\varepsilon}_{c^{\star}}(\mu_0,\mu_1)$, then the pair $(\tilde \varphi_0^{\pi^{\star}},\tilde \varphi_1^{\pi^{\star}})$ defined by  $\tilde \varphi_i^{\pi^{\star}}(x_i)=\varphi_i^{\pi^{\star}}(x_i)-\varepsilon \log(\mu_i(\{x_i\}))  $ for $x_i\in\mathcal X_i$ ($i\in\{0,1\}$), is a pair of EOT potentials for $\mathsf{EOT}^{\varepsilon}_{c_{\pi^{\star}}}(\mu_0,\mu_1)$ as defined in the statement of the lemma. 
Applying this equivalence in  \eqref{eq:limsupEGW} yields that 
\[\begin{aligned}
    \limsup_{t\downarrow 0}t^{-1}\left(\mathsf{EGW}_p^{\varepsilon}(\mu_{0,t},\mu_{1,t})-\mathsf{EGW}_p^{\varepsilon}(\mu_{0},\mu_{1})\right)
   &\leq \inf_{\substack{\mathrm{G}\in \mathbb R^{N_0\times N_1}\\ \mathrm{G}\mathbf 1_{N_1} = w_{\nu_0}-w_{\mu_0}\\ \mathrm{G}^{\intercal}\mathbf 1_{N_0} = w_{\nu_1}-w_{\mu_1}}}  \langle \mathrm{C},\mathrm{G}\rangle_{\mathrm F}
    \\
    &= \int \tilde \varphi_0^{\pi^{\star}} d(\nu_0-\mu_0) +\int \tilde\varphi_1^{\pi^{\star}} d(\nu_1-\mu_1).
\end{aligned}
\]
As the initial choice of $\pi^{\star}$ was arbitrary, the above bound can be tightened by taking the infimum over all optimal couplings for $\mathsf{EGW}_p^{\varepsilon}(\mu_0,\mu_1)$, proving the claimed result.
\qed

\subsection{Proof of \texorpdfstring{\cref{lem:differenceQuotientLowerBoundEOT}}{Lemma 9}}
\label{proof:lem:differenceQuotientLowerBoundEOT}
   For any sequence $t_n\downarrow 0$, let $(x_{t_n})_{n\in\mathbb N}$ satisfy \[x_{t_n}\in\argmin_{x\in\mathcal K_{t_n}}\left\{ \frac 12 x^{\intercal}\rM x-\varepsilon \mathsf H(x)\right\}
   \]
    where 
    \[
    \mathcal K_{t_n} \coloneqq  \left\{ x\in\mathbb R^{N_0N_1}:\mathrm Ax=\begin{pmatrix}
       w_{\mu_0}+t_n\left(w_{\nu_0}-w_{\mu_0}\right)
       \\
       w_{\mu_1}+t_n\left(w_{\nu_1}-w_{\mu_1}\right)
       \end{pmatrix},x\geq 0 
       \right\}.
    \]
    Note first that, for any choice of $\bar x$ solving $\argmin_{x\in\mathcal K}\left\{ \frac 12 x^{\intercal}\rM x-\varepsilon \mathsf H(x)\right\}$ and vector $g$ for which $\mathrm Ag=\begin{pmatrix}
       w_{\nu_0}-w_{\mu_0}
       \\
       w_{\nu_1}-w_{\mu_1}\end{pmatrix}$, we have that $\bar x+t_n g\in \mathcal K_{t_n}$ and is positive for every $n$ sufficiently large (recalling the proof of \cref{lem:differenceQuotientUpperBoundEOT}). It follows that 
       \begin{equation}
       \label{eq:upperBoundEGWPlan}
       \frac 12 x^{\intercal}_{t_n}\rM x_{t_n}-\varepsilon \mathsf H(x_{t_n})
            \leq
            \frac 12 \left(\bar x+t_n g\right)^{\intercal}\rM \left(\bar x+t_n g\right)-\varepsilon \mathsf H\left(\bar x+t_n g\right)\to \min_{x\in\mathcal K} \left\{\frac 12 x^{\intercal}\rM x -\varepsilon \mathsf H(x)\right\}.   
       \end{equation}
       Now, fix a subsequence $t_{n'}$ and note that, since $x_{t_{n'}}$ is an element of the probability simplex on $\mathbb R^{N_0N_1}$ it admits a convergent subsequence $x_{t_{n''}}\to x_{0}$ by the Bolzano-Weierstrass theorem. Note that the relation $x_{t_{n''}}\in\mathcal K_{t_{n''}}$ passes to the limit so that $x_0\in\mathcal K$. From \eqref{eq:upperBoundEGWPlan} we have, moreover, that $x_0\in\argmin_{x\in\mathcal K} \left\{\frac 12 x^{\intercal}\rM x -\varepsilon \mathsf H(x)\right\}$.

       Now, leveraging the variational form from \cref{thm:variationalForm}, we have that 
        \[
        \begin{aligned}
        \mathsf{EGW}_p^{\varepsilon}(\mu_{0,t_{n''}},\mu_{1,t_{n''}})-\mathsf{EGW}_p^{\varepsilon}(\mu_{0},\mu_{1})\hspace{-3em}
        &\\&\geq \sup_{v\in\mathbb R^{r_1}}\left\{-\frac 12 \|v\|^2+\inf_{x\in\mathcal K_{t_{n''}}}\left\{\left( \rB_1^{\intercal}v-\rB_0^{\intercal} u_{t_{n''}} \right)^{\intercal}x-\varepsilon\sH(x)  \right\} \right\}
        \\
        &%
        -\sup_{v\in\mathbb R^{r_1}}\left\{-\frac 12 \|v\|^2+\inf_{x\in\mathcal K}\left\{\left( \rB_1^{\intercal}v-\rB_0^{\intercal} u_{t_{n''}} \right)^{\intercal}x-\varepsilon\sH(x)  \right\} \right\}
        \\&%
        +\varepsilon\left(\mathsf H_{\mu_{0,t_{n''}}}+\mathsf H_{\mu_{1,t_{n''}}}-\mathsf H_{\mu_{0}}-\mathsf H_{\mu_{1}}\right)
        \end{aligned}
        \]
        for $u_{t_{n''}}=\rB_0x_{t_{n''}}$ which solves the outer minimization in the variational formulation of $\mathsf{EGW}_p^{\varepsilon}(\mu_{0,t_{n''}},\mu_{1,t_{n''}})$. Further, $\rB_0x_{t_{n''}}\to \rB_0x_0\eqqcolon u_0$ which solves the outer minimization in the variational form of $\mathsf{EGW}_p^{\varepsilon}(\mu_{0},\mu_{1})$. Now, let $v_{t_{n''}}$ be the unique solution of \[
        \sup_{v\in\mathbb R^{r_1}}\left\{-\frac 12 \|v\|^2+\inf_{x\in\mathcal K}\left\{\left( \rB_1^{\intercal}v-\rB_0^{\intercal} u_{t_{n''}} \right)^{\intercal}x-\varepsilon\sH(x)  \right\} \right\},
        \]which is strongly concave. Then,
        \begin{equation}
        \label{eq:EGWDiffLower}
        \begin{aligned}
        \mathsf{EGW}_p^{\varepsilon}(\mu_{0,t_{n''}},\mu_{1,t_{n''}})-\mathsf{EGW}_p^{\varepsilon}(\mu_{0},\mu_{1})
       &\\ 
        &\hspace{-12em}\geq \inf_{x\in\mathcal K_{t_{n''}}}\left\{\left( \rB_1^{\intercal}v_{t_{n''}}-\rB_0^{\intercal} u_{t_{n''}} \right)^{\intercal}x-\varepsilon\sH(x)\right\}-\inf_{x\in\mathcal K}\left\{\left( \rB_1^{\intercal}v_{t_{n''}}-\rB_0^{\intercal} u_{t_{n''}} \right)^{\intercal}x-\varepsilon\sH(x) \right\}
        \\&\hspace{-12em}
        +\varepsilon\left(\mathsf H_{\mu_{0,t_{n''}}}+\mathsf H_{\mu_{1,t_{n''}}}-\mathsf H_{\mu_{0}}-\mathsf H_{\mu_{1}}\right)
        \\&\hspace{-12em}=\mathsf{EOT}_{\rB_1^{\intercal}v_{t_{n''}}-\rB_0^{\intercal} u_{t_{n''}}}^{\varepsilon}(\mu_{0,t_{n''}},\mu_{1,t_{n''}})-\mathsf{EOT}_{\rB_1^{\intercal}v_{t_{n''}}-\rB_0^{\intercal} u_{t_{n''}}}^{\varepsilon}(\mu_{0},\mu_{1}).
        \end{aligned}
        \end{equation}
        Applying the same logic as in the proof of \cref{lem:EOTLipschitz}, 
       \begin{equation}
       \label{eq:EOTDiffLower}
       \begin{aligned}
            \mathsf{EOT}_{\rB_1^{\intercal}v_{t_{n''}}-\rB_0^{\intercal} u_{t_{n''}}}^{\varepsilon}(\mu_{0,t_{n''}},\mu_{1,t_{n''}})-\mathsf{EOT}_{\rB_1^{\intercal}v_{t_{n''}}-\rB_0^{\intercal} u_{t_{n''}}}^{\varepsilon}(\mu_{0},\mu_{1})&
            \\
            &\hspace{-12em}\geq \int \varphi_{0,t_{n''}}^{(\mu_0,\mu_{1,t_{n''}})} d(\mu_{0,t_{n''}}-\mu_0) + \int \varphi_{1,t_{n''}}^{(\mu_{0},\mu_{1})} d(\mu_{1,t_{n''}}-\mu_1) 
            \\
            &\hspace{-12em} =t_{n''}\int \varphi_{0,t_{n''}}^{(\mu_0,\mu_{1,t_{n''}})} d(\nu_{0}-\mu_0) + t_{n''}\int \varphi_{1,t_{n''}}^{(\mu_{0},\mu_{1})} d(\nu_{1}-\mu_1), 
        \end{aligned} 
       \end{equation}
       where $(\varphi_{0,t_{n''}}^{(\mu_0,\mu_{1,t_{n''}})},\varphi_{1,t_{n''}}^{(\mu_0,\mu_{1,t_{n''}})})$ is any pair of $\mathsf{EOT}$ potentials for $\mathsf{EOT}_{\rB_1^{\intercal}v_{t_{n''}}-\rB_0^{\intercal} u_{t_{n''}}}^{\varepsilon}(\mu_{0},\mu_{1,t_{n''}})$ satisfying the conditions of \cref{lem:EOTPotentialBounds}  and $(\varphi_{0,t_{n''}}^{(\mu_0,\mu_1)},\varphi_{1,t_{n''}}^{(\mu_0,\mu_1)})$ is defined similarly. Recall that $u_{t_{n''}}=\rB_0 x_{t_{n''}}$ which converges to $u_0=\rB_0x_0$ so that $\|u_{t_{n''}}\|\leq B$ for some fixed $B\in\mathbb R$ and all $n''\in\mathbb N$ whereas $v_{t_{n''}}$ is the unique maximizer of the smooth strongly concave function 
      \[
            f_{t_{n''}}: v\in\mathbb R^{r_1}\mapsto -\frac 12 \|v\|^2 +\inf_{x\in\mathcal K}\left\{(\rB_1^{\intercal}v-\rB_0^{\intercal}u_{t_{n''}})^{\intercal}x-\varepsilon \mathsf H(x)\right\}.
      \]
       As $|f_{t_{n''}}(v)-f_0(v)|\leq \sup_{x\in\mathcal K}\left|u_{t_{n''}}^{\intercal}\rB_0x-u_{0}^{\intercal}\rB_0x \right|\leq \sup_{x\in\mathcal K}\|\rB_0x\| \|u_{t_{n''}}-u_0\|$ for each $v\in\mathbb R^{r_1}$, we see that 
         $f_{t_{n''}}$ converges uniformly to $f_0$. Since $-f_0$ is strongly convex, $\lim_{\|v\|\to \infty} -f_0(v)=\infty$ by Corollary 11.17 in \cite{bauschke2017convex}, so that Theorem 7.33 in \cite{rockafellar1998variational} can be applied (appealing to Proposition 7.15 and Exercise 7.32 from the same reference) to obtain that $v_{t_{n''}}$ converges to $v_0$, the unique solution of 
       $\sup_{v\in\mathbb R^{r_1}}\left\{-\frac 12 \|v\|^2+\inf_{x\in\mathcal K}\left\{\left( \rB_1^{\intercal}v-\rB_0^{\intercal} u_{0} \right)^{\intercal}x-\varepsilon\sH(x)  \right\} \right\}$. Note that since $u_0=\rB_0x_0$ where $x_0$ solves $\min_{x\in\mathcal K}\left\{\frac 12 x^{\intercal}\rM x -\varepsilon\mathsf H(x)\right\}$, \cref{thm:variationalForm} asserts that $(u,v)=(\rB_0x_0,\rB_1x_0)$ is a solution of the $\inf$-$\sup$ problem in the variational form and so $v_0=\rB_1 x_0$ by uniqueness.

 With this, $u_{t_{n''}}$ and $v_{t_{n''}}$ converge to $u_0$ and $v_0$ and $\|u_{t_{n''}}\|$ and $\|v_{t_{n''}}\|$ are uniformly bounded so that the various potentials are also uniformly bounded by \cref{lem:EOTPotentialBounds}. By the Bolzano-Weierstrass theorem, there exists a further subsequence $n'''$ and pairs $(\varphi_0,\varphi_1)$, $(\psi_0,\psi_1)$ for which $(\varphi_{0,t_{n'''}}^{(\mu_0,\mu_{1,t_{n'''}})},\varphi_{1,t_{n'''}}^{(\mu_0,\mu_{1,t_{n'''}})})\to (\varphi_0,\varphi_1)$ and $(\varphi_{0,t_{n'''}}^{(\mu_0,\mu_1)},\varphi_{1,t_{n'''}}^{(\mu_0,\mu_1)})\to (\psi_0,\psi_1)$ on $\mathcal X_0\times \mathcal X_1$. Since the potentials are characterized as solutions to the Schr{\"o}dinger system \eqref{eq:SchrodingerSystem} on the entire space,  
       \[
       \begin{gathered}
            1=\int  e^{\frac{\varphi_{0,t_{n'''}}^{(\mu_0,\mu_{1,t_{n'''}})}(x)+\varphi_{1,t_{n'''}}^{(\mu_0,\mu_{1,t_{n'''}})}(y)-c_{t_{n'''}}(x,y)}{\varepsilon}} d \mu_0(x)\to \int  e^{\frac{\varphi_{0}(x)+\varphi_{1}(y)-c_{0}(x,y)}{\varepsilon}} d \mu_0(x), 
            \\
            1=\int  e^{\frac{\varphi_{0,t_{n'''}}^{(\mu_0,\mu_{1,t_{n'''}})}(x)+\varphi_{1,t_{n'''}}^{(\mu_0,\mu_{1,t_{n'''}})}(y)-c_{t_{n'''}}(x,y)}{\varepsilon}} d \mu_{1,{t_{n'''}}}(y)\to \int  e^{\frac{\varphi_{0}(x)+\varphi_{1}(y)-c_{0}(x,y)}{\varepsilon}} d \mu_{1}(y),
        \end{gathered} 
       \]
       for every $(x,y)\in\mathcal X_0\times \mathcal X_1$ 
       where $c_{t_{n'''}}$ is the cost function associated with the cost vector $\rB_1^{\intercal}v_{t_{n'''}}-\rB_0^{\intercal}u_{t_{n'''}}$ which converges to $c_0$ on $\mathcal X_0\times \mathcal X_1$. Conclude that $(\varphi_0,\varphi_1)$ satisfies the Schr{\"o}dinger system for the marginals $\mu_0,\mu_1$ and cost function $c_0$ so that the pair corresponds to EOT potentials for $\mathsf{EOT}_{c_0}^{\varepsilon}(\mu_0,\mu_1)$. Similarly, $(\psi_0,\psi_1)$ is also a pair of EOT potentials  for $\mathsf{EOT}_{c_0}^{\varepsilon}(\mu_0,\mu_1)$. By uniqueness of the EOT potentials (\cref{lem:EOTPotentialBounds}) $(\varphi_0,\varphi_1)$ and $(\psi_0,\psi_1)$ differ only by additive constants.

        Combining \eqref{eq:EGWDiffLower} and \eqref{eq:EOTDiffLower}, it follows that   
        \[
    \begin{aligned}\liminf_{t_{n'''}\downarrow 0}t_{n'''}^{-1}\left(\mathsf{EGW}_p^{\varepsilon}(\mu_{0,t_{n'''}},\mu_{1,t_{n'''}})-\mathsf{EGW}_p^{\varepsilon}(\mu_{0},\mu_{1}) \right)&
    \\
    &\hspace{-4em}\geq \int \varphi_0 d(\nu_0-\mu_0) + \int \varphi_1 d(\nu_1-\mu_1)
    \\
    &\hspace{-4em}\geq \inf_{\pi^{\star}\in\Pi^{\star}(\mu_0,\mu_1)}\int \varphi_0^{\pi^{\star}} d(\nu_0-\mu_0) + \int \varphi_1^{\pi^{\star}} d(\nu_1-\mu_1).
    \end{aligned}
        \]
        The final lower bound follows by noting that, since $c_0$ is the cost  associated with the vector $\rB_1^{\intercal}v_0 -\rB_0^{\intercal}u_0$ for $u_0=\rB_0x_0$ and $v_0=\rB_1 x_0$, $\mathsf{EOT}^{\varepsilon}_{c_0}(\mu_0,\mu_1)=\mathsf{EOT}^{\varepsilon}_{c_{\pi^{\star}_0}}(\mu_0,\mu_1)$, where $\pi^{\star}_0$ is the optimal coupling for $\mathsf{EGW}_p^{\varepsilon}(\mu_0,\mu_1)$ associated with the vector $x_0$. As such, we may take the lower bound over all optimal couplings for the corresponding EGW problem, yielding a lower bound which is independent of the choice of original sequence and subsequence. Conclude that 
        \[  \begin{aligned}\liminf_{t\downarrow 0}t^{-1}\left(\mathsf{EGW}_p^{\varepsilon}(\mu_{0,t},\mu_{1,t})-\mathsf{EGW}_p^{\varepsilon}(\mu_{0},\mu_{1}) \right)&
    \\
    &\hspace{-3em}\geq \inf_{\pi^{\star}\in\Pi^{\star}(\mu_0,\mu_1)}\int \varphi_0^{\pi^{\star}} d(\nu_0-\mu_0) + \int \varphi_1^{\pi^{\star}} d(\nu_1-\mu_1).
    \end{aligned}
        \]
       \qed

\subsection{Proof of \texorpdfstring{\cref{lem:nonemptyInterior}}{Lemma 10}}
\label{proof:lem:nonemptyInterior}
       We begin with the reverse direction. If, for each pair, $(i,j)\in [N_{\eta}]\times [N_{\eta}]$ with $i\neq j$, there exists $x'\in\supp(\eta)$ satisfying     $\kappa_0(x^{(i)},x')\neq \kappa_0(x^{(j)},x')$, we have that 
        \[
        \begin{aligned}
        \int \left|\kappa_0(x^{(i)},x')-\kappa_0(x^{(j)},x')\right|^pd\eta(x')&= \sum_{k=1}^{N_{\eta}}\eta(\{x^{(k)}\}) \left|\kappa_0(x^{(i)},x^{(k)})-\kappa_0(x^{(j)},x^{(k)})\right|^p
        \\
        &\geq \eta(\{x'\}) \left|\kappa_0(x^{(i)},x')-\kappa_0(x^{(j)},x')\right|^p>0.
        \end{aligned} 
        \] 
        It follows that $2\min_{\substack{i,j=1\\i\neq j}}^{N_{\eta}} \int \left|\kappa_0(x^{(i)},x')-\kappa_0(x^{(j)},x')\right|^pd\eta(x')\eqqcolon \underline \beta >0$. From this, we see that if $w\in\RR^{N_{\eta}}$ satisfies $w_i-w_j\leq \underline \beta$ for every $(i,j)\in [N_{\eta}]\times [N_{\eta}]$, then $w\in\Gamma_{\eta}$. In particular, if $w_i\in[-\underline \beta/2,\underline \beta/2]$ for each $i\in[N_{\eta}]$, $w \in \Gamma_{\eta}$, proving that $\Gamma_{\eta}$ has nonempty interior.

        For the opposite direction, suppose that there exists a pair $(i,j)\in[N_{\eta}]\times [N_{\eta}]$  with $i\neq j$ for which $\kappa_0(x^{(i)},x')=\kappa_0(x^{(j)},x')$ for every $x'\in\supp(\eta)$. In this case, 
        \[
            \int \left|\kappa_0(x^{(i)},x')-\kappa_0(x^{(j)},x')\right|^pd\eta(x') = 0, 
        \]
        so  $w_i-w_j\leq 0$ and $w_j-w_i\leq 0$ i.e., $w_i=w_j$ for each  $w\in\Gamma_{\eta}$ and  $\Gamma_{\eta}$ has empty interior.
\qed

\subsection{Proof of \texorpdfstring{\cref{lem:kappaVanishing}}{Lemma 13}}
\label{proof:lem:kappaVanishing}
  We start by showing that $\kappa(G,G')=\kappa(P_{\sigma}(G),P_{\sigma}(G'))$ for every $G,G'\in\mathcal G_N$. To this end, note that 
    \[
    \begin{aligned}
        \langle \mathrm{A}_G\mathbf 1, \mathrm{A}_{G'}\mathbf 1\rangle &= \sum_{i=1}^{N}\left( \sum_{j=1}^{N} (\mathrm{A}_G)_{ij}\sum_{j=1}^{N} (\mathrm{A}_{G'})_{ij}\right) 
        \\
        &= \sum_{i=1}^{N}\left( \sum_{j=1}^{N} (\mathrm{A}_G)_{\sigma(i)\sigma(j)}\sum_{j=1}^{N} (\mathrm{A}_{G'})_{\sigma(i)\sigma(j)}\right)= \langle \mathrm{A}_{P_{\sigma}(G)}\mathbf 1, \mathrm{A}_{P_{\sigma}(G')}\mathbf 1\rangle        
    \end{aligned} 
    \]
   and that 
   \[
        \langle \mathrm{A}_G, \mathrm{A}_{G'} \rangle_{\mathrm{F}} = \sum_{i,j=1}^{N} \left(\mathrm{A}_G\right)_{ij} \left(\mathrm{A}_{G'}\right)_{ij} = \sum_{i,j=1}^{N} \left(\mathrm{A}_G\right)_{\sigma(i)\sigma(j)} \left(\mathrm{A}_{G'}\right)_{\sigma(i)\sigma(j)} = \langle \mathrm{A}_{P_{\sigma}(G)}, \mathrm{A}_{P_{\sigma}(G')} \rangle_{\mathrm{F}},   
   \]  
   whereby $\kappa(G,G')=\kappa(P_{\sigma}(G),P_{\sigma}(G'))$ for each choice of $G,G'\in \mathcal G_N$ and so the coupling $\pi=(\Id,P_{\sigma})_{\sharp}\mu_0$ yields a  cost of $0$ for the GW problem, proving the claim.
\qed

\subsection{Proof of \texorpdfstring{\cref{lem:kappaFormulas}}{Lemma 14}}
\label{proof:lem:kappaFormulas}
  We begin with the first assertion. Since $G$ has $m$ edges, $\langle \mathrm{A}_G,\mathrm{A}_{G}\rangle_{\mathrm{F}} = 2m$, whereas  
    \[
\langle \mathrm{A}_G\mathbf 1,\mathrm{A}_{G}\mathbf 1\rangle = \sum_{k=1}^N \mathrm{deg}(v_k(G))^2\geq 2,
    \]
    where the inequality is saturated
    if and only if $G$ only has one edge in which case two vertices have degree $1$. Conclude that $\kappa(G,G)\geq m +4$ with equality if and only if $G$ has a single edge.

    As for the formula \eqref{eq:kappaComp}, $\langle \mathrm{A}_G,\mathrm{A}_{G_{ij}}\rangle_{\mathrm{F}} = 2$ if $G$ has an edge between the $i$-th and $j$-th vertex and $\langle \mathrm{A}_G,\mathrm{A}_{G_{ij}}\rangle_{\mathrm{F}} =0$ otherwise. Furthermore, 
          \[
\langle \mathrm{A}_G\mathbf 1,\mathrm{A}_{G_{ij}}\mathbf 1\rangle = \sum_{k=1}^N \mathrm{deg}(v_k(G)) \mathrm{deg}(v_k(G_{ij}))= \mathrm{deg}(v_i(G))+\mathrm{deg}(v_j(G)),
    \]
    proving the claimed formula.
\qed

\subsection{Proof of \texorpdfstring{\cref{lem:pushforwardB}}{Lemma 15}}
\label{proof:lem:pushforwardB}
Fix $G,G' \in \supp(\mu_0)$ for which $G\neq G'$. Then, there exists some $(i,j)\in[N]\times [N]$ for which  $(\mathrm{A}_G)_{ij}=1$ and $(\mathrm{A}_{G'})_{ij}=0$ (up to interchanging $G$ and $G'$). It follows from \eqref{eq:kappaComp} that $\kappa(G,G_{ij})$ is odd whereas $\kappa(G',G_{ij})$ is even so that $\kappa(G,G_{ij})\neq\kappa(G',G_{ij})$. Since it is assumed that $\supp(\mu_0)$ contains all graphs with a single edge, $\kappa(G,\cdot)\neq \kappa(G',\cdot)$ as functions on $\supp(\mu_0)$. By the same argument, this condition is met with $\mu_1$ in place of $\mu_0$ so that the claimed result follows from \cref{lem:vanishingGW}. 
\qed

\subsection{Proof of \texorpdfstring{\cref{lem:permutationEdges}}{Lemma 16}}
\label{proof:lem:permutationEdges}
   We first establish that $B$ must map graphs with a single edge to graphs of a single edge. For any $i,j\in[N]\times [N]$ with $i< j$, we have that $\kappa(G_{ij},G_{ij})=\kappa(B(G_{ij}),B(G_{ij}))$ by definition of $B$, noting that both $\mu_0$ and $\mu_1$ assign positive probability to all graphs with a single edge. Recall from \cref{lem:kappaFormulas} that $\kappa(G_{ij},G_{ij})=5$ and that graphs with a single edge are the only ones with this particular value, completing this step. Importantly, it also follows that $B$ is a bijection on the set of all graphs with one edge. 

    We now show that two single edge graphs share a common vertex  if and only if their image through $B$ also satisfy this property. This is a direct consequence of   \cref{lem:kappaFormulas}, which asserts that 
    \[
        \kappa(G_{ij},G_{kl}) =\begin{cases}
           0,&\text{if } \{i,j\}\cap\{k,l\}=\emptyset,
           \\
           2,&\text{otherwise,}
        \end{cases}
    \]
    provided that $\{i,j\}\neq \{k,l\}$ along with the fact that $\kappa(G_{ij},G_{kl})=\kappa(B(G_{ij}),B(G_{kl}))$ and that $B(G_{ij}),B(G_{kl})$ are graphs with a single edge. 

Finally, we show the desired result;  $B(G_{ij})=G_{\sigma(i)\sigma(j)}$ for a fixed permutation $\sigma:[N]\to[N]$. We first prove the result in the case that $N\geq 5$.  Consider the images of $(G_{1j})_{j=2}^N$ under $B$. Since $G_{12}$ and $G_{13}$ share the vertex $1$, $B(G_{12})=G_{kl}$  and $B(G_{13})=G_{mn}$ also share a vertex by the previous step so that, say, $k=m$. Similarly, the image of $G_{14}$, $B(G_{14})=G_{op}$ must share a vertex with $G_{kl}$ and $G_{kn}$. There are only two possibilities (i) $o=k$ or $p=k$, or (ii) $o=l$ and $p=n$. Now consider the image of $G_{15}$, $B(G_{15})=G_{qr}$. In case (i), we must have that $q=k$ or $r=k$ so that the four graphs considered to this point share the common vertex $k$. In case (ii), we see that $G_{qr}$ cannot simultaneously share a vertex with $G_{kl},G_{kn},G_{ln}$, as that collection of three edges forms a triangle. Conclude that only setting (i) is consistent in the case $N\geq 5$. Following this logic, the image of the $N-1$ graphs $(G_{1j})_{j=2}^{N}$ under $B$ is given by $(G_{kj})_{\substack{j=1\\j\neq k}}^{N}$ and we can identify $\sigma(1)=k$.  Applying the same argument for the remaining vertices and using the fact that $B$ is a bijection yields that there exists a permutation $\sigma:[N]\to [N]$ for which $B(G_{ij})= G_{\sigma(i)\sigma(j)}$.

Note that the case $N=2$ is trivial. For $N=3$, we have as above that $B(G_{12})=G_{kl}$  and $B(G_{13})=G_{mn}$ share a vertex so that, say, $k=m$ so that we may identify $\sigma(1)=k$. Now,  $B(G_{12})=G_{kl}$  and $B(G_{23})=G_{op}$ and, in fact, $G_{op}=G_{ln}$ since $B$ is a bijection and there are only $3$ possible edges. Conclude that $\sigma(1)=k,\sigma(2)=l$, and $\sigma(3)=n$.  
\qed

\subsection{Proof of \texorpdfstring{\cref{lem:convergenceMinimizers}}{Lemma 17}}
\label{proof:lem:convergenceMinimizers}
Recall that the minimizers of $\ell_n$ are contained in $\rB_0\mathcal K_n$ whereas those of $\ell$ are contained in $\rB_0\mathcal K$ as follows from the final remark in  \cref{thm:variationalForm}. It follows that we can optimize $\ell_n,\ell$ over the closed ball $\mathbb B_{R_0}$ of radius $R_0\coloneqq \|\rB_0\|_{1,2}$ since every element of $\mathcal K_n$ and $\mathcal K$ has $1$-norm equal to $1$. 

For any $u,u'\in\mathbb B_{R_0}$, we have that 
\[
\begin{aligned}
    &\left|\ell_n(u)-\ell_n(u')\right|\\
    &=\left|\frac 12 \|u\|^2+\inf_{x\in\mathcal K_n}\left\{\frac 12 x^{\intercal}\rB_1^{\intercal}\rB_1x-u^{\intercal }\rB_0 x\right\}-\frac 12 \|u'\|^2-\inf_{x\in\mathcal K_n}\left\{\frac 12 x^{\intercal}\rB_1^{\intercal}\rB_1x-(u')^{\intercal }\rB_0 x\right\}\right|  
    \\
    &\leq\frac {1}{2}\left|\|u\|-\|u'\| \right|\left(\|u\|+\|u'\| \right)+\sup_{x\in\mathcal K_n}\left| x^{\intercal}\rB_0^{\intercal}(u-u')\right|\leq 2R_0\|u-u'\|,
\end{aligned}
\]
so that $\ell_n$ is $2R_0$-Lipschitz continuous on $\mathbb B_{R_0}$. It follows similarly that $\ell$ has the same Lipschitz modulus. Moreover, for any $u\in\mathbb B_{R_0}$,  
\[
\begin{aligned}
    \left|\ell_n(u)-\ell(u)\right|&\leq \sup_{v\in\mathbb B_{R_1}}\left| \mathsf{OT}_{\rB_1^{\intercal}v-\rB_0^{\intercal}u}(\mu_{0,n},\mu_{1,n})-\mathsf{OT}_{\rB_1^{\intercal}v-\rB_0^{\intercal}u}(\mu_{0},\mu_{1})\right|\\&\leq C_{\rB_0,\rB_1}\left(\|w_{\mu_{0,n}}-w_{\mu_{0}}\|_1+\|w_{\mu_{1,n}}-w_{\mu_{1}}\|_1\right),  
\end{aligned}
\]
recalling \cref{lem:OTLipschitz}. Conclude that
t$\ell_n\to \ell$ pointwise on $\mathbb B_{R_0}$, noting that the right hand side of the above display converges to $0$ as $\mu_{0,n},\mu_{1,n}$ converge weakly to $\mu_0,\mu_1$. It follows from Corollary 1.8.6 in \cite{bogachev2020real} that $\ell_n$ converges uniformly to $\ell$ on $\mathbb B_{R_0}$. Setting $\ell'_n=\ell_n+\mathcal I_{\mathbb B_{R_0}}$ and  $\ell'=\ell+\mathcal I_{\mathbb B_{R_0}}$ (recall that $\mathcal I_{\mathbb B_{R_0}}$ is the indicator function of $\mathbb B_{R_0}$ as defined in \eqref{eq:indicator}),  Proposition 7.15 (b) and Theorem 7.33 in \cite{rockafellar1998variational} assert that any cluster point of minimizers of $\ell'_n$ is a minimizer of $\ell'$ and that the optimal values converge; the same implications hold, therefore, for $\ell_n$ and $\ell$ since their respective minimizers all lie in $\mathbb B_{R_0}$.  
\qed

\end{document}